\documentclass[11pt,a4paper,reqno]{amsart}
\usepackage[english]{babel}
\usepackage[T1]{fontenc}
\usepackage{verbatim}
\usepackage[shortlabels]{enumitem}
\usepackage{palatino}
\usepackage{amsmath}
\usepackage{mathabx}
\usepackage{amssymb}
\usepackage{mathrsfs}
\usepackage{amsthm}
\usepackage{amsfonts}
\usepackage{graphicx}
\usepackage{esint}
\usepackage{color}
\usepackage{mathtools}
\usepackage{overpic}
\usepackage{tikz}

\usepackage[colorlinks = true, citecolor = black]{hyperref}
\author{Tuomas Orponen, Aleksi Py\"or\"al\"a, and Guangzeng Yi}
\title{Furstenberg set theorem for transversal families of functions}
\address{Department of Mathematics and Statistics\\ University of Jyv\"askyl\"a,
	P.O. Box 35 (MaD)\\
	FI-40014 University of Jyv\"askyl\"a\\
	Finland}
\email{tuomas.t.orponen@jyu.fi} \email{aleksi.v.pyorala@jyu.fi} \email{guangzeng.m.yi@jyu.fi}
\date{\today}
\thanks{T.O. and G.Y. are supported by the European Research Council (ERC) under the European Union’s Horizon Europe research and innovation programme (grant agreement No 101087499). T.O. and A.P. are supported by the Research Council of Finland via the project \emph{Approximate incidence geometry}, grant no. 355453. A.P. has also been supported by the Research Council of Finland grant no. 321696.}

\newcommand{\R}{\mathbb{R}}

\newcommand{\N}{\mathbb{N}}

\newcommand{\Z}{\mathbb{Z}}

\newcommand{\spt}{\operatorname{spt}}
\newcommand{\Hd}{\dim_{\mathrm{H}}}

\newcommand{\diam}{\operatorname{diam}}

\newcommand{\dist}{\operatorname{dist}}

\newcommand{\m}{\mathfrak{m}}

\def\Barint_#1{\mathchoice
	{\mathop{\vrule width 6pt height 3 pt depth -2.5pt
			\kern -8pt \intop}\nolimits_{#1}}%
	{\mathop{\vrule width 5pt height 3 pt depth -2.6pt
			\kern -6pt \intop}\nolimits_{#1}}%
	{\mathop{\vrule width 5pt height 3 pt depth -2.6pt
			\kern -6pt \intop}\nolimits_{#1}}%
	{\mathop{\vrule width 5pt height 3 pt depth -2.6pt
			\kern -6pt \intop}\nolimits_{#1}}}

\numberwithin{equation}{section}

\theoremstyle{plain}
\newtheorem{thm}[equation]{Theorem}
\newtheorem{"thm"}[equation]{"Theorem"}

\newtheorem{lemma}[equation]{Lemma}
\newtheorem{"lemma"}[equation]{"Lemma"}

\newtheorem{ex}[equation]{Example}
\newtheorem{cor}[equation]{Corollary}
\newtheorem{"proposition"}[equation]{"Proposition"}
\newtheorem{proposition}[equation]{Proposition}

\newtheorem{claim}[equation]{Claim}

\theoremstyle{definition}

\newtheorem{definition}[equation]{Definition}

\newtheorem{notation}[equation]{Notation}

\theoremstyle{remark}
\newtheorem{remark}[equation]{Remark}
\newtheorem{terminology}[equation]{Terminology}

\newcommand{\nref}[1]{(\hyperref[#1]{#1})}

\DeclareMathSymbol{\intop} {\mathop}{mathx}{"B3}

\begin{document}
	
\begin{abstract} We prove an extension of the Furstenberg set theorem to families of graphs satisfying a transversality condition. We apply the result to derive bounds on $L^{p}$-norms of Fourier transforms of fractal measures supported on plane curves. 
  \end{abstract}
	
\maketitle
	
\tableofcontents

\section{Introduction}

The purpose of this paper is to prove a curvilinear generalisation of the \emph{Furstenberg set theorem} \cite{2023arXiv230808819R}. We will assume that the curves satisfy a "transversality" condition, discussed below. We first state the original "linear" result to set the scene. Given $s \in [0,1]$ and $t \in [0,2]$, a set $F \subset \R^{2}$ is called an \emph{$(s,t)$-Furstenberg set} if there exists a line family $\mathcal{L}$ with $\Hd \mathcal{L} \geq t$ such that $\Hd (F \cap \ell) \geq s$ for all $\ell \in \mathcal{L}$.

\begin{thm}[Ren-Wang '23]\label{t:renWang} Let $s \in (0,1]$ and $t \in [0,2]$. Every $(s,t)$-Furstenberg set $F \subset \R^{2}$ satisfies 
\begin{displaymath} \Hd F \geq \min\left\{s + t, \tfrac{3s + t}{2},s + 1 \right\}. \end{displaymath}
\end{thm} 

Theorem \ref{t:renWang} has already established its value in making progress in many related problems in fractal geometry and harmonic analysis, see \cite{csornyei2025improvedboundsradialprojections,MR4745881,MR4869897,2023arXiv230903068O,Orponen2024Jan,2024arXiv241108871W}. This wide applicability of Theorem \ref{t:renWang} was our motivation for pursuing a curvilinear generalisation. The specific application we had in mind was to extend the main results in \cite{Orponen2024Jan} to more general convex curves than the parabola; this will be discussed in Section \ref{s:fourierIntro}.

We then introduce the notion of "transversality" used in the paper:

\begin{definition}[Transversal family]\label{def:transversalityIntro} Let $I \subset \R$ be a compact interval, and let $\mathcal{F} \subset C^{2}(I)$. We say that $\mathcal{F}$ is a \emph{transversal family over $I$ with transversality constant $\mathfrak{T} \geq 1$}, if
\begin{equation}\label{form27}
\inf_{\theta\in I} (|f(x)- g(x)| + |f'(x) - g'(x)|) \geq \mathfrak{T}^{-1}\|f - g\|_{C^{2}(I)}, \quad f,g\in \mathcal{F}.
\end{equation}
Here $\|f\|_{C^2(I)}:=\max_{x\in I}\sum_{k=0}^2|f(x)|$.
\end{definition}

\begin{remark} The notion of a transversal family is similar to the notion of a \emph{cinematic family} introduced by Pramanik, Yang, and Zahl \cite[Definition 1.6]{PYZ22}, with the crucial difference that \cite{PYZ22} adds the term $|f''(\theta) - g''(\theta)|$ to the sum in \eqref{form27}. This difference means that the graphs of functions from a cinematic family may intersect "tangentially", whereas Definition \ref{def:transversalityIntro} rules this out. It would be interesting to know if the results below (perhaps with a different numerology) hold for cinematic families. Existing partial results in this direction can be found in \cite{MR4844556,2025arXiv250210686G,2023arXiv230705894Z}.  \end{remark}

We next give two examples of transversal families.
\begin{ex}\label{ex1} For $(a,b) \in \R^{2}$, write $f_{(a,b)}(x) := ax + b$, $x \in \R$. Then, for every $C > 0$, the family of functions $\{f_{(a,b)} : |a| \leq C\}$ is transversal on every compact interval $I \subset \R$. Thus, more informally, "lines form a transversal family". We leave the details to the reader (this case is not explicitly used in the paper, and the details are very similar to the second example right below).

As the second example, we will see (informally speaking) that "translates of a convex function form a transversal family". With an application (Corollary \ref{cor:convex}) in mind, we need to quantify this very precisely. We give the details right away, but the reader is encouraged to skip the rest of the example on first reading. 

Let $g \in C^{3}(\R)$ be a function satisfying $g''(x) > 0$ for all $x \in \R$. For $z = (a,b) \in \R^{2}$, let $g_{z}(x) := g(x - a) + b$. Thus, if $\Gamma_{g}$ denotes the graph of $g$, we have $\Gamma_{g_{z}} = z + \Gamma_{g}$.  Let $x_{0} \in \R$, and let $I = [x_{0} - 1,x_{0} + 1] \subset \R$. We claim that the family $\{g_{z} : z \in I \times \R\}$ is transversal on $5 I := [x_{0} - 5,x_{0} + 5]$ with transversality constant depending only on the following quantities:
\begin{displaymath} \mathfrak{G} := \|g\|_{C^{3}([-6,6])} \quad \text{and} \quad \mathfrak{g} := \min \{g''(x) : x \in [-6,6]\} > 0. \end{displaymath}
Crucially, the transversality constant on $5I$ is independent of $x_{0}$.

To see this, we first establish that if $z_{1},z_{2} \in I \times \R$, then
\begin{equation}\label{form28} \|g_{z_{1}} - g_{z_{2}}\|_{C^{2}(5 I)} \sim_{\mathfrak{g},\mathfrak{G}} |z_{1} - z_{2}|. \end{equation}
For the upper bound, fix $z_{j} = (a_{j},b_{j}) \in I \times \R$ and $x \in 5 I$. Then, noting that $x - a_{1},x - a_{2} \in [-6,6]$, and using the mean value theorem,
\begin{displaymath} |g_{z_{1}}(x) - g_{z_{2}}(x)| \leq |b_{1} - b_{2}| + |a_{1} - a_{2}|\|g'\|_{L^{\infty}([-6,6])}.  \end{displaymath}
Similarly $\|g_{z_{1}}^{(i)} - g_{z_{2}}^{(i)}\|_{L^{\infty}(5 I)} \leq |a_{1} - a_{2}|\|g^{(i + 1)}\|_{L^{\infty}([-6,6])}$ for $i \in \{1,2\}$. 

Concerning the lower bound, we first assume that that $|b_{1} - b_{2}| \leq 2\mathfrak{G}|a_{1} - a_{2}|$. It suffices to show that $|a_{1} - a_{2}| \lesssim_{\mathfrak{G},\mathfrak{g}} \|g_{z_{1}} - g_{z_{2}}\|_{C^{2}(5I)}$. This follows from the mean value theorem, and the strict positivity of $g''$: for every $x \in 5 I$, noting again that $x - a_{1},x - a_{2} \in [-6,6]$,
\begin{equation}\label{form29} \|g_{z_{1}} - g_{z_{2}}\|_{C^{2}(5I)} \geq |g'(x - a_{1}) - g'(x - a_{2})| \geq \mathfrak{g} \cdot |a_{1} - a_{2}|. \end{equation}
Assume next that $|b_{1} - b_{2}| \geq 2\mathfrak{G}|a_{1} - a_{2}|$. Then, for every $x \in 5 I$,
\begin{equation}\label{form30} \|g_{z_{1}} - g_{z_{2}}\|_{C^{2}(5I)} \geq |g_{z_{1}}(x) - g_{z_{2}}(x)| \geq |b_{1} - b_{2}| - \mathfrak{G} \cdot |a_{1} - a_{2}| \geq \tfrac{1}{2}|b_{1} - b_{2}|, \end{equation}
using the mean value theorem. This completes the proof of \eqref{form28}. 

We have reduced the proof of the transversality condition \eqref{form27} to verifying the inequality $|g_{z_{1}}(x) - g_{z_{2}}(x)| + |g_{z_{1}}'(x) - g_{z_{2}}'(x)| \gtrsim_{\mathfrak{g},\mathfrak{G}} |z_{1} - z_{2}|$ for all $x \in 5 I$. This follows from the intermediate estimates above. If $|b_{1} - b_{2}| \leq 2\mathfrak{G}|a_{1} - a_{2}|$, use \eqref{form29}. In the opposite case use \eqref{form30}. \end{ex}

Here is the main result of the paper, a transversal generalisation of Theorem \ref{t:renWang}:

\begin{thm}\label{t:mainIntro} Let $s \in (0,1]$ and $t \in [0,2]$. Let $I \subset \R$ be a compact interval, and let $\mathcal{F} \subset C^{2}(I)$ be a transversal family over $I$ with $\Hd \mathcal{F} \geq t$. Let $F \subset \R^{2}$ be a set satisfying $\Hd (F \cap \Gamma_{f}) \geq s$ for all $f \in \mathcal{F}$, where $\Gamma_{f} = \{(x,f(x)) : x \in I\}$ is the graph of $f$. Then,
\begin{displaymath} \Hd F \geq \min\left\{s + t, \tfrac{3s + t}{2},s + 1 \right\}. \end{displaymath} \end{thm} 

\begin{remark} Given the transversal families mentioned in Example \ref{ex1}, it may seem strange that we assume the property $\Hd (F \cap \Gamma_{f}) \geq s$ for all $f \in \mathcal{F}$, and not just some $t$-dimensional subset of $\mathcal{F}$. However, note that if $\mathcal{F}$ is one of the families in Example \ref{ex1}, any $t$-dimensional subset of $\mathcal{F}$ remains transversal (with the same constant). \end{remark} 

\subsection{Proof outline and structure of the paper} Our proof of Theorem \ref{t:mainIntro} follows Ren and Wang's proof of Theorem \ref{t:renWang}. We first establish a special case of Theorem \ref{t:renWang}, where the transversal family $\mathcal{F}$ is (almost) Ahlfors regular, see Theorem \ref{thm-furstenbergset}. This result is a transversal generalisation of \cite[Theorem 5.7]{2023arXiv230110199O}, due to the first author and Shmerkin. As in \cite{2023arXiv230110199O}, Theorem \ref{thm-furstenbergset} is derived as a (non-trivial) corollary of its own special case, Theorem \ref{proj-regular}, which we call a "projection theorem". While there are no projections visible in Theorem \ref{proj-regular}, in the linear case Theorem \ref{proj-regular} would be equivalent -- via point-line duality -- to an honest projection theorem, stated in \cite[Corollary 4.9]{2023arXiv230110199O}. In the generality of transversal families, one could define the "generalised projections" $\Pi_{\theta} \colon \mathcal{F} \to \R$ by $\Pi_{\theta}(f) := f(\theta)$; with this notation Theorem \ref{proj-regular} could be stated in the format of \cite[Corollary 4.9]{2023arXiv230110199O}.

The Ahlfors regular special case was used as a black box in \cite{2023arXiv230808819R} to prove the general case of Theorem \ref{t:mainIntro}. In fact, one of the main innovations in \cite{2023arXiv230808819R} was to identify and establish another "extreme" case -- opposite to Ahlfors regularity -- where the line family $\mathcal{L}$ (here the transversal family $\mathcal{F}$) is \emph{semi-well spaced}. Our counterpart of the semi-well spaced case of Theorem \ref{t:mainIntro} is Proposition \ref{pro-semi}. The proof of the semi-well spaced case in \cite{2023arXiv230808819R} relied on a new incidence theorem for (straight) $\delta$-tubes and $\delta$-balls, \cite[Lemma 4.8]{2023arXiv230808819R}. One of the most technical parts of the current paper is to extend \cite[Lemma 4.8]{2023arXiv230808819R} to transversal families, see Lemma \ref{mainlem} for the statement, and Appendix \ref{appA} for the proof.

Another big innovation in \cite{2023arXiv230808819R} was to prove that (up to "refinements") \textbf{every} set is an interpolation between the almost Ahlfors regular and semi-well spaced special cases. This result is so abstract that we can use it here as a black box (modulo very small details), see Proposition \ref{pro:Lip} and Lemma \ref{lem:characterization}.

The reader may wish to know if there are significant difficulties to overcome when generalising the contents of \cite{2023arXiv230110199O,2023arXiv230808819R} to transversal families. The short answer is negative, once we have the right notion of transversality. Both \cite{2023arXiv230110199O} and \cite{2023arXiv230808819R} use rescaling (and "induction on scales") arguments, where thin tubes get mapped to thicker tubes, and "tubelets" (small pieces of tubes) get mapped to squares. Such arguments seemed \emph{a priori} difficult to generalise. However, the present notion of transversality is nicely invariant under two different kinds of relevant rescaling operations (see Lemmas \ref{lem1} and \ref{lem2}). By contrast, our first attempt was to prove Theorem \ref{t:mainIntro} only in the special (described in Example \ref{ex1}) of translates of a fixed convex graph. This was unsuccessful: the necessary rescaling operations in this special case would send the associated family of functions to another family which no longer has the "original" form, preventing the use of induction-on-scales arguments.

Besides finding a scaling invariant notion of transversality, another non-trivial component of the proof is to find a curvilinear version of the \emph{high-low lemma} (see Proposition \ref{p:highLowUnweighted}). The proof of the original version (for tubes and balls), due to Guth, Solomon, Wang \cite[Proposition 2.1]{GSW}, uses the fact that the Fourier transform of a tube in $\R^{2}$ looks roughly like a tube. We do now know of any useful counterpart of this phenomenon when straight tubes are replaced by $\delta$-neighbourhoods of graphs from a transversal family. Fortunately, there exists an alternative proof of the high-low lemma, due to Cohen, Pohoata, and Zakharov \cite[Theorem 3.1]{2023arXiv230518253C} (although these authors also give credit to Roth for discovering a similar argument in the 70s). This argument is Fourier-analysis free and extends (with some effort) to the curvilinear case.

There is only one (relatively minor) "shortcut" in the argument here, compared to a combination of \cite{2023arXiv230110199O,2023arXiv230808819R}. Namely, the almost Ahlfors regular case of Theorem \ref{t:renWang} treated in \cite{2023arXiv230110199O} reduces matters to something called the \emph{$ABC$ theorem}, which is established in \cite{2023arXiv230110199O} (although relying on \cite{MR4778059,MR4706445}). In retrospect, the $ABC$ theorem is a special case of Theorem \ref{t:renWang} (for straight lines), and it is technically easier to apply Theorem \ref{t:renWang} directly than the $ABC$ theorem. To save a few pages, although at the cost of making our paper (even) less self-contained, we decided to apply Theorem \ref{t:renWang} in the part of the argument where \cite{2023arXiv230110199O} would rely on the $ABC$ theorem. For more details, see Section \ref{s:outline}.

The outline above explains what happens in Sections \ref{sec3}-\ref{sec7}, and Appendix \ref{appA}. Additionally, Section \ref{sec2} contains preliminaries, and Section \ref{s8} contain the details of one application of our curvilinear Furstenberg set theorem, introduced in Section \ref{s:fourierIntro}.

\subsection{Discretised versions of Theorem \ref{t:mainIntro}} The proof of Theorem \ref{t:mainIntro} proceeds via a $\delta$-discretised version, stated below as Theorem \ref{t:mainDiscretised}. For the terminology, see Section \ref{sec2}. 
\begin{thm}\label{t:mainDiscretised} Let $s \in (0,1]$ and $t \in [0,2]$. Then, for every $\mathfrak{T},\eta > 0$, there exist $\epsilon,\delta_{0} > 0$ (depending only on $s,t,\mathfrak{T},\eta$) such that the following holds for all $\delta \in 2^{-\N} \cap (0,\delta_{0}]$.

Let $\mathcal{F} \subset B_{C^{2}}(1)$ be a non-empty transversal family over $[-2,2]$ with constant $\mathfrak{T}$. Assume that $\mathcal{F}$ is a $(\delta,t,\delta^{-\epsilon})$-set. For each $f \in \mathcal{F}$, assume that $\mathcal{P}(f)$ is a non-empty $(\delta,s,\delta^{-\epsilon})$-set of dyadic $\delta$-squares which intersect the graph $\Gamma_{f}$, are contained in $[-1,1]^{2}$, and satisfy $|\mathcal{P}(f)| \geq M$ for some $M \in \N$ independent of $f$. Let $\mathcal{P}$ be the union of the families $\mathcal{P}(f)$. Then,
\begin{equation}\label{form35} |\mathcal{P}| \geq \delta^{\eta} \cdot \min\{\delta^{-t},\delta^{-(s + t)/2},\delta^{-1}\} \cdot M.\end{equation} \end{thm}

Theorem \ref{t:mainDiscretised} can be pieced together from slightly sharper statements in Corollary \ref{cor3} (the easiest range $t \leq s$), Corollary \ref{cor-uppercaseF} (the range $t \geq 2 - s$, in which Theorem \ref{t:renWang} was already proved by Fu and Ren \cite{MR4751206} in 2021) and finally Theorem \ref{thm-generalcase}, which covers most complex range where $s \in (0,1)$ and $t \in (s,2 - s)$. The reader may notice that the three results are formulated in terms of "nice configurations" (Definition \ref{def-niceconfiguration}) where one assumes $|\mathcal{P}(f)| = M$ instead of $|\mathcal{P}(f)| \geq M$ (as in Theorem \ref{t:mainDiscretised}). However, the hypothesis $|\mathcal{P}(f)| \geq M$, combined with easy pigeonholing, allows one to find the nice configurations to which Corollaries \ref{cor3}, \ref{cor-uppercaseF}, and Theorem \ref{thm-generalcase} may be applied.

\subsection{From Theorem \ref{t:mainDiscretised} to Theorem \ref{t:mainIntro}}

We will next describe how the statement for Hausdorff dimension of Furstenberg sets is derived from the discretised result. The proof is virtually the same as \cite[Proof of Theorem 2]{MR4844556}. Let $s\in (0,1]$, $t\in[0,2]$ and $\mathfrak T \geq 1$. Let $I\subset \R$ be a compact interval, let $\mathcal F_0 \subset C^2(I)$ be a transversal family over $I$ with constant $\mathfrak{T}$ and $\Hd \mathcal F_{0} \geq t$, and let $F_0 \subset \R^2$ be a set satisfying $\Hd (F_0\cap \Gamma_f)\geq s$ for each $f\in\mathcal F_0$. Since the claim is trivial for $t=0$, suppose that $t>0$. Fix numbers $0<s'<s$ and $0 <t' <t$, and let $u = \min\lbrace s'+t', \frac{3s'+t'}{2}, s'+1\rbrace$. Let $\eta>0$, and apply Theorem \ref{t:mainDiscretised} with parameters $s', t', \mathfrak{T}$ and $\eta$ to get constants $\varepsilon,\delta_0 > 0$. 

By assumption, $\mathcal H_\infty^{t'}(\mathcal F_0) > 0$ and $\mathcal H_\infty^{s'}(F_0 \cap \Gamma_f)>0$ for every $f\in\mathcal F_{0}$, so by countable additivity of Hausdorff content there exists $\alpha>0$ such that $\mathcal H_\infty^{t'}(\mathcal F_1) > \alpha$, where 
\begin{equation*}
    \mathcal F_1 = \mathcal F_1(\alpha) = \lbrace f \in \mathcal F:\ \mathcal H_\infty^{s'}(F_0\cap \Gamma_f) > \alpha\rbrace.
\end{equation*}
We take $k_0 = k_0(\alpha,\varepsilon, \delta_0,\eta)\in\N$ large enough so that 
\begin{equation}\label{eq-k0}
\sum_{k\geq k_0} \frac{1}{k^2} < \alpha,\qquad k_0^2 \leq \min\lbrace 2^{\varepsilon k_0}/C, 2^{\eta k_0}/C\rbrace,\qquad 2^{-k_0}< \delta_0,
\end{equation}
where $C$ is a constant depending only on $\mathfrak T$ and which will be determined later. Note that \eqref{eq-k0} also holds when $k_0$ is replaced with any larger integer.

Let $\mathcal P_0$ be an arbitrary cover of $F_0$ by dyadic squares of side length at most $2^{-k_0}$.

For $k\geq k_0$, let $\mathcal P_k$ be the collection of those $p \in \mathcal P_0$ with side length in the interval $[2^{-(k+1)} ,2^{-k})$. Then, by \eqref{eq-k0} and the pigeonhole principle, for each $f\in\mathcal F_1$ there exists $k(f)\geq k_0$ such that 
\begin{equation*}
\mathcal H_\infty^{s'}(\mathcal P_{k(f)} \cap \Gamma_f) > k(f)^{-2}.
\end{equation*}
In this section we write $\mathcal P_{k(f)} \cap \Gamma_f := \bigcup_{p\in\mathcal P_{k(f)}} p \cap \Gamma_f$. Recalling that $\mathcal H_\infty^{t'}(\mathcal F_1) > \alpha$ and using again the pigeonhole principle, there exists $k_1\geq k_0$ such that $\mathcal H_\infty^{t'}(\mathcal F_2) > k_1^{-2}$, where $\mathcal F_2 = \lbrace f\in\mathcal F_1:\ k(f) = k_1\rbrace$. Recapping, we have
\begin{equation*}
    \mathcal H_\infty^{s'} (\mathcal P_{k_1} \cap \Gamma_f) > k_1^{-2},\qquad f\in \mathcal F_2.
\end{equation*}

Fix $\delta = 2^{-k_1} \leq \delta_0$. We would now like to apply the discrete version of Frostman's lemma, \cite[Lemma 3.13]{FaO}, to $\mathcal F_2$. There is a small catch; \cite[Lemma 3.13]{FaO} is proven for subsets of $\R^3$ (the proof works for subsets of $\R^d$) while we would like to apply it for $\mathcal F_2\subset C^2(I)$. However, $\mathcal F_2$ is a subset of the transversal family $\mathcal F_0$, and by Lemma \ref{lem-Ahlforsregularity} $\mathcal F_2$ admits a $\sqrt{2}\mathfrak{T}$-bi-Lipschitz embedding into $\R^2$. Passing back and forth with this embedding, we may apply \cite[Lemma 3.13]{FaO} in the plane with the cost of letting the implicit constant in \cite[Lemma 3.13]{FaO} depend also on $\mathfrak T$. However, this is harmless. 

Let $\mathcal F \subset \mathcal F_2$ be the $(\delta, t', Ck_1^2)$-set given by \cite[Lemma 3.13]{FaO}. This $C = C(\mathfrak{T})$ is the constant present in \eqref{eq-k0}. Recalling that $Ck_1^2 \leq 2^{\varepsilon k_1} \leq \delta^{-\varepsilon}$, $\mathcal F$ is also a $(\delta, t', \delta^{-\varepsilon})$-set. Applying \cite[Lemma 3.13]{FaO} also to $\mathcal P_{k_1}\cap \Gamma_{f}$ for each $f\in\mathcal F$, we obtain $(\delta,s', C k_1^2)$-sets $\mathcal P(f) \subset \mathcal P_{k_1}\cap \Gamma_{f}$ with cardinality $|\mathcal P(f)| =: M \geq (k_1^{-2}/C)\delta^{-s'}$. Replace $\mathcal P(f)$ by $\mathcal D_\delta(\mathcal P(f))$ without changing notation. Recalling again \eqref{eq-k0}, $\mathcal P(f)$ is a $(\delta, s', \delta^{-\varepsilon})$-set with $M = |\mathcal P(f)| \geq \delta^{\eta - s'}$. We are now in position to apply Theorem \ref{t:mainDiscretised} to $\mathcal P = \bigcup_{f\in\mathcal F}\mathcal P(f)$.

By Theorem \ref{t:mainDiscretised}, $|\mathcal P| \geq \delta^{\eta}\cdot \min\lbrace \delta^{-t'}, \delta^{-(t'+s')/2}, \delta^{-1}\rbrace\cdot M \geq \delta^{2\eta - u}.$ Recall that by construction, $\mathcal P$ is a subset of $\mathcal P_0$, the original cover of $F_0$ by dyadic squares. We have
\begin{align*}
    \sum_{p\in\mathcal P_0} \diam(p)^{u-2\eta} \geq |\mathcal P| \delta^{u-2\eta} \geq 1.
\end{align*}
As the cover $\mathcal P_0$ was arbitrary, we infer that $\dim F_0 \geq u - 2\eta$. Taking $s' \to s$, $t'\to t$ and $\eta\to 0$ gives Theorem \ref{t:mainIntro}. 

\subsection{Translates of a convex curve}

We have seen in Example \ref{ex1} that translates of a fixed convex $C^{3}$-function form a transversal family. The corresponding special case of Theorem \ref{t:mainDiscretised} is recorded as Corollary \ref{cor:convex} below. For $g \colon [-6,6] \to \R$ and a square $q \subset \R^{2}$, we denote by $q + \Gamma_{g}^{1}$ the graph 
\begin{displaymath} \Gamma_{g}^{1} = \{(x,g(x)) : x \in [-1,1]\} \end{displaymath}
(the restriction to $[-1,1]$ is intentional) translated by the centre $c_{q} \in q$.

\begin{cor}\label{cor:convex}Let $s \in (0,1]$, $t \in [0,2]$, $R \geq 1$, and let $g \in C^{3}([-6,6])$ be a function whose second derivative $g''$ never vanishes. Then, for every $\eta > 0$, there exist $\epsilon = \epsilon(g,\eta,s,t) > 0$ and $\delta_{0} = \delta_{0}(\epsilon,\eta,g,s,t,R) > 0$ such that the following holds for all $\delta \in 2^{-\N} \cap (0,\delta_{0}]$.

Let $\mathcal{Q}$ be a non-empty $(\delta,t,\delta^{-\epsilon})$-set of dyadic $\delta$-squares contained in $[-R,R]^{2}$. For each $q \in \mathcal{Q}$, assume that $\mathcal{P}(q)$ is a non-empty $(\delta,s,\delta^{-\epsilon})$-set of dyadic $\delta$-squares, all of which intersect $q + \Gamma_{g}^{1}$, and satisfy $|\mathcal{P}(q)| \geq M$ for some $M \in \N$ independent of $f$. Let $\mathcal{P}$ be the union of the families $\mathcal{P}(q)$. Then,
\begin{equation}\label{form36} |\mathcal{P}| \geq \delta^{\eta} \cdot \min\{\delta^{-t},\delta^{-(s + t)/2},\delta^{-1}\} \cdot M.\end{equation}
\end{cor}

\begin{proof} This nearly follows from Theorem \ref{t:mainDiscretised}. The only catch is that the parameter $\epsilon > 0$ is not allowed to depend on $R$ (this will be important for the application in Theorem \ref{t:smoothing} below). If $\epsilon > 0$ was allowed to depend on $R$, we could just dilate all the objects in the proof by $O(1/R)$, replacing $g$ by $g_{R}(x) = R^{-1}g(Rx)$ and each $\mathcal{P}(q)$ by $\mathcal{P}_{R}(q) := R^{-1}\mathcal{P}(q) \subset [-1,1]^{2}$. Now Theorem \ref{t:mainDiscretised} could be applied. However, the $C^{2}$-norm of $g_{R}$ increases linearly in $R$, so the transversality constant of the family of translates of $g_{R}$ increases as a function of $R$. Given that "$\epsilon$" in Theorem \ref{t:mainDiscretised} depends on the transversality constant, this would end up producing an $R$-dependence to $\epsilon > 0$ in Corollary \ref{cor:convex}.

We fix the problem by initially selecting a square 
\begin{displaymath} Q_{0} := [x_{0} - 1,x_{0} + 1] \times [y_{0} - 1,y_{0} + 1] \subset [-R,R]^{2} \end{displaymath}
such that $|\mathcal{Q} \cap Q_{0}| \gtrsim |\mathcal{Q}|/R^{2}$. Since $\mathcal{Q}$ is a $(\delta,t,\delta^{-\epsilon})$-set, $Q \cap \mathcal{Q}_{0}$ is a $(\delta,t,\delta^{-2\epsilon})$-set, provided that $\delta > 0$ is sufficiently small in terms of $\epsilon$ and $R$ (as it may be). For every $q \in \mathcal{Q} \cap Q_{0}$, the set $\mathcal{P}(q)$ is contained in $[x_{0} - 2,x_{0} + 2] \times [y_{0} - \Lambda,y_{0} + \Lambda]$, where $\Lambda \sim 1 + \|g\|_{L^{\infty}([-1,1])}$.

Now, consider the family of functions associated with the translates $q + \Gamma_{g}$, for $q \in \mathcal{Q} \cap Q_{0}$. Formally,
\begin{displaymath} \mathcal{F}_{0} := \{x \mapsto g(x - a) - b : (a,b) \text{ is a centre of some } q \in \mathcal{Q} \cap Q_{0}\}. \end{displaymath} 
Note that the parameters "$a$" lie on the interval $[x_{0} - 1,x_{0} + 1]$. By Example \ref{ex1}, $\mathcal{F}_{Q_{0}}$ is transversal over the interval $[x_{0} - 5,x_{0} + 5]$ with constant depending only on
\begin{displaymath} \mathfrak{G} = \|g\|_{C^{3}([-6,6])} \quad \text{and} \quad \mathfrak{g} := \min\{g''(x) : x \in [-6,6]\}. \end{displaymath} 

 Now we translate the family $\mathcal{F}_{0}$ to the origin by setting
\begin{displaymath} \mathcal{F} := \{x \mapsto g(x - a - x_{0}) - (b - y_{0}) : x \mapsto g(x - a) - b \in \mathcal{F}_{0}\}. \end{displaymath}
Then $\mathcal{F}$ is transversal over $[-5,5]$ with constant depending only on $g$, and $\mathcal{F}$ is a $(\delta,s,\delta^{-3\epsilon})$-subset of $C^{2}([-5,5])$, provided that $\delta > 0$ is small enough; this follows from \eqref{form28} and the $(\delta,s,\delta^{-2\epsilon})$-set property of $\mathcal{Q} \cap Q_{0}$. 

For each $f = f_{(a,b)} \in \mathcal{F}$, we associate the set $\mathcal{P}(f) := \mathcal{P}(q) - (x_{0},y_{0})$, where $q \in \mathcal{Q} \cap Q_{0}$ is the square with centre $(a,b)$. Then $\mathcal{P}(f) \subset [-2,2] \times [-\Lambda,\Lambda]$ is a non-empty $(\delta,s,\delta^{-\epsilon})$-set, and every square in $\mathcal{P}(f)$ intersect the graph $\Gamma_{f}$.

We may now apply Theorem \ref{t:mainDiscretised} to the $(\delta,t,\delta^{-3\epsilon})$-set $\mathcal{F}$ and the $(\delta,s,\delta^{-\epsilon})$-sets $\mathcal{P}(f)$. The latter are not quite contained in $[-2,2]^{2}$, but they are at least contained in the set $[-2,2] \times [-\Lambda,\Lambda]$ with no $R$-dependence. To make everything rigorous, the final step is to apply the dilation $(x,y) \mapsto (x,\Lambda^{-1}y)$ to both the functions $f \in \mathcal{F}$ and the sets $\mathcal{P}(f)$. This has the effect of increasing the transversality constant by $O_{\Lambda}(1)$, but this is harmless, recalling that $\Lambda \sim 1 + \|g\|_{L^{\infty}([-1,1])}$ has no $R$-dependence. We leave the details to the reader. Now \eqref{form35} implies \eqref{form36}. \end{proof} 

\subsection{A Fourier-analytic application}\label{s:fourierIntro} A motivation for this paper was to generalise \cite[Theorem 1.1]{Orponen2024Jan} to (more) general convex curves besides the parabola. Here it is:

\begin{thm}\label{t:smoothing} Let $g \in C^{3}(\R)$ be a function whose second derivative $g''$ never vanishes, and let $\Gamma_{g}^{1} := \{(x,g(x)) : x \in [-1,1]\}$. For every $0 \leq s \leq 1$ and $t \in [0,\min\{3s,s + 1\})$, there exists $p = p(g,s,t) \geq 1$ such that the following holds. 

Let $\sigma$ be a Borel measure on $\Gamma_{g}^{1}$ satisfying $\sigma(B(x,r)) \leq r^{s}$ for all $x \in \R^{2}$ and $r > 0$. Then,
\begin{equation}\label{form34} \|\hat{\sigma}\|_{L^{p}(B(R))}^{p} \leq C_{s,t}R^{2 - t}, \qquad R \geq 1. \end{equation} 
\end{thm}

\begin{remark} The dependence of $p$ on $g,s,t$ is effective but no doubt far from sharp. In particular, the size of $p$ is influenced by the dependence between $\eta$ and $\epsilon$ in Corollary \ref{cor:convex} (this is where the $g$-dependence enters the argument). It may be that $p = 6$ always suffices. This has been verified for $s \geq 2/3$ by the third author \cite[Theorem 1.15]{2024arXiv241023080Y} (the case of the parabola was already part of \cite[Theorem 1.1]{Orponen2024Jan}). Demeter and Wang \cite[Theorem 1.3]{MR4869897} have also shown that for $s \in [0,\tfrac{1}{2}]$, the estimate \eqref{form34} holds for $p = 6$, but with the (likely) unsharp exponent $2 - 2s - \tfrac{s}{4}$, in other words
\begin{displaymath} \|\hat{\sigma}\|_{L^{6}(B(R))}^{6} \lesssim_{\epsilon} R^{2 - 2s - s/4 + \epsilon}, \qquad R \geq 1. \end{displaymath}
In the special case where $\sigma$ is a "lift" of a non-atomic self-similar measure $\mu$ on $[-1,1]$ to $\Gamma_{g}^{1}$, Algom and Khalil \cite{2025arXiv250707321A} recently showed that \eqref{form34} holds for all $t < 2$, provided $p \geq 1$ sufficiently large depending on $\mu,t$.
\end{remark} 

The proof of Theorem \ref{t:smoothing} is the same as the proof of \cite[Theorem 1.1]{Orponen2024Jan}, modulo replacing \cite[Lemma 3.3]{Orponen2024Jan} by Corollary \ref{cor:convex}, and removing an appeal to parabolic rescaling. We give the main steps in Section \ref{s8}.

	
\section{Preliminaries}\label{sec2}
\subsection*{Notation and terminology}
We use asymptotic notation $\lesssim, \gtrsim, \sim$. For example, $A \lesssim B$ means that $A \leq CB$ for a universal $C > 0$, while $A \lesssim_r B$ stands for $A \leq C(r)B$ for a positive function $C(r)$. We also use standard $O$ notation; for example $A = O(B)$, $A = O_r(B)$ are synonym with $A \lesssim B$, $A \lesssim_r B$. The notation $A \lessapprox_\delta B$, $A \gtrapprox_\delta B$, $A \approx_\delta B$ or $A \approx B$ will occasionally be used to “hide” slowly growing functions of $\delta$ (logarithmic or small negative exponentials); its precise meaning will be made explicit each time this notation is used.

\subsection{Basic definitions}
\begin{definition}[Dyadic cubes in $\R^d$]
If $n \in \mathbb{Z}$, we denote by $\mathcal{D}_{2^{-n}}(\mathbb{R}^d)$ the family of dyadic cubes in $\mathbb{R}^d$ of side-length $2^{-n}$. If $P \subset \mathbb{R}^d$ is a set, we further denote
\[\mathcal{D}_{2^{-n}}(P) := \{ q \in \mathcal{D}_{2^{-n}} : q \cap P \neq \emptyset \}.\]
Finally, we will abbreviate $\mathcal{D}_{2^{-n}} := \mathcal{D}_{2^{-n}}([0,1]^d)$ for $n \geq 0$.
\end{definition}

\begin{definition}[Covering numbers] 
Let $(X,d)$ be a metric space. For $P \subset X$ and $r>0$, we write $|P|_r$ as the minimum number of $r$-balls in $X$ needed to cover $P$.
\end{definition}

\begin{definition}[$(\delta,s,C)$-set]\label{def:spacingcondition}
Let $(X,d)$ be a metric space. Let $\delta\in (0,1]$ and $s\geq 0$. We say a bounded subset $P\subset X$ is a $(\delta,s,C)$-set if
\begin{equation}\label{def 1}
|P\cap B(x,r)|_\delta\leq Cr^s |P|_\delta, \qquad x \in X, \, \delta \leq r \leq 1. 
\end{equation}
Here $B(x,r)$ refers to a ball in $(X,d)$. A $(\delta,s,C)$-set is called a $(\delta,s)$-set if the value of the constant $C > 0$ is irrelevant. 
\end{definition}

\begin{definition}[Upper $(s,C)$-regular set]\label{def:Ahlfors}
Let $(X,d)$ be a metric space and let $C, s>0$. We say $P\subset X$ is upper $(s,C)$-regular, if for any $x\in X$ and $0<r\leq R<\infty$,
\[|P\cap B(x,R)|_r\leq C\left(\tfrac{R}{r}\right)^s.\]
We also say that $P$ is \emph{upper $(\delta,s,C)$-regular} if the estimate above holds for $\delta \leq r \leq R < \infty$.
\end{definition}

\begin{definition}[$(\delta,t,C)$-regular set]\label{def-regularse}
Let $t>0$ and $C\geq1$. A nonempty set $P\subset C^2(I)$ is called \emph{$(\delta,t,C)$-regular} if $P$ is an upper $(\delta,t,C)$-regular $(\delta,t,C)$-set. \end{definition}

\begin{definition}[$(t,C)$-Frostman measure and $(\delta,t,C)$-Frostman measure]\label{def:Frostman}
Let $(X,d)$ be a metric space. Let $t>0$ and $C\geq1$. A Borel measure $\mu$ in $X$ is called a $(t,C)$-Frostman measure if $\mu(B(x,r))\leq Cr^t$ for all $x\in X$ and $r>0$.
	
If $\delta\in (0,1]$, and the inequality $\mu(B(x,r))\leq Cr^t$ holds for $r\geq \delta$, we say that $\mu$ is a $(\delta,t,C)$-Frostman measure.
\end{definition}

\begin{definition}[$(t,C)$-regular and $(\delta, t, C)$-regular measures]\label{def-regularme}
Let $t>0$ and $C\geq1$. A non-trivial Borel measure $\mu$ with $K=\spt(\mu)\subset C^2(I)$ is called $(t,C)$-regular if
\begin{enumerate}
\item $\mu$ is a $(t,C)$-Frostman measure; and
\item $K$ is upper $(t,C)$-regular.
\end{enumerate}
If $\delta\in (0,1]$ and $\mu$ is a $(\delta,t,C)$-Frostman measure (see Definition \ref{def:Frostman}) satisfying (2) for all $\delta\leq r\leq R <+\infty$, we say that $\mu$ is a $(\delta,t,C)$-regular.
\end{definition}

Let $\pi_{\mathbf{x}}:\R^2\to\R$ denote the orthogonal projection to the $x$-axis. 

\begin{lemma}\label{lemma-almostinjection}
    Let $\mathcal F\subset C^2(I) \cap  B(1)$. For any $f\in \mathcal F$ and $r>0$, if $\mathcal P(f,r) \subset \mathcal D_r(\R^2)$ is such that for every $p\in \mathcal P(f,r)$ there exists $g\in B(f,r) \cap \mathcal F$ such that $p \cap \Gamma_g\neq\emptyset$, then
    \begin{enumerate}
        \item $|\mathcal P(f,r)|\sim|\pi_{\mathbf{x}}(\mathcal P(f,r))|$,
        \item if $P(f,r)$ is a $(r, s, C)$-set for some $s, C>0$, then $\pi_{\mathbf{x}}(\mathcal P(f,r))$ is a $(r,s, C_1)$-set for some $C_1\sim C$. 
    \end{enumerate}
\end{lemma}

\begin{proof}
    We first note that if $p\in \mathcal P(f,r)$, then ${\rm dist}(p, \Gamma_f)\lesssim r$. Thus any $I\in \pi_{\mathbf{x}}(\mathcal P(f,r))$ has at most $\sim 1$ preimages in $\mathcal P(f,r)$, and the claim (1) follows. For the second claim, note that since $f\in B(1)$, for any $I \subset \R$, the set $\mathcal P(f,r) \cap \pi_{\mathbf{x}}^{-1}(I)$ is contained in a ball of radius $\sim \diam(I)$. Claim (2) follows by combining this with (1) and using the $(\delta, s, C)$-set property of $\mathcal P(f,r)$. 
\end{proof}

\subsection{Transversal families} We repeat the following key definition from the introduction:
\begin{definition}[Transversal family]\label{def:transversality} Let $I \subset \R$ be a compact interval, and let $\mathcal{F} \subset C^{2}(I)$. We say that $\mathcal{F}$ is a \emph{transversal family on $I$ with transversality constant $\mathfrak{T} \geq 1$}, if
\begin{equation}\label{eq-trans}
\inf_{\theta\in I} (|f(\theta)- g(\theta)| + |f'(\theta) - g'(\theta)|) \geq \mathfrak{T}^{-1}\|f - g\|_{C^{2}(I)}, \quad f,g\in \mathcal{F}.
\end{equation}
Here and throughout this paper we use the definition $\|f\|_{C^2(I)}:=\max_{x\in I}\sum_{k=0}^2|f(x)|$.
\end{definition}

\begin{remark}
Transversal families are a special case of Zahl's forbidding $k$th-order tangency families (when $k=1$), see \cite[Definition 1.9]{2023arXiv230705894Z}. In particular, the set of lines $\mathcal{A}(2,1)$ restricted to any closed interval is a transversal family as it forbids 1st-order tangency.
\end{remark}

\begin{notation} A typical hypothesis in the results below will be that $\mathcal{F} \subset B_{C^{2}}(1)$ is a transversal family over $I$. The notation $B_{C^{2}}(1)$ is an abbreviation for $B_{C^{2}(I)}(0,1)$. \end{notation}

The transversality condition for $\mathcal{F}$ implies that the intersection of the $r$-neighbourhoods of two graphs of functions $f,g \in \mathcal{F}$ has small diameter. 

\begin{lemma}\label{lemma4}  Let $\mathcal{F} \subset B_{C^{2}}(1)$ be a transversal family over $[-2,2]$ with constant $\mathfrak{T} \geq 1$. Let $r > 0$, and $f,g \in \mathcal{F}$. Then, the projection of $\Gamma_{f}(r) \cap \Gamma_{g}(r)$ to the $x$-axis can be covered by $\leq 5\mathfrak{T}$ intervals of length $\leq 16r\mathfrak{T}/d(f,g)$. Here $\Gamma_{f}(r) = \{(x,y) \in \R^{2} : |y - f(x)| < r\}$.
\end{lemma}

\begin{proof} If $d(f,g) \leq 4r \mathfrak{T}$, then we observe that the $x$-projection of $\Gamma_{f}(r) \cap \Gamma_{g}(r)$ is contained in $[-2,2]$, whose length is $\leq 4 \leq 16r\mathfrak{T}/d(f,g)$. So, we may assume for the remainder of the proof that $d(f,g) > 4r\mathfrak{T}$, therefore $2r \leq \tfrac{1}{2}\mathfrak{T}^{-1}d(f,g)$. 
	
Let $\mathcal{J}$ be the set of connected components of $E := \{x \in (-2,2) : |(f - g)(x)| < 2r\}$. Since $\pi_{\mathbf{x}}(\Gamma_{f}(r) \cap \Gamma_{g}(r)) \subset E$, it suffices to show $|\mathcal{J}| \leq 5\mathfrak{T}$, and $\ell(I) \leq 4r\mathfrak{T}/d(f,g)$ for all $I \in \mathcal{J}$, where $\ell(I)$ is the length of $I$.
	
We first claim that if $I = (a,b) \in \mathcal{J}$ and $I' = (c,d) \in \mathcal{J}$ are distinct with $b \leq c$, then in fact $\dist(I,I') \geq c - b \geq \mathfrak{T}^{-1}$. To prove this, let $z \in [b,c]$ be the point where $x \mapsto |(f - g)(x)|$ attains its maximum. Then $(f - g)'(z) = 0$, which implies that $z \notin \{b,c\}$, since it follows from (e.g.) $|(f - g)(b)| \leq 2r$ and the transversality hypothesis that 
\begin{displaymath} |(f - g)'(b)| \geq \mathfrak{T}^{-1}d(f,g) - 2r \geq (2\mathfrak{T})^{-1}d(f,g) > 0. \end{displaymath}
By the mean value theorem, we may now further choose a point $\xi \in [b,z]$ such that 
\begin{displaymath}  (2\mathfrak{T})^{-1}d(f,g) \leq |(f - g)'(b)| = |(f - g)'(b) - (f - g)'(z)| = |z - b||(f - g)''(\xi)|, \end{displaymath}
or equivalently $|z - b| \geq (2\mathfrak{T})^{-1}d(f,g)/|(f - g)''(\xi)| \geq (2\mathfrak{T})^{-1}$. Since $|z - c| \geq (2\mathfrak{T})^{-1}$ by a similar argument, the claim follows. We infer that $|\mathcal{J}| \leq 5\mathfrak{T}$ (also using $\mathfrak{T} \geq 1$).

	Finally, we claim that $\ell(I) \leq 4r \mathfrak{T}/d(f,g)$ for all $I \in \mathcal{J}$. Fix $I \in \mathcal{J}$ and $x \in I$. Since $|f(x) - g(x)| \leq 2r < (2\mathfrak{T})^{-1}d(f,g)$, the transversality condition implies $|f'(x) - g'(x)| > (2\mathfrak{T})^{-1}d(f,g)$. This shows that $f - g$ is strictly monotone on $I$, and $\ell(I) \leq 4r \mathfrak{T}/d(f,g)$.  \end{proof}

We record that transversal families are upper $2$-regular subsets of $C^{2}(I)$:
\begin{lemma}\label{lem-Ahlforsregularity}
Every transversal family $\mathcal{F} \subset C^{2}(I)$ with constant $\mathfrak{T}$ admits an $\sqrt{2}\mathfrak{T}$-bi-Lipschitz embedding into $\R^{2}$. In particular, $\mathcal{F}$ is upper $(2,20\mathfrak{T}^2)$-regular.
\end{lemma}
\begin{proof} Define $A: \mathcal{F} \to \R^2$ by
\[A(f)=(f(x_{0}),f'(x_{0})),\]
where $x_{0} \in I$ is arbitrary. We claim that $A$ is an $\sqrt{2}\mathfrak{T}$-biLipschitz embedding. To see this, fix $f,g \in \mathcal{F}$, and use the transversality condition \eqref{eq-trans} to deduce
\[\|f-g\|_{C^2(I)}\leq \mathfrak{T}(|f(x_{0})-g(x_{0}))|+|f'(x_{0})-g'(x_{0})|)\leq \sqrt{2} \mathfrak{T} \cdot |A(f) - A(g)|.\]
Conversely $\|f-g\|_{C^2(I)}\geq |A(f) - A(g)|$ by the definition of the $C^2$-norm. Thus the claim holds. Then for any $f\in \mathcal{F}$ and $0<r<R<+\infty$, we have \[|B(f,R)\cap \mathcal{F}|_r\leq |B(A(f),R)|_{r/\sqrt{2}\mathfrak{T}}< 10\Big(\frac{R}{r/\sqrt{2}\mathfrak{T}}\Big)^2=20\mathfrak{T}^2 \Big(\frac{R}{r}\Big)^2.\] This completes the proof.\end{proof}

We record a few simple facts about rescaling maps on $C^{2}(I)$.
\begin{definition}[Rescaling map $T_B$] \label{def:rescalingmap} Let $B := B(f_{0},r_{0}) \subset C^{2}(I)$. We define 
\begin{displaymath} 
T_{B}(f) :=T_{f_{0},r_{0}}:= r_{0}^{-1}(f - f_{0}), \qquad f \in C^{2}(I). 
\end{displaymath} 
\end{definition}

\begin{lemma}\label{lem3}
The rescaling map $T_{B} := T_{f_{0},r_{0}}$ is a similarity in $C^{2}(I)$:
\begin{displaymath} 
\|T_{B}(f) - T_{B}(f)\|_{C^{2}(I)} = r_{0}^{-1}\|f - g\|_{C^{2}(I)}, \qquad f,g \in C^{2}(I). \end{displaymath}
\end{lemma} 

We record how the rescaling maps affect the regularity of measures on $C^{2}(I)$:

\begin{lemma}\label{lemma-rescaledregular} Let $\mu$ be a $(\delta,t,C)$-regular measure on $C^{2}(I)$, and let $B := B(f_{0},r_{0}) \subset C^{2}(I)$. Then, the renormalised measure $\mu_{B} := \mu_{f_{0},r_{0}} := r_{0}^{-t}T_{f_{0},r_{0}}(\mu)$ is $(\delta/r_{0},t,C)$-regular.
\end{lemma}
\begin{proof} 
For any $f\in C^2(I)$ and $r\geq \delta/r_0$, we have
\[\mu_B(B(f,r))=r_0^{-t}\mu(T_{f_{0},r_{0}}^{-1}(B(f,r)))=r_0^{-t}\mu(B(r_0f+f_0,r_0r))\leq Cr^t,\]
where we applied the Frostman condition of $\mu$. Moreover, if $\delta/r_0\leq r\leq R<\infty$, then
\[\begin{split}
|\spt(\mu_B)\cap B(f,R)|_r&=|T_{f_{0},r_{0}}(\spt(\mu_B))\cap B(f,R)|_r\\&=|\spt(\mu)\cap B(f_0+r_0f,r_0R)|_{r_0r}\leq C(\tfrac{R}{r})^t,
\end{split}\]
where we also used the regularity of $\mu$. Thus $\mu_B$ is $(\delta/r_{0},t,C)$-regular.
\end{proof}

Transversal families behave well under rescaling maps in $\mathcal{F}$.
\begin{lemma}\label{lem1} Let $\mathcal{F} \subset C^{2}(I)$ be a transversal family with constant $\mathfrak{T} \geq 1$, let $f_{0} \in \mathcal{F}$, and let $B := B(f_{0},r) \subset C^{2}(I)$. Then
\begin{displaymath} \mathcal{F}_{B} := T_B(\mathcal{F}):=\{T_B(f): f \in \mathcal{F}\} \subset C^{2}(I) \end{displaymath}
is a transversal family with constant $\mathfrak{T}$.
\end{lemma}

\begin{proof}
Let $f = T_B(\bar{f}) \in \mathcal{F}_{B}$ and $g = T_B(\bar{g}) \in \mathcal{F}_{B}$. Then, for any $\theta \in I$, by using transversality of $\mathcal{F}$ and Lemma \ref{lem3},
\begin{align*} |f(\theta) - g(\theta)| + |f'(\theta) - g'(\theta)| & = r^{-1}(|(\bar{f}(\theta) - \bar{g}(\theta)| + |\bar{f}'(\theta) - \bar{g}'(\theta)|)\\
& \geq r^{-1}\mathfrak{T}^{-1}\|\bar{f} - \bar{g}\|_{C^{2}(I)} = \mathfrak{T}^{-1}\|f - g\|_{C^{2}(I)}. \end{align*}
The proof is complete. \end{proof}

Besides the rescaling operation $f \mapsto (f - f_{0})/r$, we will also need to consider another one, where $f \mapsto (f(rx +x_{0}) - y_{0})/r$. Transversality is also preserved by this operation, but the intervals and constants of transversality may change, as follows:
\begin{lemma}\label{lem2}
Let $\mathcal{F}\subset C^{2}(I)$ be a transversal family with constant $\mathfrak{T} \geq 1$. Let $x_{0},y_{0} \in \R$ and $r \in (0,1]$. Then, the family
\[T_{(x_{0},y_{0}),r}(\mathcal{F}):=\{f_{(x_{0},y_{0}),r}:  f\in \mathcal{F}\} \subset C^{2}(J)\]
is transversal on any compact interval $J \subset (I - x_{0})/r$ with constant $|J|\mathfrak{T}+1$, where $f_{(x_{0},y_{0}),r}(x)=(f(rx+x_0)-y_0)/r$ and $|J|$ is the length of $J$.
\end{lemma}
\begin{proof} Fix $f,g\in\mathcal{F}$. For notational convenience, write 
\[G=f-g, \quad G_r=f_{(x_{0},y_{0}),r} -g_{(x_{0},y_{0}),r}.\]
For $z \in J$, let us record the formulae
\begin{equation}\label{to16} G(rz + x_{0}) = rG_{r}(z), \quad G'(rz + x_{0}) = G_{r}(z) \quad \text{and} \quad rG''(rz + x_{0}) = G_{r}''(z).  \end{equation} 
For $x\in J$, write $d(x) :=\max\{|G_r(x)|, |G_r'(x)|\}$. Our task is to demonstrate that $\|G_{r}\|_{C^{2}(J)} \leq d(x)(|J|\mathfrak{T} + 1)$. By the transversality of $\mathcal{F}$, and $r \leq 1$,
\begin{equation}\label{to17} \|G_{r}'\|_{L^{\infty}(J)} \leq \|G\|_{C^2(I)}\leq \mathfrak{T} \cdot \max\{|G(rx+x_0)|,|G'(rx+x_0)|\} \stackrel{\eqref{to16}}{\leq} d(x)\mathfrak{T}. \end{equation}
Next, by the mean value theorem,
\begin{displaymath} \|G_{r}\|_{L^{\infty}(J)} \leq \|G_{r}'\|_{L^{\infty}(J)}\cdot |J| + |G_{r}(x)| \stackrel{\eqref{to17}}{\leq} d(x)\mathfrak{T}|J|+|G_r(x)|\leq (\mathfrak{T}|J|+1)d(x). \end{displaymath}
Moreover, the last point in \eqref{to16} implies
\[|G''_r(z)| \leq |G''(rz+x_0)|\leq \|G\|_{C^2(I)} \stackrel{\eqref{to17}}{\leq} d(x)\mathfrak{T}, \qquad z \in J. \]
Thus also $\|G_r''\|_{L^\infty(J)}\leq d(x)\mathfrak{T}$, and the proof is complete. \end{proof}

\subsection{Dyadic cubes associated with a transversal family}
\begin{definition}[Dyadic system in $\mathcal{F}$]\label{def-pullbackdyadic}
Let $\mathcal{F}\subset C^2(I)$ be a transversal family with constant $\mathfrak{T}\geq 1$, and let $x_{0}$ be the midpoint of $I$. Let $A:\mathcal{F}\to \R^2$, $A(f)=(f(x_0),f'(x_0))$ be the bi-Lipschitz map in the proof of Lemma \ref{lem-Ahlforsregularity}. The \emph{dyadic system associated with $\mathcal{F}$}, denoted $\mathcal{D}(\mathcal{F})$, is defined as $\mathcal{D}(\mathcal{F}):=\bigcup_{r\in 2^{\Z}} \mathcal{D}_r(\mathcal{F})$, where
\[\mathcal{D}_r(\mathcal{F}):= \{A^{-1}(p) : p \in \mathcal{D}_{r}(A(\mathcal{F}))\}.\]
Thus, cubes associated with $\mathcal{F}$ are pull-backs of dyadic squares in $\R^{2}$. We will use boldface symbols $\mathbf{F},\mathbf{F}'$ etc. to denote dyadic cubes in $\mathcal{F}$. For any $r\in 2^{\Z}$ and $p\in \mathcal{D}_r(A(\mathcal{F}))$, we say $\mathbf{F}=A^{-1}(p)\in  \mathcal{D}_r(\mathcal{F})$ is a dyadic cube with side length $r$.

If $\mathcal{F}' \subset \mathcal{F}$ and $r \in 2^{\Z}$, we also define $\mathcal{D}_{r}(\mathcal{F}') := \{\mathbf{F} \in \mathcal{D}_{r}(\mathcal{F}) : \mathbf{F} \cap P \neq \emptyset\}$. 
\end{definition}

\begin{remark} If $\mathcal{F}$ is a transversal family, and $\mathcal{F}' \subset \mathcal{F}$ is an arbitrary subset, then $\mathcal{F}'$ is a transversal family on its own, and Definition \ref{def-pullbackdyadic} yields an associated dyadic system $\mathcal{D}(\mathcal{F}') = \bigcup_{r \in 2^{\Z}} \mathcal{D}_{r}(\mathcal{F}')$. Fortunately, as is clear from the construction, the families $\mathcal{D}_{r}(\mathcal{F}')$ coincide with the subsets of $\mathcal{D}_{r}(\mathcal{F})$ defined on the last line of Definition \ref{def-pullbackdyadic}. \end{remark}


		


The following result says that the dyadic covering number and ordinary ball covering number in a transversal family are comparable.
\begin{cor}\label{cor-covernumber}
Let $ P\subset \mathcal{F}$ be a bounded and non-empty subset. Then for any dyadic number $r\leq \diam(P)$, we have
\[\frac{|P|_r}{|\mathcal{D}_r(P)|}\in [\tfrac{1}{16},40\mathfrak{T}^4].\]
\end{cor}
\begin{proof}
Take $\mathbf{F}=A^{-1}(p)\in \mathcal{D}_r(P)$ for some $p\in\mathcal{D}_r(P)$. Since $p$ is contained in a $r$-ball, by Lemma \ref{lem-Ahlforsregularity} we know that $\mathbf{F}$ is contained in a ball with radius at most $\sqrt{2}\mathfrak{T}r$. By using the upper $2$-regularity of $P$, we thus get
$|P|_r\leq |\mathcal{D}_r(P)|\cdot 20\mathfrak{T}^2\cdot (\sqrt{2}\mathfrak{T})^2=40\mathfrak{T}^4|\mathcal{D}_r(P)|$. 

Secondly, for any $r$-ball $B(f,r)\subset\mathcal{F}$ with $f\in\mathcal{F}$ intersecting $P$, $A(B(f,r))$ is also contained in a $r$-ball $B(A(f),r)\subset \R^2$ by Lemma \ref{lem-Ahlforsregularity}. Since $B(A(f),r)$ intersects at most $16$ dyadic cubes in $\mathcal{D}_r(\R^2)$, we conclude that $B(f,r)$ can be covered by $\leq 16$ dyadic cubes in $\mathcal{D}_r(\mathcal{F})$ and also $|\mathcal{D}_r(P)|\leq 16 |P|_r$
\end{proof}

Finally, we define the rescaling map with respect to a dyadic cube. 
\begin{definition}[Rescaling maps $T_\mathbf{F}$] \label{def:drescalingmap}  
Let $\mathcal{F}\subset B_{C^2}(1)$ be a transversal family with constant $\mathfrak{T}\geq1$. For every ${\bf F}\in \bigcup_{r\in 2^{-\N}} \mathcal D_r (\mathcal F)$ fix an arbitrary function $f_{\bf F}\in\bf F$. For $r_0\in 2^{\N}$ and $\mathbf{F} \in \mathcal{D}_{r_0}(\mathcal{F})$, define $T_\mathbf{F}: \mathcal F\to C^2(I)$ by $T_{\bf F}(f) =r_0^{-1}(f - f_{\bf F})$.
\end{definition}
Next we record how rescaling under $T_{\bf F}$ affects regularity of sets.
\begin{lemma}\label{lemma-TFregular}
    Let $\mathcal F \subset C^2(I)$ be $(\delta, t, C)$-regular. Then for any $\Delta \geq \delta$ and ${\bf F}\in \mathcal D_{\Delta}(\mathcal F)$, the set $T_{\bf F}({\bf F}\cap \mathcal F)$ is $(\delta/\Delta, t, \overline{C})$-regular with 
    \begin{equation*}
        \overline{C} = \max\left\{C, C\cdot \tfrac{\Delta^t|\mathcal F|_\delta}{|{\bf F} \cap \mathcal F|_\delta} \right\}.
    \end{equation*}
\end{lemma}
\begin{proof}
    For any $f\in {\bf F}$ and $r\geq \delta/\Delta$, 
    \begin{align*}
        |T_{\bf F}({\bf F}\cap \mathcal F) \cap B(f, r)|_{\delta/\Delta} &= |{\bf F}\cap \mathcal F \cap B(\Delta f+f_{\bf F}, r\Delta)|_{\delta} \\
        &\leq C(r\Delta)^t |\mathcal F|_\delta = \overline{C} r^t |{\bf F}\cap \mathcal F|_{\delta} = \overline{C} r^t |T_{\bf F}({\bf F}\cap \mathcal F)|_{\delta/\Delta}
    \end{align*}
    so $T_{\bf F}({\bf F}\cap \mathcal F)$ is a $(\delta/\Delta, t, \overline{C})$-set. The upper bound in Definition \ref{def-regularse}(2) follows directly from the regularity of $\mathcal F$. 
\end{proof}

We close this subsection with a nice property that will be used in Section \ref{sec7}.
\begin{lemma}\label{lem-commutep}
Let $\mathcal{F}$ be a transversal family with constant $\mathfrak{T}\geq1$. Fix $\mathbf{F}\in\mathcal{D}_r(\mathcal{F})$. Let $\mathbf{F}_0=A^{-1}(Q)\in\mathcal{D}_\rho(\overline{\mathcal{F}})$ for some $Q\in\mathcal{D}_\rho$ with $\rho\in 2^{-\N}$, where $\overline{\mathcal{F}}=T_\mathbf{T}(\mathcal{F}\cap \mathbf{F})$. Then
\[T_\mathbf{F}^{-1}(\mathbf{F})=A^{-1}(rQ+A(f_\mathbf{F})).\]
Here $rQ+A(f_\mathbf{F}):=\{rx+A(f_\mathbf{F}): x\in Q\}$ and the operator $A$ is same as in Definition \ref{def-pullbackdyadic}.
\end{lemma}
\begin{proof}
First, by definition of $T_\mathbf{F}$
\[A(T_\mathbf{F}^{-1}(\mathbf{F}_0))=\{rA(\bar{f})+A(f_\mathbf{F}):\bar{f}\in \mathbf{F}_0\}\subset rQ+A(f_\mathbf{F}),\]
which means $T_\mathbf{F}^{-1}(\mathbf{F}_0)\subset A^{-1}(rQ+A(f_\mathbf{F}))$. Secondly, if $g\in A^{-1}(rQ+A(f_\mathbf{F}))$, then we have $A(g)=rz+A(f_\mathbf{F})$ for some $z\in Q$, which means $z=(A(g)-A(f_\mathbf{F}))/r\in A(T_\mathbf{F}(g))$. Hence $T_\mathbf{F}(g)\in A^{-1}(z)\subset A^{-1}(Q)=\mathbf{F}_0$ and $g\in T_\mathbf{F}^{-1}(\mathbf{F}_0)$, which means
\[A^{-1}(rQ+A(f_\mathbf{F}))\subset T_\mathbf{F}^{-1}(\mathbf{F}_0).\]
Consequently, $T_\mathbf{F}^{-1}(\mathbf{F}_0)=A^{-1}(rQ+A(f_\mathbf{F}))$. 
\end{proof}

\subsection{Uniform sets and branching functions.}
We will need the notions of \emph{uniform sets} and \emph{branching functions} in the context of transversal families $\mathcal{F} \subset C^{2}(I)$. These notions have previously only been defined in Euclidean spaces, so we are forced to repeat some results and argument here, even though they are straightforward adaptations of \cite[Subsection 2.3]{2023arXiv230110199O}. In what follows, we consider a transversal family $\mathcal{F} \subset B_{C^2}(1)$ (not necessarily $\delta$-separated) with constant $\mathfrak{T} \geq 1$. Recall from Lemma \ref{lem-Ahlforsregularity} that $\mathcal{F}$ is upper $(2, 20\mathfrak{T}^2)$-regular. Moreover, for any $r, R \in 2^{-\N}$ with $r \leq R$, Corollary \ref{cor-covernumber} and the upper regularity imply
\begin{equation}\label{form37} |\mathcal{D}_r(\mathcal{F}\cap\mathbf{F})|\leq 640\mathfrak{T}^4\Big(\frac{R}{r}\Big)^2, \quad \mathbf{F}\in\mathcal{D}_R(\mathcal{F}). \end{equation}

\begin{definition}[Uniform sets]\label{def:uniformset}
Let $n \ge 1$, and let $\delta = \Delta_n < \Delta_{n-1} < \cdots < \Delta_1 \le \Delta_0 = 1$ be a sequence of scales in $2^{-\N}$. We say $\mathcal{F}$ is $\{\Delta_j\}_{j=1}^n$-\emph{uniform} if there is a sequence $\{N_j\}_{j=1}^n$ with $N_j \in 2^{\mathbb{N}}$ such that \[|\mathcal{D}_{\Delta_j}(\mathcal{F} \cap \mathbf{F})| = N_j\quad\text{for all}\; j \in \{1,\dots,n\} \;\text{and for all}\; \mathbf{F} \in \mathcal{D}_{\Delta_{j-1}}(\mathcal{F}).\]
\end{definition}

It turns out that we can always find “dense uniform subsets” in a transversal family.
\begin{lemma}\label{lem-uniformsubset}
Let $T\geq \log(640\mathfrak{T}^4)$ be an integer. Let $m\in \N$ and $\delta=2^{-mT}$. Let also $\Delta_j=2^{-jT}$ for $0\leq j\leq m$, so in particular $\delta=\Delta_m$. Then there exists a $\{ \Delta_j\}_{j=0}^m$-uniform subset $\mathcal F'\subseteq \mathcal F$ such that $|\mathcal \mathcal{D}_\delta(\mathcal{F}')| \geq (6T)^{-m}|\mathcal{D}_\delta(\mathcal F)|$.
	
In particular, if $\epsilon>0$ and $T^{-1}\log(6T)\leq \epsilon$, then $|\mathcal{D}_\delta(\mathcal{F}')| \geq \delta^\epsilon|\mathcal{D}_\delta(\mathcal F)|$.
\end{lemma}
\begin{proof}
Set $\mathcal F^m := \mathcal{D}_\delta(\mathcal F)$. Given $\mathcal F^{\ell+1}$ for $\ell\in\lbrace 0,\ldots, m-1\rbrace$, set
\[\mathcal F^{\ell, k} = \bigcup\Big\{ \mathcal F^{\ell+1} \cap \mathbf{F}:\ |\mathcal{D}_{\Delta_{l+1}}(\mathcal F^{\ell+1}\cap \mathbf{F})| \in (2^{k}, 2^{(k+1)}], \mathbf{F} \in \mathcal{D}_{\Delta_\ell}(\mathcal F^{\ell+1})\Big\}.\]
Since $\mathcal{F}$ is upper $2$-regular and $T\geq \log  (640\mathfrak{T}^4)$, $k\leq 2T+\log  (640\mathfrak{T}^4)\leq 3T$. Let $k$ be the integer such that the cardinality of $\mathcal F^{\ell, k}$ is maximal, then there holds
\[|\mathcal{D}_{\Delta_{l+1}}(\mathcal F^{\ell, k})| \geq (3T)^{-1} |\mathcal{D}_{\Delta_{l+1}}(\mathcal F^{\ell+1})|.\]
We may discard at most half of elements in $\mathcal{D}_{\Delta_{l+1}}(\mathcal F^{\ell+1})\cap \mathbf{F}$ from each $\mathbf{F}$ in the definition of $\mathcal F^{\ell,k}$, then get a set $\mathcal F^\ell$ with $|\mathcal{D}_{\Delta_{l+1}}(\mathcal F^{\ell})| \geq (6T)^{-1} |\mathcal{D}_{\Delta_{l+1}}(\mathcal F^{\ell+1})|$ such that $|\mathcal{D}_{\Delta_{l+1}}(\mathbf{F}\cap \mathcal F^\ell)| = 2^{k}$ for each $\mathbf{F}\in\mathcal D_{\Delta_\ell} (\mathcal F^\ell)$. We see inductively that $|\mathcal{D}_{\Delta_{j+1}}(\mathcal{F}^l\cap \mathbf{F})|$ is constant over all $\mathbf{F}\in\mathcal{D}_{\Delta_j}(\mathcal{F}^l)$, for all $j=m-1, m-2,\ldots l$. Taking $\mathcal F' = \mathcal F^0$, the lemma follows.
\end{proof}

Lemma \ref{lem-uniformsubset} has the following useful corollary.
\begin{cor}\label{cor-uniformsets}
For every $s\in(0,2]$ and $\epsilon\in(0,1)$, there exists $\delta_0=\delta_0(\mathfrak{T},\epsilon)>0$ such that the following holds for all $\delta\in(0,\delta_0]$. Assume $\mathcal{F}$ is a $(\delta,s,\delta^{-\epsilon})$-set. Then, there exists $T\sim_{\epsilon,\mathfrak{T}}1$ and a $\{2^{-jT}\}_{j=1}^m$-uniform subset 
$\mathcal{F}'\subset\mathcal{F}$ so that $|\mathcal{D}_\delta(\mathcal{F}')|\geq\delta^\epsilon|\mathcal{D}_\delta(\mathcal{F})|$ and $\delta\in (2^{-(m+1)T}, 2^{-mT}]$. In particular, $\mathcal{F}'$ is a $(\delta,s,\delta^{-2\epsilon})$-set.
\end{cor}
\begin{proof}
Take an integer $T\geq \log (640\mathfrak{T}^4)$ so large that $T^{-1} \log(6 T) < \epsilon/2$, and then let $m \in \mathbb{N}$ be the largest number such that $\delta' = 2^{-mT} \geq \delta$. Let $\bar{\mathcal{F}} = \mathcal{D}_{\delta'}(\mathcal{F})$. Since $\delta'/\delta \leq 2^{T}$, 
\[|\bar{\mathcal{F}}| = |\mathcal{D}_{\delta'}(\mathcal{F})| \geq O(\mathfrak{T}^4)^{-1}2^{-2 T} |\mathcal{D}_{\delta}(\mathcal{F})|.\]
Next, apply Lemma \ref{lem-uniformsubset} to find a $\{2^{-jT}\}_{j=1}^{m}$-uniform subset $\mathcal{F}_1 \subset \bar{\mathcal{F}}$ with
\[|\mathcal{F}_1| \geq (\delta')^{\epsilon/2} |\bar{\mathcal{F}}| \geq O(\mathfrak{T}^4)^{-1} 2^{-2 T} \delta^{\epsilon/2} |\mathcal{D}_\delta(\mathcal{F})|.\]
Now, if $\delta > 0$ is small enough in terms of $\mathfrak{T},\epsilon$, $|\mathcal{F}_1| \geq \delta^{\epsilon} |\mathcal{D}_\delta(\mathcal{F})|$, and $\mathcal{F}':=(\cup \mathcal{F}_1)\cap\mathcal{F}$ is the desired subset of $\mathcal{F}$.
\end{proof}

If a uniform set is a $(\delta,s)$-set, we have the following nice properties.
\begin{lemma}\label{lem-subofuni}
Let $\delta\in 2^{-\N}$, and let $\mathcal{F}$ be a $(\delta,s,C_1)$-set for some $s\in(0,2]$ and $C_1>0$. Fix $\Delta\in 2^{-\N}\cap [\delta,1]$, and assume that the map $\mathbf{F}\mapsto |\mathcal{D}_\delta(\mathcal{F}\cap \mathbf{F})|$, $\mathbf{F}\in\mathcal{D}_\Delta(\mathcal{F})$ is constant. Let $\mathcal{F}'\subset\mathcal{F}$ be any subset satisfying $|\mathcal{D}_\delta(\mathcal{F}')|\geq |\mathcal{D}_\delta(\mathcal{F})|/C_2$ ($C_2>0$). Then $\mathcal{F}'$ is a $(\Delta,s,O_\mathfrak{T}(C_1C_2))$-set.
	
In particular, if $\mathcal{F}$ is also $\{2^{-jT}\}_{j=1}^m$-uniform with $T\in\N$ and $\delta=2^{-mT}$, then $\mathcal{F}'$ is a  $(\Delta_j,s,O_\mathfrak{T}(C_1C_2))$-set for any $\Delta_j=2^{-jT}$ ($j=1,2,\ldots,m$).
\end{lemma} 
\begin{proof}
Let $\mathbf{F} \in \mathcal{D}_{\Delta}(\mathcal{F}')$ be the dyadic cube maximising $|\mathcal{D}_\delta(\mathcal{F}'\cap \mathbf{F})|$ (among all cubes in $\mathcal{D}_{\Delta}(\mathcal{F}')$). Then $|\mathcal{D}_\delta(\mathcal{F})|/C_2 \leq |\mathcal{D}_\delta(\mathcal{F}')| \leq |\mathcal{D}_\delta(\mathcal{F}' \cap \mathbf{F})| \cdot |\mathcal{D}_{\Delta}(\mathcal{F}')| \leq |\mathcal{D}_\delta(\mathcal{F} \cap \mathbf{F})| \cdot |\mathcal{D}_{\Delta}(\mathcal{F}')|$, hence 
\begin{equation}\label{eq-11}
|\mathcal{D}_{\Delta}(\mathcal{F}')| \geq \frac{|\mathcal{D}_\delta(\mathcal{F})|}{C_2 |\mathcal{D}_\delta(\mathcal{F} \cap \mathbf{F})|} = \frac{|\mathcal{D}_{\Delta}(\mathcal{F})|}{C_2}.
\end{equation}
For any $\mathbf{F}_r\in\mathcal{D}_r(\mathcal{F}')$ with $r\in 2^{-\N}\cap[\Delta,1]$, we have
\begin{equation}\label{eq-12}
\begin{split}
|\mathcal{D}_\Delta(\mathcal{F}'\cap \mathbf{F}_r)|&\leq |\mathcal{D}_\Delta(\mathcal{F}\cap \mathbf{F}_r)|=\frac{|\mathcal{D}_\delta(\mathcal{F}\cap \mathbf{F}_r)|}{|\mathcal{D}_\delta(\mathcal{F}\cap \mathbf{F})|}\lesssim \frac{C_1r^s|\mathcal{D}_\delta(\mathcal{F})|}{|\mathcal{D}_\delta(\mathcal{F}\cap \mathbf{F})|}\\&=C_1r^s|\mathcal{D}_\Delta(\mathcal{F})| \stackrel{\eqref{eq-11}}{\leq}C_1C_2r^s|\mathcal{D}_\Delta(\mathcal{F}')|.
\end{split}
\end{equation}
Now for any $r$-ball $B(f,r)$ with $r\in[\Delta,1]$, we find $r_1,r_2\in 2^{-\N}\cap[\Delta,1]$ such that $r\in (r_1,r_2]$ and $r_2/r_1=2$. Also, by Corollary \ref{cor-covernumber}, $\mathcal{F}'\cap B(f,r_2)$ can covered by $\lesssim_\mathfrak{T} 1$ dyadic $r_2$-cubes $\mathbf{F}_2$. Thus
\[|\mathcal{F}'\cap B(f,r)|_\Delta\leq |\mathcal{F}'\cap B(f,r_2)|_\Delta\stackrel{\eqref{eq-12}}{\lesssim_\mathfrak{T}} C_1C_2r_2^s|\mathcal{F}'|_\Delta\lesssim C_1C_2r^s|\mathcal{F}'|_\Delta,\]
as required.
\end{proof}

\begin{definition}[Branching function]\label{def:branching}
Let $T\geq \log (640\mathfrak{T}^4)$ be an integer, and let $\mathcal{F}\subset B_{C^2}(1)$ be $\{2^{-jT}\}_{j=1}^{m}$-uniform with $\{N_j\}_{j=1}^{m}$ the associated sequence. We define the \emph{branching function} $\beta: [0,m] \to [0,\infty)$ by setting $\beta(0)=0$ and
\[\beta(j) := \frac{1}{T} \sum_{i=1}^{j} \log N_i, \qquad j \in \{1, \dots, m\}.\]
and then interpolating linearly.  \end{definition}

The hypothesis $T \geq \log(640\mathfrak{T}^{4})$ ensures that $\beta$ is $3$-Lipschitz:

\begin{remark}\label{rmk-Lip} Note that $N_{j} \leq 640\mathfrak{T}^{4} \cdot 2^{2T}$ according to \eqref{form37}, so $\beta(x) \leq \log (640\mathfrak{T}^{4}) + 2m$ for all $x \in [0,m]$, and $\beta$ is $L$-Lipschitz with constant
\begin{displaymath} L \leq \frac{\log (640\mathfrak{T}^{4}) + 2T}{T} \leq 3 \end{displaymath}
thanks to the hypothesis $T \geq \log(640\mathfrak{T}^{4})$.

The size of the branching function has the following relationship with the dyadic covering number of $\mathcal{F}$. Since $\mathcal{F}$ is $\{2^{-jT}\}_{j = 1}^{m}$-uniform, we may write
\begin{displaymath} \frac{\log |\mathcal{D}_{2^{-jT}}(\mathcal{F})|}{T} = \frac{\log |\mathcal{D}_{1}(\mathcal{F})| + \sum_{i = 1}^{j} \log N_{i}}{T} = \frac{\log |\mathcal{D}_{1}(\mathcal{F})|}{T} + \beta(j).   \end{displaymath} 
Here $0 \leq \log |\mathcal{D}_{1}(\mathcal{F})| \lesssim_{\mathfrak{T}} 1$ since $\mathcal{F} \subset B_{C^{2}}(1)$. Therefore,
\begin{equation}\label{form38} 2^{T\beta(j)} \leq |\mathcal{D}_{2^{-jT}}(\mathcal{F})| \lesssim_{\mathfrak{T}} 2^{T\beta(j)}, \qquad j \in \{0,\ldots,m\}. \end{equation} \end{remark}

The branching function encodes various spacing conditions. Before making a precise statement, we first recall the following definition from \cite[Definition 8.2]{OS23}.

\begin{definition}\label{def:Lipschitz}
Given a function $f:[a,b]\to \R$ and numbers $\epsilon>0, \sigma\in \R$, we say that $f$ is \emph{$(\sigma,\epsilon)$-superlinear on $[a,b]$} if
\[f(x)\geq f(a)+\sigma(x-a)-\epsilon(b-a),\quad x\in [a,b].\]
If $\sigma=s_f(a,b) := (f(b) - f(a))/(b - a)$, then we simply say that $f$ is $\epsilon$-superlinear on $[a,b]$.
	
We say that $f$ is \emph{$(\epsilon,\sigma)$-linear on $[a,b]$} if 
\[|f(x) - f(a) - \sigma(x - a)|\leq \epsilon |b-a|, \quad x\in [a,b].\]
\end{definition}

The Euclidean counterpart of the next lemma is \cite[Lemma 8.3]{OS23}.

\begin{lemma}\label{lem:superlinear}
For any $\epsilon>0$ and $\mathfrak{T}>1$, there exists $\Delta_0=\Delta_0(\mathfrak{T},\epsilon)>0$ such that the following holds for all dyadic $\Delta\in 2^{-\N} \cap (0,\Delta_0]$. 
	
Let $m\in\N$, and let $\mathcal{F}\subset B_{C^2}(1)$ be a $\{\Delta^i\}_{i=1}^{m}$-uniform transversal family with constant $\mathfrak{T}\geq1$. Le $\beta \colon [0,m] \to [0,\infty)$ be the branching function of $\mathcal{F}$, and let $\delta = \Delta^m$. 
\begin{itemize}
\item[\textup{(i)}] If $\beta$ is $(\sigma, \epsilon)$-superlinear on $[0,m]$, then $\mathcal{F}$ is a $(\delta, \sigma, O_{\mathfrak{T}}(\Delta^{-\sigma}\delta^{-\epsilon})$-set. Conversely, if $\mathcal{F}$ is a $(\delta, \sigma, \delta^{-\epsilon})$-set for some $\sigma \in [0,2]$, then $\beta$ is $(\sigma, 2\epsilon)$-superlinear on $[0,m]$, or
\begin{equation}\label{equ-12}
\beta(x) \geq \sigma x - 2\epsilon m, \quad x \in [0,m].
\end{equation}
		
\item[\textup{(ii)}] If $\beta$ is $(\sigma,\epsilon)$-linear on $[0,m]$, then $\mathcal{F}$ is $(\delta,\sigma, O_{\mathfrak{T}}(\Delta^{-2\sigma}\delta^{-2\epsilon}))$-regular. 
\end{itemize}
\end{lemma}
\begin{proof}
We first deduce the following key equation from the uniformity of $\mathcal{F}$, and \eqref{form38}:
\begin{equation}\label{form41} |\mathcal{D}_{\Delta^{k}}(\mathcal{F}\cap \mathbf{F})| =\frac{|\mathcal{D}_{\Delta^{k}}(\mathcal{F})|}{|\mathcal{D}_{\Delta^j}(\mathcal{F})|} \stackrel{\eqref{form38}}{\sim_{\mathfrak{T}}} \Delta^{\beta(j) - \beta(k)}, \qquad \mathbf{F}\in \mathcal{D}_{\Delta^j}(\mathcal{F}), \, k \in \{j,\ldots,m\}. \end{equation}

We then prove part (i). Note that $\Delta^{-\beta(m)} \leq |\mathcal{D}_{\delta}(\mathcal{F})|$ by \eqref{form38}. Therefore, using the $(\sigma,\epsilon)$-superlinearity of $\beta$, we obtain a "dyadic" version of the $(\delta,\sigma)$-set property:
\begin{equation}\label{equ-11}
|\mathcal{D}_{\delta}(\mathcal{F}\cap \mathbf{F})| \stackrel{\eqref{form41}}{\lesssim_{\mathfrak{T}}} \Delta^{\beta(j)}|\mathcal{D}_\delta(\mathcal{F})| \leq \Delta^{\sigma j-\epsilon m}|\mathcal{D}_\delta(\mathcal{F})|=\delta^{-\epsilon} (\Delta^j)^\sigma |\mathcal{D}_\delta(\mathcal{F})|.
\end{equation}
To prove the non-dyadic version, fix $f \in \mathcal{F}$ and and $r\in [\delta, 1]$. Take $k$ such that $r\in (\Delta^{k+1}, \Delta^k]$, then $\mathcal{F}\cap B(f,r)$ can be covered by $O_{\mathfrak{T}}(1)$ dyadic cubes with side length $\Delta^k$. Let $\mathbf{F}_\alpha^k$ be the $\Delta^k$-cube such that $|\mathcal{D}_\delta(\mathcal{F}\cap \mathbf{F}_\alpha^k)|$ attains maximum cardinality, then
\[|\mathcal{D}_\delta(\mathcal{F}\cap B(f,r))|\lesssim_{\mathfrak{T}} |\mathcal{D}_\delta(\mathcal{F}\cap \mathbf{F}_\alpha^k)|\stackrel{\eqref{equ-11}}{\lesssim_{\mathfrak{T}}}\delta^{-\epsilon}\Delta^{k\sigma} |\mathcal{D}_\delta(\mathcal{F})|\leq \delta^{-\epsilon}\Delta^{-\sigma} r^\sigma |\mathcal{D}_\delta(\mathcal{F})|.\]
Combining this with Corollary \ref{cor-covernumber}, we conclude that $\mathcal{F}$ is a $(\delta, \sigma, O_{\mathfrak{T}}(\Delta^{-\sigma}\delta^{-\epsilon}))$-set.
	
Now assume $\mathcal{F}$ is a $(\delta, s, \delta^{-\epsilon})$-set for some $s \in [0,2]$, then we prove \eqref{equ-12}. Since $\beta$ is piecewise linear, it suffices to prove \eqref{equ-12} for $x=j\in \{1,2,\ldots,m\}$. By using $(\delta, s, \delta^{-\epsilon})$-set condition and noting $\delta^{-\epsilon}=\Delta^{-m\epsilon}$, we have $|\mathcal{D}_\delta(\mathcal{F}\cap \mathbf{F})|\lesssim_\mathfrak{T} \Delta^{-m\epsilon} \Delta^{sj}|\mathcal{D}_\delta(\mathcal{F})|$ for any $\mathbf{F}\in \mathcal{D}_{\Delta^j}(\mathcal{F})$. Thus
\[\Delta^{-\beta(j)} \sim_{\mathfrak{T}} |\mathcal{D}_{\Delta^j}(\mathcal{F})|=\frac{|\mathcal{D}_\delta(\mathcal{F})|}{|\mathcal{D}_\delta(\mathcal{F}\cap \mathbf{F})|}\gtrsim_\mathfrak{T}\Delta^{\epsilon m -js}.\]
If we choose $\Delta_0=\Delta_0(\mathfrak{T},\epsilon) > 0$ small enough, this implies $\beta(j)\geq js-2\epsilon m$.
	
We next prove (ii). From (i), we already know $\mathcal{F}$ is a $(\delta, \sigma, O_{\mathfrak{T}}(\Delta^{-\sigma}\delta^{-\epsilon}))$-set. It therefore remains to prove the upper bound 
\begin{equation}\label{form42} |\mathcal{F} \cap B(x,R)|_{r} \lesssim_{\mathfrak{T}} \Delta^{-2\sigma}\delta^{-2\epsilon}(R/r)^{\sigma}, \quad x \in \mathcal{F}, \, \delta \leq r \leq R \leq 1. \end{equation}
By the $(\sigma,\epsilon)$-linearity of $\beta$ on $[0,m]$, we first deduce that, for $j,k \in [0,m] \cap \Z$ with $j \leq k$,
\begin{displaymath} \beta(k) - \beta(j) = (\beta(k) - \sigma k) - (\beta(j) - \sigma j) + \sigma(k - j) \leq 2\epsilon m + \sigma (k - j). \end{displaymath}
Plugging this into \eqref{form41} gives the following dyadic variant of \eqref{form42}:
\begin{displaymath} |\mathcal{D}_{\Delta^{k}}(\mathcal{F} \cap \mathbf{F})| \lesssim_{\mathfrak{T}} \delta^{-2\epsilon} \Delta^{(j - k)\sigma}, \qquad \mathbf{F} \in \mathcal{D}_{\Delta^{j}}(\mathcal{F}), \, k \in \{j,\ldots,m\}. \end{displaymath} 
Now \eqref{form42} follows for $x \in \mathcal{F}$ and $\delta \leq r \leq R \leq 1$ by choosing integers $0 \leq k \leq j \leq m$ in such a way that $\Delta^{j + 1} \leq r \leq \Delta^{j}$ and $\Delta^{k + 1} \leq R \leq \Delta^{k}$. We leave the details to the reader. \end{proof}

\section{High-low incidence estimates}\label{sec3}
\begin{definition}[Vertical $r$-neighbourhood]\label{def:vertical} Let $I \subset \R$ be an interval, $f\in C^2(I)$, $r>0$. The \emph{vertical $r$-neighborhood of $f$} is the set
\[\Gamma_f(r):=\{(x,y)\in I \times \R :|y-f(x)|< r\}.\]
\end{definition}

\begin{definition}[Incidences]\label{def:incidence} Let $\delta \in (0,1]$. Given a finite family $\mathcal{F}\subset C^{2}([0,1])$ and a finite family of ordinary or dyadic $\delta$-squares $\mathcal{P}$, define
\[\mathcal{I}^{\lambda}(\mathcal{F},\mathcal{P}):= |\{(f,p)\in\mathcal{F}\times \mathcal{P}: z_{p} \in \Gamma_f(\lambda\delta)\}|, \qquad \lambda \geq 2, \]
where $z_{p}$ is the centre of $p$. We note that while the parameter "$\delta$" is not visible in the notation $\mathcal{I}^{\lambda}(\mathcal{F},\mathcal{P})$, it can be deduced from the side-length of the squares in $\mathcal{P}$. \end{definition}

\begin{remark} The parameter $\lambda$ is an auxiliary "thickening" parameter, which may be viewed as an absolute constant (we could take $\lambda = 10$ everywhere in the paper). \end{remark}

We generalise the definition to allow "weight functions":
\begin{definition}[Weighted incidences]\label{def:weightedIncidence} Let $\delta \in (0,1]$. Let $\mathcal{F} \subset C^{2}([0,1])$ be a finite family, and let $\mathcal{P}$ be a finite family of ordinary or dyadic $\delta$-squares. Let $w_{\mathcal{F}} \colon \mathcal{F} \to [0,\infty)$ and $w_{\mathcal{P}} \colon \mathcal{P} \to [0,\infty)$ be functions. 
\begin{displaymath} \mathcal{I}_{w}^{\lambda}(\mathcal{F},\mathcal{P}) := \sum_{f \in \mathcal{F}} \sum_{p \in \mathcal{P}} w_{\mathcal{F}}(f)w_{\mathcal{P}}(p)\mathbf{1}_{\{z_{p} \in \Gamma_{f}(\lambda \delta)\}}. \end{displaymath}  \end{definition}

\begin{proposition}\label{p:highLowUnweighted}
Let $\delta\in 2^{-\mathbb N}$ and $\lambda \in [2,\tfrac{1}{2}\delta^{-1}]$. Let $\mathcal{F} \subset B_{C^{2}}(1)$ be a $\delta$-separated transversal family on $[-2,2]$ with constant $\mathfrak{T} \geq 1$. Let $\mathcal{P}$ be a finite disjoint family of $\delta$-squares contained in $[0,1]^2$. Then, for $S\in [2\lambda ,\delta^{-1}]$,
\begin{equation}\label{hlestimate}
\mathcal{I}^{\lambda}(\mathcal{F},\mathcal{P}) \lesssim_{\lambda} \mathbf{C}\left(S^{3}\delta^{-1} |\mathcal{F}||\mathcal{P}|\right)^{1/2}+ S^{-1}\mathcal{I}^{2}(\mathcal{F},\mathcal{P}^{S\delta}).
\end{equation}
Here $0 < \mathbf{C} \lesssim_{\mathfrak{T}} \log(1/\delta)$, and $\mathcal{I}^{2}(\mathcal{F},\mathcal{P}^{S\delta}) := |\{(f,p) \in \mathcal{F} \times \mathcal{P} : z_{p} \in \Gamma_{f}(2 S\delta)\}|$. \end{proposition}

Proposition \ref{p:highLowUnweighted} easily implies a superficially stronger weighted version:

\begin{proposition}\label{pro-highlow-curve}
Let $\delta\in 2^{-\mathbb N}$, $\lambda \in [2,\tfrac{1}{2}\delta^{-1}]$. Let $\mathcal{F} \subset B_{C^{2}}(1)$ be a $\delta$-separated transversal family over $[-2,2]$ with constant $\mathfrak{T}$. Let $\mathcal{P}$ be a finite disjoint family of $\delta$-squares contained in $[0,1]^{2}$. Let $C \geq 1$, and let $w_{\mathcal{F}} \colon \mathcal{F} \to [\delta^{C},\delta^{-C})$ and $w_{\mathcal{P}} \colon \mathcal{P} \to [\delta^{C},\delta^{-C})$ be functions. Then, for $S\in [2\lambda,\delta^{-1}]$, 
\begin{equation}\label{hlestimate11}
\mathcal{I}^{\lambda}_{w}(\mathcal{F},\mathcal{P}) \lessapprox_{C,\lambda} \mathbf{C}\Big(S^{3}\delta^{-1} \sum_{f\in\mathcal{F}}w_{\mathcal{F}}(f)^2 \sum_{p\in\mathcal{P}}w_{\mathcal{P}}(p)^2\Big)^{1/2}+S^{-1}\mathcal{I}^{2}_{w}(\mathcal{F},\mathcal{P}^{S\delta}).
\end{equation}
Here $0 < \mathbf{C} \lesssim_{\mathfrak{T}} \log(1/\delta)$.
\end{proposition}

\begin{remark} In \eqref{hlestimate11}, 
\begin{displaymath} \mathcal{I}^{2}_{w}(\mathcal{F},\mathcal{P}^{S\delta}) \stackrel{\mathrm{def.}}{=} \sum_{f \in \mathcal{F}} \sum_{p \in \mathcal{P}} w_{\mathcal{F}}(f)w_{\mathcal{P}}(p)\mathbf{1}_{\{z_{p} \in \Gamma_{f}(2 S\delta)\}}. \end{displaymath} \end{remark} 

In the sequel, we sometimes omit "$\lambda$" from the notation $\mathcal{I}_{w}^{\lambda}$ if it is clear from context.

\begin{proof}[Proof of Proposition \ref{pro-highlow-curve} assuming Proposition \ref{p:highLowUnweighted}] Since $w_{\mathcal{F}},w_{\mathcal{P}}$ take values in the interval $[\delta^{C},\delta^{-C}]$, the pigeonhole principle allows us to find dyadic values $\mathbf{w}_{\mathcal{F}},\mathbf{w}_{\mathcal{P}} \in [\tfrac{1}{2}\delta^{C},\delta^{-C}]$ and subsets $\overline{\mathcal{F}} \subset \mathcal{F}$ and $\overline{\mathcal{P}} \subset \mathcal{P}$ with the following properties:
\begin{itemize}
\item[(a)] $w_{\mathcal{F}}(f) \in [\mathbf{w}_{\mathcal{F}},2\mathbf{w}_{\mathcal{F}}]$ for all $f \in \overline{\mathcal{F}}$,
\item[(b)] $w_{\mathcal{P}}(p) \in [\mathbf{w}_{\mathcal{P}},2\mathbf{w}_{\mathcal{P}}]$ for all $p \in \overline{\mathcal{P}}$,
\end{itemize} 
and
\begin{displaymath} \mathcal{I}_{w}^{\lambda}(\mathcal{F},\mathcal{P}) \lessapprox_{C} \mathcal{I}_{w}^{\lambda}(\overline{\mathcal{F}},\overline{\mathcal{P}}) \sim \mathbf{w}_{\mathcal{F}}\mathbf{w}_{\mathcal{P}} \cdot \mathcal{I}^{\lambda}(\overline{\mathcal{F}},\overline{\mathcal{P}}). \end{displaymath}
We now fix $S \in [10,\delta^{-1}]$, and apply Proposition \ref{p:highLowUnweighted} to $\mathcal{I}(\overline{\mathcal{F}},\overline{\mathcal{P}})$:
\begin{equation}\label{to12} \mathcal{I}^{\lambda}_{w}(\mathcal{F},\mathcal{P}) \lessapprox_{C,\lambda} \mathbf{C} (S^{3}\delta^{-1} \cdot \mathbf{w}_{\mathcal{F}}^{2}\mathbf{w}_{\mathcal{P}}^{2} \cdot |\overline{\mathcal{F}}||\overline{\mathcal{P}}|)^{1/2} + \mathbf{w}_{\mathcal{F}}\mathbf{w}_{\mathcal{P}} \cdot S^{-1}\mathcal{I}^{2}(\overline{\mathcal{F}},\overline{\mathcal{P}}^{S\delta}), \end{equation}
where $\mathbf{C} \lessapprox_{\mathfrak{T}} 1$. Now it remains to note that
\begin{displaymath} \mathbf{w}_{\mathcal{F}}^{2}\mathbf{w}_{\mathcal{P}}^{2} \cdot |\overline{\mathcal{F}}||\overline{\mathcal{P}}| \leq \sum_{f \in \overline{\mathcal{F}}} w_{\mathcal{F}}^{2}(f) \sum_{p \in \overline{\mathcal{P}}} w_{\mathcal{P}}^{2}(p) \leq \sum_{f \in \mathcal{F}} w_{\mathcal{F}}^{2}(f) \sum_{p \in \mathcal{P}} w_{\mathcal{P}}^{2}(p), \end{displaymath}
and
\begin{displaymath} \mathbf{w}_{\mathcal{F}}\mathbf{w}_{\mathcal{P}} \cdot \mathcal{I}(\overline{\mathcal{F}},\overline{\mathcal{P}}^{S\delta}) \leq 4 \sum_{f \in \overline{\mathcal{F}}} \sum_{p \in \overline{\mathcal{P}}} w_{\mathcal{F}}(f)w_{\mathcal{P}}(p)\mathbf{1}_{\{z_{p} \in \Gamma_{f}(2 S\delta)\}} \leq 4\mathcal{I}_{w}^{2}(\mathcal{F},\mathcal{P}^{S\delta}). \end{displaymath}
Plugging these estimates into \eqref{to12} completes the proof of Proposition \ref{pro-highlow-curve}. \end{proof}

We then turn to the proof of Proposition \ref{p:highLowUnweighted}.

\begin{proof}[Proof of Proposition \ref{p:highLowUnweighted}] Note that if $z_p\in \Gamma_f(\lambda \delta)$ for some $f \in \mathcal{F} \subset B_{C^{2}}(1)$, then $p \subset \Gamma_{f}(2\lambda \delta)$ (note that $f$ is $1$-Lipschitz and use the triangle inequality), and therefore
\begin{equation}\label{to1} \mathcal{I}^{\lambda}(\mathcal{F},\mathcal{P}) \lesssim \lambda \cdot \delta^{-1} \int \sum_{f \in \mathcal{F}} (2\lambda \delta)^{-1} \mathbf{1}_{\Gamma_{f}(2\lambda \delta)} \sum_{p \in \mathcal{P}} \mathbf{1}_{p}. \end{equation} 
For $f \in \mathcal{F}$ fixed, we decompose $(2\lambda \delta)^{-1}\mathbf{1}_{\Gamma_{f}(2\lambda \delta)} := H_{f} + L_{f}$, where
\begin{equation}\label{to2} H_{f} := (2\lambda \delta)^{-1}\mathbf{1}_{\Gamma_{f}(2\lambda \delta)} - (S\delta)^{-1}\mathbf{1}_{\Gamma_{f}(S\delta)} \quad \text{and} \quad L_{f} := (S\delta)^{-1}\mathbf{1}_{\Gamma_{f}(S\delta)}. \end{equation} 
Note that $S\delta \geq 10\delta$ by hypothesis. Continuing from \eqref{to1},
\begin{equation}\label{to13} \mathcal{I}(\mathcal{F},\mathcal{P}) \lesssim_{\lambda} \delta^{-1} \int \sum_{f \in \mathcal{F}} H_{f} \sum_{p \in \mathcal{P}} \mathbf{1}_{p} +  \delta^{-1} \int \sum_{f \in \mathcal{F}} L_{f} \sum_{p \in \mathcal{P}} \mathbf{1}_{p} =: H + L. \end{equation}
We now proceed to prove separate estimates for the terms $H$ and $L$.
	
\subsection*{Estimating $L$} We note that if $p \cap \Gamma_{f}(S\delta) \neq \emptyset$, then $z_{p} \in p \subset \Gamma_{f}(2 S\delta)$. Consequently,
\[\begin{split}
L = \sum_{f \in \mathcal{F}} \sum_{p \in \mathcal{P}}\delta^{-1}\int  (S\delta)^{-1}\mathbf{1}_{\Gamma_{f}(S\delta)} \mathbf{1}_{p}\leq S^{-1}|\{(f,p) : z_{p} \in \Gamma_f(2 S\delta)\}|=S^{-1}\mathcal{I}^{2}(\mathcal{F},\mathcal{P}^{S\delta}).
\end{split}\]
Thus, $L$ gives rise to the second term in \eqref{hlestimate}.
	
\subsection*{Estimating $H$} Here we claim that
\begin{equation}\label{to14} H \lesssim_{\mathfrak{T}} \log(1/\delta)(S^{3}\delta^{-1}|\mathcal{F}||\mathcal{P}|)^{1/2}. \end{equation}
Since all the squares $p \in \mathcal{P}$ are contained in $[0,1]^{2}$, the term $H$ remains unchanged if we multiply the integrand in \eqref{to13} by a cut-off function $\varphi \in \mathrm{Lip}(\R^{2})$ with the properties $\mathbf{1}_{[0,1]^{2}} \leq \varphi \leq \mathbf{1}_{[-3/2,3/2]^{2}}$. Using also Cauchy-Shcwarz, and the disjointness of $\mathcal{P}$,
\begin{align} H^{2} = \delta^{-2} \left( \int \sum_{f \in \mathcal{F}} \varphi H_{f} \sum_{p \in \mathcal{P}} \mathbf{1}_{p} \right)^{2} & \leq \delta^{-2} \int \Big( \sum_{f \in \mathcal{F}} \varphi H_{f} \Big)^{2} \int \left( \sum_{p \in \mathcal{P}} \mathbf{1}_{p} \right)^{2} \notag\\
&\label{to15} = |\mathcal{P}| \sum_{f,g \in \mathcal{F}}\int (\varphi H_{f})(\varphi H_{g}). \end{align}
We claim that 
\begin{equation}\label{to9} \left| \int (\varphi H_{f})(\varphi H_{g}) \right| \lesssim_{\mathfrak{T}} \frac{S^{3}\delta}{[d(f,g) + S\delta]^{2}}, \qquad f,g \in \mathcal{F}. \end{equation}
To prove \eqref{to9}, consider first the case where $d(f,g) < 4\mathfrak{T}S\delta$. Then also $d(f,g) + S\delta \leq 5\mathfrak{T}S\delta$. In this case, using $\spt \varphi \subset [-3/2,3/2]^{2}$ and $\max\{\|H_{f}\|_{L^{\infty}},\|H_{g}\|_{L^{\infty}}\} \lesssim \delta^{-1}$,
\begin{displaymath} \left| \int (\varphi H_{f})(\varphi H_{g}) \right| \lesssim \delta^{-2}\mathcal{H}^{2}(\spt H_{f} \cap [-\tfrac{3}{2},\tfrac{3}{2}]^{2}) \lesssim S\delta^{-1} \lesssim \mathfrak{T}^{2}\frac{S^{3}\delta}{[d(f,g) + S\delta]^{2}}.  \end{displaymath} 
	
We then proceed to consider the "main" case where $f,g \in \mathcal{F}$ with $d(f,g) > 4\mathfrak{T}S\delta$. Recall the change-of-variable formula \cite[Theorem 7.26]{MR210528} for Lebesgue integrals in $\R^{d}$:
\begin{displaymath} \int h = \int (h \circ \Phi)|J_{\Phi}|, \qquad h \in L^{1}(\R^{d}), \end{displaymath}
where $\Phi \colon \R^{d} \to \R^{d}$ is assumed to be differentiable and one-to-one. We apply this formula to $h := (\varphi H_{f})(\varphi H_{g}) \in L^{1}(\R^{2})$ and $ \Phi(x,y) = (x,f(x) + y)$. In this case $|J_{\Phi}| \equiv 1$, so 
\begin{displaymath} \int (\varphi H_{f})(\varphi H_{g}) = \int_{\R} \int_{\R} (\varphi H_{f})(x,f(x) + y)(\varphi H_{g})(x,f(x) + y) \, dx \, dy. \end{displaymath}
Recall the form of the functions $H_{f}$ and $H_{g}$ from \eqref{to2}, and observe that
\begin{displaymath} H_{f}(x,f(x) + y) = 0, \qquad |y| \geq S\delta. \end{displaymath}
Therefore,
\begin{equation}\label{to3} \int (\varphi H_{f})(\varphi H_{g}) = \int_{-S\delta}^{S\delta} \int_{\R} (\varphi H_{f})(x,f(x) + y)(\varphi H_{g})(x,f(x) + y) \, dx \, dy. \end{equation}
Now, \eqref{to9} will follow once we manage to show that
\begin{equation}\label{to10} \left| \int_{\R} (\varphi H_{f})(x,f(x) + y)(\varphi H_{g})(x,f(x) + y) \, dx \right| \lesssim_{\mathfrak{T}} \frac{S^{2}}{d(f,g)^{2}}, \quad y \in [-S\delta,S\delta]. \end{equation} 
We need the following lemma about the set where the integrand does not vanish:

\begin{lemma}\label{lemma1} Assume that $y \in [-S\delta,S\delta]$ and $d(f,g) > 4\mathfrak{T}S\delta$. Then, the set
	\begin{displaymath} E_{f,g}(y) := \{x \in (-2,2) : |f(x) + y - g(x)| < S\delta\} \end{displaymath}
	is the union of $O(\mathfrak{T})$ open intervals $\mathcal{I}_{f,g}(y)$ with the following properties:
	\begin{itemize}
		\item[\textup{(1)}] $|I| \leq 2S\delta/d(f,g) < 1/2$ for all $I \in \mathcal{I}_{f,g}(y)$.
		\item[\textup{(2)}] If $I,J\in \mathcal{I}_{f,g}(y)$ are distinct, then $\dist(I,J) \geq \mathfrak{T}^{-1}$. 
		\item[\textup{(3)}] For $I \in \mathcal{I}_{f,g}(y)$ fixed, the derivative $(f - g)'$ is non-vanishing on $I$.
		\item[\textup{(4)}] If $I = (a,b) \in \mathcal{I}_{f,g}(y)$ with $[a,b] \subset (-2,2)$, then either 
		\begin{displaymath} \begin{cases} f(a) + y - g(a) = -S\delta,\\ f(b) + y - g(b) = S\delta, \end{cases} \quad \text{or} \quad \begin{cases} f(a) + y - g(a) = S\delta,\\ f(b) + y - g(b) = -S\delta. \end{cases} \end{displaymath}
		\item[\textup{(5)}] For $I \in \mathcal{I}_{f,g}(y)$ fixed, there exists a constant $a_{I} \in \R$ with $|a_{I}| \geq (2\mathfrak{T})^{-1}d(f,g)$ such that
		\begin{displaymath} |(f - g)'(x) - a_{I}| \leq 2S\delta, \qquad x \in I. \end{displaymath} 
	\end{itemize} 
\end{lemma}

\begin{proof} The set $E_{f,g}(y)$ is open, so it equals the union of its (open) connected components which we denote by $\mathcal{I}_{f,g}(y)$, then $\mathcal{I}_{f,g}(y)$ is a countable set of disjoint open intervals. We then prove the properties (1)-(4).
	
	Fix $I \in \mathcal{I}_{f,g}(y)$, and $x \in \bar{I}$ (the closure), and note that then $|f(x) - g(x)| \leq 2S\delta$. It therefore follows from the transversality hypothesis that
	\begin{equation}\label{to4} |f'(x) - g'(x)| \geq \mathfrak{T}^{-1}d(f,g)-2S\delta \geq \max\{2S\delta, (2\mathfrak{T})^{-1}d(f,g)\}, \quad x\in \bar{I}. \end{equation}
	Since $f',g' \in C(-2,2)$, we infer that $h' = f' - g'$ does not change sign on $I$, and $|I| \leq 2S\delta/d(f,g)$. This proves (1) and (3). 
	
	To prove (4), recall that $I = (a,b)$ is a component of the open set $E_{f,g}(y)$, so $a,b \notin E_{f,g}(y)$. Since on the other hand $a,b \in (-2,2)$ by assumption, this implies  
	\begin{displaymath} |f(a) + y - g(a)| = S\delta = |f(b) + y - g(b)|. \end{displaymath}
	Since $x \mapsto f(x) + y - g(x)$ is strictly monotone on $I$ by (3), we may conclude that one of the two alternatives in (4) must hold.
	
	We then prove property (2). Let $I=(a,b), J=(c,d)\in \mathcal{I}_{f,g}(y)$ with $b\leq c$. By the same argument in Lemma \ref{lemma4}, we can get $\dist(I,J)\geq \mathfrak{T}^{-1}$. This then implies $|\mathcal{I}_{f,g}(y)|\leq 5\mathfrak{T}$.
	
	Finally, to prove (5), fix $x_{0} \in I$ arbitrarily, and let $a_{I} := h'(x_{0}) = f'(x_{0}) - g'(x_{0})$. The lower bound $|a_{I}| \geq (2\mathfrak{T})^{-1}d(f,g)$ follows from \eqref{to4}. From (1), and $\|h''\|_{L^{\infty}} \leq d(f,g)$, it follows that
	\begin{displaymath} |h'(x) - a_{I}| \leq |I| \cdot \|h''\|_{L^{\infty}} \leq 2S\delta, \qquad x \in I. \end{displaymath} 
	This completes the proof of the lemma. \end{proof}
	
We then apply the lemma to the inner integral on line \eqref{to3}, with $y \in [-S\delta,S\delta]$ fixed. If $H_{g}(x,f(x) + y) \neq 0$, then $(x,f(x) + y) \in \Gamma_{g}(S\delta)$, which implies $|f(x) + y - g(x)| < S\delta$. Therefore, the set of points $x \in \R$ for which the inner integral in \eqref{to3} is non-vanishing, is contained in the union of the intervals $I \in \mathcal{I}_{f,g}(y)$:
\begin{align*} \int_{\R} & (\varphi H_{f})(x,f(x) + y)(\varphi H_{g})(x,f(x) + y) \, dx\\
& = \sum_{I \in \mathcal{I}_{f,g}(y)} \int_{I} \varphi^{2}(x,f(x) + y) H_{f}(x,f(x) + y)H_{g}(x,f(x) + y) \, dx. \end{align*} 
Recall that $\spt \varphi \subset [-\tfrac{3}{2},\tfrac{3}{2}]^{2}$, so the only non-zero terms in the sum correspond to intervals $I \in \mathcal{I}_{f,g}(y)$ with $I \cap [-\tfrac{3}{2},\tfrac{3}{2}] \neq \emptyset$. Since $|I| < 1/2$ by Lemma \ref{lemma1}(1), we conclude that $\bar{I} \subset (-2,2)$ for all such intervals. (Thus, the information in Lemma \ref{lemma1}(4) will be applicable.) Furthermore, by Lemma \ref{lemma1}(2), the number of the intervals is bounded from above by $\lesssim \mathfrak{T}$. Therefore, \eqref{to10} will follow once we manage to prove a similar estimate for each $I \in \mathcal{I}_{f,g}(y)$ with $\bar{I} \subset (-2,2)$: for such intervals we claim that
\begin{equation}\label{to11} \left| \int_{I} \varphi^{2}(x,f(x) + y) H_{f}(x,f(x) + y)H_{g}(x,f(x) + y) \, dx \right| \lesssim \frac{S^{2}}{d(f,g)^{2}}. \end{equation}
To prove \eqref{to11}, fix $I \in \mathcal{I}_{f,g}(y)$, pick a point $x_{I} \in I$, and write
\begin{align*} \int_{I} & \varphi^{2}(x,f(x) + y) H_{f}(x,f(x) + y)H_{g}(x,f(x) + y) \, dx\\
& = \int_{I} (\varphi^{2}(x,f(x) + y) - \varphi^{2}(x_{I},f(x_{I}) + y))H_{f}(x,f(x) + y)H_{g}(x,f(x) + y) \, dx\\
&\quad + \varphi^{2}(x_{I},f(x_{I}) + y) \int_{I} H_{f}(x,f(x) + y)H_{g}(x,f(x) + y) \, dx =: \Sigma_{1}(I) + \Sigma_{2}(I). \end{align*}
Since $\varphi^{2}$ is $O(1)$-Lipschitz, we find from Lemma \ref{lemma1}(1) that
\begin{displaymath} |\Sigma_{1}(I)| \lesssim |I|^{2} \cdot \|H_{f}\|_{L^{\infty}}\|H_{g}\|_{L^{\infty}} \lesssim \left(\frac{S\delta}{d(f,g)} \right)^{2} \cdot \delta^{-2} = \frac{S^{2}}{d(f,g)^{2}}.  \end{displaymath} 
It remains to show that also $|\Sigma_{2}(I)| \lesssim_{\mathfrak{T}} S^{2}/d(f,g)^{2}$. Since $|\varphi^{2}(x_{I},f(x_{I}) + y)| \leq 1$, this will follow from
\begin{equation}\label{to7} \left| \int_{I} H_{f}(x,f(x) + y)H_{g}(x,f(x) + y) \, dx \right| \lesssim \frac{S^{2}}{d(f,g)^{2}}. \end{equation}
Notice first, using \eqref{to2}, that the function $x \mapsto H_{f}(x,f(x) + y)$ on $I \subset (-2,2)$ is a constant with absolute value no larger than $(10\delta)^{-1}$. Therefore,
\begin{equation}\label{to8} \left| \int_{I} H_{f}(x,f(x) + y)H_{g}(x,f(x) + y) \, dx \right| \lesssim \delta^{-1} \left| \int_{I} H_{g}(x,f(x) + y) \, dx \right|. \end{equation} 
To analyse the final integral further, plug in the definition \eqref{to2} of $H_{g}$ to arrive at
\begin{align}\label{to5}  \int_{I} H_{g}(x,f(x) + y) \, dx & = (2\lambda\delta)^{-1}|\{x \in I : |f(x) + y - g(x)| < 2\lambda\delta\}|\\
&\label{to6}\quad - (S\delta)^{-1}|\{x \in I : |f(x) + y - g(x)| < S\delta\}|. \end{align} 
By Lemma \ref{lemma1}(3), the function $x \mapsto f(x) + y - g(x)$ is strictly monotone on $I$, so both the sets in \eqref{to5} and \eqref{to6} are intervals. In fact, the set in \eqref{to6} is precisely the interval $I =: (a,b)$ itself, whereas the set in \eqref{to5} is an open subinterval $J := (c,d) \subset I$. We assume, for example, that $x \mapsto f(x) + y - g(x)$ is strictly increasing. Then, by Lemma \ref{lemma1}(4) (and recalling that $\bar{I} \subset (-2,2)$),
\begin{displaymath} f(a) + y - g(a) = -S\delta \quad \text{and} \quad f(b) + y - g(b) = S\delta. \end{displaymath}
Similarly, because $\bar{J} \subset \bar{I} \subset (-2,2)$, and $x \mapsto f(x) + y - g(x)$ is strictly increasing on $I$,
\begin{displaymath} f(c) + y - g(c) = -2\lambda\delta \quad \text{and} \quad f(d) + y - g(d) = 2\lambda\delta. \end{displaymath}
Since by Lemma \ref{lemma1}(4) we furthermore know that $(f - g)' = a_{I} + O(S\delta)$ on $I$ (in particular $J$), we may infer that
\begin{displaymath} |J| \leq \frac{1}{a_{I} - O(S\delta)} \int_{c}^{d} \partial_{x}(f(x) + y - g(x)) \, dx = \frac{4\lambda \delta}{a_{I} - O(S\delta)}, \end{displaymath} 
and similarly $|J| \geq 4\lambda \delta/(a_{I} + O(S\delta))$. Altogether $|J| = 4\lambda \delta/(a_{I} + C_{J})$ for some constant $|C_{J}| = O(S\delta)$. By the same argument $|I| = 2S\delta/(a_{I} + C_{I})$, where $|C_{I}| = O(S\delta)$. Plugging these (approximate) equations back to \eqref{to5}-\eqref{to6}, we find
\begin{displaymath} \left|\int_{I} H_{g}(x,f(x) + y) \, dx \right| = \left|\frac{2}{a_{I} + C_{J}} - \frac{2}{a_{I} + C_{I}} \right| \leq \frac{2|C_{I}| + 2|C_{J}|}{(a_{I} + C_{I})(a_{I} + C_{J})} \lesssim_{\mathfrak{T}} \frac{S\delta}{d(f,g)^{2}}, \end{displaymath} 
using the lower bound $|a_{I}| \geq (2\mathfrak{T})^{-1}d(f,g)$ from Lemma \ref{lemma1}(4) in the final inequality. Since $S \geq 1$, and taking \eqref{to8} into account, this proves \eqref{to7}, and therefore \eqref{to9}. 

As a last step, we use \eqref{to9} to prove \eqref{to14}. We first plug \eqref{to9} into \eqref{to15}:
\[\begin{split}
H^2\lesssim_{\mathfrak{T}}|\mathcal{P}|\sum_{f\in\mathcal{F}}\sum_{\substack{g\in \mathcal{F}\\d(f,g)\leq 4\mathfrak{T}S\delta}} S\delta^{-1}+|\mathcal{P}|\sum_{f\in\mathcal{F}}\sum_{\substack{g\in \mathcal{F}\\d(f,g)> 4\mathfrak{T}S\delta}}\frac{S^{3}\delta}{d(f,g)^{2}} =: \Sigma_3+\Sigma_4.
\end{split}\]
For $\Sigma_3$, we use the upper $(2,20\mathfrak{T}^{2})$-regularity of $\mathcal{F}$ (recall Lemma \ref{lem-Ahlforsregularity}), and the assumption that $\mathcal{F}$ is $\delta$-separated, to get
\[\Sigma_3\lesssim_\mathfrak{T} S\delta^{-1}|\mathcal{P}||\mathcal{F}| \Big(\frac{S\delta}{\delta}\Big)^2 = S^3\delta^{-1}|\mathcal{P}||\mathcal{F}|.\]
Similarly, using that $\mathcal{F}$ is $\delta$-separated, upper $2$-regular, and $\diam_{C^{2}}(\mathcal{F}) \leq 1$, we may estimate $\Sigma_4$ as follows:
\[\begin{split}
\Sigma_4& \leq |\mathcal{P}|\sum_{f\in\mathcal{F}}\sum_{\substack{r\in [4\mathfrak{T}S\delta,1]\\r\in 2^{-\N}}}\sum_{\substack{g\in \mathcal{F}\\d(f,g)\in [r,2r]}}\frac{S^{3}\delta}{d(f,g)^{2}}\\&\lesssim_\mathfrak{T} S^3\delta |\mathcal{P}||\mathcal{F}| \sum_r r^{-2} \left(\frac{r}{\delta}\right)^2
\lesssim \log(1/\delta)S^3 \delta^{-1}|\mathcal{P}||\mathcal{F}|.
\end{split}\]
Combining the estimates for $\Sigma_3$ and $\Sigma_4$, we deduce \eqref{to14}. This concludes the proof of Proposition \ref{p:highLowUnweighted}. \end{proof}


\section{Projection theorem}\label{sec4}
In this section, $\mathcal{F}\subset C^2(I)$ will be a $\delta$-separated transversal family over $I$ with constant $\mathfrak{T}\geq1$, where $I\subset \R$ is a compact interval. The constant $\mathfrak{T}$ is viewed as absolute, so the constants in the statements below are allowed to depend on $\mathfrak{T}$. If it is worth emphasising $\mathfrak{T}$, we will say that $\mathcal{F}$ is a $\mathfrak{T}$-transversal family. 

Our goal is to prove the following lower bound for the $\delta$-covering number of vertical slices of $\lbrace \Gamma_f:\ f\in\mathcal F\rbrace$. Recall Definition \ref{def-regularme} and Definition \ref{def-regularse}.

\begin{thm}\label{proj-regular}
Let $s\in (0,1]$, $t\in[s,2]$. For every $u\in [0, \min\{(s+t)/2,1\})$, there exist $\epsilon, \delta_0>0$ such that the following holds for all $\delta\in(0,\delta_0]$. Assume that $\mathcal{F}\subset B(1)$ is $\delta$-separated and $(\delta,t,\delta^{-\epsilon})$-regular, and let $E\subset I$ be a $(\delta,s,\delta^{-\epsilon})$-set. Then, there exists $E'\subset E$ with $|E'|_\delta\geq \tfrac{1}{2}|E|_\delta$ such that for any $\theta\in E'$:
\begin{equation}\label{form-regular}
\Bigg|\bigcup_{f\in\mathcal{F}'} \Gamma_f\cap L_\theta\Bigg|_\delta \geq \delta^{-u},\quad \mathcal{F}'\subset\mathcal{F}, \quad|\mathcal{F}'|_\delta\geq \delta^\epsilon |\mathcal{F}|_\delta.
\end{equation}
\end{thm}
Instead of proving Theorem \ref{proj-regular} directly, we will deduce it from a measure estimate on certain \emph{high multiplicity sets} which is more convenient to prove. To proceed, we first introduce some terminology.

\begin{definition}[Multiplicity] 
Let $K \subset C^{2}(I)$, let $0 < r \leq R \leq \infty$ be dyadic numbers, and let $f \in \mathcal{F}$. For $\theta \in I$, we define the following \emph{multiplicity number}:
\begin{displaymath} 
\m_{K,\theta}(f, [r,R]) := |\{g \in K \cap B(f,R) : |f(\theta) - g(\theta)| \leq r\}|_{r}. \end{displaymath} 
\end{definition}

The quantity $\m_{K,\theta}(f, [r,R])$ captures the number, at scale $r$, of those functions whose values at $\theta$ coincide with $f(\theta)$, at resolution $r$.  

\begin{definition}[High multiplicity set] 
Let $0 < r \leq R \leq \infty$, $M > 0$, and let $\theta \in I$. For $K \subset C^{2}(I)$, we define the \emph{high multiplicity set}
\begin{displaymath} 
H_{\theta}(K,M,[r,R]) := \{f \in K : \m_{K,\theta}(f , [r,R]) \geq M\}. 
\end{displaymath}  
\end{definition}

Recall the rescaling map $T_{f_0, r_0}: f \mapsto (f - f_0)/r_0$. The following lemma discusses how the high multiplicity sets are affected by rescalings. The proof follows from the definition of $H_{\theta}(\cdots)$ and we leave it to the reader.

\begin{lemma}\label{lemma3} Let $K \subset C^{2}(I)$ be arbitrary, and let $0 < r \leq R \leq \infty$, $M > 0$, and $\theta \in I$. Then,
\begin{equation}\label{form-scale}
T_{f_{0},r_{0}}(H_{\theta}(K,M,[r,R])) = H_{\theta}(T_{f_{0},r_{0}}(K),M,[\tfrac{r}{r_{0}},\tfrac{R}{r_{0}}]), \quad f_{0} \in C^{2}(I), \, r_{0} > 0.
\end{equation}
\end{lemma}

Theorem \ref{main1} below is a generalisation of \cite[Theorem 4.7]{2023arXiv230110199O} to transversal families of functions. 

\begin{thm}\label{main1}
Let $t \in (0,2)$, $s\in (0,\min\{t,2-t\})$. For every $\sigma>(t-s)/2$, there exist $\epsilon,\delta_0> 0$ such that the following holds for all $\delta \in (0,\delta_{0}]$. Let $\mu$ be a $(\delta,t,\delta^{-\epsilon})$-regular measure supported on a transversal family $\mathcal{F}$, and let $E \subset I$ be a $(\delta,s,\delta^{-\epsilon})$-set. Then, there exists $\theta \in E$ such that
\begin{equation}\label{form-8} 
\mu(B(1) \cap H_{\theta}(\spt(\mu),\delta^{-\sigma},[\delta,1])) \leq \delta^{\epsilon}. \end{equation} 
\end{thm}

Next, we derive Theorem \ref{proj-regular} from Theorem \ref{main1}, and in the remaining parts we will be devoted to establishing Theorem \ref{main1}.

\begin{proof}[Proof of Theorem \ref{proj-regular} assuming Theorem \ref{main1}]
It suffices to show that there exists one $\theta\in E$ such that \eqref{form-regular} holds. This is because we can apply the result to $E \, \setminus \, \{\theta\}$ and repeat this process.
	
Fix $s\in (0,1]$, $t\in[s,2]$ and $u\in [0, \min\{(s+t)/2,1\})$. The parameters $\epsilon, \delta_0>0$ will be essentially inherited from Theorem \ref{main1}. We first dispose of a special case, where $\min\{(s+t)/2,1\}=1$, or in other words $s\geq 2-t$. Then $u<1$ and we can choose a new parameter $\bar{s}<2-t$ such that $u<(\bar{s}+t)/2$. Now, we note that $E$ is also a $(\delta,\bar{s},\delta^{-\epsilon})$-set. So, it suffices to consider the case where $s<2-t$. Similarly, we may assume that $t>s$. Therefore, in the sequel we suppose $t\in(s,2-s)$.
	
Pick parameters $\sigma'>\sigma>(t-s)/2$ so that $u<t-\sigma'$. Let $\mathcal{H}^0$ be the counting measure on $\mathcal{F}$ and define a measure $\mu := |\mathcal{F}|^{-1} \mathcal{H}^0$. Since $\mathcal{F}$ is a $\delta$-separated $(\delta,t,\delta^{-\epsilon})$-regular set, we can verify that $\mu$ is $(\delta,t,\delta^{-\epsilon})$-regular. 
By Theorem \ref{main1}, there exists $\theta\in E$ such that
\[\mu(B(1)\cap H_\theta(\mathcal{F},\delta^{-\sigma},[\delta,1]))\leq \delta^{\epsilon}.\]
which implies
\[|B(1)\cap H_\theta(\mathcal{F},\delta^{-\sigma},[\delta,1])|_\delta\lesssim \delta^{\epsilon}|\mathcal{F}|_\delta.\]
Here we may assume $\epsilon>0$ is so small that $u<t-\sigma'<t-\sigma-3\epsilon$. Let $\mathcal{F}'\subset\mathcal{F}$ satisfy $|\mathcal{F}'|_\delta\geq \delta^{\epsilon/2}|\mathcal{F}|_\delta$. Then, writing $\mathcal{F}'':=\mathcal{F}' \, \setminus \, H_\theta(\mathcal{F},\delta^{-\sigma},[\delta,1])$, we have $|\mathcal{F}''|_\delta \geq \tfrac{1}{2}|\mathcal{F}'|_\delta\geq \delta^{\epsilon}|\mathcal{F}|_\delta$. As a technical point, we may assume that $\diam(\mathcal{F}'')\leq \tfrac{1}{2}$ by selecting a ball $B$ of radius $1/8$ with $|B\cap\mathcal{F}''|_\delta\sim|\mathcal{F}''|_\delta$. 
	
Now we claim that 
\begin{equation}\label{eq-13}
\Bigg|\bigcup_{f\in\mathcal{F}''} \Gamma_f\cap L_\theta\Bigg|_\delta \geq \delta^{-u}.
\end{equation}
Indeed, fix $h\in \mathcal{F}''$, then by definition of $\mathcal{F}''$:
\begin{equation}\label{form-9}
|\{g\in \mathcal{F}'': |h(\theta)-g(\theta)|\leq\delta\}|_\delta\leq m_{\mathcal{F},\theta}(h,[\delta,1]) \leq \delta^{-\sigma},
\end{equation}
where we used the fact $\mathcal{F}''\subset B(h,1)$ since $\diam(\mathcal{F}'')\leq \tfrac{1}{2}$. Since
\[\Bigg|\bigcup_{f\in\mathcal{F}''} \Gamma_f\cap L_\theta\Bigg|_\delta\cdot |\{g\in \mathcal{F}'': |h(\theta)-g(\theta)|\leq\delta\}|_\delta\geq |\mathcal{F}''|_\delta,\]
the inequality \eqref{form-9} implies 
\[\Bigg|\bigcup_{f\in\mathcal{F}''} \Gamma_f\cap L_\theta\Bigg|_\delta\geq \delta^{\epsilon}|\mathcal{F}|_\delta \cdot \delta^{\sigma} \geq \delta^{2\epsilon+\sigma-t}.\]
Since $t-\sigma-3\epsilon>u$, we in particular have \eqref{eq-13} with $\delta$ choosing to be small enough.
\end{proof}

\subsection{An inductive scheme}
We prove Theorem \ref{main1} by adapting the inductive scheme from the proof of \cite[Theorem 4.7]{OS23}. Let us first introduce the notion \emph{Projection-$(s,\sigma,t)$}.
\begin{terminology}\label{ter-projection}
Let $s, \sigma \in(0,1]$ and $t\in (0,2]$. We say that \emph{Projection-$(s,\sigma,t)$} holds if the conclusion of Theorem \ref{main1} holds with the parameters $s,\sigma,t$. In other words, there exist $\epsilon,\delta_0>0$ such that whenever $\delta\in(0,\delta_0]$, $\mu$ is a $(\delta,t,\delta^{-\epsilon})$-regular measure, and $E\subset [0,1]$ is a $(\delta,s,\delta^{-\epsilon})$-set, then there exists $\theta\in E$ such that \eqref{form-8} is valid.
\end{terminology}
	
\begin{remark}
It is easy to see that
\[\emph{Projection-$(s,\sigma,t)$} \Longrightarrow \emph{Projection-$(s,\sigma',t)$}\quad\text{for all}\quad \sigma'\geq \sigma\]
with the same constants $\epsilon, \delta_0$. This follows from the inclusion 
\[H_\theta(K,\delta^{-\sigma'},[\delta,1])\subset H_\theta(K,\delta^{-\sigma},[\delta,1])\quad\text{for all}\quad\sigma'\geq \sigma.\]
\end{remark}
	
The proof of Theorem \ref{main1} will follow from iterating the next proposition:
 
\begin{proposition}\label{pro-reduction}
Let $t \in (0,2)$, $s\in (0,\min\{t,2-t\})$. For each $(t-s)/2<\sigma\leq \min\{1,t\}$, there exists $\xi=\xi(s,t,\sigma)>0$, which stays bounded away from $0$ as long as $\sigma$ stays bounded away from $(t-s)/2$, such that 
\[\emph{Projection-$(s,\sigma,t)$} \Longrightarrow \emph{Projection-$(s,\sigma-\xi,t)$}.\]
		
More precisely, assume that there exist $\epsilon_0, \Delta_0>0$ such that whenever $\Delta\in(0,\Delta_0)$, $\mu$ is a $(\Delta,t,\Delta^{-\epsilon_{0}})$-regular measure supported on a $\mathfrak{T}$-transversal family $\mathcal{F}$, and $E \subset I$ is a $(\Delta,s,\Delta^{-\epsilon_{0}})$-set, then there exists $\theta \in E$ such that 
\[ \mu\left(B(1) \cap H_{\theta}(\spt (\mu),\Delta^{-\sigma},[\Delta,1]) \right) \leq \Delta^{\epsilon_{0}}. \]
Then, there exist $\epsilon = \epsilon(s,t,\sigma,\epsilon_{0})>0$ and $\delta_{0} = \delta_{0}(s,t,\sigma,\Delta_{0},\epsilon) > 0$ such that the following holds for all $\delta \in (0,\delta_{0}]$. Let $\mu$ be a $(\delta,t,\delta^{-\epsilon})$-regular measure supported on a $\mathfrak{T}$-transversal family $\mathcal{F}$, and let $E \subset I$ be a $(\delta,s,\delta^{-\epsilon})$-set. Then, there exists $\theta \in E$ such that
\begin{equation}\label{form8} 
\mu(B(1) \cap H_{\theta}(\spt(\mu),\delta^{-\sigma+\xi},[\delta,1])) \leq \delta^{\epsilon}. \end{equation}
\end{proposition}

\begin{proof}[Proof of Theorem \ref{main1} assuming Proposition \ref{pro-reduction}]
Fix $t \in (0,2)$ and $s\in (0,\min\{t,2-t\})$. Let $\Sigma(s,t)\geq(t-s)/2$ be the infimum of the parameters $\sigma>(t-s)/2$ for which \emph{Projection-$(s,\sigma,t)$} holds. In other words, for $\sigma>\Sigma(s,t)$ there exist constants $\Delta_0(s,t,\sigma)>0$ and $\epsilon_0=\epsilon_0(s,t,\sigma)>0$ such that the hypothesis of Proposition \ref{pro-reduction} holds. By the transversality of $\mathcal{F}$, $\Sigma(s,t)\leq 1$. By the regularity of $\mu$, $\Sigma(s,t)\leq t$. Hence we have $\Sigma(s,t)\leq \min\{1,t\}$.

We claim that $\Sigma(s,t)=(t-s)/2$. To see this, we make the counter assumption that $\Sigma(s,t)>(t-s)/2$. Now, choose $\sigma>\Sigma(s,t)$ such that $\sigma-\xi<\Sigma(s,t)$, where $\xi=\xi(s,\sigma,t)$ is the constant inherited from Proposition \ref{pro-reduction}. Such a choice of "$\sigma$" is possible, since $\xi(s,\sigma,t)>0$ if $\sigma\in [\Sigma(s,t), \min\{1,t\}]$. But then Proposition \ref{pro-reduction} tells us that Projection $(s,\sigma-\xi,t)$ holds, and this contradicts the definition of "$\Sigma(s,t)$", since $\sigma-\xi<\Sigma(s,t)$.
\end{proof}

\subsection{The proof of Theorem \ref{proj-regular} in three pictures}\label{s:outline}

The "linear" case of Theorem \ref{proj-regular} -- $\mathcal{F}$ consists of affine functions -- is due to the first author and Shmerkin (case where $\mathcal{F}$ is $(\delta,t)$-regular) and Ren-Wang (case where $\mathcal{F}$ is a $(\delta,t)$-set). These results can be visualised by the picture on the left in Figure \ref{fig2}. 

The result can be summarised as follows. Let $s \in (0,1]$ and $t \in [s,2]$. Assume that $\mathcal{T}$ is a $(\delta,t)$-set of roughly horizontal $\delta$-tubes, as in Figure \ref{fig1}. Let $E \subset [0,1]$ be a $(\delta,s)$-set, and consider the vertical lines $L_{\theta} = \{\theta\} \times \R$ with $\theta \in E$. Each $\delta$-tube $T \in \mathcal{T}$ crosses $L_{\theta}$ inside some $\sim \delta$-interval $J \subset L_{\theta}$. Many tubes can cross $L_{\theta}$ inside the same interval $J$. Let us define the multiplicity $\mathfrak{m}(J)$ as the number of tubes in $\mathcal{T}$ crossing $L_{\theta}$ inside $J$. Now, the linear case of Theorem \ref{proj-regular} says that "typically" $\mathfrak{m}(J) \lessapprox \delta^{(s - t)/2}$.

\begin{figure}[h!]
\begin{center}
\begin{overpic}[scale = 0.9]{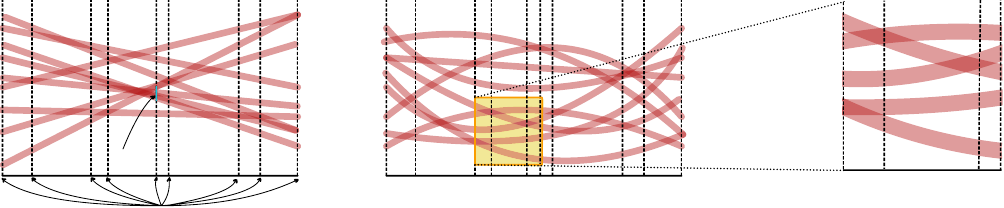}
\put(-2,2){$0$}
\put(30,2){$1$}
\put(13,-2){\small{$\theta \in E$}}
\put(11,4){\small{$J$}}
\put(49.5,7){$Q$}
\end{overpic}
\caption{Left: the (previously known) linear case of Proposition \ref{proj-regular}. Middle: general case of Proposition \ref{proj-regular}. Right: zooming into a $\sqrt{\delta}$-square where the curvilinear $\delta$-tubes look straight.}\label{fig2}
\end{center}
\end{figure}
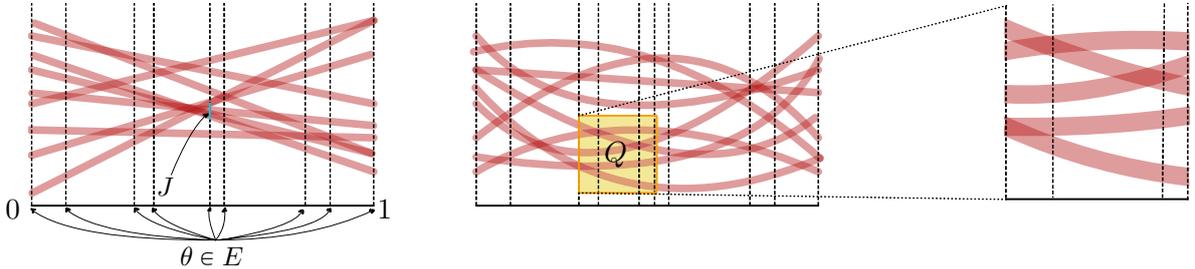

Theorem \ref{proj-regular} makes the same claim in the generality of transversal families; this is illustrated in the middle of Figure \ref{fig2}. The straight tubes are now replaced by $\delta$-neighbourhoods of the graphs of functions in the $(\delta,t)$-set $\mathcal{F} \subset C^{2}([0,1])$. We still denote these (curvilinear) tubes by $\mathcal{T}$ in this discussion. The multiplicity $\mathfrak{m}(J)$ can be defined exactly as above, and the goal is to show that "typically" $\mathfrak{m}(J) \lessapprox \delta^{(s - t)/2}$.

In the case where $\mathcal{F}$ is $(\delta,t)$-regular (as in Proposition \ref{proj-regular}), this is accomplished via Proposition \ref{pro-reduction}. Let us discuss the proof of Proposition \ref{pro-reduction}. Informally, Proposition \ref{pro-reduction} can be summarised as follows: we take for granted that "typically" $\mathfrak{m}(J) \leq \delta^{-\sigma}$, and we want to prove that "typically" $\mathfrak{m}(J) \leq \delta^{-\sigma + \xi}$, as long as $\sigma > (t - s)/2$.

The proof proceeds by a counter assumption, which roughly says that "typically" $\delta^{-\sigma + \xi} \leq \mathfrak{m}(J) \leq \delta^{-\sigma}$. Thanks to the $(\delta,t)$-regularity of $\mathcal{F}$, and the invariance of transversality under rescaling, we may use the hypothesis $\mathfrak{m}(J) \leq \delta^{-\sigma}$ effectively to pieces of $\mathcal{F}$ contained inside balls of various radii. Combined with the (typical) lower bound $\mathfrak{m}(J) \geq \delta^{-\sigma + \xi}$, this leads to a significant "self-similarity" in the organisation of the tubes $\mathcal{T}$. In particular, writing $\Delta := \sqrt{\delta}$, we will be able to find a $\Delta$-square $Q \subset [0,1]^{2}$ such that 
\begin{itemize}
\item the part of $\mathcal{T}$ entering $Q$ corresponds to a (rescaled) $(\Delta,t)$-subset $\mathcal{F}_{Q} \subset \mathcal{F}$, 
\item the "typical" multiplicities of $\delta$-intervals $J \subset L_{\theta} \cap Q$ exceed $\Delta^{-\sigma + \xi} \gg \Delta^{(s - t)/2}$.
\end{itemize} 
Once these properties have been reached, we perform the "zoom-in" depicted on the right hand side of Figure \ref{fig2}. Since $\mathcal{F} \subset C^{2}([0,1])$, any "curvature" of the $\delta$-tubes $\mathcal{T}$ is essentially invisible inside $Q$. Therefore, the rescaled picture is amenable to the "linear" case of Theorem \ref{proj-regular}, and we may apply the result of Ren and Wang to complete proof. Even though $\mathcal{F}$ is assumed $(\delta,t)$-regular, in the final step we need the "linear" case for $(\delta,t)$-sets: we are only able to guarantee that the tubes on the right of Figure \ref{fig2} form a $(\Delta,t)$-set, but not -- at least without additional effort -- that they form a $(\Delta,t)$-regular set.

\medskip
	
The remainder of this section focuses on proving Proposition \ref{pro-reduction}. 

\subsection{An auxiliary result}
This part contains an auxiliary result (Proposition \ref{prop2}) which allows us to "upgrade" the hypothesis of Proposition \ref{pro-reduction} into a stronger one. Let us first introduce the notion of locally high multiplicity sets.

\begin{definition}[Locally high multiplicity sets] Let $\delta \in (0,\tfrac{1}{2}]$, $\rho \in 2^{-\N}$, $\sigma \in (0,1]$, and let $\theta \in I$. For $K \subset C^{2}(I)$, we define the \emph{local high multiplicity set}
\begin{displaymath} H_{\theta,\mathrm{loc}}(K,\sigma,\delta,\rho) := \bigcup_{\delta \leq r \leq R \leq 8} H_{\theta}(K,4(R/r)^{\sigma},[r,R]), \end{displaymath}
where the union ranges over dyadic radii $r,R \in 2^{-\N} \cap [\delta,8]$ satisfying $r/R \leq \rho$.
\end{definition}

Now we state the auxiliary proposition which will be used later.
	
\begin{proposition}\label{prop2} Let $s,\sigma \in [0,1]$, $t \in [0,2]$. Assume that \emph{Projection-$(s,\sigma,t)$} holds for $\Delta_{0},\epsilon_{0} > 0$. That is, whenever $\Delta\in(0,\Delta_0]$, $\mu$ is a $(\Delta,t,\Delta^{-\epsilon_{0}})$-regular measure supported on a $\mathfrak{T}$-transversal family $\mathcal{F}$, and $E \subset I$ is a $(\Delta,s,\Delta^{-\epsilon_{0}})$-set, then there exists $\theta \in E$ such that 
\begin{equation}\label{form9} \mu\left(B(1) \cap H_{\theta}(\spt (\mu),\Delta^{-\sigma},[\Delta,1]) \right) \leq \Delta^{\epsilon_{0}}.  \end{equation}
		
Then, for every $\eta \in (0,1]$, there exists $\epsilon = \epsilon(\eta,\epsilon_{0}) > 0$ and $\delta_{0} = \delta_{0}(\Delta_{0},\epsilon,\eta) > 0$ such that the following holds for all $\delta \in (0,\delta_{0}]$. Let $\mu$ be a $(\delta,t,\delta^{-\epsilon})$-regular measure supported on a $\mathfrak{T}$-transversal family $\mathcal{F}$, and let $E \subset I$ be a $(\delta,s,\delta^{-\epsilon})$-set. Then, there exists $\theta \in E$ such that
\begin{equation}\label{form888} \mu(B(1) \cap H_{\theta,\mathrm{loc}}(\spt(\mu),\sigma,\delta,\delta^{\eta})) \leq \delta^{\epsilon}. \end{equation}   \end{proposition}
\begin{remark}\label{rmk-constant}
The proof below shows that it is enough to take $\epsilon<\epsilon_0\eta/C$ for a certain absolute constant $C\geq1$.
\end{remark}

The proof is identical to that of \cite[Theorem 4.19]{OS23}, relying only on dyadic pigeonholing in ${\rm spt}(\mu)$ and on regularity of $\mu$. However, we include it here for the reader's convenience.

\begin{proof}[Proof of Proposition \ref{prop2}] Assume to the contrary that \eqref{form888} fails for every $\theta \in E$. Abbreviate $K := \spt(\mu)$, and
\begin{displaymath} K_{\theta} := B(1) \cap H_{\theta,\mathrm{loc}}(K,\sigma,\delta,\delta^{\eta}). \end{displaymath}
Thus $\mu(K_{\theta}) \geq \delta^{\epsilon}$ for all $\theta \in E$. By definition, for every $f \in K_{\theta}$, there exist dyadic radii $\delta \leq r \leq \delta^{\eta}R \leq 8\delta^{\eta}$, depending on both $f$ and $\theta$, such that
\begin{displaymath} f \in H_{\theta}(K,4\left(\tfrac{R}{r} \right)^{\sigma},[r,R]) \quad \Longleftrightarrow \quad \mathfrak{m}_{K,\theta}(f , [r,R]) \geq 4\left(\tfrac{R}{r} \right)^{\sigma}. 
\end{displaymath}
By standard pigeonholing (note that both $r,R$ only have $\lesssim \log(1/\delta)$ possible values), and at the cost of replacing the lower bound $\mu(K_{\theta}) \geq \delta^{\epsilon}$ by $\mu(K_{\theta}) \geq \delta^{2\epsilon}$, we may assume that $r_{f,\theta} = r_{\theta}$ and $R_{f,\theta} = R_{\theta}$ for all $f \in K_{\theta}$. Similarly, by replacing $E$ by a subset which remains a $(\delta,s,\delta^{-2\epsilon})$-set, we may assume that $r_{\theta} = r \in [\delta,8]$ and $R_{\theta} = R \in [r,8]$ for all $\theta \in E$ (we keep the notation "$E$" for this subset). Applying Corollary \ref{cor-uniformsets} (replacing $E$ by a further $(\delta,s,\delta^{-2\epsilon})$-subset), we may assume that $E$ is $\{2^{-jT}\}_{j = 1}^{m}$-uniform for some $T \sim_{\epsilon} 1$, and that $\delta = 2^{-mT}$.
		
Next, let $\mathcal{B}_{R}$ be a minimal cover of $K$ by $R$-balls. By the $(\delta,t,\delta^{-\epsilon})$-regularity of $\mu$, we have $|\mathcal{B}_{R}| \leq \delta^{-\epsilon}R^{-t}$. Notice that
\begin{displaymath} 
\sum_{B \in \mathcal{B}_{R}} \sum_{\theta \in E} \mu(K_{\theta} \cap B) = \sum_{\theta \in E} \sum_{B \in \mathcal{B}_{R}} \mu(K_{\theta} \cap B) \geq \sum_{\theta \in E} \mu(K_{\theta}) \geq \delta^{\epsilon}|E|. 
\end{displaymath}
Consequently, there exists a ball $B \in \mathcal{B}_{R}$ with the property
\begin{displaymath} \sum_{\theta \in E} \mu(K_{\theta} \cap B) \geq \frac{\delta^{\epsilon}|E|}{|\mathcal{B}_{R}|} \geq \delta^{2\epsilon}|E|R^{t}. \end{displaymath}
As a further consequence, and using the Frostman bound $\mu(K_{\theta} \cap B) \leq \delta^{-\epsilon}R^{t}$, there exists a subset $E' \subset E$ of cardinality $|E'| \geq \delta^{4\epsilon}|E|$ such that
\begin{equation}\label{form32} \mu(K_{\theta} \cap B) \geq 16 \cdot \delta^{4\epsilon}R^{t}, \qquad \theta \in E'. \end{equation}
We note that
\begin{equation}\label{form31} K_{\theta} \cap B \subset H_{\theta}(K \cap 2B,4\left(\tfrac{R}{r}\right)^{\sigma},[r,R]), \end{equation}
where $2B$ is the ball concentric with $B$ and radius $2R$. This nearly follows from the definition of $K_{\theta}$ (and our pigeonholing of $r,R$), but we added the intersection with $2B$. This is legitimate, since (recalling that $B$ is a disc of radius $R$)
\begin{displaymath} 4\left(\tfrac{R}{r} \right)^{\sigma} \leq \mathfrak{m}_{K,\theta}(f , [r,R]) \stackrel{\mathrm{def.}}{=} |\{g \in K \cap B(f,R) : |f(\theta) - g(\theta)| \leq r\}|_{r}, \quad f \in B \cap K_{\theta}, \end{displaymath}
and the right-hand side does not change if we replace "$K$" by "$K \cap 2B$".
Now, as in Lemma \ref{lemma-rescaledregular}, the measure $\mu_{4B} = (4R)^{-t} T_{4B}(\mu)$ is $(\delta/(4R),t,\delta^{-\epsilon})$-regular. Since $r/R \leq \delta^{\eta}$, we see that $\mu_{4B}$ is also $(\Delta,t,\Delta^{-\epsilon/\eta})$-regular with $\Delta := r/(4R) \leq \delta^{\eta}$. Furthermore, by Lemma \ref{lemma3} applied with $r_{0} = 4R$, we have
\begin{align*} T_{4B}(H_{\theta}(K \cap 2B,M,[r,R])) & = H_{\theta}(T_{4B}(K)\cap B(\tfrac{1}{2}), M,\left[\tfrac{r}{4R},\tfrac{1}{4} \right]) \\
& \subset B(1) \cap H_{\theta}(T_{4B}(K),M,\left[\tfrac{r}{4R},1\right]), \end{align*}
where we abbreviated $M := 4(R/r)^{\sigma} \geq (4R/r)^{\sigma}$ (recall that $\sigma \leq 1$). As a consequence,
\begin{align} \mu_{4B}\big(B(1) \cap & H_{\theta}(T_{4B}(K), \left(\tfrac{4R}{r} \right)^{\sigma},\left[\tfrac{r}{4R},1\right])\big) \notag\\
& \geq (4R)^{-t}\mu(H_{\theta}(K \cap 2B,M,[r,R])) \notag\\
&\label{form33} \stackrel{\eqref{form31}}{\geq} \frac{\mu(K_{\theta} \cap B)}{16R^{t}} \stackrel{\eqref{form32}}{\geq} \delta^{4\epsilon} \geq \left(\tfrac{r}{4R} \right)^{4\epsilon/\eta}, \qquad \theta \in E'. \end{align}
We now claim that \eqref{form33} violates our assumption that \emph{Projection-$(s,\sigma,t)$} holds at scale $\Delta = r/(4R)$. Namely, since $E$ is a $\{2^{-jT}\}_{j = 1}^{m}$-uniform $(\delta,s,\delta^{-\epsilon})$-set (with $T \sim_{\epsilon} 1$), and $|E'| \geq \delta^{4\epsilon}|E|$, we infer from \cite[Corollary 2.19]{2023arXiv230110199O} that $E'$ is a $(\Delta,s,\delta^{-C\epsilon})$-set, assuming that $\delta > 0$ is small enough in terms of $\epsilon$. Since $\Delta \leq \delta^{\eta}$, we see that $E'$ is a $(\Delta,s,\Delta^{-C\epsilon/\eta})$-set.
		
Finally, recall that $\mu_{4B}$ is a $(\Delta,t,\Delta^{-\epsilon/\eta})$-regular measure, and $\mu_{4B}$ is supported on the family $\mathcal{F}_{4B}$ -- which is $\mathfrak{T}$-transversal according to Lemma \ref{lem1}.  Now, if $\epsilon = \epsilon(\eta,\epsilon_{0}) > 0$ is small enough, the inequality \eqref{form33} (for all $\theta \in E'$) violates our assumption \eqref{form9} at scale $\Delta$. To be precise, recalling the parameters "$\Delta_{0},\epsilon_{0}$" in the statement of the theorem, the contradiction will ensue if we have taken $\delta > 0$ so small that $r/(4R) \leq \delta^{\eta} \leq \Delta_{0}$, and $\epsilon > 0$ so small that $C\epsilon/\eta \leq \epsilon_{0}$. 
\end{proof}

\subsection{Proof of Proposition \ref{pro-reduction}.} We now begin the proof of Proposition \ref{pro-reduction} in earnest. The whole proof consists of four steps.
	
\subsection*{Step 1: Fixing parameters.}  Fix $(s,t,\sigma)$ as in Proposition \ref{pro-reduction}:
\begin{equation}\label{form-parameters}
t\in(0,2),\quad s\in(0,\min\{t,2-t\}), \quad \text{and}~~\tfrac{t-s}{2}<\sigma\leq \min\{1,t\}.
\end{equation}
Let $\delta_0, \epsilon, \xi\in(0,\tfrac{1}{2}]$ be small parameters whose values will be determined in the proof, then we make our main counter assumption: there exist
a $(\delta,t,\delta^{-\epsilon})$-regular measure $\mu$ supported on a $\delta$-separated transversal family $\mathcal{F}$, and a $(\delta,s,\delta^{-\epsilon})$-set $E \subset I$ such that
\begin{equation}\label{form-counterap1} \mu(B(1) \cap H_{\theta}(\spt(\mu),\delta^{-\sigma+\xi},[\delta,1])) \geq \delta^{\epsilon}, \quad \theta\in E. \end{equation}
We first let $\xi<\xi_0$, where $\xi_0>0$ satisfies
\begin{equation}\label{form-parameters2}
\tfrac{t-s}{2}+5\xi_0<\sigma \leq \min\{1,t\}.
\end{equation}

Let us discuss the parameters $\xi,\epsilon,\delta_0$ further. The parameter "$\xi$" is the most important one. We will see that \eqref{form-counterap1} implies a contradiction against Theorem \ref{discretised projection} which can be derived from the Furstenberg set theorem \cite[Theorem 1.1]{2023arXiv230808819R}, if $\xi$ is chosen sufficiently small in terms of $\sigma, \xi_0, s, t$. Thus, the contradiction will ensue if
\begin{equation}\label{form-xi}
\xi=o_{s,t,\sigma,\xi_0}(1).
\end{equation}
Here $o_{p_1p_2\cdots p_n}(1)$ refers–and will refer to a function of the parameters $p_1,\cdots,p_n$ which is continuous at $0$ and vanishes at $0$. This will show that Proposition \ref{pro-reduction} actually holds with some (maximal) constant "$\xi$" satisfying \eqref{form-xi}. The constant $\epsilon>0$ in Proposition \ref{pro-reduction} is allowed to depend on both $\xi$, and the constants $\epsilon_0$ for which the "inductive hypothesis" in Proposition \ref{pro-reduction} holds. The constant $\delta_0$ is additionally allowed to depend on $\mathfrak{T}, \epsilon$ and $\Delta_0$. Thus, to reach a contradiction, we will need that
\begin{equation}\label{form-epsilon}
\epsilon=o_{\xi,\epsilon_0}(1)~~\text{and}~~\delta_0=o_{\Delta_0,\epsilon,\mathfrak{T},\xi}(1).
\end{equation}
When in the sequel we write "...by choosing $\delta>0$ sufficiently small" we in fact mean "...by choosing the upper bound $\delta_0$ for $\delta$ sufficiently small". Also, we may write "Note that $\epsilon\leq \xi^{10}\epsilon_0$ by \eqref{form-epsilon}", or something similar. This simply means that the requirement "$\epsilon\leq \xi^{10}\epsilon_0$" should–at that moment–be added to the list of restrictions for the function $o_{\xi,\epsilon_0}(1)$. Finally, we will often use the following abbreviation: an upper bound of the form $C_{\xi,\epsilon}\delta^{-C\epsilon}$ will be abbreviated to $\delta^{-O(\epsilon)}$. Indeed, if $\delta$ is sufficiently small, then $C_{\xi,\epsilon}\leq \delta^{-\epsilon}$, and hence $C_{\xi,\epsilon}\delta^{-C\epsilon}\leq \delta^{-(1+C)\epsilon}$.
	
Without loss of generality, we will further assume $I=[0,1]$. For each $\theta\in I$, let $L_\theta$ be the vertical line $\{(x,y)\in\mathbb R^2: x=\theta\}$. Denote $K:=\spt(\mu)\subset \mathcal{F}$.
	
To streamline our exposition, we introduce the following definition. When drawing a comparison between the arguments in this section and those in \cite[Section 4]{2023arXiv230110199O}, our range operator $\mathcal{R}_\theta$ corresponds to the orthogonal projection $\pi_\theta$.
\begin{definition}\label{def:tube}
For each $\theta\in I=[0,1]$, define the range operator $\mathcal{R}_\theta: \mathcal{F}\to \R$ by $\mathcal{R}_\theta(f):=f(\theta)$ for any $f\in \mathcal{F}$. Moreover, for every $\theta \in E$ and $J\subset \R$ define the set $\Xi_{\theta, J}:=\{f\in K: f(\theta)\in J\} \subset \mathcal{R}_\theta^{-1}(J)$. We will refer to the sets $\Xi_{\theta,J}$ as bundles.
\end{definition}

\begin{figure}[h!]
\begin{center}
\begin{overpic}[scale = 0.8]{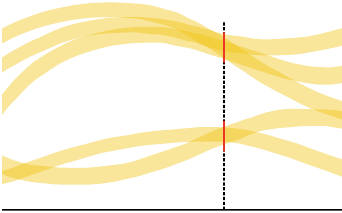}
\put(-4,-1){$0$}
\put(100,-1){$1$}
\put(64,-7){$\theta$}
\put(58,48){\small{$J_{1}$}}
\put(58,20){\small{$J_{2}$}}
\end{overpic}
\caption{Two bundles $\Xi_{\theta,J_{1}}$ and $\Xi_{\theta,J_{2}}$, where $J_{1},J_{2} \subset [0,\infty)$ are $\Delta$-intervals. The curved orange stripes represent $\Delta$-balls in $C^{2}([0,1])$.}\label{fig1}
\end{center}
\end{figure}
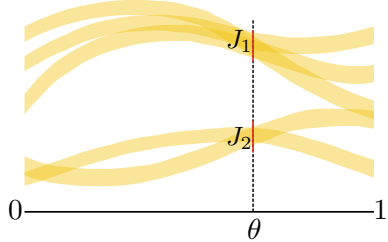

Figure \ref{fig1} shows how bundles might look like. Bundles are subsets of $C^{2}([0,1])$, but they can be visualised in the plane by drawing the graphs of the associated functions.

\subsection*{Step 2: Finding a branching scale for $E$}\label{sub-branching}
Eventually we will restrict the functions in $K$ to an interval $J$ of length $\Delta:= \delta^{1/2}$ and derive a contradiction after rescaling. One important property we will need is that if $J\in \mathcal{D}_{\delta^{1/2}}(E)$, then the renormalization $E_J$ is also a $(\delta^{1/2},s)$-set. This is not true in general. However, the following proposition can fix this problem by replacing $\delta$ with a new scale $\bar{\delta}$ while keeping our counter assumption \eqref{form-counterap1}.
	
\begin{proposition}\label{pro-set E}
There exists a scale $\bar{\delta}\in[\delta, \delta^{\sqrt{\xi}/12}]$ and a $\bar{\delta}$-separated $(\bar{\delta},s,\bar{\delta}^{-O_\xi(\epsilon)})$-set $E_{\bar{\delta}}\subset E$ which is $\{2^{-jT}\}_{j=1}^{\bar{m}}$-uniform and has the following properties.
\begin{itemize}
\item [\textup{(P1)}]\phantomsection \label{P1} For $\bar{\Delta}:=\bar{\delta}^{1/2}$ and $J\in\mathcal{D}_{\bar{\Delta}}(E_{\bar{\delta}})$, the renormalization $(E_{\bar{\delta}})_J$ is a $(\bar{\Delta}, s-\sqrt{\xi}, \bar{\Delta}^{-O_\xi(\epsilon)})$-set.
\item [\textup{(P2)}]\phantomsection \label{P2} We have
\begin{equation}\label{form-counterap2}
\mu(B(1)\cap H_\theta(K,\bar{\delta}^{-\sigma+\bar{\xi}},[\bar{\delta},3]))\geq \bar{\delta}^{O_\xi(\epsilon)},\quad \theta\in E_{\bar{\delta}},
\end{equation}
where $O_\xi(\epsilon)=12\epsilon/\sqrt{\xi}$ and $\bar{\xi}=13\sqrt{\xi}$.
\end{itemize}
\end{proposition}
The proof does not have major differences when compared to that of \cite[Proposition 4.32]{OS23}, but we still include the argument for completeness. 

First, for $\epsilon>0$, we use \cite[Lemma 2.15]{2023arXiv230110199O} (which is just the Euclidean version of Lemma \ref{lem-uniformsubset}) to extract a $\{2^{-jT}\}_{j=1}^m$-uniform subset $E'\subset E$ with $|E'|\geq \delta^\epsilon|E|$, then $E'$ is a $(\delta,s,\delta^{-2\epsilon})$-set. We replace $E$ by $E'$ without changing notation. Since $E$ is now $\{2^{-jT}\}_{j=1}^m$-uniform, and a subset of $[0,1]$, we may associate to it an $1$-Lipschitz branching function $\beta:[0,m]\to[0,m]$. Since $E$ is a $(\delta,s,\delta^{-\epsilon})$-set, it follows from \cite[Lemma 2.22]{2023arXiv230110199O} that
\[\beta(x)\geq sx-\epsilon m-C,\quad x\in [0,m].\]
Therefore, the renormalised function $f(x):=\tfrac{1}{m}\beta(mx)$, defined on $[0,1]$, is also $1$-Lipschitz
and satisfies
\[f(x)\geq \tfrac{1}{m}(smx-\epsilon m-C)\geq sx-\epsilon-\tfrac{C}{m},\quad x\in [0,1].\]
Since $\delta=2^{-mT}$ with $T\sim_\epsilon1$, we may assume that $C/m\leq \epsilon$ by choosing $\delta=o_\epsilon(1)$. Hence $f(x)\geq sx-2\epsilon$ for $x\in[0,1]$. Since $2\epsilon \leq \sqrt{\xi}/6$ by \eqref{form-epsilon}, we can apply \cite[Corollary 2.12]{2023arXiv230110199O} to find a point $a\in[\sqrt{\xi}/12, \tfrac{1}{3}]$ with the property
\[f(x)-f(a)\geq (s-\sqrt{\xi})(x-a),\quad x\in [a,1].\]
In terms of the original branching function $\beta$, this means that there exists a point $\bar{m}:=am\in[\sqrt{\xi}m/12,m/3]$ with the property
\begin{equation}\label{form-31}
\beta(x)-\beta(\bar{m})\geq (s-\sqrt{\xi})(x-\bar{m}), \quad x\in[\bar{m},2\bar{m}].
\end{equation}
The inequality would even hold for $x\in[\bar{m},m]$, but we only need it for $x\in[\bar{m},2\bar{m}]$.
	
Set $\bar{\delta}:=2^{-2\bar{m}T}$ and $\bar{\Delta}:=\bar{\delta}^{1/2}=2^{-\bar{m}T}$. It now follows from \eqref{form-31} and \cite[Lemma 2.24]{2023arXiv230110199O} that if $J\in \mathcal{D}_{\bar{\Delta}}(E)$, then $E_J$ is a $(\bar{\Delta},s-\sqrt{\xi},C_\epsilon)$-set, where $C_\epsilon\lesssim 2^{T}\lesssim_\epsilon 1$. Since $\bar{m}\geq \sqrt{\xi}m/12$, we notice that
	\begin{equation}\label{form-32}
		\delta=2^{-mT}\geq (2^{-2\bar{m}T})^{6/\sqrt{\xi}}=\bar{\delta}^{6/\sqrt{\xi}},
	\end{equation}
	as desired in Proposition \ref{pro-set E}. On the other hand, since $\bar{m}\leq m/3$, we have $2\bar{m}\leq 2m/3$, and therefore $\bar{\delta}$ is also substantially larger than $\delta$:
	\begin{equation}\label{form-33}
		\delta \bar{\delta}^{-1}=2^{(2\bar{m}-m)T}\leq 2^{-mT/3}=\delta^{1/3}.
	\end{equation}
	The ratio $\delta/\bar{\delta}$ will appear in our calculations later, and \eqref{form-32} will allow us to assume that it is "as small as needed" by choosing $\delta>0$ small enough. The scale $\bar{\delta}\in[\delta,\delta^{\sqrt{\xi}/12}]$ has now been fixed, and simply the choice $E_{\bar{\delta}}:=E$ (or at least a $\bar{\delta}$-net inside $E$) would satisfy \nref{P1}. However, to achieve \nref{P2}, some additional work is required, and we will ultimately replace $E$ with a refined subset $E_{\bar{\delta}}$. 
	
	The following auxiliary result and its proof were taken from \cite[Lemma 2.4]{TO2024proj}.
	\begin{lemma}\label{lem-newaux}
		Let $R, C_{reg}\geq1$ and $t\in[0,2]$. Let $\mu$ be a $(\delta,t,C_{reg})$-regular measure supported on a $\delta$-separated transversal family $\mathcal{F}$ on $I$. Abbreviate $K:=\spt(\mu)$, $\mu_1:=\mu|_{B(1)}$ and $H(...):=H_\theta(K,...)$. Let $\kappa\in(0,1]$, $1\leq M\leq N$ and $\delta, r\in 2^{-\N}$ with $\delta\leq r$. Then for any fixed $\theta\in I$,
		\begin{equation}\label{form-newaxu}
			\begin{split}
				\mu_1(H(N,[\delta,R]))\leq \kappa &+(1+\kappa)\mu_1(H(M,[3\delta,3r]))\\
				&+\mu_1(H(N,[\delta,R])\cap H(c(\kappa)\tfrac{N}{M},[r,3])),
			\end{split}
		\end{equation}
		where $c(\kappa)=c\kappa^2/(R^2C_{reg}^3)$ for an absolute constant $c>0$.
	\end{lemma}
	\begin{proof}
		Fix $\kappa\in (0,1]$ and let $\eta=\eta(R,C_{reg},\kappa)>0$ be so small that
		\begin{equation}\label{form-eta}
			(1-\eta)^{-1}\leq (1+\kappa)\quad \text{and} \quad CC_{reg}\eta R^t\leq \kappa,
		\end{equation}
		where $C>0$ is an absolute large constant to be determined in \eqref{form-41} below.
		
		Let $\{J'\}$ be a minimal cover of $\mathcal{R}_\theta(B(1)\cap H(N,[\delta,R]))$ by disjoint half-open intervals of length $\delta$, then ${\bf \Xi}_\theta=\{\Xi_{\theta,J'}\}$ (recall Definition \ref{def:tube}) is a disjoint cover of $B(1)\cap H(N,[\delta,R])$. In particular, every $\Xi\in{\bf \Xi}_\theta$ contains at least one $f_\Xi\in B(1)\cap H(N,[\delta,R])$. Moreover, let us check that
		\begin{equation}\label{form-41}
			|{\bf \Xi}_\theta|\leq R^t C C_{reg} \delta^{-t}/N.
		\end{equation}
		To see this, it follows from definition of $H_\theta(...)$ that 
		\[\left|\{g\in K\cap B(f_\Xi,R): |f_\Xi(\theta)-g(\theta)|\leq \delta\}\right|_\delta\geq N.\]
		In particular, if $\{B_i\}$ is a minimal cover of $B(3R)\cap K$ by $\delta$-balls, then there are $\geq N$ balls intersecting $\Xi$. By regularity of $\mu$, we have
		$|B(3R)\cap K|_\delta \leq C_{reg}(3R/\delta)^t$. Since each $\delta$-ball can intersect $\leq3$ different $\Xi$, this easily implies \eqref{form-41}.
		
		We say $\Xi\in{\bf \Xi}_\theta$ is \textbf{good} (denoted by ${\bf \Xi}_{good}$) if
		\[\mu_1(\Xi\cap H(M,[3\delta,3r]))\geq (1-\eta)\mu_1(\Xi\cap H(N,[\delta,R]))\]
		and otherwise is \textbf{bad} (denoted by ${\bf \Xi}_{bad}$). Then we deduce that
		\[\begin{split}
			\mu_1(H(N,[\delta,R]))&\leq (1-\eta)^{-1}\sum_{\Xi\in {\bf \Xi}_{good}}\mu_1(\Xi\cap H(M,[3\delta,3r]))+\sum_{\Xi\in {\bf \Xi}_{bad}}\mu_1(\Xi\cap H(N,[\delta,R]))\\
			&\stackrel{\eqref{form-eta}}{\leq }(1+\kappa)\mu_1(H(M,[3\delta,3r]))+\sum_{\Xi\in {\bf \Xi}_{bad}}\mu_1(\Xi\cap H(N,[\delta,R])).
		\end{split}\]
		To proceed, we further split ${\bf \Xi}_{bad}$ into ${\bf \Xi}_{bad}^{l}$ and ${\bf \Xi}_{bad}^{h}$, where
		\[{\bf \Xi}_{bad}^{l}:=\{\Xi\in {\bf \Xi}_{bad}: \mu_1(\Xi\cap H(N,[\delta,R]))\leq \eta N\delta^t\}\]
		and ${\bf \Xi}_{bad}^{h}={\bf \Xi}_{bad} \, \setminus \, {\bf \Xi}_{bad}^{l}$. By \eqref{form-41},
		\[\sum_{\Xi\in {\bf \Xi}_{bad}^{l}} \mu_1(\Xi\cap H(N,[\delta,R]))\leq R^tCC_{reg} \eta \stackrel{\eqref{form-eta}}{\leq} \kappa.\]
		It then suffices to show
		\begin{equation}\label{form-42}
			\sum_{\Xi\in {\bf \Xi}_{bad}^h}\mu_1(\Xi\cap H(N,[\delta,R]))\leq \mu_1(H(N,[\delta,R])\cap H(c(\kappa)\tfrac{N}{M},[r,3])),
		\end{equation}
		where we choose $c(\kappa)\sim C_{reg}^{-1} \eta^2\sim \kappa^2/(R^2C_{reg}^3)$.
		
		Fix $\Xi=\Xi_{\theta, J'}\in {\bf \Xi}_{bad}^{h}$, we claim that
		\begin{equation}\label{form-43}
			B(1)\cap \Xi\subset H(c\eta^2\tfrac{N}{M},[r,3])
		\end{equation}
		for a suitable constant $c\sim C_{reg}^{-1}$. Clearly, \eqref{form-43} implies \eqref{form-42}. To show \eqref{form-43}, first using the definition of ${\bf \Xi}_{bad}$, and then ${\bf \Xi}_{bad}^{h}$,
		\[\mu_1(\Xi \, \setminus \, H(M,[3\delta,3r]))\geq \eta\cdot\mu_1(\Xi\cap H(N,[\delta,R]))\geq \eta^2N\delta^t.\]
		Let
		\begin{equation}\label{form-44}
			\{f_i: 1\leq i\leq m\}\subset B(1)\cap K\cap \Xi \, \setminus \, H(M,[3\delta,3r])
		\end{equation}
		be a maximal $\delta$-separated subset. Since $\mu_1(B(f_i,\delta))\leq C_{reg}\delta^t$, we have 
		\begin{equation}\label{form-45}
			m\geq \eta^2N/C_{reg}.
		\end{equation}
		We now use $\{f_i\}$ to prove \eqref{form-43}. The rough idea is that since $f_i$ lie in the complement of $H(M,[3\delta, 3r])$, it takes $\gtrsim m/M\gtrsim_{C_{reg}}\eta^2N/M$ "$r$-bundles" to cover them, and this eventually gives \eqref{form-43}.
		
		Unwrapping the definitions, \eqref{form-43} is equivalent to
		\begin{equation}\label{form-46}
			\left|\{g\in K\cap B(f,3): |f(\theta)-g(\theta)|\leq r\}\right|_r \geq c\eta^2\tfrac{N}{M},\quad f\in B(1)\cap \Xi.
		\end{equation}
		Let $\mathcal{J}_1=\{B\}$ be a minimal cover of $\{g\in K\cap B(f,3): |f(\theta)-g(\theta)|\leq r\}$ by $r$-balls. Since $f, f_i\in B(1)\cap \Xi$, each $f_i$ is contained in some $B\in\mathcal{J}_1$. Now if \eqref{form-46} fails, then at least one $B\in\mathcal{J}_1$ contains a subset $Y\subset \{f_i\}$ with cardinality 
		\[|Y|\gtrsim \frac{m}{c\eta^2N/M}\stackrel{\eqref{form-45}}{\geq} \frac{M}{cC_{reg}}\geq CM,\]
		where we choose $c\sim C_{reg}^{-1}$ small enough and $C\geq1$ is an absolute constant. Note
		\[\diam (Y)\leq \diam (B)\leq 2r.\]
		This shows that $Y\subset H(M,[3\delta,3r])$ if $C\geq1$ is large enough, which contradicts our choice \eqref{form-44}. 
        \end{proof}

	To proceed, we use Lemma \ref{lem-newaux} to establish \nref{P2}. We apply Lemma \ref{lem-newaux} to the $(\delta,t,\delta^{-\epsilon})$-regular measure $\mu$ by taking
	\[R=1,~~r=\bar{\delta},~~\kappa=\delta^{2\epsilon},~~N=\delta^{-\sigma+\xi}~~\text{and}~~M=(\delta \bar{\delta}^{-1})^{-\sigma},\]
	from which we deduce
	\begin{equation}\label{form-34}
		\begin{split}
			c(\kappa)\cdot\tfrac{N}{M}&=c\delta^{7\epsilon}\cdot\delta^{-\sigma+\xi}\cdot(\delta \bar{\delta}^{-1})^{\sigma}=c\cdot\bar{\delta}^{-\sigma}\cdot\delta^{\xi+7\epsilon}\\
			&\geq c\cdot\bar{\delta}^{-\sigma}\cdot\delta^{2\xi}\stackrel{\eqref{form-32}}{\geq}\bar{\delta}^{-\sigma+13\sqrt{\xi}}.
		\end{split}
	\end{equation}
	In the last line, we first took $7\epsilon<\xi$ and then chose $\delta>0$ small enough such that $c\geq \delta^{\sqrt{\xi}}$. In the sequel, we abbreviate $\bar{\xi}:=13\sqrt{\xi}$.
	
	For each fixed $\theta\in E$, we recall \eqref{form-counterap1} and write \eqref{form-newaxu} with our choices for parameters:
	\begin{align} \delta^{\epsilon} & \leq \mu(B(1)\cap H_\theta(K,\delta^{-\sigma+\xi},[\delta,1])) \notag\\
			&\label{form-35} \leq \delta^{2\epsilon}+2\mu(B(1)\cap H_\theta(K,(\delta \bar{\delta}^{-1})^{-\sigma},[3\delta,3\bar{\delta}]))\\
			&\qquad+\mu(B(1)\cap H_\theta(K,\bar{\delta}^{-\sigma+\bar{\xi}},[\bar{\delta},3])). \notag
	\end{align}
	We claim that if $\epsilon>0$ is sufficiently small relative to $\epsilon_0$, then
	\begin{equation}\label{form-36}
		\sum_{\theta\in E}2\mu(B(1)\cap H_\theta(K,(\delta \bar{\delta}^{-1})^{-\sigma},[3\delta,3\bar{\delta}])) \leq \delta^{2\epsilon} |E|.
	\end{equation}
	To prove $\eqref{form-36}$, let $\mathcal{C}$ be a minimal cover of $K\cap B(1)$ by $3\bar{\delta}$-balls. By the regularity of $\mu$,
	\begin{equation}\label{form-37}
		|\mathcal{C}|\leq \delta^{-\epsilon}\cdot \bar{\delta}^{-t}.
	\end{equation}
	Then we decompose
	\begin{equation}\label{form-38}
		2\mu(B(1)\cap H_\theta(K,(\delta \bar{\delta}^{-1})^{-\sigma},[3\delta,3\bar{\delta}])) \leq \sum_{B\in \mathcal{C}}2\mu(B\cap H_\theta(K,(\delta \bar{\delta}^{-1})^{-\sigma},[3\delta,3\bar{\delta}])).
	\end{equation}
	To treat the individual terms, we next consider $\mu_B=(3\bar{\delta})^{-t}T_B\mu$, and write
	\begin{equation}\label{form-39}
		\mu(B\cap H_\theta(K,(\delta \bar{\delta}^{-1})^{-\sigma},[3\delta,3\bar{\delta}]))\stackrel{\eqref{form-scale}}{=}(3\bar{\delta})^t\mu_B(B(1)\cap H_\theta(T_B(K),\Delta^{-\sigma},[\Delta,1])),
	\end{equation}
	where $\Delta:=\delta/\bar{\delta}$.
	By Lemma \ref{lemma-rescaledregular}, the measure $\mu_B$ is $(\Delta,t,\delta^{-\epsilon})$-regular, where
	\[\delta^{-\epsilon}\stackrel{\eqref{form-33}}{\leq} (\delta/\bar{\delta})^{-3\epsilon}=\Delta^{-3\epsilon}. \]
	In particular, $\mu_B$ is $(\Delta,t,\Delta^{-\epsilon_0})$-regular, assuming $3\epsilon\leq \epsilon_0$. Since we may also assume $\Delta\leq \delta^{1/3}\leq \Delta_0$, the hypothesis of Proposition \ref{pro-reduction} is applicable to $\mu_B$. We will use this to show that
	\begin{equation}\label{form-40}
		\sum_{\theta\in E}\mu_B(B(1)\cap H_\theta(T_B(K),\Delta^{-\sigma},[\Delta,1]))\leq \delta^{10\epsilon}|E|.
	\end{equation}
	Assume that \eqref{form-40} is false. Then, since $\mu_B(B(1))\leq \delta^{-\epsilon}$, there exists a subset $E'\subset E$ with the properties $|E'|\geq \delta^{20\epsilon}|E|$ and
	\begin{equation}\label{form-11}
		\mu_B(B(1)\cap H_\theta(T_B(K),\Delta^{-\sigma},[\Delta,1]))\geq\delta^{20\epsilon},\quad \forall \theta\in E'.
	\end{equation}
	Since $E$ is a uniform $(\delta,s,\delta^{-\epsilon})$-set, $E'$ is a $(\Delta,s,\delta^{-O(\epsilon)})$-set by \cite[Corollary 2.20]{2023arXiv230110199O} (the Euclidean counterpart of Lemma \ref{lem-subofuni}). Since $\Delta\leq \delta^{1/3}$, $E'$ is also a $(\Delta,s,\Delta^{-O(\epsilon)})$-set. Therefore, if we take $\epsilon\leq \epsilon_0/C$ for some absolute constant $C>0$, the hypothesis of Proposition \ref{pro-reduction} implies that
	\[\mu_B(B(1)\cap H_\theta(T_B(K),\Delta^{-\sigma},[\Delta,1]))\leq \Delta^{-\epsilon_0}=(\delta/\bar{\delta})^{\epsilon_0}\leq \delta^{\epsilon_0/3}\]
	for some $\theta\in E'$. This violates \eqref{form-11} if $20\epsilon<\epsilon_0/3$, and the ensuing contradiction shows that \eqref{form-40} must be valid. Consequently,
	\[\begin{split}
		&\sum_{\theta\in E}\sum_{B\in\mathcal{C}}2\mu(B\cap H_\theta(K,(\delta \bar{\delta}^{-1})^{-\sigma},[3\delta,3\bar{\delta}]))\\
		&\stackrel{\eqref{form-39}}{=}2(3\bar{\delta})^t\sum_{B\in\mathcal{C}}\sum_{\theta\in E}\mu_B(B(1)\cap H_\theta(T_B(K),\Delta^{-\sigma},[\Delta,1]))\\
		&\stackrel{\eqref{form-40}}{\leq}2(3\bar{\delta})^t\sum_{B\in\mathcal{C}}\delta^{10\epsilon}|E|\stackrel{\eqref{form-37}}{\leq}\delta^{7\epsilon}|E|,
	\end{split}\]
	which establishes our claim \eqref{form-36}. Combining \eqref{form-35} and \eqref{form-36}, we have
	\[\begin{split}
		\delta^\epsilon|E|&\leq \sum_{\theta\in E}\mu(B(1)\cap H_\theta(K,\delta^{-\sigma+\xi},[\delta,1]))\\
		&\leq \sum_{\theta\in E}\mu(B(1)\cap H_\theta(K,\bar{\delta}^{-\sigma+\bar{\xi}},[\bar{\delta},3]))+2\delta^{2\epsilon}|E|,
	\end{split}\]
	and consequently
	\[\sum_{\theta\in E}\mu(B(1)\cap H_\theta(K,\bar{\delta}^{-\sigma+\bar{\xi}},[\bar{\delta},3]))\geq \tfrac{1}{2}\delta^\epsilon|E|.\]
	Therefore, there exists a subset $E'\subset E$ with $|E'|\geq \delta^{2\epsilon}|E|\geq \bar{\delta}^{12\epsilon/\sqrt{\xi}}|E|$ with the property
	\begin{equation}\label{form-12}
		\mu(B(1)\cap H_\theta(K,\bar{\delta}^{-\sigma+\bar{\xi}},[\bar{\delta},3]))\geq \delta^{2\epsilon}\geq \bar{\delta}^{12\epsilon/\sqrt{\xi}},~~\theta\in E'.
	\end{equation}
	This verifies property \nref{P2}.
	
	A small technicality remains: the scale $\bar{\delta}$ was chosen so that $E_J$ is a $(\bar{\Delta},s-\sqrt{\xi}, C_\epsilon)$-set for all $J\in \mathcal{D}_{\bar{\Delta}}(E)$ (with $\bar{\Delta}=\bar{\delta}^{1/2}$), but since $E'\subset E$ is only a subset with $|E'|\geq \delta^{2\epsilon}|E|$, this property may be violated. To fix this, apply \cite[Corollary 2.16]{2023arXiv230110199O} to find a further $\{2^{-jT}\}_{j=1}^m$-uniform subset $E''\subset E'$ with $|E''|\geq \delta^{\epsilon}|E'|\geq \delta^{3\epsilon}|E|$. Since both $E, E''$ are uniform, we have
	\begin{equation}\label{eq-14}
		|E''\cap J|\geq \delta^{3\epsilon}|E\cap J|, \quad J\in \mathcal{D}_{\bar{\Delta}}(E'').
	\end{equation}
	Now if follows from a combination of \eqref{eq-14} and the $(\bar{\Delta},s-\sqrt{\xi}, C_\epsilon)$-set property of $E_J$, that $E''_J$ is a $(\bar{\Delta},s-\sqrt{\xi}, \delta^{-O(\epsilon)})$-set for all $J\in \mathcal{D}_{\bar{\Delta}}(E'')$. Finally, since $\delta^{-O(\epsilon)}\leq \bar{\Delta}^{-O_\xi(\epsilon)}$, we see that $E''_J$ is a $(\bar{\Delta},s-\sqrt{\xi},\bar{\Delta}^{-O_\xi(\epsilon)})$-set for all $J\in \mathcal{D}_{\bar{\Delta}}(E'')$.
	
	In Proposition \ref{pro-set E}, we desired the set $E_{\bar{\delta}}$ to be $\bar{\delta}$-separated. This is finally achieved by choosing one point from each interval $J\in \mathcal{D}_{\bar{\Delta}}(E'')$, and calling the result $E_{\bar{\delta}}$. This does not violate the property of the blow-ups $E''_J$ established just above, since the $(\bar{\Delta},s-\sqrt{\xi})$-set property of $E''_J$ only cares about the behaviour of $E''$ between the scales $\bar{\delta}$ and $\bar{\Delta}$. The proof of Proposition \ref{pro-set E} is complete.

\subsection*{Notation} In the following we simplify the notations by removing the "bars". We assume that $\bar{\delta}=\delta$ (thus $\bar{\Delta}=\sqrt{\delta}$) and $E_{\bar{\delta}}=E$). We also rename $\bar{\xi}:=\xi$. The only change from our original setup is that certain constants of the form $\delta^\epsilon$ have to be replaced by $\delta^{O_\xi(\epsilon)}$. Notably, $E$ is a $(\delta,s,\delta^{-O_\xi(\epsilon)})$-set, and $\mu$ is $(t,\delta^{-O_\xi(\epsilon)})$-regular. Since we are seeking a contradiction if $\xi>0$ is small enough in terms of $(s,t,\sigma)$, and $\epsilon$ is small enough in terms of $(s,t,\sigma, \epsilon_0,\xi)$, this difference will be completely irrelevant. Additionally, $E_J$ is a $(\Delta,s-\sqrt{\xi},\Delta^{-O_\xi(\epsilon)})$-set for all $J\in \mathcal{D}_\Delta(E)$.

\subsection*{Step 3: Defining the sets $K_\theta$}
Recall our counter assumption \eqref{form-counterap2} (also \eqref{form-counterap1}): 
\begin{equation}\label{form-counterap3}
\mu(B(1)\cap H_\theta(K,\delta^{-\sigma+\xi},[\delta,1]))\geq \delta^{C_\xi \epsilon},\quad \theta\in E,
\end{equation}
where $C_\xi\geq1$ depends only on $\xi$. Recall also our hypothesis that \emph{Projection-$(s,\sigma,t)$} holds. Therefore, we can apply Proposition \ref{prop2} with $\eta=\sqrt{\epsilon}$ and choose $\epsilon$ small enough such that $2C_\xi\epsilon<\eta \epsilon_0/C$ (recall Remark \ref{rmk-constant}). This yields the following conclusion for at least half of the points $\theta \in E$:
	\begin{equation}\label{form-loc}
		\mu(B(1)\cap H_{\theta,loc}(K,\sigma,\delta,\delta^{\sqrt{\epsilon}}))\leq \delta^{2C_\xi \epsilon}.
	\end{equation}
    We replace $E$ by this subset keeping the notation. At this point, the $(\Delta,s-\xi,\Delta^{-O_\xi(\epsilon)})$-property of the renormalizations $E_J$ might have failed, but this can be fixed by replacing the new $E$ with a further $\{2^{-jT}\}_{j=1}^m$-uniform subset, just as we did in \textbf{Step 2.}
    
	Now for each $\theta\in E$, we define
	\begin{equation}\label{def:Ktheta}
		K_\theta:=B(1)\cap H_\theta(K,\delta^{-\sigma+\xi},[\delta, 1]) \, \setminus \, H_{\theta,\mathrm{loc}}(K,\sigma,\delta,\delta^{\sqrt{\epsilon}}).
	\end{equation}
	From \eqref{form-counterap3} and \eqref{def:Ktheta}, we get
	\begin{equation}\label{form-100}
		\mu(K_\theta)\geq \delta^{C_\xi\epsilon}-\delta^{2C_\xi\epsilon}=\delta^{O_\xi(\epsilon)},\quad \theta\in E.
	\end{equation}
	
	We next state two non-concentration properties for $K_\theta$. Their proofs are identical to those of \cite[Lemmas 4.53 and 4.59]{OS23}, but we give the details for completeness.    
	\begin{lemma}\label{lem-p1}
		Let ${B}$ be an open $C^{2}(I)$-ball of radius $\Delta\in [\delta^{1-\sqrt{\epsilon}},1]$. Then
		\begin{equation}\label{form-p1}
			|\mathcal{R}_\theta(K_\theta\cap {B})|_\delta
			\gtrsim \delta^{O_\xi(\epsilon)-t}\cdot (\tfrac{\delta}{\Delta})^\sigma \cdot \mu(K_\theta\cap {B}),\quad \theta\in E.
		\end{equation}
	\end{lemma}
    \begin{proof}
		Let $\{f_i(\theta):f_i\in K_\theta\cap {B},~1\leq i\leq N_{B}\}$ be a maximal $\delta$-separated subset of 
		\[\mathcal{R}_\theta(K_\theta\cap {B})=\{f(\theta): f\in K_\theta\cap {B}\}.\] 
		Since $f_i\in K_\theta\cap {B}$, we get by the definition of $K_\theta$
		\[f_i\notin H_{\theta,\mathrm{loc}}(K,\sigma,\delta,\delta^{\sqrt{\epsilon}}) \quad \Longrightarrow \quad f_i\notin H_\theta(K,4(\tfrac{\Delta}{\delta})^\sigma, [8\delta,8\Delta]),\]
		noting that $\delta/\Delta\leq \delta^{\sqrt{\epsilon}}$ by assumption. Unwrapping the definition even further,
		\begin{equation}\label{form-228}
			|\{g\in K\cap B(f_i,8\Delta): |g(\theta)-f_i(\theta)|\leq 8\delta\}|_{8\delta}\leq 4(\tfrac{\Delta}{\delta})^\sigma.
		\end{equation}
		Let $\mathcal{B}_\delta$ be a minimal cover of $K_\theta \cap {B}$ by $\delta$-balls in $\mathcal F$. Then
		\begin{equation}\label{form-229}
			\sum_{i=1}^{N_{B}}\left|\{g\in K_\theta\cap {B}: |g(\theta)-f_i(\theta)|\leq \delta\}\right|_\delta \geq |\mathcal{B}_\delta|.
		\end{equation}
		From \eqref{form-228}, we deduce that
		\[\begin{split}
			&\left|\{g\in K_\theta\cap {B}: |g(\theta)-f_i(\theta)|\leq \delta\}\right|_\delta\\
			&\leq \left|\{g\in K\cap B(f_i, 2\Delta): |g(\theta)-f_i(\theta)|\leq \delta\}\right|_\delta\lesssim (\tfrac{\Delta}{\delta})^\sigma,
		\end{split}\]
		which implies $N_{B}\gtrsim (\delta/\Delta)^\sigma |\mathcal{B}_\delta|$. By the $(\delta, t, \delta^{-O_\xi(\epsilon)})$-regularity of $\mu$, we also have $\mu(K_\theta\cap {B})\leq |\mathcal{B}_\delta|\delta^{t-O_\xi(\epsilon)}$. As a consequence,
		\[|\mathcal{R}_\theta(K_\theta\cap {B})|_\delta\sim N_{B}\gtrsim \delta^{O_\xi(\epsilon)-t}\cdot (\tfrac{\delta}{\Delta})^\sigma \cdot \mu(K_\theta\cap {B}),\]
	which completes the proof. \end{proof}

	\begin{lemma}\label{lem-p2}
		Let $\theta\in E$ and $\Delta\in [\delta^{1-\sqrt{\epsilon}},1]$. Let $\Xi_{\theta,J} = \{f \in K : f(\theta) \in J\}$ be as in Definition \ref{def:tube}, where $J\subset \R$ is a $\Delta$-interval. Then
		\begin{equation}\label{form-p2}
			|\mathcal{R}_\theta(K_\theta\cap \Xi_{\theta,J})|_\delta=|\mathcal{R}_\theta(K_\theta)\cap J|_\delta
			\lesssim \delta^{-O_\xi(\epsilon)-\xi}\cdot (\tfrac{\Delta}{\delta})^{t-\sigma}.
		\end{equation}
		In particular, we have
		\begin{equation}\label{form-p22}
			|\mathcal{R}_\theta(K_\theta)|_\delta\lesssim \delta^{-O_\xi(\epsilon)-\xi}\cdot \delta^{\sigma-t}.
		\end{equation}
	\end{lemma}
  
	\begin{proof}
		We may assume $K_\theta\cap \Xi_{\theta,J}\neq \emptyset$ since otherwise there is nothing to prove. Let $\{I_k\}$ be a minimal $\delta$-cover of $\mathcal{R}_\theta(K_\theta\cap \Xi_{\theta,J})$, then ${\bf \Xi}^\delta_\theta:=\{\Xi_{\theta, I_k}\}$ is a cover of $K_\theta\cap \Xi_{\theta,J}$ such that each $\Xi_{\theta, I_k}$ contains at least one function in $K_\theta\cap \Xi_{\theta,J}$. It then suffices to show
		\begin{equation}\label{form-230}
			|{\bf \Xi}^\delta_\theta|\lesssim \delta^{-O_\xi(\epsilon)-\xi}\cdot (\tfrac{\Delta}{\delta})^{t-\sigma}.
		\end{equation}
		To see this, fix $\Xi_{\theta,I_k}$ and $f\in K_\theta\cap \Xi_{\theta,I_k}$. By definition, $f\in H_\theta(K,\delta^{-\sigma+\xi},[\delta,1])$, hence
		\begin{equation}\label{form-231}
			|\{g\in K\cap B(2): |g(\theta)-f(\theta)|\leq \delta\}|_\delta\geq \m_{K,\theta}(f , [\delta,1])\geq \delta^{-\sigma+\xi}.
		\end{equation}
		Summing over all $I_k$, we can deduce that
		\[|K\cap B(2)\cap \Xi_{\theta,J}|_\delta\gtrsim \delta^{-\sigma+\xi}\cdot |{\bf \Xi}^\delta_\theta|.\]
		On the other hand, fix any $f_0\in K_\theta\cap \Xi_{\theta,J}$, then recall from definition of $K_\theta$ that
		\[f_0\notin H_{\theta, \mathrm{loc}}(K,\sigma,\delta,\delta^{\sqrt{\epsilon}}) \quad \Longrightarrow \quad f_0\notin H_\theta (K,4\Delta^{-\sigma},[8\Delta,8]),\]
		noting $8\Delta/8=\delta/\Delta\leq \delta^{\sqrt{\epsilon}}$ by assumption. Unwrapping the definition, we get
		\[\left|\{g\in K\cap B(f_0,8): |g(\theta)-f_0(\theta)|\leq 8\Delta\}\right|_{8\Delta}\leq 4\Delta^{-\sigma}.\]
		This also implies that
		\[|K\cap B(2)\cap \Xi_{\theta,J}|_\Delta \lesssim \Delta^{-\sigma}.\]
		Since $\mu$ is $(\delta,t,\delta^{-O_\xi(\epsilon)})$-regular, $|K\cap {B}|_\delta \leq \delta^{-O_\xi(\epsilon)} (\Delta/\delta)^t$ for any $\Delta$-ball ${B}$. Hence
		\[\delta^{-\sigma+\xi}\cdot |{\bf \Xi}^\delta_\theta|\lesssim |K\cap B(2)\cap \Xi_{\theta,J}|_\delta\lesssim \delta^{-O_\xi(\epsilon)} \Delta^{-\sigma} (\tfrac{\Delta}{\delta})^t.\]
		Dividing by $\delta^{-\sigma+\xi}$ implies \eqref{form-230} and thus \eqref{form-p2}. Since $K_\theta\subset B(1)$, \eqref{form-p22} follows from \eqref{form-p2} by taking $\Delta=1$.
\end{proof}

\subsection*{Step 4: Choosing a heavy $\Delta$-bundle}
Recall the sets $K_\theta$ defined in \eqref{def:Ktheta}, which satisfy $\mu(K_\theta)\geq \delta^{O_\xi(\epsilon)}$ for all $\theta\in E$ by \eqref{form-100}. Also, recall that $E_J$ is a $(\Delta,s-\xi,\delta^{-O_\xi(\epsilon)})$-set for any $J\in \mathcal{D}_\Delta(E)$. Here $\Delta=\delta^{1/2}$. Since $\mu(B(1))\leq \delta^{-\epsilon}$, it follows from Cauchy-Schwarz inequality that
\[\begin{split}
	\sum_{\theta, \theta'\in E\cap J} \mu(K_\theta\cap K_{\theta'})=\int\Big(\sum_{\theta\in E\cap J}\textbf{1}_{K_\theta}\Big)^2 d\mu\geq \frac{\Big(\int \sum_{\theta\in E\cap J}\textbf{1}_{K_\theta} d\mu\Big)^2}{\mu(B(1))}\geq \delta^{O_\xi(\epsilon)} |E\cap J|^2.
\end{split}\]
In particular, there exists $\theta_0\in E\cap J$ such that
\[\sum_{\theta\in E\cap J} \mu(K_{\theta_0}\cap K_{\theta})\geq \delta^{O_\xi(\epsilon)} |E\cap J|.\]
From this inequality we further deduce that there exists a subset $E'\cap J\subset E\cap J$ with $|E'\cap J|\geq \delta^{O_\xi(\epsilon)}|E\cap J|$ such that $\mu(K_{\theta_0}\cap K_\theta)\geq \delta^{O_\xi(\epsilon)}$ for any $\theta\in E'\cap J$. Since the renormalization $E'_J$ remains a $(\Delta,s-\xi,\delta^{-O_\xi(\epsilon)})$-set, we simplify notation by assuming that
\begin{equation}\label{form-ap4}
	\mu(K_{\theta_0}\cap K_\theta)\geq \delta^{O_\xi(\epsilon)},\quad \theta\in E\cap J.
\end{equation}
Since our problem is translation-invariant, we may assume that $J=[0,\Delta]$. Then $\theta_0\in[0,\Delta]$, and in order to further simplify notation we assume $\theta_0 = 0$. Then $K_{\theta_0}=K_0$ and $ \mathcal{R}_{\theta_0}=\mathcal R_0$.

Let $\mathcal{B}_\Delta$ be a minimal cover of $B(1)\cap K$ by $\Delta$-balls in $\mathcal{F}$. By the regularity of $K$
\[|\mathcal{B}_\Delta|\leq \Delta^{-O_\xi(\epsilon)-t}.\]
We also note that
\begin{equation}\label{form-50}
	\sum_{\theta\in E\cap J}\sum_{{B}\in\mathcal{B}_\Delta}\mu(K_0\cap K_\theta\cap {B})\geq \sum_{\theta\in E\cap J}\mu (K_0\cap K_\theta)\stackrel{\eqref{form-ap4}}{\geq} \delta^{O_\xi(\epsilon)}|E\cap J|.
\end{equation}
A ball ${B}\in\mathcal{B}_\Delta$ is called \textbf{light} (denoted by $\mathcal{B}_\Delta^{\mathrm{light}}$) if
\[\frac{1}{|E\cap J|}\sum_{\theta\in E\cap J}\mu(K_0\cap K_\theta\cap {B})\leq \Delta^{t+C_\xi\epsilon},\]
where $C_\xi\geq1$ is a constant to be determined in a moment. Observe that
\begin{equation}\label{form-51}
	\sum_{\theta\in E\cap J}\sum_{{B}\in\mathcal{B}_\Delta^{\mathrm{light}}}\mu(K_0\cap K_\theta\cap {B})\leq |\mathcal{B}_\Delta||E\cap J|\Delta^{t+C_\xi\epsilon}\leq \Delta^{C_\xi\epsilon-O_\xi(\epsilon)}|E\cap J|,
\end{equation}
hence if we define $\mathcal{B}_\Delta^{\mathrm{heavy}}:=\mathcal{B}_\Delta \, \setminus \, \mathcal{B}_\Delta^{\mathrm{light}}$, then
\begin{equation}\label{form-52}
	\sum_{\theta\in E\cap J}\sum_{{B}\in\mathcal{B}_\Delta^{\mathrm{heavy}}}\mu(K_0\cap K_\theta\cap {B})\geq (\delta^{O_\xi(\epsilon)}-\Delta^{C_\xi\epsilon-O_\xi(\epsilon)})|E\cap J|\geq \delta^{O_\xi(\epsilon)}|E\cap J|,
\end{equation}
if choosing $C_\xi>5 O_\xi(\epsilon)$ in \eqref{form-51}.

We make the following simple observation about the heavy balls:
\begin{equation}\label{form-53}
	\mu(K_0\cap {B})\geq \frac{1}{|E\cap J|}\sum_{\theta\in E\cap J}\mu(K_0\cap K_\theta\cap {B})\geq \delta^{O_\xi(\epsilon)}\Delta^t,\quad {B}\in \mathcal{B}_\Delta^{\mathrm{heavy}}.
\end{equation}
Substituting \eqref{form-53} into Lemma \ref{lem-p1} gives
\begin{equation}\label{form-54}
	|\mathcal{R}_{0}(K_0\cap {B})|_\delta \gtrsim \delta^{O_\xi(\epsilon)}\cdot \Delta^{\sigma-t},\quad {B}\in \mathcal{B}_\Delta^{\mathrm{heavy}}.
\end{equation}

Let $\{J_i\}$ be a minimal $\Delta$-cover of $\mathcal{R}_{0}(\cup\mathcal{B}_\Delta^{\mathrm{heavy}})$. The collection ${\bf \Xi}_\Delta:=\{\Xi_{0,J_i}\}$ forms a cover of all heavy balls. In particular, each $\Xi_{0,J_i}$ meets at least one ball in $\mathcal{B}_\Delta^{\mathrm{heavy}}$. Since for each ${B}\in\mathcal{B}_\Delta^{\mathrm{heavy}}$, there exist at most three intervals $J_i$ such that $\mathcal R_{0}(K_0\cap {B}) \cap \mathcal R_{0}(\Xi_{0, J_i})\neq\emptyset$, we deduce that
\begin{equation}\label{form-56}
	\delta^{-O_\xi(\epsilon)-\xi}\cdot \delta^{\sigma-t}\stackrel{\eqref{form-p22}}{\gtrsim} |\mathcal{R}_{0}(K_0)|_\delta \stackrel{ \eqref{form-54}}{\gtrsim} \delta^{O_\xi(\epsilon)}\cdot \Delta^{\sigma-t}\cdot |{\bf \Xi}_\Delta|,
\end{equation}
which easily implies
\begin{equation}\label{form-55}
	|{\bf \Xi}_\Delta|\leq \Delta^{\sigma-t-3\xi}.
\end{equation}
by choosing $\epsilon=o_\xi(1)$ so small that $O_\xi(\epsilon)\leq \xi$. 

For each $\Xi\in {\bf \Xi}_\Delta$, define
\begin{equation}\label{form-57}
	\mathcal{B}(\Xi):=\{B\in \mathcal{B}_\Delta^{\mathrm{heavy}}: {B}\cap \Xi\neq \emptyset\}.
\end{equation}
Since ${\bf \Xi}_\Delta$ is a cover of $\mathcal{B}_\Delta^{\mathrm{heavy}}$, we get by \eqref{form-52}
\begin{equation}\label{form-58}
	\sum_{\Xi\in {\bf \Xi}_\Delta}\sum_{\theta\in E\cap J}\sum_{{B}\in\mathcal{B}(\Xi)}\mu(K_0\cap K_\theta\cap {B}\cap \Xi)\geq \delta^{O_\xi(\epsilon)}|E\cap J|.
\end{equation}
A bundle $\Xi\in {\bf \Xi}_\Delta$ is called \emph{heavy} (denoted by ${\bf \Xi}_\Delta^{\mathrm{heavy}}$) if
\begin{equation}\label{form-59}
	\sum_{\theta\in E\cap J}\sum_{{B}\in\mathcal{B}(\Xi)}\mu(K_0\cap K_\theta\cap {B}\cap \Xi)\geq \Delta^{t-\sigma+4\xi}|E\cap J|.
\end{equation}
With this definition, we see that
\[\begin{split}&\sum_{\Xi\in {\bf \Xi}_\Delta \, \setminus{\bf \Xi}_\Delta^{\mathrm{heavy}}}\sum_{\theta\in E\cap J}\sum_{{B}\in\mathcal{B}(\Xi)}\mu(K_0\cap K_\theta\cap {B}\cap \Xi)\\&\leq |{\bf \Xi}_\Delta|\cdot \Delta^{t-\sigma+4\xi}\cdot|E\cap J|\stackrel{\eqref{form-55}}{<}\tfrac{1}{2}\delta^{O_\xi(\epsilon)}|E\cap J|.\end{split}\]
Therefore, ${\bf \Xi}_\Delta^{\mathrm{heavy}}$ is not empty. From now on, we fix one heavy bundle $\Xi := \Xi_{0, J_i}\in {\bf \Xi}_\Delta^{\mathrm{heavy}}$. By translating $K$ further, we may assume $J_i = J=[0,\Delta]$. Moreover, we abbreviate $\mathcal B :=\mathcal B(\Xi)$ and establish the following lemma for $\mathcal{B}$. By passing to a sub-family $\mathcal{B}^\ast\subset\mathcal{B}$ such that \eqref{form-59} still holds if replacing $\mathcal{B}$ by $\mathcal{B}^\ast$, we may assume that $\mathcal{B}$ is $100\mathfrak{T}\Delta$-separated, that is, $d(B_1,B_2)\geq 100\mathfrak{T}\Delta$ for any $B_1,B_2\in\mathcal{B}$.

\begin{lemma}\label{lem-setB}
	The ball family $\mathcal{B}$ has the following properties.
	\begin{enumerate}
		\item[\textup{(i)}] For any $f\in \mathcal{F}$ and $R\in [\delta^{-\sqrt{\epsilon}}\Delta, 1]$, we have \[|\mathcal{B}\cap B(f,R)|=|\{B\in \mathcal{B}: B\cap B(f,R)\neq \emptyset\}|\lesssim_\mathfrak{T} (\tfrac{R}{\Delta})^\sigma.\]
		
		\item[\textup{(ii)}] $\Delta^{-\sigma + 6\xi} \leq |\mathcal B| \lesssim_\mathfrak{T} \Delta^{-\sigma}$.
	\end{enumerate}
\end{lemma}
\begin{proof}
	We begin by verifying (i). Fix $f\in \mathcal{F}$ and $R\in [\delta^{-\sqrt{\epsilon}}\Delta, 1]$. Let $B\in \mathcal{B}\cap B(f,R)$, then ${B}\in \mathcal{B}_\Delta^{\mathrm{heavy}}$ and ${B}\cap K_0\neq \emptyset$ according to \eqref{form-53}. Fix $f_0\in {B}\cap K_0$ and recall
	\[f_0\notin H_{0,\mathrm{loc}}(K,\sigma, \delta, \delta^{\sqrt{\epsilon}}) \quad \Longrightarrow \quad f_0 \notin H_{0}(K,4(\tfrac{R}{\Delta})^\sigma, [8\Delta,8R]),\]
	where $\Delta / R \leq \delta^{\sqrt{\epsilon}}$ and $8R\leq 8$. Unwrapping the definition further, we obtain
	\begin{equation}\label{equa-upper}
		|\{g\in K\cap B(f_0,8R): |g(0)-f_0(0)|\leq 8\Delta\}|_{8\Delta}\leq 4(\tfrac{R}{\Delta})^\sigma.
	\end{equation}
	By the triangle inequality, we know $B(f,R)\subset B(f_0,3R)$, so we can deduce
	\[\cup (\mathcal{B}\cap B(f,R))\subset B(f_0,8R).\]
	By definition of $\mathcal{B}$, each $B'\in \mathcal{B}$ intersects $\Xi=\Xi_{0,J}$, so $\mathcal{R}_0(\cup \mathcal{B})$ is contained in the $2\Delta$-neighborhood of $J$ which is a $5\Delta$-interval. In particular, we have
	\[\cup (\mathcal{B}\cap B(f,R))\cap K\subset \{g\in K\cap B(f_0,8R): |g(0)-f_0(0)|\leq 8\Delta\}.\]
	By upper $2$-regularity of $\mathcal{F}$ and \eqref{equa-upper}, this implies
	\[|\mathcal{B}\cap B(f,R)\}|\lesssim_\mathfrak{T} |\cup (\mathcal{B}\cap B(f,R))\cap K|_\Delta\leq (\tfrac{R}{\Delta})^\sigma,\]
	which completes the proof of property (i).
	
	We then move to property (ii). By using property (i) with $R=1$, we simply get $|\mathcal{B}|\lesssim_\mathfrak{T} \Delta^{-\sigma}$.
	On the other hand, we deduce from \eqref{form-59} that there exists a subset $E'\cap [0,\Delta]\subset E\cap [0,\Delta]$ of cardinality $|E'\cap [0,\Delta]|\geq \Delta^{5\xi} |E\cap [0,\Delta]|$ such that
	\[\sum_{{B}\in\mathcal{B}(\Xi)}\mu(K_0\cap K_\theta\cap {B}\cap \Xi)\geq \Delta^{t-\sigma+5\xi}, \quad \theta\in E'\cap [0,\Delta].\]
	As we have done many times before, we replace $E\cap [0,\Delta]$ by $E'\cap [0,\Delta]$ without changing notation: the only property of $E'\cap [0,\Delta]$ we will need is that $E'_J= S(E'\cap [0,\Delta])$ remains a $(\Delta, s-\xi, \Delta^{-O_\xi(\epsilon)})$-set. Thus we will assume in the sequel that
	\begin{equation}\label{form-ap5}
		\sum_{{B}\in\mathcal{B}(\Xi)}\mu(K_0\cap K_\theta\cap {B}\cap \Xi)\geq \Delta^{t-\sigma+5\xi}, \quad \theta\in E\cap [0,\Delta].
	\end{equation}
	In particular, \eqref{form-ap5} implies $|\mathcal{B}|\geq \Delta^{-\sigma+6\xi}$ and thus we get property (ii).
\end{proof}

Next, write $\mathcal A:= \mathcal D_\delta (\mathcal R_0(K_0 \cap \Xi))\subset [0,\Delta]$. By passing to a sub-family of $\mathcal{A}$ with comparable cardinality, we may assume that $d(J'_1,J'_2)\geq 100\delta $ for all distinct $J'_1, J'_2\in\mathcal{A}$. Recall we also assume that $d(B_1,B_2)\geq 100\mathfrak{T}\Delta$ for any $B_1,B_2\in\mathcal{B}$. Then for each $\theta\in E\cap [0,\Delta]$ and $B\in\mathcal{B}$, we define
\begin{equation}\label{def:Atheta}
	\mathcal{A}_{\theta,B}:=\{J'\in\mathcal{A}: \Xi_{0,J'}\cap K_0\cap K_\theta\cap B\neq\emptyset\}
\end{equation}
and
\begin{equation}\label{def:Btheta}
	\mathcal{B}_\theta:=\{B\in\mathcal{B}: |\mathcal{A}_{\theta,B}|\geq \Delta^{\sigma-t+9\xi}\}.
\end{equation}
An easy observation shows that $\mathcal{A}_{\theta,B}\subset \mathcal{A}_{0,B}$ and $\mathcal{B}_\theta\subset \mathcal{B}_0$ for any $\theta\in E\cap [0,\Delta]$.

Let $S: [0,\Delta] \to [0,1]$ be the map defined by $S(x) = x/\Delta$.  We close this step by stating a lemma that will serve as a key tool in {\bf{Step 5}} for deriving a contradiction.

\begin{lemma}\label{lem-bundle}
	Under the notations defined above, we have the following properties.
	\begin{enumerate}
		\item[\textup{(i)}] For every $\theta\in E\cap [0,\Delta]$, $\mathcal B_\theta$ is a $(\Delta, \sigma, \delta^{-5\xi})$-set with $|\mathcal B_\theta|\geq \Delta^{-\sigma +8\xi}$. 
		
		\item[\textup{(ii)}] The set of dyadic $\Delta$-intervals $S(\mathcal A):=\{J'\in\mathcal A: S(J')\}$ is a $(\Delta, t-\sigma, \delta^{-6\xi})$-set. 
		
		\item[\textup{(iii)}] For any $\theta\in E\cap[0,\Delta]$, $|\mathcal{R}_\theta(K_\theta\cap \Xi)|_\delta\lesssim \Delta^{\sigma-t-2\xi}$. In particular, $|\mathcal A| \lesssim \Delta^{\sigma-t-2\xi}$. 
	\end{enumerate}
\end{lemma}

\begin{proof}
	To prove (i), we first show that for any $I'\in\mathcal{D}_\delta(\mathcal{R}_{0}(K))$ there holds
	\begin{equation}\label{form-64}
		\mu(\Xi_{0,I'}\cap K_0\cap {B})\lesssim \Delta^{-\sigma}\cdot \delta^{t-O_\xi(\epsilon)}, \quad  {B}\in \mathcal{B}.
	\end{equation}
	To see this, take ${B}\in \mathcal{B}$ and $f_0\in \Xi_{0,I'}\cap K_0\cap {B} \neq \emptyset$, then
	\[
	f_0\notin H_{0,\mathrm{loc}}(K,\sigma,\delta,\delta^{\sqrt{\epsilon}}) \quad \Longrightarrow \quad f_0\notin H_{0}(K,4(\tfrac{\Delta}{\delta})^\sigma, [8\delta, 8\Delta]),
	\]
	or in other words
	\[|\{g\in K\cap B(f_0,8\Delta): |g(0)-f_0(0)|\leq 8\delta\}|_{8\delta}\leq 4(\tfrac{\Delta}{\delta})^\sigma\sim \Delta^{-\sigma}.\]
	Since $f_0(0)\in I'$ and ${B}\subset B(f_0,8\Delta)$, this implies $|\Xi_{0,I'}\cap K_0\cap {B}|_\delta\lesssim \Delta^{-\sigma}$. Hence \eqref{form-64} follows easily from the regularity of $\mu$. 
	
	Recall $\mathcal{B}_\theta=\{B\in\mathcal{B}: |\mathcal{A}_{\theta,B}|\geq \Delta^{\sigma-t+9\xi}\}$, and let $\mathcal B_\theta^c = \mathcal B \, \setminus \, \mathcal B_\theta$.  By definition, for any ball ${B} \in \mathcal{B}_\theta^c$ we have $\Xi_{0, I'} \cap K_0\cap K_\theta\cap B\neq \emptyset$ for at most $\Delta^{\sigma-t+9\xi}$ different $I'\in \mathcal A$. Applying \eqref{form-64} for each of those intervals leads to
	\[\sum_{I'\in\mathcal A} \mu(\Xi_{0,I'}\cap K_0\cap K_\theta\cap {B})\leq \Delta^{t+8\xi}\]
	for every ${B}\in \mathcal{B}_\theta^c$, assuming $O_\xi(\epsilon)<\xi$. Summing over ${B}\in\mathcal{B}_\theta^c$ and using Lemma \ref{lem-setB}(ii),
	\begin{equation*}
		\sum_{{B}\in\mathcal{B}_\theta^c}\sum_{I'\in\mathcal{A}} \mu(\Xi_{0,I'}\cap K_0\cap K_\theta\cap {B})\leq |\mathcal{B}|\cdot \Delta^{t+8\xi}\leq \Delta^{t-\sigma+6\xi}.
	\end{equation*}
	On the other hand, we have the following lower bound for the full sum:
	\[\sum_{I'\in\mathcal{A}}\sum_{{B}\in \mathcal{B}} \mu(\Xi_{0,I'}\cap K_0\cap K_\theta\cap {B})\geq \sum_{B\in \mathcal{B}} \mu(K_0\cap K_\theta\cap {B}\cap \Xi)\stackrel{\eqref{form-ap5}}{\geq} \Delta^{t-\sigma+5\xi}.\]
	From the above two inequalities, we see that the full sum cannot be dominated by the part over ${B}\in \mathcal B_\theta^c$. Consequently,
	\begin{align}\label{eq-point1}
		\Delta^{t-\sigma+5\xi}&\leq 2 \sum_{I'\in\mathcal{A}}\sum_{B\in \mathcal{B}_\theta} \mu(\Xi_{0,I'}\cap K_0\cap K_\theta\cap B)\nonumber\\
		&\stackrel{\eqref{form-64}}{\lesssim} |\mathcal{A}|\cdot |\mathcal{B}_\theta|\cdot \Delta^{-\sigma}\cdot \delta^{t-O_\xi(\epsilon)}\stackrel{\eqref{form-p2}}{<} |\mathcal{B}_\theta|\cdot \Delta^{t-3\xi},
	\end{align}
	which implies $|\mathcal{B}_\theta|\geq \Delta^{-\sigma+8\xi}$. From this lower bound, and Lemma \ref{lem-setB}(i), we deduce that for any $R\geq \Delta$ and $f\in \mathcal{F}$,
	\[|\mathcal{B}_\theta\cap B(f,R)|\leq |\mathcal{B}\cap B(f, \delta^{-\sqrt{\epsilon}}R)|\lesssim \big(\tfrac{\delta^{-\sqrt{\epsilon}}R}{\Delta}\big)^{\sigma}\leq \delta^{-\sqrt{\epsilon}\sigma -4\xi} R^\sigma|\mathcal{B}_\theta|.\]
	Therefore $\mathcal B_\theta$ is a $(\Delta, \sigma, \delta^{-5\xi})$-set if we choose $\sqrt{\epsilon}\sigma<\xi$, as claimed in part (i).
	
	Now we verify property (ii). Taking $\Delta=1$ in Lemma \ref{lem-p1} and using \eqref{form-100}, we get
	\begin{equation}\label{form-101}
		|\mathcal{R}_{0}(K_0)|_\delta\gtrsim \delta^{O_\xi(\epsilon)}\delta^{\sigma-t}.
	\end{equation}
	From Lemma \ref{lem-p2} and \eqref{form-101}, we deduce that $\mathcal{R}_{0}(K_0)$ is a $(\delta,t-\sigma,\delta^{-O_\xi(\epsilon)-\xi})$-set. Then we know that $S(\mathcal{A})\neq\emptyset$ is a $(\Delta,t-\sigma,\overline{C})$-set, where
	\[\overline{C}:=\delta^{-O_\xi(\epsilon)-\xi}\cdot\frac{\Delta^{t-\sigma}|\mathcal{R}_{0}(K_0)|_\delta}{|\mathcal A|}.\]
	By \eqref{form-p22} we have
	\begin{equation*}\label{form-102}
		|\mathcal{R}_{0}(K_0)|_\delta\leq \delta^{-O_\xi(\epsilon)}\delta^{\sigma-t-\xi}.
	\end{equation*}
	Also, from the inequality \eqref{eq-point1} we get $|\mathcal{A}|\geq \delta^{O_\xi(\epsilon)}\Delta^{\sigma-t+5\xi}$. Therefore, by the two inequalities above, $\overline{C}<\delta^{-O_\xi(\epsilon)-5\xi}$. We conclude that $S(\mathcal{A})$ is a $(\Delta,t-\sigma,\delta^{-6\xi})$-set if we choose $O_\xi(\epsilon)<\xi$.

	It remains to show (iii). For every $\theta \in E \cap [0,\Delta]$ and $f\in K_\theta\cap \Xi$, we have \[|f(\theta)-f(0)|\leq |\theta|\leq \Delta.\]
	This shows that $\mathcal{R}_\theta(K_\theta\cap \Xi)$ is contained in a $3\Delta$-interval, say $J_\theta$, and by definition
	\begin{equation}\label{T0}
		\Xi=\{f\in K: f(0)\in [0,\Delta]\}\subset \{f\in K: f(\theta)\in J_\theta\}=\Xi_{\theta, J_\theta}.
	\end{equation}
	Thus, applying Lemma \ref{lem-p2} to $\Xi_{\theta, J_\theta}$ and taking $\epsilon$ small enough such that $O_\xi(\epsilon)\leq \xi$, we find that
	\begin{equation}\label{form-60}
		|\mathcal{R}_\theta(K_\theta\cap \Xi)|_\delta\lesssim \delta^{-O_\xi(\epsilon)-\xi}\cdot \Delta^{\sigma-t}\leq \Delta^{\sigma-t-2\xi},\quad \theta\in E\cap [0,\Delta],
	\end{equation}
	as desired. 
\end{proof}

\subsection*{Step 5: contradicting $\delta$-discretised projection theorem}
We first state Theorem \ref{discretised projection}, which is the "linear" special case of Theorem \ref{proj-regular}, and follows from the $\delta$-discretised Furstenberg estimate \cite[Theorem 4.1]{2023arXiv230808819R}. We will not repeat the details of this deduction, since they are in the literature, although a little scattered. To fill in the details, the reader should check how (the projection theorem) \cite[Corollary 6.1]{2023arXiv230110199O} is deduced from (the Furstenberg set estimate) \cite[Theorem 5.35]{2023arXiv230110199O} . The argument is the same here, although the numerology is different (since both \cite[Theorem 5.35]{2023arXiv230110199O} and \cite[Corollary 6.1]{2023arXiv230110199O} are unsharp). Finally, the reader should note that Theorem \ref{discretised projection} is the "dual" version of a $\delta$-discretised projection theorem with exactly same format as \cite[Corollary 6.1]{2023arXiv230110199O} (thus, the $\delta$-covering number of slices in Theorem \ref{discretised projection} corresponds to the $\delta$-covering number of the projections in \cite[Corollary 6.1]{2023arXiv230110199O}). 

\begin{definition}\label{def:tubes}
	Let $\delta\in (0,1)$. A \emph{$\delta$-tube} $T\subset [-1,1]^2$ is a rectangle of side lengths $(2\times \delta)$. We say two $\delta$-tubes $T_1, T_2$ are \emph{distinct} if $|T_1\cap T_2|\leq \tfrac{1}{10}|T_i|$ with $i=1,2$.
\end{definition}

\begin{definition}\label{def:spacing of tube}
	Let $\delta\in (0,1)$ and $s\in [0,2]$. Let $\mathcal{T}$ be a distinct family of $\delta$-tubes in $[-1,1]^2$. We say $\mathcal{T}$ is a $(\delta,s,\mathbf{C})$-set if for any $10(2\times r)$-rectangle $\mathbb{T}$ with $r\in[\delta,1]$ 
	\[|\mathcal{T}\cap \mathbb{T}|\leq \mathbf{C} r^s|\mathcal{T}|,\]
	where $\mathcal{T}\cap \mathbb{T}:=\{T\in \mathcal{T}: T\subset \mathbb{T}\}$.
\end{definition}

\begin{thm}\label{discretised projection}
	Let $s\in (0,1]$ and $t\in [s, 2-s]$. Then for every $u\in (0,\tfrac{s+t}{2})$, there exist $\epsilon=\epsilon(s,t,u)>0$ and $\delta_0=\delta_0(s,t,u)>0$ such that the following holds for all $\delta\in(0,\delta_0]$.
	
	Assume that $\mathcal{T}$ is a $(\delta,t,\delta^{-\epsilon})$-set of distinct $\delta$-tubes in $[-1,1]^2$, and $P\subset [0,1]$ is a $\delta$-separated $(\delta,s,\delta^{-\epsilon})$-set. Assume that $P\subset \pi_{\mathtt x}(T)$ for each $T\in\mathcal{T}$. Then, there exists $P'\subset P$ with $|P'|\geq\tfrac{1}{2}|P|$ such that for $\theta\in P'$:
	\[\Bigg|\bigcup_{T\in\mathcal{T}'} T\cap L_\theta\Bigg|_\delta\geq \delta^{-u},\quad \text{ for all } \mathcal{T}'\subset\mathcal{T} \quad \text{with} \quad |\mathcal{T}'|\geq \delta^\epsilon |\mathcal{T}|.\]
\end{thm}

Roughly speaking, our strategy is to apply Lemma \ref{lem-bundle} to construct a family of $t$-dimensional $8\Delta(2\times\Delta)$-rectangles in $[-8\Delta,8\Delta]^2$ such that, for each $\theta \in E \cap [0, \Delta]$, a large sub-family intersects the vertical line $L_\theta$ in dimension at most $(t - \sigma)$. Moreover, after rescaling by the factor $8\Delta$, the corresponding family of $\Delta$-tubes becomes a $(\Delta, t)$-set, while the renormalized set $E_J \subset [0,1]$ is a $(\Delta, s - \sqrt{\xi})$-set by Proposition \ref{pro-set E}. This eventually enables us to reach a contradiction with Theorem \ref{discretised projection}.

Before the construction, we record the following simple but useful lemma. The proof uses only elementary calculus.
\begin{lemma}\label{lem-simple}
	For any $f, g \in C^2([0,1])$ with $\|f-g\|_{ C^2([0,1])}\leq 2\Delta$, if $|f(0)-g(0)|\leq \delta$, then 
	\[|f(\theta)-g(\theta)|\leq 3\delta,\qquad \theta \in [0,\Delta].\]
\end{lemma}

Now we are ready to construct the tube family. We claim that for every $J'\in \mathcal A= \mathcal D_\delta(\mathcal R_0(K_0\cap \Xi))\subset [0,\Delta]$ and $B\in \mathcal{B}_0$, there exists a rectangle $R_{{B}, J'}$ of side lengths $8\Delta(2\times \Delta)$ such that
\[\Bigg(\bigcup_{f\in {B}\cap \Xi_{0,J'}} \Gamma_{f|_{[0,\Delta]}}\Bigg)\subset R_{B,J'}\subset [-8\Delta,8\Delta]^2.\]
Indeed, fix any $f := f_{B,J'} \in B \cap \Xi_{0,J'}$, let $\ell_0=\{(x,y)\in \R^2: y=f'(0)x+f(0)\}$ be the tangent line of $\Gamma_f$ at $0$, then for any $g\in\Xi_{0,J'}\cap B$ and $x\in[0,\Delta]$, we have by triangle inequality and Lemma \ref{lem-simple} that
\[\dist((x,g(x)),\ell_0)=\frac{|f'(0)x+f(0)-g(x)|}{\sqrt{1+f'(0)^2}}\leq 4\delta.\]
This means $\Gamma_{g|_{[0,\Delta]}}\subset \ell_0[4\delta]$ where $\ell_0[4\delta]$ is the $4\delta$-neighbourhood of $\ell_0$. Taking $R_{B, J'}$ to be a truncated part of $\ell_0[4\delta]$ with its $x$-projection containing $[0,\Delta]$, the claim follows. In the sequel, the \emph{slope of $R_{B, J'}$} means the slope of the centre line $\ell_0$.

Write $\mathcal{R} :=\{R_{B,J'}: {B}\in \mathcal B_0, J'\in\mathcal{A}\}$. Then, by Lemma \ref{lem-bundle}(i) and Lemma \ref{lem-setB}(ii),
\begin{equation}\label{form-91}
	\Delta^{-t+20\xi}\leq|\mathcal{R}|\leq \Delta^{-t-3\xi}.
\end{equation}
For $x\in\R^2$, let $\bar{S}(x) = x/8\Delta$. Then $\bar{S}(\mathcal{R}):=\{\bar{S}(R_{{B},J'}): R_{{B},J'}\in\mathcal{R}\}\subset [-1,1]^2$ is a family of $\Delta$-tubes. Abbreviate
\[T_{{B},J'}:=\bar{S}(R_{{B},J'})\quad\text{and}\quad \mathcal T:=\bar{S}(\mathcal{R})\]

Recall (from above Lemma \ref{lem-setB}) that $d(B_1,B_2)\geq 100\mathfrak{T}\Delta$ for all distinct $B_1, B_2\in\mathcal{B}$, and (from above \eqref{def:Atheta}) that $d(I_1,I_2)\geq 100\delta $ for all distinct $I_1, I_2\in\mathcal{A}$. This condition ensures that the tube family we constructed above is a distinct family.

\begin{lemma}\label{lem-distinct}
	The tube family $\mathcal{T}$ is distinct in the sense of Definition \ref{def:tubes}.
\end{lemma}
\begin{proof}
	Fix $R_{B_1,J'_1}\neq R_{B_2,J'_2}\in \mathcal{R}$. If $B_1=B_2$ and $J'_1\neq J'_2$, then since $d(J'_1,J'_2)\geq 100\delta$ we deduce that $R_{B_1,J'_1}\cap R_{B_2,J'_2}=\emptyset$. If $B_1\neq B_2$, then $d(B_1,B_2)\geq 100\mathfrak{T}\Delta$. We may also assume $R_{B_1,J'_1}\cap R_{B_2,J'_2}\neq\emptyset$ as otherwise $T_{B_1,J_1'}$ and $T_{B_2,J_2'}$ are clearly distinct. Fix $(z,w)\in R_{B_1,J'_1}\cap R_{B_2,J'_2}$ and let $f_1 \in B_{1}$ and $f_2 \in B_{2}$ be the functions used to define $R_{B_{1},J_{1}'}$ and $R_{B_{2},J_{2}'}$, so in particular $d(f_{1},f_{2}) \geq 100\mathfrak{T}\Delta$, and
	\begin{displaymath} \Gamma_{f_1|_{[0,\Delta]}} \subset R_{B_1,J'_1} \quad \text{and} \quad \Gamma_{f_2|_{[0,\Delta]}} \subset R_{B_2,J'_2}. \end{displaymath}
	By definition, the slopes of $R_{B_1,J'_1}, R_{B_2,J'_2}$  are $f_1'(0), f_2'(0)$, respectively. By the triangle inequality $|f_1(z)-f_2(z)|\leq 50\delta$, hence by the transversality of $\mathcal{F}$ we further deduce that $|f_1'(z)-f_2'(z)|\geq 99\Delta$ if $\delta$ is sufficiently small. Using the mean value theorem and the triangle inequality, $|f_1'(0)-f_2'(0)|\geq 97\Delta$. By simple geometry, the area of the intersection
	\[|R_{B_1,J'_1}\cap R_{B_2,J'_2}|< \frac{\delta^2}{\Delta}=\Delta\cdot\delta.\]
	After rescaling by $8\Delta$, we get the distinctness of $\mathcal{T}$.
\end{proof}

We next verify that $\mathcal T$ satisfies the following non-concentration condition.
\begin{lemma}\label{lem-rectangles}
	The tube family $\mathcal T$ is a $(\Delta,t,\Delta^{-50\xi})$-set.
\end{lemma}
\begin{proof}
	Take an arbitrary $10(2\times r)$-rectangle $T_r$ with $r\in [\Delta, 1]$. We aim to show 
	\begin{equation}\label{form-93}
		|\mathcal T \cap T_r|\leq \Delta^{-50\xi}r^t |\mathcal T|,
	\end{equation}
	where $\mathcal T \cap T_r=\{T_{{B},J'}\in \mathcal{T}: T_{{B},J'}\subset T_r\}$. After rescaling back, it suffices to show
	\begin{equation}\label{form-94}
		|\mathcal{R}\cap R_{\Delta r}|\leq \Delta^{-50\xi}r^t |\mathcal{R}|,
	\end{equation}
	where $R_{\Delta r}$ is a $80\Delta(2\times r)$-rectangle and $\mathcal{R}\cap R_{\Delta r}=\{R_{{B},J'}\in {\bf R}: R_{{B},J'}\subset R_{\Delta r}\}$. 
	
	First, we claim that
	\begin{equation}\label{form45} |\{B\in \mathcal{B}_0: \exists J'\in\mathcal A \text{ such that } R_{B,J'}\subset R_{\Delta r}\}| \leq \delta^{-6\xi}r^{\sigma}|\mathcal{B}_{0}|, \end{equation}
	provided that $\delta > 0$ is small enough. We will show that the balls $B$ as in \eqref{form45} are contained in a single ball of radius $\sim_{\mathfrak{T}} r$, and use the $(\Delta,\sigma)$-property from Lemma \ref{lem-bundle}(i).
	
	If there exists a rectangle $R_{B,J'}\subset R_{\Delta r}$, the angle between the longer sides of $R_{B,J'}$ and $R_{\Delta r}$ is $\lesssim r$. Since the slope of $R_{B,J'}$ is $\leq 1$, the slope of $R_{\Delta r}$ is $\lesssim 1$. Let ${B}_1, {B}_2\in \mathcal{B}_0$ be such that $R_{{B}_1, J'_1}, R_{{B}_2, J'_2}\subset R_{\Delta r}$ for some $J'_1, J'_2\in \mathcal A$, and fix $f_1\in {B}_1, f_2\in {B}_2$ such that $\Gamma_{f_1|_{[0,\Delta]}}$ and $\Gamma_{f_2|_{[0,\Delta]}}$ are contained in $R_{\Delta r}$. Since the slope of $R_{\Delta r}$ is $\lesssim 1$, $|f_1(x)  -f_2(x)| \leq C\Delta r$ for every $x\in [0,\Delta]$, where $C > 0$ is absolute. Furthermore, $|f_1'(x_0) - f_2'(x_0)| \lesssim r$ for at least one $x_0\in[0,\Delta]$, since otherwise either $|f_{1}(0) - f_{2}(0)| > C\Delta r$ or $|f_1(\Delta) - f_2(\Delta)| > C\Delta r$ by the mean value theorem. It then follows from the transversality of $\mathcal{F}$ that
	\[2r \gtrsim |f_1(x_0)-f_2(x_0)|+|f_1'(x_0)-f_2'(x_0)|\geq \mathfrak{T}^{-1}\|f_1-f_2\|_{C^2(I)} \geq \mathfrak{T}^{-1}\dist_{C^{2}}(B_{1},B_{2}).\]
	This establishes the $r$-ball containment asserted below \eqref{form45}, and then from Lemma \ref{lem-bundle}(i), we obtain
	\[|\{B\in \mathcal{B}_0: \exists J'\in\mathcal A \text{ such that } R_{B,J'}\subset R_{\Delta r}\}|\lesssim \delta^{-5\xi}(Ar)^\sigma|\mathcal{B}_0|. \]
	This implies \eqref{form45} if $\delta$ small enough in terms of $\mathfrak{T}$. 
	
	Second, for each fixed $B \in \mathcal B_0$, we estimate the number of rectangles in $\{R_{{B},J'}: J'\in \mathcal A\}$ contained in $R_{\Delta r}$. If there exists $R_{B,J'}\subset R_{\Delta r}$, the slope of $R_{\Delta,r}$ is $\lesssim1$, hence $J'$ is contained in the $2\delta$-neighborhood of $R_{\Delta r}\cap L_0$ which is an interval $I_{\Delta r}$ of length $\sim \Delta r$ (note $\delta\leq \Delta r$). Now it suffieces to estimate the number of $J'\in \mathcal{A}$ such that $J'\subset I_{\Delta r}$. After rescaling by $\Delta$, the problem is equivalent to estimating the number of $\Delta$-intervals in $S(\mathcal A)$ contained in an interval $I_r$ of length $\sim r$. By Lemma \ref{lem-bundle}(ii), $|S(\mathcal A)\cap I_r|_\Delta\lesssim \delta^{-6\xi}r^{t-\sigma}|\mathcal A|$.
	
	Finally, recalling that $|\mathcal A| \lesssim \Delta^{\sigma-t-2\xi}$ by Lemma \ref{lem-bundle}(iii), and $|\mathcal{B}_{0}| \lesssim_{\mathfrak{T}} \Delta^{-\sigma}$ by Lemma \ref{lem-setB}(ii),
	\[|\mathcal{R}\cap R_{\Delta r}| \stackrel{\eqref{form45}}{\lesssim} \delta^{-12\xi}r^\sigma|\mathcal{B}_{0}|\cdot r^{t-\sigma}|\mathcal A| \stackrel{\eqref{form-91}}{\leq} \Delta^{-50\xi}r^{t}|\mathcal{R}|.\]
	This completes the proof of \eqref{form-93}.
\end{proof}

To proceed, let $E^\ast\subset E$ be a maximal $8\delta$-separated subset. Since $E_J=E_{[0,\Delta]} \subset [0,1]$ is a $(\Delta,s-\sqrt{\xi}, \Delta^{-O_\xi(\epsilon)})$-set by Proposition \ref{pro-set E}, it follows that $\bar{S}(E^\ast)=\{\theta/8\Delta:\theta\in E^\ast\}\subset [0,\tfrac{1}{8}]$ is also a $(\Delta,s-\sqrt{\xi}, \Delta^{-O_\xi(\epsilon)})$-set. From our construction, the $x$-projection of each rectangle in $\mathbf{R}$ contains $[0,\Delta]$, thus $\bar{S}(E^\ast)\subset \pi_{\mathtt x}(T)$ for each $T\in\mathcal{T}$.

The following lemma will eventually conclude the whole proof.

\begin{lemma}\label{lem-subset}
	For each $\bar \theta \in \bar{S}(E^\ast)$, there is a sub-family $\mathcal T_{\bar \theta} \subset \mathcal T$ with $|\mathcal T_{\bar \theta}|\geq \Delta^{22\xi} |\mathcal T|$ such that
	\begin{equation}\label{form-104}
		\Bigg|\bigcup_{T \in \mathcal T_{\bar \theta}} T\cap L_{\bar{\theta}}\Bigg|_\Delta\leq \Delta^{\sigma-t-3\xi}.	
	\end{equation}
\end{lemma}
\begin{proof}
	Fix one $\theta\in E^\ast\subset E$ and let $\bar{\theta}=\theta/8\Delta\in \bar{S}(E^\ast)$. We begin by defining
	\[\mathcal{R}^{\ast} := \{R_{{B},J'}: {B}\in \mathcal{B}_\theta,\ J'\in \mathcal A_{\theta,B}\}.\]
	Recall $\mathcal{B}_\theta \subset \mathcal{B}_0$ and $\mathcal{A}_{\theta,B}\subset\mathcal{A}$ for any $\theta \in E \cap [0, \Delta]$, hence $\mathcal{R}^* \subset \mathcal{R}$. Correspondingly, we define a subfamily of $\mathcal{T}$:
	\[\mathcal T_{\bar{\theta}}:=\{T_{{B},J'}: {B}\in \mathcal{B}_\theta, J'\in \mathcal A_{\theta,B} \}=\bar{S}(\mathcal{R}^*)\subset \mathcal{T}.\] 
	Using the definition of $\mathcal{B}_\theta$ and Lemma \ref{lem-bundle} (i), we obtain the estimate:
	\[|\mathcal{R}^\ast|\geq \Delta^{\sigma-t+9\xi}\Delta^{-\sigma+8\xi}\stackrel{\eqref{form-91}}{\gtrsim}\Delta^{20\xi}|\mathcal{R}|,\]
	which implies $|\mathcal T_{\bar{\theta}}|\geq \Delta^{22\xi} |\mathcal T|$. Moreover, let $B\in\mathcal{B}_\theta$ and $J'\in \mathcal{A}_{\theta,B}$, then for any $f\in B\cap \Xi_{0,J'}\subset \Xi$ we can find $g\in \Xi_{0,J'}\cap K_0\cap K_\theta\cap B\subset K_\theta\cap \Xi$ such that $|f(0)-g(0)|\leq\delta$. Then by Lemma \ref{lem-simple}, $|f(\theta)-g(\theta)|\leq 3\delta$, which means $\mathcal{R}_\theta(\bigcup_{B\in\mathcal{B}_\theta, J'\in \mathcal{A}_{\theta,B}}B\cap \Xi_{0,J'})$ belongs to the $3\delta$-neighborhood of $\mathcal{R}_\theta(K_\theta\cap \Xi)$. By using this fact, we deduce
	\[\begin{split}
		\Bigg|\bigcup_{T_{{B},J'}\in \mathcal T_{\bar{\theta}}} T_{{B},J'}\cap L_{\bar{\theta}}\Bigg|_\Delta &\sim \Bigg|\bigcup_{R_{{B},J'}\in \mathcal{R}^\ast}R_{{B},J'}\cap L_\theta\Bigg|_\delta\sim \Bigg|\mathcal{R}_\theta\Bigg(\bigcup_{B\in\mathcal{B}_\theta, J'\in \mathcal{A}_{\theta,B}}B\cap \Xi_{0,J'}\Bigg)\Bigg|_\delta\\
		&\lesssim |\mathcal{R}_\theta(K_\theta\cap \Xi)|_\delta\stackrel{\eqref{form-60}}{\lesssim} \Delta^{\sigma-t-2\xi},
	\end{split}\]
	proving \eqref{form-104}.
\end{proof}

In conclusion, we have established the following results:
\begin{itemize}
	\item[(i)] $\bar{S}(E^\ast) \subset [0,1]$ is a $(\Delta,s-\sqrt{\xi}, \Delta^{-O_\xi(\epsilon)})$-set, thus also a $(\Delta,s,\Delta^{-2\sqrt{\xi}})$-set.
	
	\item[(ii)] $\mathcal{T}$ is a $(\Delta,t,\Delta^{-50\xi})$-set of distinct $\delta$-tubes contained in $[-1,1]^2$.
	
	\item[(iii)] For each $\bar{\theta}\in \bar{S}(E^\ast)$, there exists $\mathcal{T}_{\bar{\theta}}\subset \mathcal{T}$ with $|\mathcal{T}_{\bar{\theta}}|\geq \Delta^{22\xi} |\mathcal{T}|$ such that \eqref{form-104} holds.
\end{itemize}
Recall from assumption \eqref{form-parameters} that $t\in (s,2-s)$. By applying Theorem \ref{discretised projection}, there exists $\bar{\theta}\in \bar{S}(E^\ast)$ such that
\begin{equation}\label{equ-lower}
	\Bigg|\bigcup_{T\in \mathcal T_{\bar{\theta}}} T\cap L_{\bar{\theta}}\Bigg|_\Delta\geq \Delta^{-u}, \quad u \in \left(0, \tfrac{s + t}{2} \right),
\end{equation}
provided $\max\{2\sqrt{\xi}, 50\xi\}<\bar{\epsilon}(s,t,u)$ and $\Delta<\bar{\delta}_0(s,t,u)$, where $\bar{\epsilon},\bar{\delta}_0$ are the constants from Theorem \ref{discretised projection}. In particular, taking $u = \tfrac{s + t}{2} - \xi_0$ in \eqref{equ-lower} and recalling \eqref{form-104},
\[\tfrac{s + t}{2} - \xi_0 \leq t - \sigma + 3\xi,\]
which implies
\[\sigma \leq \tfrac{t - s}{2} + \xi_0 + 3\xi < \tfrac{t - s}{2} + 4\xi_0,\]
contradicting our assumption \eqref{form-parameters2}. This completes the whole proof of Proposition~\ref{pro-reduction}.


\section{Regular case}\label{sec5}

	Let $\mathcal X \subset C^2([-2,2])\cap B(1)$ be a transversal family with constant $\mathfrak T\geq 1$. As before, $\mathcal X$ is equipped with the metric induced by the $C^2$ norm $\Vert \cdot \Vert_{C^2([-2,2])}$. Throughout this section, for any $r\in 2^{-\N}$, the symbol $\mathcal D_r$ refers to the dyadic partition of $\mathcal X$ given by Definition \ref{def-pullbackdyadic}. For example, for any $\mathcal F\subset \mathcal X$, 
\[
\mathcal D_r (\mathcal F) = \lbrace {\bf F} \in \mathcal D_{r}(\mathcal X):\ {\bf F}\cap \mathcal F\neq \emptyset\rbrace.
\]
We begin by defining a \emph{nice configuration}, a discrete analogue of a Furstenberg set. 

\begin{definition}
    Let $\mathcal F$ be a collection of subsets of $\mathcal X$ and $\delta>0$. The collection $\mathcal F$ is called $\delta$-\emph{admissible} if there exist constants $C_1, C_2>0$ depending only on $\mathcal X$ such that
    \begin{enumerate}
    \item each ${\bf f}\in\mathcal F$ is contained in a ball $B_{\bf f} \subset C^2([-2,2])$ of radius $C_1\delta$, and 
    \item the balls $B_{\bf f}$ have $C_2$-bounded overlap, that is, $\sum_{{\bf f}\in\mathcal F} \mathbf{1}_{B_{\bf f}} \leq C_2$.
    \end{enumerate}
\end{definition}

Examples of $\delta$-admissible sets include $\delta$-separated subsets of $\mathcal X$ (when a function $f$ is identified with $\lbrace f\rbrace$) and, by Corollary \ref{cor-covernumber}, the sets $\mathcal D_\delta(\mathcal F)$ for $\mathcal F\subset \mathcal X$. A property of a $\delta$-admissible set $\mathcal F$ that we will use without further reference is that 
\begin{equation*}
    |\mathcal F| \sim |\mathcal F|_\delta
\end{equation*}
where the implicit constant depends only on $C_1,C_2$ in the definition of admissibility. When $\mathcal F$ is $\delta$-admissible and $\delta\leq \Delta\leq 1$, we use the conventions $|\mathcal F|_\Delta := |\bigcup_{{\bf f} \in \mathcal F} {\bf f}|_\Delta$ and $\mathcal D_{\Delta}(\mathcal F):= \mathcal D_{\Delta}(\bigcup_{{\bf f}\in\mathcal F} {\bf f})$. 
    
\begin{definition}[Nice configuration]\label{def-niceconfiguration}
Fix $s\in [0,1]$, $\delta\in 2^{-\N}$, $C>0$ and $M\geq1$. Let $\mathcal F$ be a $\delta$-admissible collection of subsets of $\mathcal X$, and ${\mathcal{P}}\subset \mathcal D_\delta([0,1]^2)$. We say that the pair $(\mathcal{F}, {\mathcal{P}})$ is a $(\delta, s, C, M)$-nice configuration if for every ${\bf f}\in \mathcal F$ there exists $\mathcal P({\bf f})\subset \mathcal P$ such that $|\mathcal P({\bf f})| = M$ and $p\cap \Gamma_{\bf f}\neq\emptyset$ for every $p\in \mathcal P({\bf f})$, where 
\[
\Gamma_{\bf f} := \bigcup_{f\in {\bf f}}\Gamma_f.
\]
\end{definition}

The following theorem is the main result of this section. The notation $A \lessapprox_\delta B$ means that $A\leq (C\log\delta^{-1})^C B$ for some $C>0$ depending only on $\mathfrak{T}$, the transversality constant of $\mathcal X$. The notations $\gtrapprox_\delta$ and $\approx_\delta$ are defined similarly. The notation $A\sim B$ means that $C^{-1}A\leq B\leq CA$ for some $C\geq 1$ depending only on $\mathfrak T$.
\begin{thm}\label{thm-furstenbergset}
Let $s \in (0,1]$, $t\in [s,2 - s]$ and $0\leq u < \frac{s+t}{2}$. There exist $\varepsilon = \varepsilon(s,t,u)>0$ and $\delta_0 = \delta_0(s,t,u, \mathfrak{T})>0$ such that the following holds for all $\delta\in (0,\delta_0]$. 

Let $(\mathcal{F}, {\mathcal{P}})$ be a $(\delta, s, \delta^{-\varepsilon}, M)$-nice configuration, where $\mathcal{F}$ is $(\delta, t, \delta^{-\varepsilon})$-regular. Then, 
\[\Big|\bigcup_{f\in\mathcal{F}} {\mathcal{P}}(f)\Big| \geq M \cdot \delta^{-u}.\]
\end{thm}

Theorem \ref{thm-furstenbergset} is a generalisation of (the dual version of) \cite[Theorem 5.7]{2023arXiv230110199O} to the setting where $\mathcal F$ is a transversal family of $C^2$-functions instead of affine maps. Theorem \ref{thm-furstenbergset} is deduced from Theorem \ref{proj-regular}, combined with an induction-on-scales mechanism that was introduced in \cite{2023arXiv230110199O}. The mechanism translates into our setting without too much difficulty. Indeed, once we have Theorem \ref{proj-regular} in our use, the transversality condition is not confronted elsewhere in the proof of Theorem \ref{thm-furstenbergset}, save for the definition of the dyadic system we introduced in Definition \ref{def-pullbackdyadic}. Switching between covers by dyadic cubes and balls in $\mathcal X$ induces the $\mathfrak{T}$-dependence on constants similarly as in Lemma \ref{cor-covernumber}.

In fact, a compactness argument shows that the constant $\varepsilon = \varepsilon(s,t,u) > 0$ in Theorem \ref{thm-furstenbergset} can be taken independent of the $t$-variable. We state this result in the following form:
\begin{cor}\label{cor1}
Let $s \in (0,1)$ and $t\in [s,2-s]$. For every $\eta>0$, there exists $\epsilon=\epsilon(\eta,s)>0$ and $\delta_0=\delta_0(\eta,s,t,\mathfrak{T})>0$ such that the following holds for all $\delta\in (0,\delta_0)$. Let $(\mathcal{F}, {\mathcal{P}})$ be a $(\delta, s, \delta^{-\epsilon}, M)$-nice configuration, where $\mathcal{F}$ is $(\delta, t, \delta^{-\epsilon})$-regular. Then, 
\[\bigg|\bigcup_{f\in\mathcal{F}} {\mathcal{P}}(f)\bigg| \geq M \cdot \delta^{-\tfrac{s+t}{2}+\eta}.\]
\end{cor}

The proof is the same as given below  \cite[Theorem 6.1]{2023arXiv230808819R}, but we repeat the details.

\begin{proof}[Proof of Corollary \ref{cor1}] Fix $s \in (0,1]$ and $\eta \in (0,s]$. For every $t \in [s,2 - s]$, let 
\begin{displaymath} \epsilon_{0} := \epsilon_{0}(s,t,(s + t - \eta)/2) > 0 \end{displaymath}
be the constant provided by Theorem \ref{thm-furstenbergset}. Write $\epsilon_{t} := \tfrac{1}{2}\min\{\epsilon_{0},\eta\}$. Thanks to the compactness of $[s,2 - s]$, there exists a finite set $\{t_{1},\ldots,t_{N}\} \subset [s,2 - s]$ such that 
\begin{displaymath} [s,2 - s] \subset \bigcup_{j = 1}^{N} B(t_{j},\epsilon_{t_{j}}). \end{displaymath}
Write $\epsilon := \epsilon(s,\eta) := \min\{\epsilon_{t_{j}} : 1 \leq j \leq N\} > 0$. Now, we claim that if $t \in [s,2 - s]$, and $(\mathcal{F},\mathcal{P})$ is a $(\delta,s,\delta^{-\epsilon},M)$-nice configuration, where $\mathcal{F}$ is $(\delta,t,\delta^{-\epsilon})$-regular, then 
\begin{equation}\label{form44} \bigg|\bigcup_{f\in\mathcal{F}} {\mathcal{P}}(f)\bigg| \geq M \cdot \delta^{-(s+t)/2 +\eta}. \end{equation}
Fix $t \in [s,2 - s]$, and find $1 \leq j \leq N$ such that $|t - t_{j}| \leq \epsilon_{t_{j}}$. Let $(\mathcal{F},\mathcal{P})$ be a $(\delta,s,\delta^{-\epsilon},M)$-nice configuration, where $\mathcal{F}$ is $(\delta,t,\delta^{-\epsilon})$-regular. We claim that $(\mathcal{F},\mathcal{P})$ is automatically a $(\delta,s,\delta^{-\epsilon_{0}(s,t_{j},(s + t_{j} - \eta)/2)},M)$-nice configuration, where $\mathcal{F}$ is $(\delta,t_{j},\delta^{-\epsilon_{0}(s,t_{j},(s + t_{j} - \eta)/2})$-regular. The niceness part is clear, since $\delta^{-\epsilon} \leq \delta^{-\epsilon_{t_{j}}} \leq \delta^{-\epsilon_{0}(s,t_{j},(s + t_{j})/2 - \eta)}$. 

To check the $(\delta,t_{j},\delta^{-\epsilon_{0}(s,t_{j},(s + t_{j})/2 - \eta)})$-regularity, we use the $(\delta,t,\delta^{-\epsilon})$-regularity of $\mathcal{F}$:
\begin{itemize} 
\item[(a)] $\mathcal{F}$ is a $(\delta,t,\delta^{-\epsilon})$-set, so for every $r \in [\delta,1]$,
\begin{displaymath} |\mathcal{F} \cap B(x,r)|_{\delta} \leq \delta^{-\epsilon}r^{t}|\mathcal{F}|_{\delta} \leq \delta^{-\epsilon - |t - t_{j}|}r^{t_{j}}|\mathcal{F}|_{\delta} \leq \delta^{-\epsilon_{0}(s,t_{j},(s + t_{j} - \eta)/2)}r^{t_{j}}|\mathcal{F}|_{\delta}. \end{displaymath}
The final inequality follows from $\epsilon + |t - t_{j}| \leq 2\epsilon_{t_{j}} \leq \epsilon_{0}(s,t_{j},(s + t_{j} - \eta)/2)$. Thus $\mathcal{F}$ is a $(\delta,t_{j},\delta^{-\epsilon_{0}(s,t_{j},(s + t_{j} - \eta)/2})$-set. 
\item[(b)] $\mathcal{F}$ is upper $(\delta,t,\delta^{-\epsilon})$-regular, so for every $\delta \leq r \leq R \leq 1$,
\begin{align*} |\mathcal{F} \cap B(x,R)|_{r} & \leq \delta^{-\epsilon}(R/r)^{t} \leq \delta^{-\epsilon - |t - t_{j}|}(R/r)^{t_{j}}\\
& \leq \delta^{-\epsilon - |t - t_{j}|}r^{t_{j}}|\mathcal{F}|_{\delta} \leq \delta^{-\epsilon_{0}(s,t_{j},(s + t_{j} - \eta)/2)}(R/r)^{t_{j}}, \end{align*} 
where the final inequality is deduced as in (a). We have now shown that $\mathcal{F}$ is upper $(\delta,t_{j},\delta^{-\epsilon_{0}(s,t_{j},(s + t_{j} - \eta)/2)})$-regular.
\end{itemize}
By the definition of the parameter $\epsilon_{0}(s,t_{j},(s + t_{j} - \eta)/2)$, we may now infer from Theorem \ref{thm-furstenbergset} that
\begin{displaymath} \bigg|\bigcup_{f\in\mathcal{F}} {\mathcal{P}}(f)\bigg| \geq M \cdot \delta^{-(s + t_{j} - \eta)/2} \geq M \cdot \delta^{-(s + t)/2 + \eta},\end{displaymath}
where the final inequality follows from $|t - t_{j}| \leq \epsilon_{t_{j}} \leq \eta/2$. This proves \eqref{form44}. \end{proof} 

Before the proof of Theorem \ref{thm-furstenbergset}, we introduce auxiliary results which are used to find suitable refinements of the configuration $(\mathcal F, \mathcal P)$ in the statement of the main theorem. 

\begin{definition}[Refinement of $\mathcal{F}$]\label{def:refinement1}
	Let $\mathcal F$ be a $\delta$-admissible collection of subsets of $\mathcal X$. We say $\mathcal{F}'$ is a refinement of $\mathcal{F}$ if $\mathcal{F}'\subset \mathcal{F}$ and $|\mathcal{F}'|\approx_\delta |\mathcal{F}|$. Let $\Delta\in (\delta,1)$, we say $\mathcal{F}'$ is a refinement of $\mathcal{F}$ at scale $\Delta$ if $\mathcal{F}'\subset \mathcal{F}$ and $|\mathcal{F}'|_\Delta\approx_\Delta |\mathcal{F}|_\Delta$. 
\end{definition} 

\begin{definition}[Refinement of a nice configuration]\label{def:refinement2}
	Let $(\mathcal{F}_0,\mathcal{P}_0)$ be a $(\delta,s,C_0,M_0)$-nice configuration. We say $(\mathcal{F},\mathcal{P})\subset (\mathcal{F}_0,\mathcal{P}_0)$ is a refinement of $(\mathcal{F}_0,\mathcal{P}_0)$ if $|\mathcal{F}| \approx_\delta |\mathcal{F}_0|$, and for each ${\bf f}\in \mathcal{F}$ there is $\mathcal{P}({\bf f})\subset \mathcal{P}_0({\bf f})\cap \mathcal{P}$ such that $\sum_{{\bf f}\in \mathcal{F}} |\mathcal{P}({\bf f})|\gtrapprox_\delta M_0\cdot |\mathcal{F}_0|_\delta$. 
\end{definition}

\begin{definition}\label{def:refinement3}
	Let $\delta\leq\Delta\in 2^{-\N}$, $s\in[0,1]$ and $C,M>0$. We say a configuration $(\mathcal{F}_\Delta,\mathcal{P}_\Delta)\subset \mathcal{D}_\Delta(\mathcal{F})\times \mathcal{D}_\Delta$ covers a $(\delta,s,C,M)$-nice configuration $(\mathcal{F},\mathcal{P})$, if
	\[\sum_{\mathbf{F}\in \mathcal{F}_\Delta}\sum_{{\bf f}\in \mathbf{F}\cap \mathcal{F}}\sum_{Q\in \mathcal{P}_\Delta(\mathbf{F})} |\mathcal{P}({\bf f})\cap Q| \gtrapprox_\delta M\cdot |\mathcal{F}|_\delta,\]
	where $\mathcal{P}_\Delta(\mathbf{F})=\{Q\in\mathcal{P}_\Delta: \exists {\bf f}\in \mathcal{F}\cap\mathbf{F} \text{ such that }Q\cap\Gamma_{\bf f}\neq \emptyset\}$. \end{definition}

The following proposition combines \cite[Proposition 4.3]{2023arXiv230808819R} and \cite[Proposition 4.1]{OS23} and translates them into our setting:

\begin{proposition}\label{pro-pigeonhole}
	Fix $\delta \in 2^{-\N}, s\in [0,1], C>0, M\in\N$ and let $\Delta \in (\delta,1)$. Let $(\mathcal{F}, \mathcal{P})$ be a $(\delta, s, C, M)$-nice configuration, where $\mathcal F\subset \mathcal X$ or $\mathcal F \subset \mathcal D_\delta(\mathcal X)$. There exists a refinement $(\mathcal{F}_1, \mathcal{P}_1)$ of $(\mathcal{F}, \mathcal{P})$ such that $\mathcal{F}_1$ is $\lbrace 1,\Delta, \delta\rbrace$-uniform, and 
	\begin{enumerate}
		\item[\textup{(i)}] \label{pigeonhole(1)}$(\mathcal{F}_1^\Delta, \mathcal{P}_1^\Delta)$ is a $(\Delta, s, C_\Delta, M_\Delta)$-nice configuration, where
		\begin{itemize}
			\item $\mathcal{F}_1^\Delta = \mathcal{D}_\Delta(\mathcal{F}_1)$, 
			\item for any $\mathbf{F} \in \mathcal{F}_1^\Delta$, $\mathcal{P}_1^\Delta(\mathbf{F}) := \mathcal{D}_\Delta( \bigcup_{{\bf f} \in \mathcal{F}_1 \cap \mathbf{F}} \mathcal{P}_1({\bf f}))$, and $\mathcal{P}_1^\Delta:=\bigcup_{\mathbf{F}\in\mathcal{F}_1^\Delta}\mathcal{P}_1^\Delta(\mathbf{F})$,
			\item $C_\Delta \lessapprox_\delta C$ and $M_\Delta\in\N$;
		\end{itemize}
            \item[\textup{(ii)}] \label{pigeonhole(2)} there exists a constant $H\approx_\delta M\cdot |\mathcal F|/|\mathcal P_1^\Delta|$ such that 
            \begin{equation*}
                |\lbrace ({\bf f}, p) \in \mathcal F_1 \times \mathcal D_\delta:\ p \in \mathcal P_1({\bf f})\ \text{and}\ p\subset Q\rbrace| \gtrsim H, \qquad Q \in \mathcal P_1^\Delta,
            \end{equation*}
		\item[\textup{(iii)}] \label{pigeonhole(3)} any refinement $(\mathcal{F}_\Delta, \mathcal{P}_\Delta)$ of $(\mathcal{F}_1^\Delta, \mathcal{P}_1^\Delta)$ covers $(\mathcal{F},\mathcal{P})$, that is, 
		\[\sum_{\mathbf{F}\in\mathcal{F}_\Delta}\sum_{{\bf f}\in \mathbf{F}\cap\mathcal{F}}\sum_{Q\in \mathcal{P}_\Delta(\mathbf{F})}|Q\cap \mathcal{P}({\bf f})|\gtrapprox_\delta M\cdot|\mathcal{F}|.\]
	\end{enumerate}
\end{proposition}

\begin{proof}
    We assume $\mathcal F \subset \mathcal D_\delta(\mathcal X)$: the proof works ad verbatim in the case $\mathcal F \subset \mathcal X$. By \cite[Lemma 2.15]{2023arXiv230110199O}, for every ${\bf f}\in \mathcal F$ there exists a $\lbrace 1,\Delta,\delta\rbrace$-uniform subset ${\mathcal{P}}_1({\bf f})\subset {\mathcal{P}}({\bf f})$ with $|{\mathcal{P}}_1({\bf f})|\approx_\delta M$. Let $m({\bf f}) = |\mathcal D_\Delta({\mathcal{P}}_1({\bf f}))|$. For each $Q \in \mathcal D_\Delta(\mathcal P_1({\bf f}))$ with $Q\cap {\mathcal{P}}_1({\bf f})\neq \emptyset$, it follows from $\lbrace 1,\Delta,\delta\rbrace$-uniformity of $\mathcal P_1({\bf f})$ that 
    \begin{equation}\label{eq-P1uniform}
|Q\cap {\mathcal{P}}_1({\bf f})| \approx_\delta \frac{M}{m({\bf f})}.
    \end{equation}
    By pigeonholing and Lemma \ref{lem-uniformsubset}, there exists $m\in2^{\N}$ and a $\lbrace 1,\Delta,\delta\rbrace$-uniform subset $\mathcal F_1\subseteq \mathcal F$ with $|\mathcal F_1|\approx_\delta |\mathcal F|$ and for every ${\bf f}\in \mathcal F_1$, $m \sim m({\bf f})$. 

    For each ${\bf F}\in \mathcal D_\Delta(\mathcal F_1)$, there exists $X({\bf F})\in 2^{\N}$ and a subset ${\mathcal{P}}_1({\bf F})\subseteq \mathcal D_\Delta$ such that for each $Q\in {\mathcal{P}}_1({\bf F})$, 
    \begin{equation}\label{def-X(F)}
|\lbrace {\bf f}\in\mathcal F_1\cap {\bf F}:\ Q\in \mathcal D_\Delta({\mathcal{P}}_1({\bf f}))\rbrace|\sim X({\bf F})
    \end{equation}
    and 
    \begin{equation}\label{def-X(F)2}
        X({\bf F})\cdot |{\mathcal{P}}_1({\bf F})| \approx_\delta |\mathcal F_1\cap {\bf F}|\cdot m.
    \end{equation}
    To see this, write $X^Q({\bf F}) := |\lbrace {\bf f}\in\mathcal F_1\cap {\bf F}:\ Q\in \mathcal D_\Delta({\mathcal{P}}_1({\bf f}))\rbrace|$ so that 
    \[
\sum_Q X^Q({\bf F}) = |\lbrace ({\bf f}, Q):\ {\bf f} \in \mathcal F_1\cap {\bf F},\ Q\in \mathcal D_\Delta (\mathcal P_1({\bf f}))\rbrace| = |\mathcal F_1 \cap {\bf F}| \cdot m.
    \]
Since $X^Q({\bf F})$ can take $\lessapprox_\delta 1$ many dyadic values, there exists $X({\bf F})\in 2^{\N}$ and $\mathcal P_1({\bf F}) \subset \mathcal D_\Delta$ such that $X^Q({\bf F})\sim X({\bf F})$ for every $Q\in \mathcal P_{1}({\bf F})$ and \eqref{def-X(F)}, \eqref{def-X(F)2} hold. By reducing the cardinality of $\mathcal F_1$ by at most a factor of $\approx_\Delta 1$, we may suppose that there exists $X\in 2^{\N}$ such that $X({\bf F}) = X$ for every ${\bf F}\in \mathcal D_\Delta(\mathcal F_1)$. Let $M_\Delta := \frac{|\mathcal F_1\cap {\bf F}| \cdot m}{X}$ and note that $M_\Delta$ does not depend on ${\bf F}$ since $\mathcal F_1$ is $\lbrace 1,\Delta,\delta\rbrace$-uniform. By \eqref{def-X(F)2}, $|\mathcal P_1({\bf F})|\approx_\delta M_\Delta$ for every ${\bf F}\in\mathcal D_\Delta(\mathcal F_1)$. 
\begin{claim}\label{claim-deltasC-set}
For each ${\bf F}\in \mathcal D_\Delta (\mathcal F_1)$, the set ${\mathcal{P}}_1({\bf F})$ is a $(\Delta, s, C_\Delta)$-set for $C_\Delta \approx_\delta C$.
\end{claim}
\begin{proof}[Proof of Claim]
For any $\rho \in (\Delta, 1)$ and any $x\in\R^2$, $r\geq \Delta$, we have 
\begin{align*}
|{\mathcal{P}}_1({\bf F}) \cap B(x,r)| \cdot X \cdot \frac{M}{m} &\leq |\lbrace ({\bf f}, Q):\ {\bf f} \in \mathcal F_1\cap {\bf F},\ Q\in \mathcal D_\Delta (\mathcal P_1({\bf f}))\cap B(x,r)\rbrace|\cdot \frac{M}{m} \\
&\leq|\lbrace ({\bf f}, p):\ {\bf f}\in \mathcal F_1\cap {\bf F},\ p \in \mathcal P_1({\bf f})\cap B(x,2r)\rbrace| \\
&=|\mathcal F_1 \cap {\bf F}| \cdot |{\mathcal{P}}_1({\bf f})\cap B(x,2r)|
\end{align*}
upon rearranging which we obtain
\begin{align*}
|{\mathcal{P}}_1({\bf F}) \cap B(x,r)| &\leq \frac{|\mathcal F_1 \cap {\bf F}| \cdot |{\mathcal{P}}_1({\bf f})\cap B(x,r)|\cdot m}{M \cdot X}\\
&\leq Cr^s \frac{|\mathcal F_1 \cap {\bf F}|\cdot m}{X} \approx_\delta Cr^s \cdot |{\mathcal{P}}_1({\bf F})|.
\end{align*}
On the second line we used the $(\delta,s,C)$-set property of ${\mathcal{P}}_1({\bf f})$. 
\end{proof}

Now, for every $f\in{\bf F}$, ${\bf F} \in \mathcal D_\Delta(\mathcal F_1)$, we replace ${\mathcal{P}}_1({\bf f})$ by $\bigcup_{Q\in \mathcal P_1({\bf F})}(\mathcal P_1({\bf f})\cap Q)$. This results only in negligible loss of cardinality: 
\begin{align*}
\sum_{{\bf f}\in\mathcal F_1} |{\mathcal{P}}_1({\bf f})| &= \sum_{{\bf F}\in\mathcal D_\Delta(\mathcal F_1)} \sum_{{\bf f}\in\mathcal F_1\cap {\bf F}} \sum_{Q \in {\mathcal{P}}_1({\bf F})} |{\mathcal{P}}_1({\bf f})\cap Q|\\
&=\sum_{{\bf F}\in\mathcal D_\Delta(\mathcal F_1)} \sum_{Q \in {\mathcal{P}}_1({\bf F})} \sum_{f\in\mathcal F_1\cap {\bf F}} |{\mathcal{P}}_1({\bf f})\cap Q|\\
&\approx_\delta \sum_{{\bf F}\in \mathcal D_\Delta(\mathcal F_1)} |{\mathcal{P}}_1({\bf F})|\cdot \frac{M}{m} X
\approx_\delta M\cdot |\mathcal F_1| \approx_\delta M\cdot |\mathcal F|.
\end{align*}
In the third inequality we used $\sum_{f\in\mathcal F_1\cap {\bf F}} |{\mathcal{P}}_1(f)\cap Q|_\delta \approx_\delta X\cdot \frac{M}{m}$ which follows from \eqref{eq-P1uniform} and \eqref{def-X(F)}, and the fourth inequality follows from \eqref{def-X(F)2} and uniformity of $\mathcal F_1$. By definition of ${\mathcal{P}}_1({\bf f})$, we have $\mathcal D_\Delta(\bigcup_{{\bf f}\in \mathcal F_1\cap {\bf F}} {\mathcal{P}}_1({\bf f})) = {\mathcal{P}}_1({\bf F})$ for any ${\bf F}\in \mathcal D_\Delta(\mathcal F_1) =: \mathcal F_1^\Delta$ which completes the proof of (i). 

Claim (ii) follows easily from \eqref{eq-P1uniform} and the inequality $|\mathcal F_1| \approx_\delta |\mathcal F|$: For any $Q \in \mathcal P_1^\Delta := \bigcup_{{\bf F} \in \mathcal F_1^\Delta} \mathcal P_1({\bf F})$, 
\begin{align*}
|\lbrace (f, p) \in \mathcal F_1 \times \mathcal D_\delta:\ p \in \mathcal P(f)\ \text{and}\ p\subset Q\rbrace| = \sum_{{\bf f}\in\mathcal F_1} |\mathcal P_1({\bf f})\cap Q| = |\mathcal F_1| \cdot \frac{M}{m} \gtrapprox_\delta \frac{|\mathcal F| \cdot M}{|P_1^\Delta|}.
\end{align*}
To prove claim (iii), let $(\mathcal F_\Delta, \mathcal P_\Delta)$ be a refinement of $(\mathcal F_1^\Delta, \mathcal P_1^\Delta)$. Since $\mathcal F_\Delta\subset \mathcal F_1^\Delta$ and $\mathcal P_\Delta \subset \mathcal P_1^\Delta$, it follows from \eqref{eq-P1uniform} that $|Q\cap \mathcal P({\bf f})|\gtrapprox \frac{M}{m}$ for every $Q\in P_\Delta$ and ${\bf f}\in {\bf F}\cap \mathcal F_1$ with ${\bf F}\in \mathcal F_\Delta$. Therefore 
\begin{align*}
\sum_{{\bf F}\in \mathcal F_\Delta} \sum_{{\bf f}\in {\bf F}\cap \mathcal F} \sum_{Q\in \mathcal P_\Delta({\bf F})} |Q\cap \mathcal P({\bf f})| &\geq \sum_{{\bf F}\in \mathcal F_\Delta} |\mathcal P_\Delta({\bf F})| \cdot X\cdot \frac{M}{m} \gtrapprox_\delta M_\Delta \cdot |F_1^\Delta| \cdot X \cdot \frac{M}{m}\\
&= \frac{|\mathcal F_1\cap {\bf F}| \cdot |F_1^\Delta| \cdot m \cdot X \cdot M}{X \cdot m}
= |\mathcal F_1| \cdot M \gtrapprox_\delta |\mathcal F|\cdot M
\end{align*}
where in the second inequality we used $\sum_{{\bf F}\in \mathcal F_\Delta} |\mathcal P_\Delta({\bf F})| \gtrapprox_\delta M_\Delta \cdot |\mathcal F_1^\Delta|$ which follows from the definition of a refinement configuration. This completes the proof of the proposition.
\end{proof}

The following result is \cite[Proposition 5.2]{OS23} translated to our setting. It provides the basis for the induction on scales in the proof of Theorem \ref{thm-furstenbergset}. Let $\pi_{\mathtt x}:\R^2\to\R$ denote the orthogonal projection to the $x$-axis.
\begin{proposition}\label{prop-52OS}
Let $\delta, \Delta\in 2^{-\N}$ with $\delta\leq \Delta$. Let $(\mathcal F_0, {\mathcal{P}}_0)$ be a $(\delta, s, C_1, M)$-nice configuration, where $\mathcal F_0\subset\mathcal X$ or $\mathcal F_0\subset \mathcal D_\delta(\mathcal X)$. Then there exist sets $\mathcal F\subset \mathcal F_0$ and ${\mathcal{P}}({\bf f})\subset {\mathcal{P}}_0({\bf f})$ for each ${\bf f}\in\mathcal F$ such that, denoting ${\mathcal{P}} := \bigcup_{{\bf f}\in\mathcal F} {\mathcal{P}}({\bf f})$, the following hold: 
\begin{enumerate}
\item[\textup{(i)}] \label{prop52-1} $|\mathcal D_\Delta(\mathcal F)|\approx_\delta |\mathcal D_\Delta(\mathcal F_0)|$ and $|\mathcal F\cap {\bf F}|\approx_\delta |\mathcal F_0\cap {\bf F}|$ for all ${\bf F}\in\mathcal D_\Delta (\mathcal F)$. 
\medskip  
\item[\textup{(ii)}] \label{prop52-2} $|{\mathcal{P}}({\bf f})|\gtrapprox_\delta |{\mathcal{P}}_0({\bf f})| = M$ for every ${\bf f}\in\mathcal F$.
\medskip
\item[\textup{(iii)}] \label{prop52-3} There exist a set ${\mathcal{P}}_\Delta\subset \mathcal D_\Delta({\mathcal{P}})$ and numbers $C_\Delta \approx_\delta C_1$ and $M_\Delta \geq 1$ such that $(\mathcal D_\Delta(\mathcal F), {\mathcal{P}}_\Delta)$ is a $(\Delta, s, C_\Delta, M_\Delta)$-nice configuration. Moreover, the associated families ${\mathcal{P}}_\Delta({\bf F}):=\{Q\in {\mathcal{P}}_\Delta: \exists~f\in\mathbf{F} \text{ such that }Q\cap\Gamma_f\neq\emptyset\} \subset \mathcal P_\Delta$ satisfy 
\begin{equation*}
\bigcup {\mathcal{P}}({\bf f}) \subset \bigcup {\mathcal{P}}_\Delta({\bf F}), \qquad {\bf f}\in \mathcal F\cap {\bf F},\ {\bf F}\in \mathcal D_\Delta(\mathcal F).
\end{equation*}
\item[\textup{(iv)}] \phantomsection \label{prop52-4} For each ${\bf F}\in \mathcal D_\Delta(\mathcal F)$ there exist $C_{\bf F} \approx_\delta C_1$, $M_{\bf F}\geq 1$, and a set ${\mathcal{P}}_{\bf F}\subset \mathcal D_{\delta/\Delta}(\R^2)$ such that $(T_{\bf F}(\mathcal F\cap {\bf F}), {\mathcal{P}}_{\bf F})$ is a $(\delta/\Delta, s, C_{\bf F}, M_{\bf F})$-nice configuration, where $T_{\bf F}$ denotes the rescaling map of Definition \ref{def:drescalingmap}. Moreover, 
\begin{equation}\label{eq-projection}
\mathcal D_{\delta/\Delta}(\pi_{\mathtt x}[{\mathcal{P}}_{\bf F}( T_{\bf F}({\bf f}))]) = \mathcal D_{\delta/\Delta} (\pi_{\mathtt x}({\mathcal{P}}({\bf f}))), \qquad {\bf f}\in \mathcal F \cap {\bf F},
\end{equation}
and
\begin{equation}\label{eq-growth}
\frac{|{\mathcal{P}}_0|}{M}\gtrapprox_\delta \frac{|{\mathcal{P}}_\Delta|}{M_\Delta} \cdot \left( \max_{{\bf F}\in \mathcal D_\Delta(\mathcal F)} \frac{|{\mathcal{P}}_{\bf F}|}{M_{\bf F}}\right).
\end{equation}
\end{enumerate}
\end{proposition}

\begin{proof}
We assume $\mathcal F_0 \subset \mathcal D_\delta(\mathcal X)$: the proof works ad verbatim in the case $\mathcal F_0 \subset \mathcal X$. The verification of points (i) - (iii) is exactly analogous to \cite[Proof of Proposition 5.2]{OS23}, relying only on dyadic pigeonholing. The only geometric difference brought by $\mathcal F$ being a collection of $C^2$-functions instead of linear maps arises in point (iv): For example, in \eqref{eq-projection} we consider the orthogonal projection of ${\mathcal{P}}_{\bf F}(T_{\bf F}({\bf f}))$ instead of the slope set of its dual, and the rescaling map $T_{\bf F}$ acts on $\mathcal F$ instead of $\R^2$. However, we include the full proof for the convenience of the reader. 

Fix ${\bf F} \in \mathcal D_{\Delta}(\mathcal F_0)$ and apply Proposition \ref{pro-pigeonhole} to the nice configuration $(\mathcal F_0 \cap {\bf F}, \mathcal P_0)$. This yields a refinement $(\overline{ \mathcal F}_{\bf F}, \overline{ \mathcal P}_{\bf F})$ such that $(\mathcal D_\Delta(\overline{\mathcal F}_{\bf F}), (\overline{ \mathcal P}_{\bf F})^\Delta)$ is a $(\Delta, s, C_\Delta, M_\Delta)$-nice configuration, and 
\begin{equation}\label{eq-HF}
    |\lbrace ({\bf f},p) \in \overline{\mathcal F}_{\bf F} \times \mathcal D_\delta:\ p \in \mathcal P_0({\bf f})\ \text{and}\ p \subset Q\rbrace| \gtrsim H_{\bf F},\ \qquad Q\in (\overline{ \mathcal P}_{\bf F})^\Delta,
\end{equation}
for $H_{\bf F}\approx_\delta M\cdot |\overline{\mathcal F}_{\bf F}|/|(\overline{ \mathcal P}_{\bf F})^\Delta|$. 

By pigeonholing, we find $\overline{M}_\Delta \geq 1$ and a set $\overline{\mathfrak F} \subset \mathcal D_\Delta(\mathcal F_0)$ with $|\overline{\mathfrak F}|\approx_\Delta |\mathcal D_\Delta(\mathcal F_0)|$ such that $\overline{M}_\Delta \leq |(\overline{ \mathcal P_{\bf F}})^\Delta| \leq 2\overline{M}_\Delta$ for all ${\bf F} \in \overline{\mathfrak F}$. 

Let 
\begin{equation*}
    \overline{\mathcal P}_\Delta := \bigcup_{{\bf F}\in \overline{\mathfrak F}} (\overline{ \mathcal P_{\bf F}})^\Delta.
\end{equation*}
We will next pigeonhole further to ensure that $|\mathcal P_0 \cap Q|$ is roughly constant over all $Q$ in a substantial subset of $\overline{\mathcal P}_\Delta$. To this end, define 
\begin{equation*}
    \mathcal I(\overline{\mathfrak F}, \overline{\mathcal P}_\Delta) := |\lbrace ({\bf F}, Q) \in \overline{\mathfrak F} \times \overline{\mathcal P}_\Delta:\ Q \in (\overline{\mathcal P}_{\bf F})^\Delta\rbrace |
\end{equation*}
and note that $\mathcal I(\overline{\mathfrak F}, \overline{\mathcal P}_\Delta) \sim |\overline{\mathfrak F}|\cdot \overline{M}_\Delta$. For $j\geq 1$, let 
\begin{equation*}
    \overline{\mathcal P}_{\Delta, j} := \lbrace Q \in \overline{\mathcal P}_\Delta:\ 2^{j-1} < |\mathcal P_0\cap Q| \leq 2^j\rbrace.
\end{equation*}
Since $|\mathcal D_\delta \cap Q| \leq \delta^{-2}$, we have 
\begin{equation*}
    |\overline{\mathfrak F}|\cdot \overline{M}_\Delta \sim \mathcal I(\overline{\mathfrak F}, \overline{P}_\Delta) = \sum_{2^j \leq \delta^{-2}} |\lbrace ({\bf F}, Q) \in \overline{\mathfrak F} \times \overline{\mathcal P}_{\Delta,j}:\ Q \in (\overline{\mathcal P}_{\bf F})^{\Delta}\rbrace |.
\end{equation*}
Therefore we may pick $j \leq 2\log(1/\delta)$ such that if 
\begin{equation*}
    \mathcal P_\Delta := \overline{\mathcal P}_{\Delta, j},\ \mathcal{P}_\Delta({\bf F}) := \mathcal D_\Delta \cap (\overline{P}_{\bf F})^\Delta,
\end{equation*}
then we have 
\begin{equation}
    \sum_{{\bf F}\in \overline{\mathfrak F}} |\mathcal P_\Delta({\bf F})| = |\lbrace ({\bf F}, Q) \in \overline{\mathfrak F}\times \mathcal P_\Delta:\ Q \in \mathcal P_\Delta({\bf F})\rbrace| \approx_\delta |\overline{\mathfrak F}|\cdot \overline{M}_\Delta.
\end{equation}
Let $N_\Delta := 2^j$ for this index ``$j$'', so 
\begin{equation}\label{eq-NDelta-def}
    |\mathcal P_0 \cap Q| \sim N_\Delta,\qquad Q\in \mathcal P_\Delta({\bf F}) \subset \mathcal D_\Delta.
\end{equation}
Since $\overline{M}_\Delta \sim |(\overline{P}_{\bf F})^\Delta| \geq |\mathcal P_\Delta({\bf F})|$ for ${\bf F}\in \overline{\mathfrak F}$, we find a subset $\mathfrak F\subset \overline{\mathfrak F}$ of cardinality $|\mathfrak F|\approx_\delta|\overline{\mathfrak F}|$ such that 
\begin{equation}
    |\mathcal P_\Delta({\bf F})|\approx_\delta \overline{M}_\Delta \sim |(\overline{P}_{\bf F})^\Delta|, \qquad {\bf F} \in \mathfrak F.
\end{equation}

We reduce the families $\mathcal P_\Delta({\bf F})$ so that they have common cardinality $M_\Delta \approx_\delta \overline{M}_\Delta$. Since $(\overline{P}_{\bf F})^\Delta$ is a $(\Delta, s, C_\Delta)$-set, also $\mathcal P_\Delta({\bf F})$ remains a $(\Delta, s, C_\Delta)$-set with constant $C_\Delta\approx_\delta C_1$. We now finalise the definition of $\mathcal P_\Delta$ by setting 
\begin{equation}\label{eq-defPdelta}
    \mathcal P_\Delta := \bigcup_{{\bf F}\in \mathfrak{F}} \mathcal P_\Delta({\bf F}).
\end{equation}

We next begin defining the families $\mathcal P({\bf f}) \subset \mathcal P_0({\bf f})$. For every ${\bf F}\in \mathfrak{F}$, let 
\begin{equation}
    \overline{\mathcal P}({\bf f}) := \bigcup_{Q\in \mathcal P_\Delta({\bf F})} (\mathcal P_0({\bf f}) \cap Q), \qquad {\bf f}\in\overline{\mathcal F}_{\bf F}.
\end{equation}
Since the final families $\mathcal P({\bf f})$ will be subsets of $\overline{\mathcal P}({\bf f})$ and $\mathcal D_\Delta(\mathcal F) = \mathfrak F$, this completes the proof of claim (iii). 

The current issue is that we do not have $|\overline{\mathcal P}({\bf f})|\approx_\delta M$ for all ${\bf f}\in \overline{\mathcal F}_{\bf F}$. Recall that $|\mathcal P_\Delta({\bf F})|\approx_\delta |(\overline{\mathcal P}_{\bf F})^\Delta|$ for ${\bf F}\in\mathfrak F$. Recalling \eqref{eq-HF} where $H_{\bf F} \approx_\Delta M \cdot|\overline{\mathcal F}_{\bf F}|/|\mathcal P_\Delta({\bf F})|$, we have 
\begin{align*}
    \sum_{{\bf f} \in \overline{\mathcal F}_{\bf F}}|\overline{P}({\bf f})| &= \sum_{{\bf f} \in \overline{\mathcal F}_{\bf F}} \sum_{Q \in \mathcal P_\Delta({\bf F})} |\mathcal P_0({\bf f})\cap Q| \\
    &= \sum_{Q \in \mathcal P_\Delta({\bf F})} |\lbrace ({\bf f}, p) \in \overline{\mathcal F}_{\bf F}\times \mathcal D_\delta:\ p \in \mathcal P_0({\bf f})\ \text{and}\ p\subset Q\rbrace| \\
    &\gtrapprox_\delta |\mathcal P_\Delta({\bf F})|\cdot H_{\bf F} \approx_\delta M\cdot |\overline{\mathcal F}_{\bf F}|.
\end{align*}
Since $|\overline{\mathcal P}({\bf f})|\leq |\mathcal P_0({\bf f})| = M$ for all ${\bf f}\in \overline{\mathcal F}_{\bf F}\subset \mathcal F_0$, we find a subset $\mathcal F_{\bf F}\subset \overline{\mathcal F}_{\bf F}$ of cardinality $|\mathcal F_{\bf F}|\approx_\delta |\overline{\mathcal F}_{\bf F}|$ such that 
\begin{equation}\label{eq-P(f)card}
    |\overline{\mathcal P}({\bf {\bf f}})|\approx_\delta M,\qquad {\bf f} \in \mathcal F_{\bf F}.
\end{equation}
We define 
\begin{equation}
    \mathcal F:= \bigcup_{{\bf F}\in \mathfrak F} \mathcal F_{\bf F}\quad \text{and}\quad \overline{\mathcal P} := \bigcup_{{\bf f}\in \mathcal F} \overline{P}({\bf f}),
\end{equation}
so that $\mathcal F$ satisfies claim (i), and claim (ii) is satisfied for ${\bf f}\in\mathcal F$ and the families $\overline{P}({\bf f})$. In the sequel, we will further refine $\overline{\mathcal P}({\bf f})$ into $\mathcal P({\bf f})$, but claim (ii) remains valid as long as $|\mathcal P({\bf f})|\approx_\delta |\overline{\mathcal P}({\bf f})|$. The family $\mathcal D_\Delta(\mathcal F) = \mathfrak F$ will remain intact throughout the sequel, but the sets $\mathcal F_{\bf F}$ will be further refined once more while maintaining $|\mathcal F_{\bf F}|\approx_\delta |\mathcal F_0\cap {\bf F}|$. Clearly this will not influence the validity of claim (i).

Denote $\overline \delta = \delta/\Delta$ and let ${\bf f}\in \mathcal F$. For $T \in \pi_{\mathbf{x}}^{-1}(\mathcal D_{\overline \delta}(\R))$ with $T\cap \overline{\mathcal P}({\bf f})\neq \emptyset$, the collection 
\begin{equation*}
\Psi({\bf f},T) := \overline{\mathcal{P}}({\bf f})\cap T = \lbrace p \in \mathcal D_\delta:\ p \in \overline{\mathcal P}({\bf f})\ \text{and}\ p\subset T\rbrace
\end{equation*}
of dyadic squares is called a \emph{square packet}. Note that $\pi_{\mathtt x} (\Psi({\bf f},T))$ is contained in $\pi_{\mathtt x}(T)$, a dyadic interval of length $\overline \delta$. By partitioning $\overline{\mathcal P}({\bf f})$ into square packets and pigeonholing, we find $m({\bf f}), M({\bf f})\in 2^{\N}$ and square packets ${\Psi}({\bf f}, T_1),\ldots, {\Psi}({\bf f}, T_{M({\bf f})})\subset \overline{\mathcal{P}}({\bf f})$ such that 
\begin{equation}\label{eq-m(f)M(f)}
|\Psi({\bf f}, T_j)|\sim m({\bf f}) \quad \text{and} \quad m({\bf f})\cdot M({\bf f}) \approx_\delta |\overline{\mathcal P}({\bf f})| \approx_\delta M.
\end{equation}
We now define the set ${\mathcal{P}}({\bf f})$ as
$$
{\mathcal{P}}({\bf f}) := \bigcup_{j=1}^{M({\bf f})} \Psi({\bf f}, T_j)
$$
and note that $|\mathcal P({\bf f})| \approx_\delta |\overline{\mathcal P}({\bf f})|$. 

We will next define the families ${\mathcal{P}}_{\bf F} \subset \mathcal D_{\overline \delta}({\mathcal{P}})$ of claim (iv). Fix ${\bf F}\in \mathcal D_\Delta(\mathcal F)$ for the rest of the proof. By replacing $\mathcal F \cap {\bf F}$ with a subset of cardinality $\gtrapprox_\delta |\mathcal F \cap \bf F|$, we find $M_{\bf F}\in2^\N$ such that $M({\bf f}) \sim M_{\bf F}$ for every ${\bf f}\in \mathcal F \cap {\bf F}$. It follows that 
\begin{equation}\label{eq-squarepacketsize}
m({\bf f}) \sim |\Psi({\bf f}, T_j)| \approx_\delta M/M_{\bf F}, \qquad {\bf f}\in\mathcal F\cap {\bf F}, \, 1 \leq j \leq M({\bf f}). 
\end{equation}
For each $j$, let $p_j \in \mathcal D_\delta$ be a dyadic square intersecting both $\Gamma_{\bf f}$ and the left side of $T_j$, and set 
\begin{equation}\label{eq-P'(f)}
{\mathcal{P}}'({\bf f}) := \lbrace p_1,\ldots, p_{M({\bf f})}\rbrace.
\end{equation}
Then $\pi_{\mathtt x} ({\mathcal{P}}'({\bf f}))$ is a $\overline \delta/2$-separated set with $|\pi_x (\mathcal P'({\bf f}))| = |\mathcal P'({\bf f})| = M_{\bf F}$. We claim $\pi_{\mathtt x} ({\mathcal{P}}'({\bf f}))$ is a $(\overline \delta, s, C_{\bf F})$-set with $C_{\bf F}\approx_\delta C_1$. Since $\mathcal P({\bf f})$ is a $(\delta, s, C')$-set with $C'\approx_\delta C_1$ by \eqref{eq-P(f)card} and the inequality $|\mathcal P({\bf f})| \approx_\delta |\overline{\mathcal P}({\bf f})|$, for any $x\in\R^2$ and $r \geq \overline{\delta}$ we have
\begin{align*}
|{\mathcal{P}}'({\bf f}) \cap B(x,r)|_{\overline \delta} & \leq |\lbrace 1 \leq j \leq M({\bf f}):\ p_j\subset B(x,2r)  \rbrace| 
\stackrel{\eqref{eq-m(f)M(f)}}{\lesssim} \frac{1}{m({\bf f})}|\mathcal P({\bf f}) \cap B(x,3r)|\\
&\lessapprox_\delta \frac{M}{m({\bf f})} C_1 r^s  \stackrel{\eqref{eq-squarepacketsize}}{\approx_\delta} C_1 M_{\bf F} r^s = C_1 |\mathcal P'({\bf f})|_{\overline \delta}\cdot r^s.
\end{align*}
This verifies that $\mathcal P'({\bf f})$ is a $(\overline{\delta}, s, C_{\bf F})$- set with $C_{\bf F}\approx_\delta C_1$. By Lemma \ref{lemma-almostinjection}, so is $\pi_{\mathbf{x}} (\mathcal P'({\bf f}))$.

We let the rescaling map $T_{\bf F}$ (initially defined on $C^{2}(I)$) act on $\R^2$ by 
$$
T_{\bf F}(x,y) = (x, \Delta^{-1}(y-f_{\bf F}(x))), \qquad {\bf F} \in \mathcal{D}_{r}(\mathcal{F}).
$$
Thus $T_{\bf F}(x,g(x)) = (x,T_{\bf F}(g)(x))$ for all $g\in {\bf F}$. Since $T_{\bf F}$ preserves the $x$-coordinate, the family $T_{\bf F}({\mathcal{P}}'({\bf f}))$ is a collection of $\overline{\delta}/2$-separated sets, each of which is contained in a $\delta \times 2\overline \delta$-rectangle by the assumption ${\bf F}\subset  B_{C^{2}}(1)$; we let 
\[
{\mathcal{P}}_{\bf F}(T_{\bf F}({\bf f})) := \mathcal D_{\overline \delta} (T_{\bf F}({\mathcal{P}}'({\bf f}))), \qquad \mathbf{f} \in \mathcal{F} \cap \mathbf{F},\]
so that $|{\mathcal{P}}_{\bf F}(T_{\bf F}({\bf f}))| \sim |{\mathcal{P}}'({\bf f})| \sim M_{\bf F}$. Using again that $T_{\bf F}$ preserves the $x$-coordinate,
\begin{align*}
    \mathcal D_{\overline{\delta}}(\pi_{\mathbf{x}}[\mathcal P_{\bf F}(T_{\bf F}({\bf f}))]) = \mathcal D_{\overline{\delta}}(\pi_{\mathbf{x}}(\mathcal P'({\bf f}))) =\lbrace \pi_{\mathbf{x}}(T_j):\ 1\leq j \leq M({\bf f})\rbrace =\mathcal D_{\overline{\delta}}(\pi_{\mathbf{x}}(\mathcal P({\bf f}))).
\end{align*}
This verifies \eqref{eq-projection}. Moreover, $\mathcal P_{\bf F}(T_{\bf F}({\bf f}))$ is a $(\overline \delta, s, C_{\bf F})$-set: for any $x\in \R^2$ and $r\geq  \overline{\delta}$, 
\begin{align*}
    |\mathcal P_{\bf F}(T_{\bf F}({\bf f})) \cap B(x,r)|_{\overline \delta} &\sim |T_{\bf F}(\mathcal P'({\bf f})) \cap B(x,r)|_{\overline \delta} 
    \lesssim |\pi_{\mathbf{x}} (\mathcal P'({\bf f})) \cap \pi_x(B(x,r))|_{\overline{\delta}} \\
    &\lesssim C_{\bf F}r^s |\pi_{\mathbf{x}}(\mathcal P'({\bf f}))|_{\overline{\delta}} \leq C_{\bf F} r^s |T_{\bf F}(\mathcal P'({\bf f}))|_{\overline{\delta}} \sim C_{\bf F}r^s|\mathcal P_{\bf F}(T_{\bf F}({\bf f}))|,
\end{align*}
where we used the equality $|T_{\bf F}(\mathcal P'({\bf f}))| = |\pi_{\mathbf{x}}(T_{\bf F}(\mathcal P'({\bf f})))| = |\pi_{\mathbf{x}}(\mathcal P'({\bf f}))|$ and the $(\overline{\delta}, s, C_{\bf F})$-set property of $\pi_{\mathbf{x}}(\mathcal P'(\mathbf{f}))$. We now define the set $\mathcal P_{\bf F}$ by 
$$
{\mathcal{P}}_{\bf F} := \bigcup_{{\bf f}\in\mathcal F \cap {\bf F}} {\mathcal{P}}_{\bf F}(T_{\bf F}({\bf f})) \subset \mathcal D_{\overline \delta} (\R^2),
$$
and claim that $(T_{\bf F}(\mathcal F\cap {\bf F}), \mathcal P_{\bf F})$ is a $(\overline{\delta}, s, C_{\bf F}, M_{\bf F})$-nice configuration. Indeed, the required properties for $\mathcal P_{\bf F}(T_{\bf F}({\bf f}))$ have been verified above, and the $\overline{\delta}$-admissibility of $T_{\bf F}(\mathcal F\cap {\bf F})$ follows from the $\delta$-admissibility of $\mathcal F\cap {\bf F}$ since $T_{\bf F}$ is a similarity which scales distances by $\Delta^{-1}$. 

Finally, we move on to proving \eqref{eq-growth}. Recall the definition of $\mathcal P_\Delta$ from \eqref{eq-defPdelta}. First note that by \eqref{eq-NDelta-def},
    $$
|{\mathcal{P}}_0| \geq |{\mathcal{P}}_\Delta| \cdot\min_{Q\in {\mathcal{P}}_\Delta} |{\mathcal{P}}_0 \cap Q| \sim |{\mathcal{P}}_\Delta|\cdot N_\Delta.
    $$
Denote ${\mathcal{P}}({\bf F}) := \bigcup_{{\bf f}\in\mathcal F\cap{\bf F}} \bigcup_{Q \in {\mathcal{P}}_\Delta({\bf F})} ({\mathcal{P}}_0({\bf f})\cap Q) \subset {\mathcal{P}}_0$ and note that by point (iii),
    $$
|{\mathcal{P}}({\bf F})| \leq \sum_{Q \in {\mathcal{P}}_\Delta({\bf F})}| {\mathcal{P}}_0\cap Q|\lesssim M_\Delta \cdot N_\Delta.
    $$
By these inequalities, \eqref{eq-growth} boils down to proving 
$$
|{\mathcal{P}}({\bf F})|\gtrapprox_\delta \frac{|{\mathcal{P}}_{\bf F}|}{M_{\bf F}} \cdot M, \quad \mathbf{F} \in \mathcal{D}_{\Delta}(\mathcal{F}).
$$

Denote ${\mathcal{P}}' := \bigcup_{ {\bf f}\in \mathcal F\cap {\bf F}} {\mathcal{P}}'({\bf f})$, where $\mathcal P'({\bf f})$ is as in \eqref{eq-P'(f)}. Then, for $p\in {\mathcal{P}}'$, there exists a square packet $\Psi({\bf f}_p,T_p)$ (with ${\bf f}_p \in \mathcal{F} \cap \mathbf{F}$ not necessarily unique) such that $p$ intersects both $\Gamma_{{\bf f}_p}$ and the left side of $T_p$. Now, if we could choose the square packets $\Psi(\mathbf{f}_p, T_p)$ in such a way that for each pair $p \neq q\in {\mathcal{P}}'$ we had $\Psi({\bf f}_p, T_p)\cap \Psi({\bf f}_q, T_q) = \emptyset$, then
\begin{equation}\label{eq-completion}
|{\mathcal{P}}({\bf F})| \geq \left|\bigcup_{{\bf f}\in \mathcal F\cap {\bf F}}{\mathcal P}({\bf f})\right| \geq \sum_{p\in {\mathcal{P}}'} |\Psi({\bf f}_p, T_p)| \stackrel{\eqref{eq-m(f)M(f)}}{\approx_\delta} |{\mathcal{P}}'| \cdot \frac{M}{M_{\bf F}} \gtrsim |{\mathcal{P}}_{\bf F}| \cdot \frac{M}{M_{\bf F}}
\end{equation}
as desired, where in the last inequality we used 
\[
|\mathcal P'| \geq |T_{\bf F}(\mathcal P')|_{\overline{\delta}} \sim |\mathcal D_{\overline{\delta}}(T_{\bf F}(\mathcal P'))| = |\mathcal P_{\bf F}|.
\]
The following lemma shows that \eqref{eq-completion} holds when $\mathcal P'$ is replaced with a large subset, thus completing the proof of Proposition \ref{prop-52OS}.
\begin{lemma}\label{lemma-largesubsetofP'}
There exists $\mathcal P'' \subset {\mathcal{P}}'$ with $|\mathcal{P}'|\sim |\mathcal{P}''|$ such that $\Psi({\bf f}_p, T_p) \cap \Psi({\bf f}_q, T_q)=\emptyset$ for any $p\neq q\in {\mathcal{P}}''$ and for any choice of the square packets.
\end{lemma}
\begin{proof}
Let $p, q \in {\mathcal{P}}'$ be such that $\Psi({\bf f}_p, T_p) \cap \Psi({\bf f}_q, T_q)\neq \emptyset$ for some ${\bf f}_p, {\bf f}_q \in \mathbf{F}$. We claim that ${\rm dist}(\pi_\mathtt{y} (p), \pi_\mathtt{y} (q)) \lesssim \delta$, where $\pi_{\mathtt y}$ denotes the orthogonal projection to the $y$-axis. 

To prove this claim, we first note that necessarily $T_p = T_q$ and $\pi_{\mathbf{x}}(p) = \pi_{\mathbf{x}}(q)$. Let $o \in \Psi({\bf f}_p, T_p) \cap \Psi({\bf f}_q, T_q)$, and let $x_p, x_q, x_{o,p}, x_{o,q}\in\R$ be any points and $f_p, f_{o,p} \in {\bf f}_p$, $f_q, f_{o,q} \in {\bf f}_q$ be any functions such that 
\begin{align*}
    (x_p, f_p(x_p)) \in p, \quad (x_q, f_q(x_q)) \in q, \quad (x_{o,p}, f_{o,p}(x_{o,p})), (x_{o,q}, f_{o,q}(x_{o,q})) \in o.
\end{align*} 
It follows from the triangle inequality that
\begin{alignat}{2}\label{eq-piybound}
{\rm dist}(\pi_\mathtt{y} (p), \pi_\mathtt{y}( q))&\leq &&|f_p(x_p) -  f_q(x_q)| \nonumber\\
&\leq\, && |f_p(x_p) - f_{o,p}(x_p)|+ |(f_{o,p}- f_{o,q})(x_p) - (f_{o,p}- f_{o,q})(x_{o,p}))|  \nonumber\\
& && + |f_{o,q}(x_p) - f_{o,q}(x_q)| + |f_{o,p}(x_{o,p}) - f_{o,q}(x_{o,q})|  \nonumber\\
& && + |f_{o,q}(x_{o,q}) - f_{o,q}(x_{o,p})| + |f_{o,q}(x_q) - f_q(x_q)|.
\end{alignat}
Since $\pi_{\mathbf{x}}(p) = \pi_{\mathbf{x}}(q)$, we have $|x_p - x_q|\leq \delta$. Furthermore, using that $\diam({\bf f}_p), \diam({\bf f}_q) \lesssim \delta$ and $\diam({\bf F})\lesssim \Delta$ and the mean value theorem, we have the following bounds for the summands in \eqref{eq-piybound}: 
\begin{align*}
    &|f_p(x_p) - f_{o,p}(x_p)| \leq \Vert f_p - f_{o,p}\Vert_{C^2([-2,2])}\lesssim \delta, \\
    &|(f_{o,p}- f_{o,q})(x_p) - (f_{o,p}- f_{o,q})(x_{o,p}))|\lesssim \Delta |x_p - x_{o,p}| \leq \Delta \overline\delta = \delta, \\
    &|f_{o,q}(x_p) - f_{o,q}(x_q)| \leq |x_p - x_q|\leq \delta, \\
    &|f_{o,p}(x_{o,p}) - f_{o,q}(x_{o,q})| \leq \diam(o) \lesssim \delta, \\
    &|f_{o,q}(x_{o,q}) - f_{o,q}(x_{o,p})| \leq |x_{o,p} - x_{o,q}| \leq \delta, \\
    &|f_{o,q}(x_q) - f_q(x_q)| \leq \Vert f_{o,q} - f_q\Vert_{C^2([-2,2])} \lesssim \delta.
\end{align*}
Inserting these into \eqref{eq-piybound}, we find that 
\[
{\rm dist}(\pi_\mathtt{y}(p), \pi_\mathtt{y}(q)) \lesssim \delta.
\]
The set ${\mathcal{P}}''$ is now obtained by choosing for each $I\in\pi_{\mathtt x} \mathcal (\mathcal{P}')$ a $\sim\delta$-separated subset of the column $\lbrace p'\in \mathcal P':\ \pi_{\mathtt x}(\mathcal{P}') = I\rbrace$. Since we discard at most a $\sim 1$-proportion of the squares of $\mathcal P'$, we have $|{\mathcal{P}}''|\gtrsim |{\mathcal{P}}'|$. Moreover, for every $p\neq q \in \mathcal P''$, $\Psi(f_p, T_p)\cap \Psi(f_q, T_q)=\emptyset$, since either $T_p \neq T_q$ or $\pi_{\mathtt{y}} (p)$ and $\pi_{\mathtt{y}} (q)$ are $\sim\delta$-separated.
\end{proof}
Applying \eqref{eq-completion} with the set $\mathcal P''$ of Lemma \ref{lemma-largesubsetofP'} in place of $\mathcal P'$ completes the proof of Proposition \ref{prop-52OS}.
\end{proof}

We also state for later use the following multi-scale version of Proposition \ref{prop-52OS}. Each $\approx$ is an abbreviation of $\approx_{\delta}$.

\begin{proposition}\label{pro-multiscaledocom0}
Fix $N\geq 2$ and sequence of scales $\{\Delta_j\}_{j=0}^N\subset 2^{-\N}$ with
\[0<\delta=\Delta_N<\Delta_{N-1}<\cdots<\Delta_1<\Delta_0=1.\]
Let $(\mathcal{F}_0,\mathcal{P}_0)$ be a a $(\delta,s,C,M)$-nice configuration, where $\mathcal F_0\subset \mathcal X$ or $\mathcal F_0 \subset \mathcal D_\delta(\mathcal X)$. Then there exists $\mathcal{F}\subset \mathcal{F}_0$ such that:
\begin{itemize}
\item [\textup{(D1)}]\phantomsection \label{D1} $|\mathcal{D}_{\Delta_j}(\mathcal{F})|\approx |\mathcal{D}_{\Delta_j}(\mathcal{F}_0)|$ and $|\mathcal{F}\cap \mathbf{F}|\approx|\mathcal{F}_0\cap\mathbf{F}|$ for any $\mathbf{F}\in \mathcal{D}_{\Delta_j}(\mathcal{F})$ with $1\leq j \leq N$.
		
\item [\textup{(D2)}]\phantomsection \label{D2} For every $\mathbf{F}\in \mathcal{D}_{\Delta_j}(\mathcal{F})$ with $1\leq j \leq N-1$, there exist numbers $C_\mathbf{F}\approx C$ and $M_\mathbf{F}\geq1$, and a family of dyadic cubes $\mathcal{P}_\mathbf{F}\subset \mathcal{D}_{\Delta_{j+1}/\Delta_j}$ such that $(T_\mathbf{F}(\mathcal{F}\cap \mathbf{F}), \mathcal{P}_\mathbf{F})$ is a $(\Delta_{j+1}/\Delta_j, s, C_\mathbf{F}, M_\mathbf{F})$ nice configuration.
\end{itemize}
Furthermore, the families $\mathcal{P}_\mathbf{F}$ can be chosen such that if $\mathbf{F}_j\in \mathcal{D}_{\Delta_j}(\mathcal{F})$ with $1\leq j \leq N-1$, then
\[\frac{|\mathcal{P}_0|}{M}\gtrapprox \prod_{j=0}^{N-1}\frac{|\mathcal{P}_{\mathbf{F}_j}|}{M_{\mathbf{F}_j}}.\]
\end{proposition}
\begin{proof}
We prove the statement by induction on $N$. The case $N=2$ is exactly Proposition \ref{prop-52OS}. Suppose then that the claim holds for $N-1$, and let us verify it for $N$. Apply Proposition \ref{prop-52OS} to $(\mathcal F_0, \mathcal P_0)$ with $\delta = \Delta_N$ and $\Delta = \Delta_{N-1}$; let $\mathcal F'\subset \mathcal F_0,\mathcal P\subset \mathcal P_0$ and $\mathcal P_{\Delta_{N-1}}\subset \mathcal D_{\Delta_{N-1}}(\mathcal P)$ be the resulting families. The set $\mathcal P_{\Delta_{N-1}}$ appears in Proposition \ref{prop-52OS} (iii). The property (D2) holds for $\mathcal F'$ and $j=N-1$, by Proposition \ref{prop-52OS} (iv). Apply then the inductive assumption to the configuration $(\mathcal D_{\Delta_{N-1}}(\mathcal F'), \mathcal P_{\Delta_{N-1}})$ which is $(\Delta_{N-1}, s, C_{\Delta_{N-1}}, M_{\Delta_{N-1}})$-nice by Proposition \ref{prop-52OS} (iii). This yields a set $\mathcal F''\subset \mathcal D_{\Delta_{N-1}}(\mathcal F')$: we define 
\begin{equation*}
\mathcal F = \bigcup_{{\bf F}\in \mathcal D_{\Delta_{N-1}}(\mathcal  F'')} \mathcal F' \cap {\bf F} \subset \mathcal F_0.
\end{equation*}
This ensures that $\mathcal F\cap {\bf F} = \mathcal F' \cap {\bf F}$ for ${\bf F} \in \mathcal D_{\Delta_{N-1}}(\mathcal F)$, and that $\mathcal D_{\Delta_j}(\mathcal F) = \mathcal D_{\Delta_j}(\mathcal F'')$ for $j\geq N-1$, and so the property (D1) holds for every $1\leq j \leq N$ by the induction hypothesis. To see why the final claim holds, note first that by the inductive assumption,
\begin{equation}\label{eq-indassumption}
    \frac{|\mathcal P_{\Delta_{N-1}}|}{M_{\Delta_{N-1}}} \gtrapprox_{\Delta_{N-1}} \prod_{j=0}^{N-2} \frac{|\mathcal P_{{\bf F}_j}|}{M_{{\bf F}_j}}
\end{equation} 
for any $\mathbf{F}_j\in \mathcal{D}_{\Delta_j}(\mathcal{F})$, $1\leq j \leq N-2$. On the other hand, by \eqref{eq-growth} of Proposition \ref{prop-52OS} (iv), 
\begin{equation*}
    \frac{|\mathcal P_0|}{M} \gtrapprox_\delta \frac{|\mathcal P_{\Delta_{N-1}}|}{M_{\Delta_{N-1}}} \cdot \left( \max_{{\bf F}\in \mathcal D_{\Delta_{N-1}}(\mathcal F')} \frac{|\mathcal P_{\bf F}|}{M_{\bf F}}\right).
\end{equation*}
Since $\mathcal D_{\Delta_{N-1}}(\mathcal F') \supset \mathcal D_{\Delta_{N-1}}(\mathcal F)$, inserting \eqref{eq-indassumption} into the above completes the proof.
\end{proof}
	
We are now ready to prove Theorem \ref{thm-furstenbergset}. 

\begin{proof}[Proof of Theorem \ref{thm-furstenbergset}]
Let 
\begin{equation*}
    v = \tfrac{1}{2}(u + \min \lbrace \tfrac{s+t}{2}, 1\rbrace) \in [0,1),
\end{equation*}
and let $\varepsilon, \delta_0>0$ be small enough so that Theorem \ref{proj-regular} holds with $v$ in place of $u$ and with constants $10\varepsilon$ and $\delta_0$. Additionally, assume $\varepsilon>0$ is small enough so that $v-6\varepsilon>u$. Choose $\delta_0', \varepsilon'\in (0,\tfrac{1}{2}]$ so that 
\begin{equation*}
\delta_0' \leq (\delta_0)^{\varepsilon^{-3}}, \quad 
\varepsilon' \leq \tfrac{1}{10} \varepsilon^4.
\end{equation*}
We will now prove that the conclusion of Theorem \ref{thm-furstenbergset} holds when $0<\delta\leq \delta_0'$ and $(\mathcal F_0, \mathcal P_0)$ is a $(\delta, s, \delta^{-\varepsilon'}, M)$-nice configuration, where $\mathcal F_0$ is $(\delta,t,\delta^{-\varepsilon'})$-regular. In the following, we replace $\mathcal F_0$ by $\mathcal D_\delta(\mathcal F_0)$ without changing notation: this allows us to formally apply Proposition \ref{prop-52OS} to $(\mathcal F_0, \mathcal P_0)$, and does not affect the conclusion of Theorem \ref{thm-furstenbergset}. Write
\begin{align*}
    \Delta &= \delta^{\varepsilon^3}\leq \delta_0, \\
    \Delta_j &= \Delta^{-j} \delta, \quad 0\leq j \leq \varepsilon^{-3}.
\end{align*}
Notice that $\Delta_{j} \leq \Delta_{j + 1}$ and $\Delta_j \leq \Delta$ for $j=1,\ldots, \varepsilon^{-3}-1 =: J$. 

As in \cite[Proof of Theorem 5.7]{2023arXiv230110199O}, we start by finding a sequence of $(\Delta_j, s, \lessapprox_\delta \delta^{-\varepsilon'}, M_j)$-nice configurations $(\mathcal{F}_j, {\mathcal{P}}_j)_{j=0}^J$ such that $\mathcal F_j$ is $(\Delta_j, t, \lessapprox_\delta \delta^{-\varepsilon'})$-regular. 

By pigeonholing the quantity $|\mathcal P_0({\bf f}) \cap Q|$ first with respect to $Q\in\mathcal D_\Delta(\mathcal P_0({\bf f}))$ and then with respect to ${\bf f} \in \mathcal F_0$, we may replace $\mathcal P_0$ and $\mathcal F_0$ by subsets such that the cardinalities of $\mathcal{F}_{0}$ and $\mathcal P_0({\bf f})$ are reduced by a factor of $\approx_\delta 1$ and for every ${\bf f}\in\mathcal F_0$, 
\begin{equation*}
    |\mathcal P_0({\bf f}) \cap Q| \equiv H_0, \qquad {\bf f} \in \mathcal F_0, \quad Q \in \mathcal D_\Delta(\mathcal P_0({\bf f})).
\end{equation*}
for some $H_0 \in 2^\N$.  

Suppose then that for some $j\leq J-1$, the $(\Delta_j, s, \lessapprox_\delta \delta^{-\varepsilon'}, M_j)$-nice configuration $(\mathcal F_j, {\mathcal{P}}_j)$ has been defined. We begin the construction of $(\mathcal F_{j+1}, \mathcal P_{j+1})$ by pigeonholing a subset $\overline{\mathcal F}_j \subset \mathcal F_j$ such that 
\begin{enumerate}
\item \label{property1} $|\overline{\mathcal F}_j | \approx_{\Delta_j} |\mathcal F_j|$, 
\item \label{property2} $|\overline{\mathcal F}_j \cap {\bf F}| \equiv A_j$ for some $A_j\in2^\N$, for every ${\bf F}\in \mathcal D_{\Delta_{j+1}}(\overline{\mathcal F}_j)$. 
\end{enumerate}
Then apply Proposition \ref{prop-52OS} to the configuration $(\overline{\mathcal F}_j, {\mathcal{P}}_j)$, with $\Delta_j$ and $\Delta_{j+1}$ in place of $\delta$ and $\Delta$. This proposition gives a ``refinement configuration'' $(\hat{\mathcal F}_j, \hat{{\mathcal{P}}}_j)$ (the configuration $(\mathcal F, {\mathcal{P}})$ in the statement of Proposition \ref{prop-52OS}) and a ``covering configuration'' $(\mathcal{F}_{j+1}, {\mathcal{P}}_{j+1})$ which is $(\Delta_{j+1}, s, C_{\Delta_{j+1}}, M_{\Delta_{j+1}})$-nice (the configuration $(\mathcal D_{\Delta}(\mathcal F), {\mathcal{P}}_\Delta)$ in the statement of Proposition \ref{prop-52OS}), with the following properties:
\begin{enumerate}[(i)]
    \item \label{prop-i} $|\mathcal D_{\Delta_{j+1}}(\hat{\mathcal F_j})| \approx_\delta |\mathcal D_{\Delta_{j+1}} (\overline{\mathcal F}_j)|$ and $|\hat{\mathcal F}_j\cap {\bf F}| \approx_\delta |\overline{\mathcal F}_j \cap {\bf F}|$ for every ${\bf F}\in \mathcal D_{\Delta_{j+1}}(\hat{\mathcal F})$. Also, it follows from the above properties \eqref{property1} and \eqref{property2} that $|\hat{\mathcal F}_j|\approx_\delta|\mathcal D_{\Delta_{j+1}}(\overline{\mathcal F}_j)|\cdot A_j = |\overline{\mathcal F}_j| \approx_\delta |\mathcal F_j|$. 

    \item \label{prop-ii} $|\hat{\mathcal P}_j({\bf F})| \approx_\delta |\mathcal P_j({\bf F})|$ for ${\bf F}\in\hat{\mathcal F}_j$. 

    \item \label{prop-iii} The configuration $(\mathcal F_{j+1}, \mathcal P_{j+1})$ is $(\Delta_{j+1}, s, \lessapprox_\delta \delta^{-\varepsilon'}, M_{j+1})$-nice for some $M_{j+1} \geq 1$, and $\mathcal F_{j+1} = \mathcal D_{\Delta_{j+1}}(\hat{\mathcal F}_j)$. Moreover, 
    \begin{equation}\label{eq_PjPj+1}
        \bigcup \hat{\mathcal P}_j({\bf f}) \subset \bigcup \mathcal P_{j+1}({\bf F}), \qquad {\bf f}\in \hat{\mathcal F}_j\cap {\bf F},\ {\bf F} \in \mathcal F_{j+1}.
    \end{equation}

    \item \label{prop-iv} For each ${\bf F}\in \mathcal F_{j+1}$ there exists $M_{\bf F}\geq 1$ and a family of squares $\mathcal P_{\bf F} \subset \mathcal D_{\Delta}$ such that the configuration $(T_{\bf F}(\hat{\mathcal F}_j \cap {\bf F}), \mathcal P_{\bf F})$ is $(\Delta, s, \lessapprox_\delta \delta^{-\varepsilon'},M_{\bf F})$-nice, where 
    \begin{equation*}
        T_{\bf F}(\hat{\mathcal F}_j \cap {\bf F}) = \lbrace T_{\bf F}({\bf f}):\ {\bf f}\in \hat{\mathcal F}_j\cap {\bf F}\rbrace. 
    \end{equation*}
    Moreover, 
    \begin{equation}\label{eq-equalprojections}
        \mathcal D_{\Delta}(\pi_{\mathbf{x}} (\mathcal P_{\bf F}(T_{\bf F}({\bf f}))) = \mathcal D_{\Delta} (\pi_{\mathbf{x}}(\hat{\mathcal P}_j({\bf f}))),\qquad {\bf f}\in \hat{\mathcal F}_j\cap {\bf F}.
    \end{equation}

    \item \label{prop-v}
    \begin{equation}\label{eq-growth2}
        \frac{|\mathcal P_j|}{M_j} \gtrapprox_\delta \frac{|\mathcal P_{j+1}|}{M_{j+1}} \cdot \frac{|\mathcal P_{\bf F}|}{M_{\bf F}},\qquad {\bf F}\in \mathcal F_{j+1}.
    \end{equation}
\end{enumerate}
By some additional pigeonholing, we may ensure that the number 
\begin{equation}\label{eq-regularproof-uniformity}
H_{j+1} := |\mathcal P_{j+1}({\bf F}) \cap Q|,\qquad {\bf F} \in \mathcal F_{j+1}, Q \in \mathcal D_{\Delta}(\mathcal P_{j+1}({\bf F}))
\end{equation}
depends only on $j$. First pigeonhole $|\mathcal P_{j+1}({\bf F}) \cap Q|$ with respect to $Q \in \mathcal D_{\Delta}(\mathcal P_{j+1}({\bf F}))$ for ${\bf F}$ fixed. Then pigeonhole this number with respect to ${\bf F}$. Finally, discard a few squares from each $\mathcal P_{j+1}({\bf F}) \cap Q$ to achieve equality.  Combining this with \ref{prop-ii} above, we also find
\begin{equation}\label{eq-regularproof-uniformity2}
|\hat{{\mathcal{P}}}_j({\bf F})|_\Delta \approx_\delta |{\mathcal{P}}_j({\bf F})|_\Delta, \qquad {\bf F}\in \hat{\mathcal F}_j,\ 0\leq j \leq J.
\end{equation}
This follows by rearranging the inequality
\begin{equation*}
    H_j \cdot |\mathcal P_{j}({\bf F})|_{\Delta} = |\mathcal P_j({\bf F})| \approx_\delta |\hat{\mathcal P}_j({\bf F})|\leq H_j\cdot |\hat{\mathcal P}_j({\bf F})|_\Delta.
\end{equation*}
We also record here that for all $0\leq j\leq J-1$ and ${\bf F}\in \mathcal F_{j+1}$, using Lemma \ref{lemma-almostinjection}, 
\begin{equation}\label{eq-M_F}
    M_{\bf F} \stackrel{\rm def.}= |\mathcal P_{\bf F}(T_{\bf F}({\bf f}))| \stackrel{\eqref{eq-equalprojections}}\sim |\pi_{\mathbf{x}} (\hat{\mathcal P}_j({\bf f}))|_\Delta \stackrel{\eqref{eq-regularproof-uniformity2}}{\approx_\delta} |\mathcal P_j({\bf f})|_\Delta,\qquad {\bf f}\in \hat{\mathcal F}_j\cap {\bf F}.
\end{equation}
Let us then verify that the sets $\mathcal F_j$ indeed have the claimed regularity. 
\begin{lemma}\label{lemma-Fjregular}
    For $0\leq j\leq J$, the family $\mathcal F_j$ is $(\Delta_j, t, \lessapprox_\delta \delta^{-\varepsilon'})$-regular.
\end{lemma}
\begin{proof}
We prove the claim by induction on $j$. 

For $j=0$ this holds by assumption. Assume then that $\mathcal F_j$ is $(\Delta_j, t, \lessapprox_\delta \delta^{-\varepsilon'})$-regular for $j<J$. Since $\mathcal F_{j+1}\subset \mathcal D_{\Delta_{j+1}}(\mathcal F_0)$, for any $\Delta_{j+1} \leq r\leq R\leq 1$ and $f \in \bigcup_{{\bf f}\in \mathcal F_{j+1}} {\bf f}$ we have
\begin{equation*}
    |\mathcal F_{j+1} \cap B(f, R)|_r \lesssim |\mathcal F_0\cap B(f,R)|_r \leq \delta^{-\varepsilon'} (R/r)^t
\end{equation*}
which is Definition \ref{def-regularse}(2). Therefore we only need to verify that $\mathcal F_{j+1}$ is a $(\Delta_{j+1}, t, \lessapprox \delta^{-\varepsilon'})$-set. By property \ref{prop-i} and the inductive assumption, $\hat{\mathcal F}_j$ is a $(\Delta_j, t, \lessapprox_\delta \delta^{-\varepsilon'})$-set. Recalling that $\mathcal F_{j+1} = \mathcal D_{\Delta_{j+1}}(\hat{\mathcal F}_j)$ and using \ref{prop-i}, (\ref{property2}) and Corollary \ref{cor-covernumber} (cf. \cite[Lemma 2.17]{OS23} for details) we have
\begin{equation*}
    |\mathcal F_{j+1} \cap B(f, r)|_{\Delta_{j+1}} \lessapprox_\delta \delta^{-\varepsilon'} r^s | \hat{\mathcal F}_j|_{\Delta_{j+1}} \lesssim \delta^{-\varepsilon'} r^s |\mathcal F_{j+1}|_{\Delta_{j+1}}
\end{equation*}
for any $r\geq \Delta_{j+1}$ and $f\in \bigcup_{{\bf f}\in \mathcal F_{j+1}} {\bf f}$ which is what was claimed.
\end{proof}

Since $\mathcal F_{j+1} = \mathcal D_{\Delta_{j+1}}(\hat{\mathcal F}_j)$ for $j=0,\ldots, J-1$, there exist ${\bf F}_j \in \hat{\mathcal F}_j \subset \mathcal D_{\Delta_j}(\mathcal F)$ such that ${\bf F}_{j} \subset {\bf F}_{j+1}$ for $j=0,\ldots, J-1$. Denote $N_j := M_{{\bf F}_j} \approx_\delta |{\mathcal{P}}_{j-1}({\bf F}_{j-1})|_\Delta$, where the last inequality is \eqref{eq-M_F}, and note that for $1\leq j \leq J-1$, 
\begin{equation}\label{eq-Njincreasing}
    N_j \approx_\delta |P_{j-1}({\bf F}_{j-1})|_\Delta \stackrel{\eqref{eq-regularproof-uniformity2}}{\approx_\delta} |\hat{\mathcal P}_{j-1}({\bf F}_{j-1})|_\Delta \stackrel{\eqref{eq_PjPj+1}}\leq |\mathcal P_j({\bf F}_j)|_\Delta \approx_\delta N_{j+1}.
\end{equation}
Our goal in the remainder of the proof is to show that 
\begin{equation}\label{eq-goalforj}
    \frac{|{\mathcal{P}}_{{\bf F}_j}|}{N_j} \geq \Delta^{-v+2\varepsilon}, \qquad j \in \lbrace 0,\ldots, J\rbrace \, \setminus \, E,
\end{equation}
where $E \subset \{0,\ldots,J\}$ is a small "exceptional set" with $|E|\leq 2\varepsilon^{-1}$. Given \eqref{eq-goalforj}, we will then have
\begin{equation*}
    \frac{|\mathcal P_0|}{M_0} \gtrapprox_\delta \prod_{j=1}^J \frac{|\mathcal P_{{\bf F}_j}|}{N_j} \geq (\Delta^{-v+2\varepsilon})^{\varepsilon^{-3} -2-2\varepsilon^{-1}} \geq \delta^{-v+6\varepsilon} \geq \delta^{-u}
\end{equation*}
by repeated application of \eqref{eq-growth2}, recalling that $\Delta = \delta^{\varepsilon^3}$ and $J = \varepsilon^{-3} - 1$. This will complete the proof of the theorem. We now move on to proving \eqref{eq-goalforj} with $E:= \lbrace j\in \lbrace 1,\ldots, J-1\rbrace:\ N_{j+1} \geq \Delta^{-\varepsilon} N_j\rbrace$, and begin by proving the desired bound for $|E|$. Since $N_{J} = M_{{\bf F}_{J}} \lesssim \mathcal D_{\Delta}([-2,2]) \leq 4\Delta^{-1}$ by Lemma \ref{lemma-almostinjection}, and $N_{j+1} \gtrapprox_\delta N_j$ for every $j$ by \eqref{eq-Njincreasing},
\begin{align*}
4\Delta^{-1} &\geq N_{J} \geq \frac{N_{J}}{N_1} =\prod_{j=1}^{J-1} \frac{N_{j+1}}{N_j} = \prod_{j\in E} \frac{N_{j+1}}{N_j} \prod_{j\not\in E}\frac
{N_{j+1}}{N_j} \\
&\geq (\Delta^{-\varepsilon})^{|E|}(C \log (1/\delta)^{-C})^{J-|E|}\geq (\Delta^{-\varepsilon})^{|E|}(C \log (1/\delta)^{-C})^{\varepsilon^{-3}}
\geq 4\Delta^{1 -\varepsilon|E|}
\end{align*}
as long as $\delta>0$ is small enough. The bound $|E|\leq 2\varepsilon^{-1}$ follows from this inequality.

Fix now an index $j\in \lbrace 1,\ldots, J-1\rbrace \, \setminus \, E$ for the rest of the proof, and write ${\bf F} := {\bf F}_j$ and ${\bf G} := \hat{\mathcal{F}}_{j-1}\cap {\bf F} \subset \mathcal D_{\Delta_{j-1}}(\mathcal F_{0})$. Define
\[\Theta = \mathcal{D}_\Delta (\pi_{\mathtt x}({\mathcal{P}}_j({\bf F}))) \subset \mathcal D_\Delta(\R)\]
and for ${\bf f}\in {\bf G}$,
\[\Theta_{\bf f} = \mathcal{D}_\Delta (\pi_{\mathtt x}(\hat{\mathcal{P}}_{j-1}({\bf f}))) \subset \Theta.\]
The inclusion holds by \ref{prop-iii}. It follows from the definition of $N_j$ and Lemma \ref{lemma-almostinjection} that 
\[
|\Theta| \stackrel{\eqref{eq-M_F}}{\approx_\delta} N_{j+1} \leq \Delta^{-\varepsilon} N_j \stackrel{\eqref{eq-M_F}}\sim \Delta^{-\varepsilon} |\mathcal D_\Delta (\pi_{\mathtt x} (\hat{\mathcal{P}}_{j-1}({\bf f})))| = \Delta^{-\varepsilon}|\Theta_{\bf f}|
\]
for every ${\bf f}\in \mathbf{G}$. As a consequence, we claim that if
\[\mathbf{G}_\theta := \lbrace {\bf f} \in \mathbf{G} :\ \theta \in \Theta_{\bf f}\rbrace,\]
then there exists $\Theta'\subset \Theta$ such that $|\Theta'|\gtrapprox_\delta \Delta^\varepsilon|\Theta|$ and $|\mathbf{G}_\theta| \gtrapprox_\delta \Delta^\varepsilon|\mathbf{G}|$ for every $\theta\in\Theta'$: this follows from 
\begin{equation*}
    \sum_{\theta \in \Theta} |\mathbf{G}_\theta| = |\lbrace ({\bf f},\theta) \in \mathbf{G}\times \Theta:\ \theta \in \Theta_{\bf f}\rbrace| = \sum_{{\bf f} \in \mathbf{G}} |\Theta_{\bf f}| \gtrapprox_\delta \Delta^{\epsilon}|\mathbf{G}| |\Theta|.
\end{equation*}
Since $\mathbf{G}\subset \mathcal D_{\Delta_{j-1}}(\mathcal F_{0})$, it follows from Lemma \ref{cor-covernumber} that there exist $\Delta_{j-1}$-separated subsets $\mathcal G\subset \mathbf{G}$ and $\mathcal G_\theta\subset \mathbf{G}_\theta$ with $|\mathcal G| \sim |\mathbf{G}|$ and $|\mathcal G_\theta|\sim |\mathbf{G}_\theta|$. We aim to apply Theorem \ref{proj-regular} to the rescaled set $T_{\bf F}(\mathcal G)$ and its subsets $T_{\bf F}(\mathcal G_\theta)$. Let us verify the needed regularity properties.

\begin{lemma}\label{lemma-regular}
Assuming that $\delta>0$ is small enough, the set $T_{\bf F}(\mathcal G)$ is $(\Delta, t, \Delta^{-\varepsilon})$-regular, and $\Theta'$ is a $(\Delta, s, \Delta^{-3\varepsilon})$-set.
\end{lemma}
\begin{proof}

Note first that $\hat{\mathcal F}_{j - 1}$ is $(\Delta_{j - 1}, t, \lessapprox_\delta \delta^{-\varepsilon'})$-regular by Lemma \ref{lemma-Fjregular} and property \ref{prop-i}. Applying Lemma \ref{lemma-TFregular} to $\hat{\mathcal F}_{j-1}\cap {\bf F}$, we find that $T_{\bf F}(\mathbf{G})$ is $(\Delta, t, \overline{C})$-regular with 
\begin{equation}\label{form39} \overline{C} \approx_\delta \max\left\{\delta^{-\varepsilon'},\delta^{-\varepsilon'}\cdot \frac{\Delta_{j}^t|\hat{\mathcal F}_{j-1}|}{|\hat{\mathcal F}_{j-1} \cap {\bf F}|}\right\}.
 \end{equation}
Therefore our main task is to find a lower bound for $|\hat{\mathcal F}_{j-1} \cap {\bf F}|$. First of all, since $\hat{\mathcal F}_{j - 1} \subset \mathcal D_{\Delta_{j - 1}}(\mathcal F_0)$, we have
\begin{equation}\label{eq-hatupperbound}
    |\hat{\mathcal F}_{j - 1}| \leq \delta^{-\varepsilon'}\Delta_{j - 1}^{-t} \quad \text{and} \quad |\mathcal D_{\Delta_{j}}(\hat{\mathcal F}_{j - 1})|\leq \delta^{-\varepsilon'}\Delta_{j}^{-t}
\end{equation}
by the $(\delta, t,\delta^{-\varepsilon'})$-regularity of $\mathcal F_0$. Second, by \ref{prop-i}, we have 
\[
|\hat{\mathcal F}_{j - 1}\cap {\bf F}'|\approx_\delta |\overline{\mathcal{F}}_{j - 1} \cap {\bf F}'| = A_{j - 1},\qquad {\bf F}' \in \mathcal D_{\Delta_{j}}(\hat{\mathcal F}_{j - 1}).
\]
Finally, from \eqref{eq-hatupperbound} and the $(\Delta_{j - 1}, t, \lessapprox_\delta \delta^{-\varepsilon'})$-set property of $\hat{\mathcal F}_{j - 1}$, we deduce 
\[
\delta^{-\varepsilon'}\Delta_{j - 1}^{-t} \lessapprox_\delta |\mathcal D_{\Delta_{j}}(\hat{\mathcal F}_{j - 1})| \cdot A_j \leq \delta^{-\varepsilon'}\Delta_{j}^{-t}\cdot A_{j - 1},
\]
upon rearranging which we find that
\[
|\hat{\mathcal F}_{j - 1} \cap {\bf F}| \approx_{\delta} A_{j - 1}\gtrapprox_\delta \delta^{2\varepsilon'}\Delta^{-t}.
\]
Consequently, the second term inside the $\max$ on line \eqref{form39} satisfies
\[
\delta^{-\varepsilon'}\cdot \frac{\Delta_{j}^t|\hat{\mathcal F}_{j-1}|}{|\hat{\mathcal F}_{j-1} \cap {\bf F}|}\lessapprox_\delta \delta^{-\varepsilon'} \frac{\Delta_{j}^t \delta^{-\varepsilon'} \Delta_{j - 1}^{-t}}{\delta^{2\varepsilon'}\Delta^{-t}} = \delta^{-4\varepsilon'}.
\]
In particular $\overline{C} \lessapprox_{\delta} \delta^{-4\varepsilon'}$. Since $\varepsilon'$ was chosen to be smaller than $\frac{1}{10}\varepsilon^4$, this shows that $\overline{C} \leq \Delta^{-\varepsilon}$ when $\delta$ is small enough. Since $|\mathcal G|\sim |\mathbf{G}|$, this completes the proof of the first claim.

For the second claim, it follows from property \ref{prop-iii} and Lemma \ref{lemma-almostinjection} that $\pi_{\mathbf{x}}(\mathcal P_{j}({\bf F})) \subset \mathcal{D}_{\Delta_{j}}(\R)$ is a $(\Delta, s, \delta^{-\varepsilon'})$-set. Recall that by \eqref{eq-regularproof-uniformity2}, $|\mathcal P_j({\bf F}) \cap Q| = H_j$ for every $Q\in \mathcal D_\Delta(\mathcal P_j({\bf F}))$. By Lemma \ref{lemma-almostinjection}, also $|\pi_{\mathbf{x}}(\mathcal P_j({\bf F})) \cap J| \sim H_j$ for every $J \in \mathcal D_\Delta(\pi_{\mathbf{x}}(\mathcal P_j({\bf F}))) = \Theta$. Combining these, we find that for any interval $I\subset \R$ with $\diam (I) \geq \Delta \geq \Delta_j$,  
\begin{align*}
    |\Theta \cap I| &\lesssim |\lbrace J\in \Theta:\ J \subset 2I \rbrace| 
    \lesssim H_j^{-1} |\pi_{\mathbf{x}}(\mathcal P_{j}({\bf F})) \cap 2I| \\&\lesssim \delta^{-\varepsilon'} H_j^{-1} \diam(I)^s | \pi_{\mathbf{x}}(\mathcal P_{j}({\bf F}))| \sim \delta^{-\varepsilon'} \diam(I)^s |\Theta|.
\end{align*}
Therefore $\Theta$ is a $(\Delta, s, \delta^{-\varepsilon'})$-set. Since $|\Theta'| \gtrapprox_\delta \Delta^\varepsilon |\Theta|$, it follows that $\Theta'$ is a $(\Delta, s, \Delta^{-\varepsilon}\delta^{-\varepsilon'})$-set. Recalling that $\varepsilon'\leq \varepsilon^4$ and $\delta = \Delta^{\varepsilon^{-3}}$, we have $\Delta^{-\varepsilon}\delta^{-\varepsilon'} \leq \Delta^{-2\varepsilon}$ and therefore $\Theta'$ is a $(\Delta, s, \Delta^{-3\varepsilon})$-set. 
\end{proof}

Let $E\subset \Theta'$ be a $\Delta$-separated subset with $|E| \sim |\Theta'|$. Clearly $E$ is a $(\Delta, s, \Delta^{-3\varepsilon})$-set. By Theorem \ref{proj-regular}, there exists a set $E'\subset E$ with $|E'| \geq |E|/2$ such that for every $\theta\in E'$, 
\[
\left|\bigcup_{f \in T_{\bf F}(\mathcal G_\theta)} \Gamma_f \cap L_\theta\right|_\Delta \geq \Delta^{-v}.
\]

Recall that if ${\bf f}\in T_{\bf F}(\mathbf{G}_\theta)$, then by definition there exists at least one $Q({\bf f})\in \mathcal P_{\bf F}$ with $L_\theta\cap Q({\bf f})\neq\emptyset$ such that $\Gamma_f \cap Q({\bf f})\neq \emptyset$ for some $f\in {\bf f}$. Moreover, if ${\bf f}, {\bf g}\in T_{\bf F}(\mathbf{G}_{\theta})$ are such that $\min_{f\in{\bf f}, g\in {\bf g}}|f(\theta) - g(\theta)|\geq 2\Delta$, then $Q({\bf f}) \neq Q({\bf g})$. Recalling that $T_{\bf F}(\mathcal G_\theta)$ is a $4\Delta$-separated subset of $T_{\bf F}(\mathbf{G}_\theta)$, it follows that 
\[
|\lbrace Q\in \mathcal P_{\bf F}:\ L_\theta\cap Q\neq \emptyset\rbrace| \geq \left|\bigcup_{f \in T_{\bf F}(\mathcal G_\theta)} \Gamma_f \cap L_\theta\right|_\Delta \geq \Delta^{-v}
\]
Since 
\[|E'| \gtrapprox_\delta \Delta^{\varepsilon}|\Theta| \approx_\delta \Delta^{\varepsilon} N_{j+1} \gtrapprox_\delta \Delta^\varepsilon N_j,\]
it follows that 
\[
\frac{|{\mathcal{P}}_{{\bf F}}|}{N_j} \gtrapprox_\delta \Delta^\varepsilon\frac{|{\mathcal{P}}_{{\bf F}}|}{|E'|} \geq \Delta^\varepsilon\sum_{\theta \in E'} \frac{|\lbrace Q \in {\mathcal{P}}_{\bf F}:\ L_\theta \cap Q\neq\emptyset\rbrace|}{|E'|} \geq \Delta^{-v+\varepsilon}.
\]
Thus $\frac{|{\mathcal{P}}_{{\bf F}}|}{N_j} \geq \Delta^{-v+2\varepsilon}$ when $\delta>0$ is small enough. This completes the proof.
\end{proof}

	
\section{Semi-well spaced case}\label{sec6}
In this section, $\mathcal{F}\subset B_{C^2}(1)$ will always be a $\delta$-separated transversal family with constant $\mathfrak{T}\geq1$. The notation $\lessapprox_\delta$ means $\leq C (\log\delta^{-1})^C$ for some $C>0$ independent of $\delta$, and likewise for $\gtrapprox_\delta$ and $\approx_\delta$.

\subsection{Two auxiliary results} Recall the definition of \emph{$(\delta,s,C,M)$-nice configuration} $(\mathcal{F},\mathcal{P})$ from Definition \ref{def-niceconfiguration}. In particular, the transversal family $\mathcal{F}$ in a $(\delta,s,C,M)$-nice configuration is assumed to be $\delta$-separated. We now introduce some new terminology.

\begin{definition}\label{def:notation1}
	Let $\delta \leq \Delta\in 2^{-\N}$. For each $p\in \mathcal{D}_\delta$, define 
	\begin{displaymath} N_{\Delta,b}(p):=|\{\mathbf{F}\in \mathcal{D}_\Delta(\mathcal{F}): |2p\cap \mathbf{F}\cap \mathcal{F}|\geq b\}|, \end{displaymath}
	where $\lambda p\cap \mathbf{F}\cap \mathcal{F}:=\{f\in \mathbf{F}\cap \mathcal{F}: z_p\in\Gamma_f(\lambda \delta)\}$, $\lambda \in \N$ (thus, $\delta$ is implicitly determined by the side-length of $p$). For each $q\in \mathcal{D}_\Delta$, define 
	\begin{displaymath} \lambda q \cap \mathcal{D}_\Delta(\mathcal{F}):=\{\mathbf{F}\in \mathcal{D}_\Delta(\mathcal{F}): \lambda q \cap \mathbf{F} \cap \mathcal{F} \neq \emptyset\}. \end{displaymath}
	Here and throughout this paper $z_p$ is the centre of the cube $p$.
\end{definition}

The following proposition is \cite[Proposition 4.4]{2023arXiv230808819R} translated to our notation:

\begin{proposition}\label{pro-nicecondition}
	Let $s\in [0,1]$ and $\lambda \geq 4$. Then the following holds for all $\Delta \in 2^{-\N} \cap (0,(2\lambda)^{-1}]$, and for $\delta = \Delta^{n}$ with $n \in \N$.

	 Fix $C_0, M_0\geq1$. Let $(\mathcal{F}_0,\mathcal{P}_0)$ be a $(\delta, s, C_0,M_0)$-nice configuration with $\mathcal{F}_0 \subset B_{C^{2}}(1)$ a non-empty $(\delta,2 - s,C_0)$-set. Then there exists a refinement $(\mathcal{F},\mathcal{P})$ of $(\mathcal{F}_0,\mathcal{P}_0)$ such that 
	\begin{equation}\label{form-441}
		|\lambda p^{\Delta^{k}}\cap \mathcal{D}_{\Delta^{k}}(\mathcal{F})|\lessapprox C_0^2\Delta^{-3}\cdot \Delta^{k}|\mathcal{D}_{\Delta^{k}}(\mathcal{F})|, \qquad p \in \mathcal{P}, \, k \in \{1,\ldots,n\}.
	\end{equation}
	Here $p^{\Delta^{k}}$ is the dyadic $\Delta^{k}$-cube containing $p$. Here, and in the proof, the notation "$\lessapprox$" means "$\leq C_1(\log\delta^{-1})^{C_2}$" with $C_1, C_2>0$ depending on $\lambda,n,\mathfrak{T}$. 
\end{proposition}

\begin{proof}
We prove Proposition \ref{pro-nicecondition} by induction on $n$. Consider first the case $n = 1$. Then $\Delta = \delta$ and $k =1$, and \eqref{form-441} looks like this (no refinement is needed):
\begin{displaymath} |\lambda p \cap \mathcal{D}_{\delta}(\mathcal{F}_{0})| \lessapprox C_{0}^{2}\delta^{-3} \cdot \delta|\mathcal{D}_{\delta}(\mathcal{F}_{0})|, \qquad p \in \mathcal{P}_{0}. \end{displaymath}
The idea is that $|\lambda p \cap \mathcal{D}_{\delta}(\mathcal{F}_{0})| \lesssim_{\lambda,\mathfrak{T}} \delta^{-1}$, and $\delta |\mathcal{D}_{\delta}(\mathcal{F}_{0})| \gtrsim C_{0}^{-1}$ by the $(\delta,t,C_{0})$-set hypothesis. So, the inequality holds with an extra $C_{0}\delta^{-2}$ to spare. Let us justify the claim $|\lambda p \cap \mathcal{D}_{\delta}(\mathcal{F}_{0})| \lesssim_{\lambda,\mathfrak{T}} \delta^{-1}$. Let $f,g\in\mathcal{F}_0$ satisfy $d(f,g)\geq 4\lambda \mathfrak{T}\delta$ and $z_p=(x_p,y_p)\in\Gamma_f(\lambda \delta)\cap \Gamma_g(\lambda \delta)$, then by triangle inequality $|f(x_p)-g(x_p)|\leq 2\lambda \delta$. By the transversality of $\mathcal{F}_0$,
\[|f'(x_p)-g'(x_p)|\geq \frac{d(f,g)}{\mathfrak{T}}-2\lambda \delta\geq \frac{d(f,g)}{2\mathfrak{T}}.\]
Let $\mathcal{F}^\ast \subset \cup [\lambda p^\delta\cap \mathcal{D}_\delta(\mathcal{F}_0)]$ be a maximal $2\lambda \mathfrak{T}\delta$-separated set. Defining $A^\ast:\mathcal{F}^\ast\to [-1,1]$ by $A^\ast(f)=f'(x_p)$, it follows that $A^\ast$ is a bi-Lipschitz embedding. Hence $|\mathcal{F}^\ast|_\delta \lesssim_\mathfrak{T}\delta^{-1}$. Since $\mathcal{F}_0$ is upper $(2,20\mathfrak{T}^{2})$-regular, we have $|\lambda p^\delta\cap \mathcal{D}_\delta(\mathcal{F}_0)|\lesssim_{\lambda,\mathfrak{T}} |\mathcal{F}^\ast|_\delta\lesssim_{\mathfrak{T}} \delta^{-1}$. 
	
In the sequel, we assume that the statement is true for $n-1$ and prove case $n$.
		
We apply Proposition \ref{pro-pigeonhole} to $(\mathcal{F}_0, \mathcal{P}_0)$ with scales $\delta,\Delta^{n-1}$. This gives us a refinement $(\mathcal{F}_1, \mathcal{P}_1)$ of $(\mathcal{F}_0, \mathcal{P}_0)$ and a $(\Delta^{n-1}, s, C_{\Delta^{n-1}}, M_{\Delta^{n-1}})$-nice configuration $(\mathcal{F}_1^{\Delta^{n-1}}, \mathcal{P}_1^{\Delta^{n-1}})$, where $\mathcal{F}_{1}$ is $\{1,\Delta^{n - 1},\delta\}$-uniform. We wish to apply induction to $(\mathcal{F}_1^{\Delta^{n-1}}, \mathcal{P}_1^{\Delta^{n-1}})$. To this end we claim that $\mathcal{F}_1^{\Delta^{n-1}} = \mathcal{D}_{\Delta^{n - 1}}(\mathcal{F}_{1})$ is a $(\Delta^{n - 1},t,C)$-set with $C \lessapprox C_{0}$. This is true because $\mathcal{F}_{1}$ is a $\{\delta,\Delta^{n - 1},1\}$-uniform $(\delta,t,C)$-set with $C \lessapprox_{\delta} C_{0}$, see Lemma \ref{lem-subofuni}. 
		
Using the inductive hypothesis, there is a refinement $(\mathcal{F}_{\Delta^{n-1}}, \mathcal{P}_{\Delta^{n-1}})$ of $(\mathcal{F}_1^{\Delta^{n-1}}, \mathcal{P}_1^{\Delta^{n-1}})$ such that for any $1\leq k\leq n-1$, 
\begin{equation}\label{form-442}
|\lambda \mathbf{p}^{\Delta^{k}}\cap \mathcal{D}_{\Delta^{k}}(\mathcal{F}_{\Delta^{n-1}})|\lessapprox C_0^2\Delta^{-3}\cdot \Delta^{k}|\mathcal{D}_{\Delta^{k}}(\mathcal{F}_{\Delta^{n-1}})|, \qquad \mathbf{p} \in \mathcal{P}_{\Delta^{n - 1}}.
\end{equation}
Write 
\begin{equation}\label{form17} \mathcal{F}_2:=\bigcup_{\mathbf{F}\in \mathcal{F}_{\Delta^{n-1}}} (\mathcal{F}_1\cap \mathbf{F}) \end{equation}
and $\mathcal{P}_2:=\bigcup_{f\in \mathcal{F}_2} \mathcal{P}_2(f)$, where for each $f\in \mathcal{F}_2$ \[\mathcal{P}_2(f):=\{p\in\mathcal{P}_1(f): p^{\Delta^{n-1}}\in \mathcal{P}_{\Delta^{n-1}}\}.\]
From property (2) of Proposition \ref{pro-pigeonhole}, we know $(\mathcal{F}_{\Delta^{n-1}}, \mathcal{P}_{\Delta^{n-1}})$ covers $(\mathcal{F}_1, \mathcal{P}_1)$, that is,
\begin{equation}\label{form-443}
\sum_{f\in\mathcal{F}_2} |\mathcal{P}_2(f)| = \sum_{\mathbf{F}\in\mathcal{F}_{\Delta^{n-1}}}\sum_{f\in \mathbf{F}\cap\mathcal{F}_1}\sum_{\mathbf{p}\in \mathcal{P}_{\Delta^{n-1}}(\mathbf{F})}|\mathbf{p}\cap \mathcal{P}_1(f)|\gtrapprox_\delta M_0|\mathcal{F}_1|.
\end{equation}
From \eqref{form-443} we can pigeonhole a number $r\in 2^\N \cap [1,O_{\lambda,\mathfrak{T}}(1)\delta^{-1}]$ such that by defining 
\begin{equation}\label{form18} \mathcal{P}:=\{p\in\mathcal{P}_2: |\lambda p\cap \mathcal{F}_2|\sim r\} \end{equation}
and $\mathcal{P}(f):=\mathcal{P}_2(f)\cap \mathcal{P}$ for any $f\in\mathcal{F}_2$, there holds
\begin{equation}\label{equ-incidence}
\sum_{f\in\mathcal{F}_2} |\mathcal{P}(f)|\gtrapprox M_0|\mathcal{F}_1|.
\end{equation}
We remind that for each $p\in\mathcal{P}$ there are two properties:
\begin{itemize}
\item $|\lambda p\cap \mathcal{F}_2|\sim r$, where $r\lesssim_{\lambda,\mathfrak{T}} \delta^{-1}$ by the same argument as in case $n=1$;
\item $p^{\Delta^{n-1}}\in \mathcal{P}_{\Delta^{n-1}}$, thus \eqref{form-442} applies to $p^{\Delta^{n-1}}$.
\end{itemize}
	
Now we verify \eqref{form-441} for $(\mathcal{F}_2,\mathcal{P})$. Fix $p\in\mathcal{P}$ and $w=\Delta^k$ with $1\leq k\leq n$. If $1\leq k\leq n-1$, then \eqref{form-442} applied to $p^{\Delta^{n-1}}$ gives
\[|\lambda p^{w} \cap \mathcal{D}_{w}(\mathcal{F}_{2})| = |\lambda (p^{\Delta^{n-1}})^w\cap \mathcal{D}_w(\mathcal{F}_{\Delta^{n-1}})|\lessapprox C_0^2\Delta^{-3}\cdot w|\mathcal{D}_w(\mathcal{F}_2)|,\]
where we used $\mathcal{D}_w(\mathcal{F}_2) = \mathcal{D}_{w}(\mathcal{D}_{\Delta^{n - 1}}(\mathcal{F}_{2})) =\mathcal{D}_w(\mathcal{F}_{\Delta^{n-1}})$, see \eqref{form17}. 
	
It remains to consider the case $k=n$, that is, demonstrate that
\begin{equation}\label{form19} r \stackrel{\eqref{form18}}{\sim} |\lambda p \cap \mathcal{F}_{2}| \lessapprox C_{0}^{2}\Delta^{-3} \cdot \delta|\mathcal{F}_{2}|, \qquad p \in \mathcal{P}. \end{equation} 
We apply Proposition \ref{p:highLowUnweighted} to $(\mathcal{F}_2, \mathcal{P})$ with $S :=\Delta^{-1} \geq 2\lambda$. This gives
\begin{equation}\label{eq-552}
\begin{split}
\mathcal{I}^{\lambda}(\mathcal{F}_2,\mathcal{P}) &=|\{(f,p)\in\mathcal{F}_2\times\mathcal{P}: z_p\in \Gamma_f(\lambda \delta)\}|\\&\lesssim_{\lambda} \mathbf{C}\Big(\Delta^{-3}\delta^{-1}|\mathcal{P}| |\mathcal{F}_2| \Big)^{1/2}+\Delta \cdot \mathcal{I}^{2}(\mathcal{F}_2,\mathcal{P}^{S\delta}),
\end{split}
\end{equation}
where $0 < \mathbf{C} \lesssim_{\mathfrak{T}} \log(1/\delta)$. Let us also recall (from the statement of Proposition \ref{p:highLowUnweighted}) that $\mathcal{I}^{2}(\mathcal{F}_{2},\mathcal{P}^{S\delta}) = |\{(f,p) \in \mathcal{F}_{2} \times \mathcal{P} : z_{p} \in \Gamma_{f}(2 S\delta)\}|$. Since $2 S\delta = 2\Delta^{n - 1}$, and $\lambda \geq 4$, and every $f \in \mathcal{F}_{0}$ is $1$-Lipschitz, $z_{p} \in \Gamma_{f}(2 S\delta)$ implies $z_{p^{\Delta^{n - 1}}} \in \Gamma_{f}(\lambda \Delta^{n - 1})$. Therefore
\begin{displaymath} \mathcal{I}(\mathcal{F}_2,\mathcal{P}^{S\delta}) \leq  |\{(f,p) \in \mathcal{F}_{2} \times \mathcal{P} : z_{p^{\Delta^{n - 1}}} \in \Gamma_{f}(\lambda \Delta^{n - 1})\}|. \end{displaymath}
Recall that $\mathcal{I}(\mathcal{F}_2,\mathcal{P})\gtrapprox M_0|\mathcal{F}_2|$ by \eqref{equ-incidence}, and $\mathcal{I}(\mathcal{F}_2,\mathcal{P})\gtrsim r|\mathcal{P}|$. If the first term in \eqref{eq-552} dominates, then
\[(M_0|\mathcal{F}_2|\cdot r|\mathcal{P}|)^{1/2}\lessapprox \mathcal{I}(\mathcal{F}_2,\mathcal{P})\lesssim\mathbf{C} \left(\Delta^{-3}\delta^{-1}|\mathcal{P}| |\mathcal{F}_2|\right)^{1/2},\]
which implies $r\lessapprox \Delta^{-3} \delta^{-1}M_0^{-1}$. Since $|\mathcal{F}_2|\gtrapprox |\mathcal{F}_0|\geq C_0^{-1} \delta^{s - 2}$ ($\mathcal{F}_0$ is a $(\delta,2 - s,C_0)$-set), $M_0 \geq C_0^{-1} \delta^{-s}$ ($P_0(f)$ is a $(\delta,s,C_0)$-set for each $f\in\mathcal{F}_0$), we arrive at \eqref{form19}.
	
If the second term in \eqref{eq-552} dominates, then
\[\begin{split}
r|\mathcal{P}| &\lesssim \Delta \cdot |\{(f,p)\in \mathcal{F}_2\times \mathcal{P}: z_{p^{\Delta^{n-1}}} \in \Gamma_f(\lambda \Delta^{n-1})\}|\\
&\leq \Delta \sum _{p\in\mathcal{P}} |10\lambda p^{\Delta^{n-1}}\cap \mathcal{D}_{\Delta^{n-1}}(\mathcal{F}_2)|\cdot \sup_{\mathbf{F}\in \mathcal{D}_{\Delta^{n-1}}(\mathcal{F}_2)} |\mathbf{F}\cap \mathcal{F}_2|\\
&\stackrel{\eqref{form-442}}{\lessapprox} \Delta \cdot |\mathcal{P}|\big(C_0^2\Delta^{-3}\cdot \Delta^{n-1}|\mathcal{D}_{\Delta^{n-1}}(\mathcal{F}_2)|\big)\cdot \frac{|\mathcal{F}_2|}{|\mathcal{D}_{\Delta^{n-1}}(\mathcal{F}_2)|},
\end{split}\]
which implies
\[r\lessapprox C_0^2\Delta^{-3}\cdot \Delta^{n} |\mathcal{F}_2|=C_0^2\Delta^{-3}\cdot \delta|\mathcal{F}_2|.\]
This proves case $n$ with a multiplicative constant $\leq C_{1}(\log \tfrac{1}{\delta})^{C_2}$, where $C_1, C_2>0$ both only depend on $n,\lambda, \mathfrak{T}$. 
\end{proof}

As a by-product of Proposition \ref{pro-nicecondition}, we obtain the following Furstenberg set estimate:
\begin{cor}\label{cor-uppercaseF}
	Let $s\in [0,1]$. For each $\epsilon>0$, there exists $\delta_0=\delta_0(\mathfrak{T},\epsilon)>0$ such that the following holds for $\delta\in (0,\delta_0]$. Let $(\mathcal{F},\mathcal{P})$ be a $(\delta,s,\delta^{-\epsilon/20},M)$-nice configuration with $\mathcal{F} \subset B_{C^{2}}(1)$ a non-empty $(\delta,2 - s,\delta^{-\epsilon/20})$-set. Then $|\mathcal{P}| \geq \delta^{-1+\epsilon}M$.
\end{cor}
\begin{proof} Fix $\epsilon > 0$ and $\delta > 0$. We only treat explicitly the special case where $n := 4/\epsilon \in \N$ and $\Delta := \delta^{\epsilon/4} \in 2^{-\N}$. We apply Proposition \ref{pro-nicecondition} to $\delta,\Delta$, and $\lambda = 4$ and $C_{0} := \delta^{-\epsilon/20}$, noting that $\delta = \Delta^{n}$. This yields a refinement $(\mathcal{F}',\mathcal{P}')$ of $(\mathcal{F},\mathcal{P})$ such that for any $p\in\mathcal{P}'$, 
	\begin{equation}\label{form20}
		|4p\cap \mathcal{F}'|\lessapprox_{\delta,\mathfrak{T},\epsilon} \delta^{-2\eta}\Delta^{-3}\cdot \delta|\mathcal{F}'| = \delta^{-\epsilon/10 - 3\epsilon/4} \cdot \delta|\mathcal{F}'|. 
	\end{equation}
	Here $\lessapprox_{\delta,\mathfrak{T},\epsilon}$ means $\leq C_1 (\log\tfrac{1}{\delta})^{C_2}$ where both $C_1 ,C_2$ only depend on $\epsilon, \mathfrak{T}$. In particular, $C_{1}(\log \tfrac{1}{\delta})^{C_{2}} \leq \delta^{-\epsilon/10}$ for $\delta > 0$ sufficiently small, depending only on $\epsilon,\mathfrak{T}$. For such $\delta > 0$, \eqref{form20} becomes
	\begin{displaymath} |4p \cap \mathcal{F}'| \leq \delta^{1-19\epsilon/20}|\mathcal{F}'|, \qquad p \in \mathcal{P}'. \end{displaymath}
	Next, note that if $f \in \mathcal{F}'$, $p \in \mathcal{P}'$ with $p \cap \Gamma_{f} \neq \emptyset$ (as in the definition of a nice configuration), then $z_{p} \in \Gamma_{f}(4\delta)$, and therefore $f \in 4p \cap \mathcal{F}'$. This yields
	\[M|\mathcal{F}| \lessapprox_\delta \sum_{f\in\mathcal{F}'}|\mathcal{P}'(f)| \leq\sum_{p\in \mathcal{P}'}|4p\cap \mathcal{F}'| \leq \delta^{1 - 19\epsilon/20}|\mathcal{F}||\mathcal{P}|, \]
	and consequently $|\mathcal{P}| \geq \delta^{-1 + \epsilon}M$, provided that $\delta > 0$ is small enough that the $\lessapprox_{\delta}$ constant is bounded by $\delta^{-\epsilon/20}$.\end{proof}

\subsection{Furstenberg estimates for semi-well spaced case} In this section, we prove Furstenberg estimates in a situation where $\mathcal{F}$ is a $(\delta, 2-s)$-set at large scales and a Katz-Tao $(\delta,s)$-set at small scales. The next proposition is our counterpart of \cite[Proposition 4.6]{2023arXiv230808819R}.
\begin{proposition}\label{pro-semi}
	Let $s\in (0,1]$, $\epsilon>0$, and $\mathfrak{T} \geq 1$. Then, there exists $\Delta_{0} = \Delta_{0}(\epsilon,\mathfrak{T}) > 0$ such that the following holds for all $\delta,\Delta \in 2^{-\N} \cap (0,\Delta_{0}]$ with $\delta \leq \Delta$.
	
	Let $\mathcal{F} \subset B_{C^{2}}(1)$ be a $\delta$-separated transversal family over $[-2,2]$ with constant $\mathfrak{T}$. Let $(\mathcal{F},\mathcal{P})$ be a $(\delta,s,\delta^{-\epsilon^2},M)$-nice configuration. Suppose that $\mathcal{F}$ satisfies the spacing conditions
	\begin{align}
		|\mathcal{F}\cap B(f,r)|\leq \delta^{-\epsilon^2}\cdot r^{2-s}\cdot |\mathcal{F}|,\quad \Delta\leq r\leq1, \label{form-highscale}\\
		|\mathcal{F}\cap B(f,r)|\leq \delta^{-\epsilon^2} \big(\tfrac{r}{\delta}\big)^s,\quad \delta\leq r\leq\Delta. \label{form-lowscale}
	\end{align}
	Then $|\mathcal{P}|\geq \delta^{60\epsilon}\cdot \min\{M|\mathcal{F}|, M^{3/2}|\mathcal{F}|^{1/2}, \delta^{-1}M\}$.
\end{proposition}

The next corollary follows from Proposition \ref{pro-semi} by taking $\Delta=(\delta^s|\mathcal{F}|)^{1/(2s-2)}$.

\begin{cor}\label{cor2}
	Let $s\in (0,1)$, $\epsilon>0$ and $\eta=\tfrac{\epsilon^2}{3600}$. There exists $\delta_0=\delta_0(s,\mathfrak{T},\epsilon)>0$ such that the following holds for any $\delta\in (0,\delta_0]$. Let $(\mathcal{F},\mathcal{P})$ be a $(\delta,s,\delta^{-\eta},M)$-nice configuration with $\mathcal{F} \subset B_{C^{2}}(1)$ and $|\mathcal{F}| \in [\delta^{-s},\delta^{s - 2}]$. Assume that $\mathcal{F}$ satisfies
	\begin{equation}\label{form26}
		|\mathcal{F}\cap \mathbf{F}|\leq \delta^{-\eta}\cdot \max\{\rho^{2-s}|\mathcal{F}|, \big(\tfrac{\rho}{\delta}\big)^s\}, \qquad \mathbf{F}\in\mathcal{D}_\rho(\mathcal{F}), \rho\in [\delta,1].
	\end{equation}
	Then
	\begin{equation}\label{form-semiwellcase}
		|\mathcal{P}| \geq \delta^{-\tfrac{s}{2}+\epsilon}|\mathcal{F}|^{1/2}M.
	\end{equation}
\end{cor}

\begin{proof} The hypothesis $|\mathcal{F}| \in [\delta^{-s},\delta^{s - 2}]$ guarantees that $\Delta := (\delta^{s}|\mathcal{F}|)^{1/(2s - 2)} \in [\delta,1]$. Now the hypothesis \eqref{form26} implies \eqref{form-highscale}-\eqref{form-lowscale} with constant $\epsilon' = \sqrt{\eta}$, because $\rho^{2 - s}|\mathcal{F}| = (\rho/\delta)^{s}$ for $\rho = \Delta$. Finally, our choice $\eta = \epsilon^{2}/3600$ is equivalent to $60\epsilon' = \epsilon$. \end{proof}

The proof of Proposition \ref{pro-semi} relies on Lemma \ref{mainlem}, below, which generalises \cite[Lemma 4.8]{2023arXiv230808819R} to the setup of transversal families. We postpone the lengthy proof of Lemma \ref{mainlem} in the appendix. 

\begin{lemma}\label{mainlem}
	Let $\mathcal{F} \subset B_{C^{2}}(1)$ be a finite transversal family over $[-2,2]$ with constant $\mathfrak{T}\geq 1$. For every $\epsilon\in\frac{1}{\mathbb N}$, there exists $\Delta_0=\Delta_0(\epsilon,\mathfrak{T})>0$ such that the following holds for all $\delta,\Delta \in 2^{-\N} \cap (0,\Delta_{0}]$ with $\delta \leq \Delta$.
	
	Assume that $\mathcal{F}$ is $\{\Delta^{k\epsilon}\}_{j = 1}^{\epsilon^{-1}}$-uniform: for each $1 \leq k \leq \epsilon^{-1}$, there exists $N_{k} \in 2^{\N}$ such that 
	\begin{displaymath} |\mathbf{F} \cap \mathcal{F}| \in [N_{k},2N_{k}], \qquad \mathbf{F} \in \mathcal{D}_{\Delta^{k\epsilon}}(\mathcal{F}). \end{displaymath}
	Assume further that $\dist(\mathbf{F}_1, \mathbf{F}_2) \geq \Delta$ for distinct $\mathbf{F}_1,\mathbf{F}_2 \in \mathcal{D}_{\Delta}(\mathcal{F})$. For $a \geq 2$, $b \geq 1$, and $ab\geq \delta^{1-2\epsilon} |\mathcal{F}|$, assume that a family $\mathcal{P}_{a,b}\subset \mathcal{D}_\delta$ satisfies
	\begin{itemize}
		\item [\textup{(i)}]\phantomsection \label{i} $|6p^\rho \cap \mathcal{D}_\rho(\mathcal{F})| \leq \Delta^{-\epsilon} \cdot \rho |\mathcal{D}_\rho(\mathcal{F})|$ for all $p \in \mathcal{P}_{a,b}$ and $\rho \in \{ \Delta^\epsilon, \Delta^{2\epsilon}, \dots, \Delta\}$,
		\item [\textup{(ii)}]\phantomsection \label{ii} $N_{\Delta,b}(p) \geq a$ for all $p \in \mathcal{P}_{a,b}$.
	\end{itemize}
	Then
	\begin{equation}\label{form-221}
		|\mathcal{P}_{a,b}| \leq \delta^{-10\epsilon} \frac{|\mathcal{F}|^2}{a^3 b^2}.
	\end{equation}
\end{lemma}

Before proceeding to the proof of Proposition \ref{pro-semi}, we need one more auxiliary lemma, which is the translation of \cite[Lemma 4.11]{2023arXiv230808819R} to the setting of transversal families.

Since the proof is identical to the original, we omit it here, except for Remark \ref{rem3}.

\begin{lemma}\label{lem-auxi2} Let $s\in [0,1]$ and $\delta,\rho \in 2^{-\N}$ with $\delta \leq \rho$. Let $\mathcal{F} \subset B_{C^{2}}(1)$ be a $\delta$-separated transversal family over $[-2,2]$ with constant $\mathfrak{T}\geq1$. 
		
	Fix $\mathbf{F}\in\mathcal{D}_\rho(\mathcal{F})$, and assume that $\mathcal{F}_\mathbf{F}=\mathcal{F}\cap \mathbf{F}$ is a Katz-Tao $(\delta,s,K)$-set. For each $f\in\mathcal{F}_\mathbf{F}$, let $\mathcal{P}(f)\subset\{p\in \mathcal{D}_\delta: p\cap \Gamma_f\neq \emptyset\}$ be a $(\delta,s,K)$-set of dyadic $\delta$-cubes such that $|\mathcal{P}(f)|\sim M\in 2^{\N}$. 
	
	Then there exists a subset $\mathcal{F}'_\mathbf{F}\subset \mathcal{F}_\mathbf{F}$ with $|\mathcal{F}'_\mathbf{F}|\geq \tfrac{1}{2}|\mathcal{F}_\mathbf{F}|$, and for each $f\in\mathcal{F}'_\mathbf{F}$ there is a subset $\mathcal{P}'(f)\subset \mathcal{P}(f)$ with $|\mathcal{P}'(f)|\geq \tfrac{1}{2}|\mathcal{P}(f)|$ such that for each $p\in \mathcal{P}':=\bigcup_{f\in\mathcal{F}'_\mathbf{F}}\mathcal{P}'(f)$, we have
	\[|\{f\in \mathcal{F}'_\mathbf{F}: p\in\mathcal{P}'(f)\}|\lesssim_{\mathfrak{T}} K^2 \log\tfrac{1}{\delta}.\]
\end{lemma}

\begin{remark}\label{rem3} The only "geometric" property of transversal families needed in the proof of Lemma \ref{lem-auxi2} is the following one. If $f_{1},f_{2} \in \mathcal{F}$ are distinct with $d(f_{1},f_{2}) = w \geq \delta$, then $|\mathcal{P}(f_{1}) \cap \mathcal{P}(f_{2})| \lesssim_{\mathfrak{T}} KM(\delta/w)^{s}$. This follows from Lemma \ref{lemma4} taking $r := \delta$.  \end{remark} 

Lemma \ref{lem-auxi2} implies the case $t = s$ for Furstenberg set estimates. 
\begin{cor}\label{cor3}
	Let $s \in [0,1]$, $\mathfrak{T} \geq 1$, and $\epsilon > 0$. There exists $\delta_0=\delta_0(\mathfrak{T},\epsilon)>0$ such that the following holds for $\delta\in (0,\delta_0]$. Let $(\mathcal{F},\mathcal{P})$ be a $(\delta,s,\delta^{-\epsilon/3},M)$-nice configuration, where $\mathcal{F} \subset B_{C^{2}}(1)$ is a Katz-Tao $(\delta,s,\delta^{-\epsilon/3})$-set. Then $|\mathcal{P}| \geq \delta^{\epsilon}|\mathcal{F}|M$.
\end{cor}

\begin{proof}
	We apply Lemma \ref{lem-auxi2} with $\rho=1$. The conclusion is that we can find $\mathcal{F}'\subset \mathcal{F}$ with $|\mathcal{F}'|\geq \tfrac{1}{2}|\mathcal{F}|$, and for each $f\in\mathcal{F}'$ there is a subset $\mathcal{P}'(f)\subset \mathcal{P}(f)$ with $|\mathcal{P}'(f)|\gtrsim M$ such that for each $p\in \mathcal{P}':=\bigcup_{f\in\mathcal{F}'}\mathcal{P}'(f)$, we have
	\[|\{f\in \mathcal{F}': p\in\mathcal{P}'(f)\}|\lesssim_{\mathfrak{T}} \delta^{-2\epsilon/3} \log\tfrac{1}{\delta}.\]
	This implies
	\[\begin{split}
		|\mathcal{P}|\geq \Big|\bigcup_{f\in\mathcal{F}'} \mathcal{P}'(f)\Big|\gtrsim_{\mathfrak{T}} \frac{\sum_{f\in\mathcal{F}'} |\mathcal{P}'(f)|}{\delta^{-2\epsilon/3}\log\tfrac{1}{\delta}}\gtrsim \delta^{2\epsilon/3}(\log\tfrac{1}{\delta})^{-1} |\mathcal{F}|M.
	\end{split}\]
	If $\delta > 0$ is small enough, this yields $|\mathcal{P}| \geq \delta^{\epsilon}|\mathcal{F}|M$, as desired. \end{proof}

We are now ready to prove Proposition \ref{pro-semi}. The proof is virtually the same as the proof of \cite[Proposition 4.6]{2023arXiv230808819R}, but we include the details for the reader's convenience.
\begin{proof}[Proof of Proposition \ref{pro-semi}.] During the proof, we will always assume that $\delta > 0$ is small enough in terms of $\epsilon,\mathfrak{T}$ without mentioning this explicitly.

We start with three preliminary reductions. The first one is that $\mathbf{F} \mapsto |\mathcal{F} \cap \mathbf{F}|$ is roughly constant on $\mathcal{D}_{\Delta}(\mathcal{F})$; this can be achieved by pigeonholing, at the cost of reducing the cardinality of $\mathcal{F}$ by a factor of $\sim \log \tfrac{1}{\delta}$. 

The second reduction is that if $\mathcal{D}_{\Delta}(\mathcal{F})$ is $\Delta$-separated. This is possible to achieve because for each $\mathbf{F} \in \mathcal{D}_{\Delta}(\mathcal{F})$, there are $\lesssim_{\mathfrak{T}} 1$ "neighbouring" cubes $\mathbf{F}' \in \mathcal{D}_{\Delta}(\mathcal{F})$ with $\dist(\mathbf{F},\mathbf{F}') < \Delta$. Now we apply Brook's theorem \cite{MR396344} to find a subset $\mathcal{F}_{\Delta} \subset \mathcal{D}_{\Delta}(\mathcal{F})$ with $|\mathcal{F}_{\Delta}| \gtrsim_{\mathfrak{T}} |\mathcal{D}_{\Delta}(\mathcal{F})|$ whose elements are $\Delta$-separated. By the first reduction, replacing $\mathcal{F}$ by $\mathcal{F} \cap \mathcal{F}_{\Delta} := \{f \in \mathcal{F} : f \in \cup \mathcal{F}_{\Delta}\}$ reduces the cardinality of $\mathcal{F}$ by a factor $\sim_{\mathfrak{T}} 1$. 

The third reduction is that $\mathcal{F}$ is $\{\Delta^{j\epsilon}\}_{j = 1}^{\epsilon^{-1}}$-uniform, provided that $\Delta > 0$ is sufficiently small in terms of $\epsilon,\mathfrak{T}$. This follows from Lemma \ref{lem-uniformsubset} applied to the scale $\Delta$, with $m = 1/\epsilon$, and with $2^{-T} := \Delta^{\epsilon/10}$. Now, if $\Delta$ is so small that $T^{-1}\log(6T) \leq \epsilon^{2}$, the lemma produces a $\{\Delta^{j(\epsilon/10)}\}_{j = 1}^{10\epsilon^{-1}}$-uniform subset $\mathcal{F}_{\Delta} \subset \mathcal{D}_{\Delta}(\mathcal{F})$ with $|\mathcal{F}_{\Delta}| \geq \Delta^{\epsilon^{2}}|\mathcal{D}_{\Delta}(\mathcal{F})| \geq \delta^{\epsilon^{2}}|\mathcal{D}(\mathcal{F})|$. In particular, by the first reduction, replacing $\mathcal{F}$ by $\mathcal{F} \cap \mathcal{F}_{\Delta}$ only reduces the cardinality of $\mathcal{F}$ by a factor $\delta^{\epsilon^{2}}$. Consequently, \eqref{form-highscale}-\eqref{form-lowscale} remain valid with $\delta^{-2\epsilon^{2}}$ in place of $\delta^{-\epsilon^{2}}$.

	We then start the proof in earnest. We immediately dispose of two special cases. In the first special case $\Delta\leq100\delta$. Then it follows from \eqref{form-highscale} that $\mathcal{F}$ is a $(\delta, 2-s, \delta^{-2\epsilon^2})$-set. Therefore, Corollary \ref{cor-uppercaseF} gives $|\mathcal{P}|\geq \delta^{-1+\epsilon} M$.
	
	In the second special case $\Delta \geq \delta^{8\epsilon}$. Now it follows from \eqref{form-lowscale} that $\mathcal{F}$ is a Katz-Tao $(\delta, s, \delta^{-17\epsilon})$-set, and therefore Corollary \ref{cor3} implies $|\mathcal{P}| \geq \delta^{60\epsilon}|\mathcal{F}|M$.
	
	In the sequel, we assume that $100\delta < \Delta < \delta^{8\epsilon}$. Our goal is to apply Lemma \ref{mainlem} to a suitable refinement of the pair $(\mathcal{F}, \mathcal{P})$, satisfying the conditions of Lemma \ref{mainlem}. First, using Proposition~\ref{pro-pigeonhole}, we obtain a refinement $(\mathcal{F}_1, \mathcal{P}_1)$ of $(\mathcal{F}, \mathcal{P})$, along with a $(\Delta, s, C_\Delta, M_\Delta)$-nice covering configuration $(\mathcal{F}_1^\Delta, \mathcal{P}_1^\Delta)$, where $C_\Delta \approx_\Delta \delta^{-\epsilon^2}$, and such that $\mathcal{F}_{1}$ is $\{1,\Delta,\delta\}$-uniform, and $\mathcal{F}_{1}^{\Delta} = \mathcal{D}_{\Delta}(\mathcal{F}_{1})$. Then $\mathcal{F}_1^\Delta$ is a $(\Delta, 2-s, C_0)$-set with $C_0 \lessapprox_{\Delta, \mathfrak{T}} \delta^{-\epsilon^2}$, which follows from \eqref{form-highscale} plus the $\{1,\Delta,\delta\}$-uniformity of $\mathcal{F}_{1}$, see \cite[Lemma 2.17]{2023arXiv230110199O}.

	To proceed, we apply Proposition \ref{pro-nicecondition} to $(\mathcal{F}_1^\Delta, \mathcal{P}_1^\Delta)$ with $\Delta=(\Delta^{\epsilon/10})^n$ and $n=10\epsilon^{-1}$. We obtain a refinement $(\mathcal{F}_\Delta, \mathcal{P}_\Delta)$ of $(\mathcal{F}_1^\Delta, \mathcal{P}_1^\Delta)$ such that for all $w\in \{\Delta^\epsilon/10, \Delta^{2\epsilon/10}, \cdots, \Delta\}$,	\begin{equation}\label{eq-556a}
		|6\mathbf{p}^w\cap \mathcal{D}_w(\mathcal{F}_\Delta)|\lessapprox_{\epsilon,\mathfrak{T}} C_0^2\Delta^{-3\epsilon/10}\cdot w|\mathcal{D}_w(\mathcal{F}_\Delta)|, \qquad \mathbf{p} \in \mathcal{P}_{\Delta}.
	\end{equation}
	In particular, since $\Delta < \delta^{8\epsilon}$, and $C_{0}^{2} \lessapprox_{\Delta,\mathfrak{T}} \delta^{-\epsilon^{2}}$, we have $C_{0}^{2} \leq \Delta^{-\epsilon/3}$ for $\Delta > 0$ small enough, and \eqref{eq-556a} implies
	\begin{equation}\label{eq-556} |6\mathbf{p}^w\cap \mathcal{D}_w(\mathcal{F}_\Delta)| \leq \Delta^{-\epsilon}\cdot w|\mathcal{D}_w(\mathcal{F}_\Delta)|, \qquad \mathbf{p} \in \mathcal{P}_{\Delta}. \end{equation}
	Write $\mathcal{F}_2=\bigcup_{\mathbf{F}\in \mathcal{F}_\Delta}(\mathcal{F}_1\cap \mathbf{F})$ and $\mathcal{P}_2=\bigcup_{f\in \mathcal{F}_2} \mathcal{P}_2(f)$, where $\mathcal{P}_2(f):=\{p\in\mathcal{P}_1(f): p^\Delta\in \mathcal{P}_\Delta\}$. Then by Proposition \ref{pro-pigeonhole} (2) applied to the refinement $(\mathcal{F}_\Delta,\mathcal{P}_\Delta)$, we have 
	\begin{equation}\label{equ-23}
		|\{(f,p) \in \mathcal{F}_{2} \times \mathcal{P}_{2} : p \in \mathcal{P}_2(f)| \approx_\delta M|\mathcal{F}|.
	\end{equation}
	By definition, $\mathcal{D}_\Delta(\mathcal{P}_2(f))\subset \mathcal{P}_\Delta(\mathbf{F})$, where $f\in \mathbf{F}\in \mathcal{D}_\Delta(\mathcal{F}_2)$. Note also that $\mathcal{F}_2$ is $\{1,\Delta,\delta\}$-uniform since $\mathcal{F}_1$ is $\{1,\Delta,\delta\}$-uniform.
	
	Next, note that hypothesis \eqref{form-lowscale} implies that each $\mathcal{F}_{2} \cap \mathbf{F}$ with $\mathbf{F} \in \mathcal{D}_{\Delta}(\mathcal{F}_{2})$ is a Katz-Tao $(\delta,s,O_{\mathfrak{T}}(\delta^{-\epsilon^{2}}))$-set. Therefore, we can use Lemma \ref{lem-auxi2} to find $\mathcal{F}_{3,\mathbf{F}}\subset \mathcal{F}_2\cap \mathbf{F}$ with $|\mathcal{F}_{3,\mathbf{F}}|\geq \tfrac{1}{2}|\mathcal{F}_2\cap\mathbf{F}|$, and for each $f\in\mathcal{F}_{3,\mathbf{F}}$ a subset $\mathcal{P}_{3,\mathbf{F}}(f)\subset \mathcal{P}_2(f)$ with $|\mathcal{P}_{3,\mathbf{F}}(f)|\geq \tfrac{1}{2}|\mathcal{P}_2(f)|$ such that
	\begin{equation}\label{form23} |\{f\in \mathcal{F}_{3,\mathbf{F}}: p\in\mathcal{P}_{3,\mathbf{F}}(f)\}| \lessapprox_{\epsilon,\mathfrak{T}} \delta^{-2\epsilon^{2}}, \qquad p \in \mathcal{P}_{3,\mathbf{F}} := \bigcup_{f \in \mathcal{F}_{3,\mathbf{F}}} \mathcal{P}_{3,\mathbf{F}}(f). \end{equation}
	Define $\mathcal{F}_3=\bigcup_{\mathbf{F}\in \mathcal{D}_\Delta(\mathcal{F}_2)}\mathcal{F}_{3,\mathbf{F}}$ and $\mathcal{P}_3=\bigcup_{\mathbf{F}\in \mathcal{D}_\Delta(\mathcal{F}_2)}\mathcal{P}_{3,\mathbf{F}}$, and for $f \in \mathcal{F}_{3,\mathbf{F}}$, define $\mathcal{P}_{3}(f) := \mathcal{P}_{3,\mathbf{F}}(f)$. Thus $|\mathcal{F}_3|\sim|\mathcal{F}_2|$ and $|\mathcal{P}_3(f)|\sim|\mathcal{P}_2(f)|$ for $f \in \mathcal{F}_{3}$, so \eqref{equ-23} implies
	\begin{equation}\label{form24} |\{(f,p) \in \mathcal{F}_{3} \times \mathcal{P}_{3} : p \in \mathcal{P}_{3}(f)\}| \gtrapprox_{\delta} M|\mathcal{F}|. \end{equation} 
 Further, \eqref{form23} implies for $\delta > 0$ sufficiently small that
	\begin{equation}\label{eq-553}
		b(\mathbf{F},p) := |\{f\in \mathcal{F}_{3}\cap \mathbf{F}: p\in\mathcal{P}_{3}(f)\}|\leq \delta^{-\epsilon}, \qquad \mathbf{F} \in \mathcal{D}_{\Delta}(\mathcal{F}_{3}), \, p \in \mathcal{P}_{3}.
	\end{equation}
	At this point we might have lost the $\{\Delta^{j(\epsilon/10)}\}_{j = 1}^{10\epsilon^{-1}}$ uniformity of $\mathcal{F}_{3}$. But since $\mathbf{F} \mapsto |\mathcal{F}_{3} \cap \mathbf{F}|$ remains roughly constant on $\mathcal{D}_{\Delta}(\mathcal{F}_{3})$, this property can be reinstated exactly as in the third initial reduction. Of course $\mathcal{D}_{\Delta}(\mathcal{F}_{3}) \subset \mathcal{D}_{\Delta}(\mathcal{F})$ remains $\Delta$-separated. 
	
	Next, write
	\begin{displaymath} |\{f \in \mathcal{F}_{3} : p \in \mathcal{P}_{3}(f)\}| = \sum_{\mathbf{F} \in \mathcal{D}_{\Delta}(\mathcal{F}_{3})} b(\mathbf{F},p) \sim \sum_{b \in 2^{\N}} b \cdot |\{\mathbf{F} \in \mathcal{D}_{\Delta}(\mathcal{F}_{3}) : b(\mathbf{F},p) \sim b\}|, \quad p \in \mathcal{P}_{3}. \end{displaymath} 
	Since $b(\mathbf{F},p) \leq \delta^{-\epsilon}$ according to \eqref{eq-553}, the only non-zero terms in the series correspond to $b \leq \delta^{-\epsilon}$. Therefore, we may pigeonhole $b(p) \in 2^{\N} \cap [0,\delta^{-\epsilon}]$ such that 
	\begin{displaymath} |\{f \in \mathcal{F}_{3} : p \in \mathcal{P}_{3}(f)\}| \approx_{\delta} b(p) \cdot |\{\mathbf{F} \in \mathcal{D}_{\Delta}(\mathcal{F}_{3}) : b(\mathbf{F},p) \sim b(p)\}| =: b(p) \cdot a(p), \end{displaymath}
	for $p \in \mathcal{P}_{3}$. Next, we pigeonhole a subset $\mathcal{P}_{3}' \subset \mathcal{P}$, and fixed dyadic numbers $b \leq \delta^{-\epsilon}$ and $a \in \N$ such that $b(p) = b$ and $a(p) \sim a$ for all $p \in \mathcal{P}_{3}'$, and
	\begin{displaymath} \sum_{p \in \mathcal{P}_{3}} a(p)b(p) \approx_{\delta} ab|\mathcal{P}_{3}'| \leq \delta^{-\epsilon}a|\mathcal{P}_{3}'|. \end{displaymath}
	With this notation,
	\begin{equation}\label{form25} M|\mathcal{F}| \stackrel{\eqref{form24}}{\lessapprox_{\delta}} \sum_{p \in \mathcal{P}_{3}} |\{f \in \mathcal{F}_{3} : p \in \mathcal{P}_{3}(f)\}| \approx_{\delta} \sum_{p \in \mathcal{P}_{3}} a(p)b(p) \leq \delta^{-\epsilon}a|\mathcal{P}_{3}'|. \end{equation} 
 The plan is to estimate the cardinality of $\mathcal{P}_{3}'$ from above via Lemma \ref{mainlem}. We now verify the hypotheses of the lemma for the pair $(\mathcal{F}_3, \mathcal{P}_3')$. First, $\mathcal{F}_{3}$ needs to be $\{\Delta^{j\epsilon}\}_{j = 1}^{\epsilon^{-1}}$-uniform, and $\mathcal{D}_{\Delta}(\mathcal{F}_{3})$ needs to be $\Delta$-separated. These properties were discussed below \eqref{eq-553} (note that $\{\Delta^{j(\epsilon/10)}\}_{j = 1}^{10\epsilon^{-1}}$-uniformity implies $\{\Delta^{j\epsilon}\}_{j = 1}^{\epsilon^{-1}}$-uniformity).
 
 Hypothesis (i) of Lemma \ref{mainlem} asks us to check that $|6p^{w} \cap \mathcal{D}_{w}(\mathcal{F}_{3})| \leq \Delta^{-\epsilon} \cdot w|\mathcal{D}_{w}(\mathcal{F}_{3})|$ for all $p \in \mathcal{P}_{3}'$, and $w \in \{\Delta^{\epsilon},\Delta^{2\epsilon},\ldots,\Delta\}$. This follows from \eqref{eq-556}, since $\mathcal{D}_{\Delta}(\mathcal{P}_{3}') \subset \mathcal{P}_{\Delta}$, and $\mathcal{D}_{\Delta}(\mathcal{F}_{3}) = \mathcal{D}_{\Delta}(\mathcal{F}_{2}) = \mathcal{F}_{\Delta}$.
 
 Finally, hypothesis (ii) of Lemma \ref{mainlem} asks us to check that 
 \begin{displaymath} |\{\mathbf{F} \in \mathcal{D}_{\Delta}(\mathcal{F}_{3}) : |2p \cap \mathbf{F} \cap \mathcal{F}_{3}| \gtrsim b\}| \stackrel{\mathrm{def.}}{=} N_{\Delta,b}(p) \gtrsim a, \qquad p \in \mathcal{P}_{3}'. \end{displaymath}
By the definition of the numbers $a,b$, we know that $a \sim a(p)$ and $b = b(p)$ for all $p \in \mathcal{P}_{3}'$. Unwrapping the definitions further, 
 \begin{displaymath} |\{\mathbf{F} \in \mathcal{D}_{\Delta}(\mathcal{F}_{3}) : b(\mathbf{F},p) \sim b\}| \sim a, \qquad p \in \mathcal{P}_{3}', \end{displaymath}
 where $b(\mathbf{F},p) = |\{f \in \mathcal{F}_{3} \cap \mathbf{F} : p \in \mathcal{P}_{3}(f)\}|$. Now, it suffices to note that $p \in \mathcal{P}_{3}(f)$ implies $z_{p} \in \Gamma_{f}(2\delta)$, and therefore $|2p \cap \mathbf{F} \cap \mathcal{F}_{3}| \geq b(\mathbf{F},p)$.
 
 The final hypothesis of Lemma \ref{mainlem} is that $ab \geq \delta^{1 - 2\epsilon}|\mathcal{F}_{3}|$. If this fails, then in particular $a \leq \delta^{1 - 2\epsilon}|\mathcal{F}_{3}| \leq \delta^{1 - 2\epsilon}|\mathcal{F}|$, and
 \begin{displaymath} M|\mathcal{F}| \stackrel{\eqref{form25}}{\lessapprox_{\delta}} \delta^{-\epsilon}a|\mathcal{P}_{3}'| \leq \delta^{1 - 3\epsilon}|\mathcal{F}||\mathcal{P}|, \end{displaymath} 
 which can be rearranged to $|\mathcal{P}| \gtrapprox_{\delta} \delta^{-1 + 3\epsilon}M$, and the proof is complete.
 
 Assume next that $ab \geq \delta^{1 - 2\epsilon}|\mathcal{F}_{3}|$. Then \eqref{form-221} implies
 \begin{displaymath} M|\mathcal{F}| \stackrel{\eqref{form25}}{\lessapprox_{\delta}} \delta^{-\epsilon}a|\mathcal{P}_{3}'| \leq \delta^{-11\epsilon} \frac{|\mathcal{F}|^{2}}{a^{2}}, \end{displaymath} 
 therefore $a \lessapprox_{\delta} \delta^{-10\epsilon}|\mathcal{F}|^{1/2}/M^{1/2}$. By one more application of \eqref{form25},
 \begin{displaymath} |\mathcal{P}| \geq |\mathcal{P}_{3}'| \gtrapprox_{\delta} \delta^{\epsilon}\frac{M|\mathcal{F}|}{a} \gtrapprox_{\delta} \delta^{11\epsilon}M^{3/2}|\mathcal{F}|^{1/2}. \end{displaymath} 
This completes the proof.
\end{proof}

	
\section{General case}\label{sec7}
In this section, we prove the intermediate case for Furstenberg set theorem.
\begin{thm}\label{thm-generalcase}
	Let $s \in (0,1)$ and $t\in (s,2-s)$. For every $\epsilon>0$, there exists $\eta=\eta( s,\mathfrak{T},\epsilon)>0$ and $\delta_0=\delta_0(s,t,\mathfrak{T},\epsilon)>0$ such that the following holds for all $\delta\in (0,\delta_0]$. Let $(\mathcal{F}, {\mathcal{P}})$ be a $(\delta, s, \delta^{-\eta}, M)$-nice configuration, where $\mathcal{F}$ is a $(\delta, t, \delta^{-\eta})$-set. Then, 
	\[\bigg|\bigcup_{f\in\mathcal{F}} {\mathcal{P}}(f)\bigg| \geq M \cdot \delta^{-\tfrac{s+t}{2}+\epsilon}.\]
\end{thm}
Theorem \ref{thm-generalcase} also holds at the endpoints $t \in \{s,2 - s\}$, but these cases have already been covered by Corollary \ref{cor-uppercaseF} and Corollary \ref{cor3}, so we omit them from the statement of Theorem \ref{thm-generalcase} (the application of Corollary \ref{cor3} to the case $t = s$ also requires the fact, see \cite[Lemma 2.7]{OS23}, that $(\delta,s)$-sets contains Katz-Tao $(\delta,s)$-sets of cardinality $\approx \delta^{-s}$).

Recall Definition \ref{def:Lipschitz}. The next lemma combines \cite[Lemmas 2.6 and 2.7]{2023arXiv230110199O}.
\begin{proposition}\label{pro:Lip}
Fix $0<s<t<u$ and $d,m\geq 1$. For every $0<\epsilon<\min\{u-s,\tfrac{1}{2}\}$, there is $\tau=\tau(d,\epsilon,s,t,u)>0$ such that the following holds: for any piecewise affine $d$-Lipschitz function $f: [0,m]\to \R$ with $f(0)=0$ such that
\[f(x)\geq tx-\epsilon^2 m, \quad x\in [0,m] \quad\text{and}\quad f(m)\leq (t+\epsilon^2)m,\]
there exists a sequence of non-overlapping intervals $\{[c_j,d_j]\}_{j=1}^n$ contained in $[0,m]$ such that
\begin{itemize}
\item [\textup{(i)}] for each $j$, at least one of the following alternatives holds:
\begin{itemize}
\item [\textup{(a)}] $(f,c_j,d_j)$ is $\epsilon$-linear with $s_f(c_j,d_j)\in [s,u]$,
\item [\textup{(b)}] $(f,c_j,d_j)$ is $\epsilon$-superlinear with $s_f(c_j,d_j)\in [s,u]$ and
\begin{equation}\label{equ-551}
f(x)\geq \min\{f(c_j)+u(x-c_j),f(d_j)-s(d_j-x)\}-\epsilon(d_j-c_j).
\end{equation}
\end{itemize}
\item [\textup{(ii)}] $d_j-c_j\geq \tau m$ for all $j$.
\item[\textup{(iii)}] $\left| [0,m] \, \setminus \, \bigcup_j [c_j,d_j] \right|\lesssim_{s,t,u} \epsilon m$. 
\end{itemize}
\end{proposition}

Applying Proposition \ref{pro:Lip} to the branching function of $\mathcal{F}$ yields the following description of $\mathcal{F}$ at different scales. This is explained in \cite[Remark 6.5]{2023arXiv230808819R}, but we give the details for completeness. Recall from Definition \ref{def:branching} the meaning of the branching function for a $\{\Delta^{j}\}_{j = 1}^{n}$-uniform $\mathfrak{T}$-transversal family $\mathcal{F} \subset C^{2}([-2,2])$. Recall also from Remark \ref{rmk-Lip} that such branching functions are $3$-Lipschitz, provided that $\Delta > 0$ is sufficiently small in terms of the transversality constant $\mathfrak{T}$. 
\begin{lemma}\label{lem:characterization}
For any $\epsilon>0$ and $\mathfrak{T}\geq1$, there exists $\Delta_0=\Delta_0(\mathfrak{T})>0$ such that the following holds for all $\Delta\in (0,\Delta_0]\cap 2^{-\N}$. 
	
Let $n\geq 1$ be an integer. Let $\mathcal{F}$ be a $\{\Delta^j\}_{j=1}^n$-uniform $\mathfrak{T}$-transversal family with branching function $\beta \colon [0,n] \to [0,\infty)$. Assume that the conditions in Proposition \ref{pro:Lip} are satisfied with $f=\beta$, $d= 3$, $u=2-s$ and $m=n$. Let $\{[c_j,d_j]\}_{j=1}^n$ be the intervals obtained from Proposition \ref{pro:Lip}. Fix one $[c_j,d_j]\subset [0,n]$, write 
\begin{displaymath} t_j=s_f(c_j,d_j), \qquad \rho_j=\Delta^{d_j-c_j}, \quad \text{and} \quad \mathcal{F}_j=T_\mathbf{F}(\mathcal{F}\cap \mathbf{F}) \end{displaymath}
for $\mathbf{F}\in \mathcal{D}_{\Delta^{c_j}}(\mathcal{F})$, where $T_\mathbf{F}(g):=(g-g_\mathbf{F})/\Delta^{c_j}$ for some $g_\mathbf{F}\in \mathbf{F}$. Then the following holds.
\begin{itemize}
\item [\textup{(i)}] If $[c_j,d_j]$ is type \textup{(a)} in Proposition \ref{pro:Lip}\textup{(i)}, then $\mathcal{F}_j$ is $(\rho_j,t_j,\Delta^{-(d_j-c_j)\epsilon-5})$-regular.
\item [\textup{(ii)}] If $[c_j,d_j]$ is type \textup{(b)} in Proposition \ref{pro:Lip}\textup{(i)}, then $|\mathcal{D}_{\rho_j}(\mathcal{F}_j)|\sim_\mathfrak{T}\rho_j^{-t_j}$ and 
\[|\mathcal{F}_j\cap B(g,r)|_{\rho_j}\leq \Delta^{-(d_j-c_j)\epsilon-5}\cdot \max\left\{r^{2-s}|\mathcal{F}_j|_{\rho_j}, \left(\tfrac{r}{\rho_j}\right)^s\right\}, \quad g \in \mathcal{F}_{j}, \, r \in [\rho_{j},1].\]
\end{itemize}
\end{lemma}
\begin{proof} To prove (i), note that $\mathcal{F}_{j} = T_{\mathbf{F}}(\mathcal{F} \cap \mathbf{F})$ is $\{\Delta^{j}\}_{j = 1}^{n - c_{j}}$-uniform with branching function $\beta_{j}(x) = \beta(x + c_{j}) - \beta(c_{j})$ (to be precise, $\mathcal{F}_{j}$ is uniform relative to the dyadic system obtained by mapping all the dyadic cubes associated with $\mathcal{F}$ under $T_{\mathbf{F}}$). Using the $(t_{j},\epsilon)$-linearity of $\beta$ on $[c_{j},d_{j}]$, is easy to check that $\beta_{j}$ is $(t_{j},\epsilon)$-linear on $[0,d_{j} - c_{j}]$. Therefore, Lemma \ref{lem:superlinear}(ii) implies that is $\mathcal{F}_{j}$ is $(\Delta^{d_{j} - c_{j}},t_{j},O_{\mathfrak{T}}(\Delta^{-4 - \epsilon(d_{j} - c_{j})}))$-regular. In particular $\mathcal{F}_{j}$ is $(\Delta^{d_{j} - c_{j}},t_{j},\Delta^{-5 - \epsilon(d_{j} - c_{j})})$-regular if $\Delta > 0$ is sufficiently small in terms of $\mathfrak{T}$.

To prove (ii), we first deduce from the uniformity of $\mathcal{F}$ that, for $0\leq k\leq d_j-c_j$,
	\begin{equation}\label{equ-15}
			\big|\mathcal{D}_{\Delta^{d_{j} - c_{j}}}(\mathcal{F}_j\cap \mathbf{F}')\big| \sim_\mathfrak{T}  \Delta^{\beta(c_j+k)-\beta(d_j)}, \qquad \mathbf{F}' \in \mathcal{D}_{\Delta^k}(\mathcal{F}_j).
	\end{equation}
	This follows from \eqref{form41}, recalling that the branching function of $\mathcal{F}_{j}$ is $\beta_{j}(x) = \beta(x + c_{j}) - \beta(c_{j})$, thus $\beta_{j}(k) - \beta_{j}(d_{j} - c_{j}) = \beta(c_{j} + k) - \beta(d_{j})$. With $k = 0$, we obtain in particular
	\begin{equation}\label{equ-19}
		|\mathcal{D}_{\Delta^{d_{j} - c_{j}}}(\mathcal{F}_j)|\sim_\mathfrak{T} \Delta^{f(c_j)-f(d_j)}.
	\end{equation}

Now, the hypothesis that $[c_{j},d_{j}]$ is type (b) in Proposition \ref{pro:Lip}\textup{(i)} means that $\beta$ is $\epsilon$-superlinear on $[c_{j},d_{j}]$, and
	\begin{equation*}
		\beta(x)\geq \min\{\beta(c_j)+(2-s)(x-c_j),\beta(d_j)-s(d_j-x)\}-\epsilon(d_j-c_j).
	\end{equation*}
	By using this inequality, \eqref{equ-15} and \eqref{equ-19}, for any $0\leq k\leq d_j-c_j$ and $\mathbf{F}' \in \mathcal{D}_{\Delta^k}(\mathcal{F}_j)$,
	\[
		\big|\mathcal{D}_{\Delta^{d_{j} - c_{j}}}(\mathcal{F}_j\cap \mathbf{F}')\big|\lesssim_\mathfrak{T}\Delta^{-\epsilon(d_j-c_j)}\cdot  \max\{\Delta^{k(2-s)} |\mathcal{D}_{\rho_j}(\mathcal{F}_j)|, \Delta^{-s(d_j-c_j-k)}\}. \]
		This is the "dyadic" version of the claim in part (ii). The non-dyadic version follows by comparing dyadic and ball covering numbers like in the proof of Lemma \ref{lem:superlinear}.		This completes the proof of Lemma \ref{lem:characterization}. \end{proof}

We are now ready to complete the proof of Theorem \ref{thm-generalcase}. We restate Proposition \ref{pro-multiscaledocom0} below for the reader's convenience. The reader should also recall Corollary \ref{cor1} for regular case and Corollary \ref{cor2} for semi-well spaced case.
\begin{proposition}\label{pro-multiscaledocomA}
Fix $N\geq 2$ and sequence of scales $\{\Delta_j\}_{j=0}^N\subset 2^{-\N}$ with
\[0<\delta=\Delta_N<\Delta_{N-1}<\cdots<\Delta_1<\Delta_0=1.\]
Let $(\mathcal{F}_0,\mathcal{P}_0)$ be a $(\delta,s,C,M)$-nice configuration. Then there exists $\mathcal{F}\subset \mathcal{F}_0$ such that:
\begin{itemize}
\item [\textup{(D1)}]\phantomsection \label{D1} $|\mathcal{D}_{\Delta_j}(\mathcal{F})|\approx_\delta |\mathcal{D}_{\Delta_j}(\mathcal{F}_0)|$ and $|\mathcal{F}\cap \mathbf{F}|\approx_\delta |\mathcal{F}_0\cap\mathbf{F}|$ for any $\mathbf{F}\in \mathcal{D}_{\Delta_j}(\mathcal{F})$ with $1\leq j \leq N$.
		
\item [\textup{(D2)}]\phantomsection \label{D2} For every $\mathbf{F}\in \mathcal{D}_{\Delta_j}(\mathcal{F})$ with $1\leq j \leq N-1$, there exist numbers $C_\mathbf{F}\approx_\delta C$ and $M_\mathbf{F}\geq1$, and a family of dyadic cubes $\mathcal{P}_\mathbf{F}\in \mathcal{D}_{\Delta_{j+1}/\Delta_j}$ such that $(T_\mathbf{F}(\mathcal{F}\cap \mathbf{F}), \mathcal{P}_\mathbf{F})$ is a $(\Delta_{j+1}/\Delta_j, s, C_\mathbf{F}, M_\mathbf{F})$ nice configuration.
\end{itemize}
Furthermore, the families $\mathcal{P}_\mathbf{F}$ can be chosen such that if $\mathbf{F}_j\in \mathcal{D}_{\Delta_j}(\mathcal{F})$ with $1\leq j \leq N-1$, then
\[\frac{|\mathcal{P}_0|}{M}\gtrapprox_\delta \prod_{j=0}^{N-1}\frac{|\mathcal{P}_{\mathbf{F}_j}|}{M_{\mathbf{F}_j}}.\]
\end{proposition}

\begin{proof}[Proof of Theorem \ref{thm-generalcase}]
Fix $\epsilon>0$. Let $\delta_0=\delta_0(s,t,\mathfrak{T},\epsilon) > 0$ be a constant to be determined later, and fix $\delta \in (0,\delta_{0}]$. Let $\eta_1(s,\tfrac{\epsilon}{2})$ be the constant determined in Corollary \ref{cor1} and let $\eta_2=\tfrac{(\epsilon/2)^2}{3600}$. Choose $\eta_0=\eta_0(s,\epsilon) <\tfrac{1}{2}\min\{\eta_1,\eta_2\}$. Then fix $0<\eta<\eta_0^2/4$, and let $(\mathcal{F},\mathcal{P})$ be a $(\delta,s,\delta^{-\eta},M)$-nice configuration, where $\mathcal{F}$ is a $(\delta,t,\delta^{-\eta})$-set. Replacing $\mathcal{F}$ by a suitable $(\delta,t,\delta^{-2\eta})$-subset without changing notation (see \cite[Lemma 2.7]{OS23}) we may and will assume that $|\mathcal{F}|\leq \delta^{-t}$. 
	
Let $T$ be an integer such that $\log(6T)/T<\eta^2$ and $T>\log(640\mathfrak{T}^4)$. Up to a refinement by using Lemma \ref{lem-uniformsubset}, we may further assume that $\mathcal{F}$ is $\{2^{-jT}\}_{j=1}^m$-uniform with $\delta=2^{-mT}$ and the associated sequence $\{N_j\}_{j=1}^m$. Let $\beta$ be the branching function of $\mathcal{F}$; since $|\mathcal{F}| \leq \delta^{-t}$, it holds $\beta(m) \leq tm$. Moreover, since $\mathcal{F}$ is a $(\delta,t,\delta^{-\eta})$-set, by Lemma \ref{lem:superlinear} 
\begin{displaymath} \beta(x) \geq tx - 2\eta m\geq tx-\eta_0^2m, \qquad x \in [0,m], \end{displaymath}
if $T=T(\mathfrak{T},\eta)$ is chosen sufficiently large. 
	
We now apply Proposition \ref{pro:Lip} with parameters $f=\beta$, $s, t$, $u=2-s > t$ and $\epsilon=\eta_0$. This gives $\tau=\tau(s,t,\eta_0) > 0$ and a sequence of intervals $\{[c_j,d_j]\}_{j=1}^n$. Let 
\[0<\delta=\Delta_{2n+1}<\Delta_{2n}<\cdots<\Delta_1<\Delta_0=1\]
be the sequence generated by $\{[c_j,d_j]\}_{j=1}^n$, that is, $\Delta_k=2^{-c_jT}$ or $\Delta_k=2^{-d_jT}$ when $k\in [1,2n]$. Then partition $\mathfrak{N}=\{0,1,2,\cdots,2n\}$ into 
\begin{displaymath} \mathcal{G} := \{k\in \mathfrak{N} :  \Delta_k=2^{-c_jT} \text{ for some } c_j\}, \end{displaymath}
and $\mathcal{B}=\mathfrak{N} \, \setminus \, \mathcal{G}$. Write $\rho_j=\Delta_{j+1}/\Delta_j$ for $j\in\mathfrak{N}$. Then we have the following properties.  
\begin{itemize}
\item[(E1)]\phantomsection\label{E1} By Proposition \ref{pro:Lip} (ii)-(iii), $\rho_j^{-1}\geq \delta^{-\tau}$ for each $j\in \mathcal{G}$ and $\prod_{j\in \mathcal{B}} \rho_j^{-1}\leq \delta^{-O_{s,t}(\eta_0)}$.
		
\item[(E2)]\phantomsection\label{E2} By definition of $\beta$ and recalling $t_j=s_f(c_j,d_j) \in [s,u] = [s,2 - s]$, 
\[\prod_{j\in \mathcal{G}} \rho_j^{-t_j}\geq |\mathcal{F}|\cdot \prod_{j\in \mathcal{B}} \rho_j^3\geq \delta^{-t+\eta} \cdot \delta^{O_{s,t}(\eta_0)}=\delta^{-t+O_{s,t}(\eta_0)}.\]
		
\item[(E3)]\phantomsection\label{E3} By Lemma \ref{lem:characterization}, for each $j\in \mathcal{G}$ and $\mathbf{F}\in \mathcal{D}_{\Delta_j}(\mathcal{F})$, either
\begin{itemize}
\item [(1)] $\mathcal{F}_j :=T_\mathbf{F}(\mathcal{F}\cap \mathbf{F})$ is $(\rho_j,t_j,\rho_j^{-\eta_0}2^{5T})$-regular set, or
\item [(2)] $|\mathcal{D}_{\rho_j}(\mathcal{F}_j)|\sim_\mathfrak{T}\rho_j^{-t_j}$ and for any $g \in \mathcal{F}_j$ and $r\in [\rho_j,1]$,
\[\begin{split}|\mathcal{F}_j\cap B(g,r)|_{\rho_j}\leq 2^{5T}\rho_j^{-\eta_0}\cdot \max\Big\{r^{2-s}|\mathcal{F}_j|_{\rho_j}, \Big(\frac{r}{\rho_j}\Big)^s\Big\}.\end{split}\]
\end{itemize}
\end{itemize}
To proceed, we apply Proposition \ref{pro-multiscaledocomA} to $(\mathcal{F},\mathcal{P})$ with scale sequence $\{\Delta_j\}_{j=0}^{2n+1}$. Then for each $0\leq j \leq 2n$ and $\mathbf{F}_j\in\mathcal{D}_{\Delta_j}(\mathcal{F})$ we can find $\mathcal{P}_{\mathbf{F}_j}\in \mathcal{D}_{\rho_j}$ such that $(T_{\mathbf{F}_j}(\mathcal{F}\cap \mathbf{F}_j), \mathcal{P}_{\mathbf{F}_j})$ is a $(\rho_j, s, C_{\mathbf{F}_j}, M_{\mathbf{F}_j})$ nice configuration and
\[\frac{|\mathcal{P}|}{M}\gtrapprox_\delta \prod_{j=0}^{2n}\frac{|\mathcal{P}_{\mathbf{F}_j}|}{M_{\mathbf{F}_j}},\]
where $C_{\mathbf{F}_j}\lessapprox_\delta \delta^{-\eta}$ and $M_{\mathbf{F}_j}\geq1$. For each $j\in \mathcal{G}$, we are either in case (1) or in case (2) of \nref{E3}. Our plan is to Corollary \ref{cor1} for case (1) and apply Corollary \ref{cor2} for case (2). Before that, we need to choose the parameters correctly. Recall $\eta_0<\tfrac{1}{2}\min\{\eta_1,\eta_2\}$. By choosing $m$ large enough (thus $\delta_0=\delta_0(s,t,\mathfrak{T},\epsilon)$ small enough) such that $10/m\tau<\tfrac{1}{2}\min\{\eta_,\eta_2\}$, then we have
\[\eta_0+\frac{5}{m\tau}<\eta_1\quad \text{and} \quad \eta_0+\frac{5}{m\tau}<\eta_2, \]
which implies (by \nref{E1} $2^{m\tau T}=\delta^{-\tau}\leq \rho_j^{-1}$)
\[\rho_j^{-\eta_0}2^{5T}<\rho_j^{-\eta_1}, \quad 2^{5T}\rho_j^{-\eta_0}<\rho_j^{-\eta_2}\quad \text{and} \quad C_{\mathbf{F}_j}\leq \rho_j^{-\min\{\eta_1,\eta_2\}}.\]
Consequently, for each $j\in\mathcal{G}$, we can apply Corollary \ref{cor1} for case (1) and apply Corollary \ref{cor2} for case (2), thus in both cases we get $|\mathcal{P}_{\mathbf{F}_j}|/M_{\mathbf{F}_j} \geq \rho_j^{-(s+t_j)/2+\epsilon/2}$. Consequently, we deduce using \nref{E1} and \nref{E2} that
\[\begin{split}
\frac{|\mathcal{P}|}{M}&\gtrapprox _\delta \prod_{j\in\mathcal{G}} \rho_j^{-\tfrac{s+t_j}{2}+\tfrac{\epsilon}{2}}=\prod_{j\in\mathcal{G}} \rho_j^{-\tfrac{s}{2}+\tfrac{\epsilon}{2}}\cdot \prod_{j\in\mathcal{G}} \rho_j^{-\tfrac{t_j}{2}}\\ & \geq \delta^{(-1+O_{s,t}(\eta_0))(\tfrac{s}{2}-\tfrac{\epsilon}{2})}\cdot \delta^{-\tfrac{t}{2}+O_{s,t}(\eta_0)}\geq \delta^{-\tfrac{s+t}{2}+\tfrac{\epsilon}{2}+O_{s,t}(\eta_0)}.
\end{split}\]
Finally, we choose $\delta<\delta_0$ small enough and $O_{s,t}(\eta_0)<\epsilon/5$, which leads to the desired estimate
\[\frac{|\mathcal{P}|}{M}\geq \delta^{-\tfrac{s+t}{2}+\epsilon}.\]
This completes the proof of Theorem \ref{thm-generalcase}. \end{proof}


	\section{Proof of Theorem \ref{t:smoothing}}\label{s8}
	
	Here we give the main steps in the proof of Theorem \ref{t:smoothing}, restated below.

	\begin{thm}\label{t:smoothing2} Let $g \in C^{3}(\R)$ be a function whose second derivative $g''$ never vanishes, and let $\Gamma_{g}^{1} := \{(x,g(x)) : x \in [-1,1]\}$. For every $0 \leq s \leq 1$ and $t \in [0,\min\{3s,s + 1\})$, there exists $p = p(g,s,t) \geq 1$ such that the following holds. 

Let $\sigma$ be a Borel measure on $\Gamma_{g}^{1}$ satisfying $\sigma(B(x,r)) \leq r^{s}$ for all $x \in \R^{2}$ and $r > 0$. Then,
\begin{displaymath} \|\hat{\sigma}\|_{L^{p}(B(R))}^{p} \leq C_{s,t}R^{2 - t}, \qquad R \geq 1. \end{displaymath} 
\end{thm}

The proof of Theorem \ref{t:smoothing2} is the same as the proof of \cite[Theorem 1.1]{Orponen2024Jan}, modulo replacing \cite[Lemma 3.3]{Orponen2024Jan} by Corollary \ref{cor:convex}, and removing an appeal to parabolic rescaling. We only explain the main steps -- and in particular why parabolic rescaling is no longer needed, see Remark \ref{rem4}. The first step is to deduce the following Proposition \ref{prop32} from Corollary \ref{cor:convex}. The case $g(x) = x^{2}$ of this proposition is \cite[Proposition 4.3]{Orponen2024Jan}.

\begin{proposition}\label{prop32} Let $g \in C^{3}(\R)$ be a function whose second derivative $g''$ never vanishes. For all $s \in (0,1]$, $t \in [0,2]$, $R > 0$, and $\kappa > 0$, there exist $\epsilon = \epsilon(g,\kappa,s,t) > 0$ and $\delta_{0} = \delta_{0}(\epsilon,\kappa,g,s,t,R) > 0$ such that the following holds for all $\delta \in (0,\delta_{0}]$ and $k \geq 1$. Assume that $\mu,\sigma$ are Borel probability measures with 
\begin{displaymath} \spt \mu \subset B(R) \subset \R^{2} \quad \text{and} \quad \spt \sigma \subset \Gamma_{g}^{1}. \end{displaymath}
Assume additionally that $I^{\delta}_{t}(\mu) \leq \delta^{-\epsilon}$ and $I^{\delta}_{s}(\sigma) \leq \delta^{-\epsilon}$. Write $\Pi := \mu \ast \sigma$, and assume that $E \subset \R^{2}$ is a Borel set with $\Pi^{k}(E) \geq \delta^{\epsilon}$. Then,
\begin{displaymath} |E|_{\delta} \geq \delta^{-\gamma(s,t) + \kappa}, \qquad \gamma(s,t) := \min\left\{s + t,\tfrac{3s + t}{2},s + 1\right\}. \end{displaymath} \end{proposition}

\begin{remark} The only difference between Proposition \ref{prop32} and \cite[Proposition 4.3]{Orponen2024Jan} is that Proposition \ref{prop32} allows $\spt \mu \subset B(R)$, whereas \cite[Proposition 4.3]{Orponen2024Jan} assumed $\spt \mu \subset B(1)$. This is why we wanted to allow the cubes $\mathcal{Q}$ in Corollary \ref{cor:convex} to lie in $[-R,R]^{2}$; now Corollary \ref{cor:convex} implies Proposition \ref{prop32}, following the proof of \cite[Proposition 4.3]{Orponen2024Jan}.  \end{remark} 
	
The next step in the proof of Theorem \ref{t:smoothing} is to deduce the following proposition from Proposition \ref{prop32}. The case $g(x) = x^{2}$ is \cite[Proposition 4.7]{Orponen2024Jan}:

\begin{proposition}\label{prop4} Let $g \in C^{3}(\R)$ be a function whose second derivative $g''$ never vanishes. For all $s \in (0,1]$, $t \in [0,2]$, $\kappa \in (0,1]$, and $R > 0$, there exist $\epsilon = \epsilon(g,\kappa,s,t) > 0$, $k_{0} = k_0(g,\kappa,s,t) \in \N$, and $\delta_{0} = \delta_{0}(g,\kappa,s,t,R) > 0$ such that the following holds for all $\delta \in (0,\delta_{0}]$. Assume that $\mu,\sigma$ are Borel probability measures with 
\begin{displaymath} \spt \mu \subset B(R) \quad \text{and} \quad \spt \sigma \subset \Gamma_{g}^{1}. \end{displaymath}
Assume additionally that $I^{\delta}_{t}(\mu) \leq \delta^{-\epsilon}$ and $I^{\delta}_{s}(\sigma) \leq \delta^{-\epsilon}$. Write $\Pi := \mu \ast \sigma$. Then,
\begin{displaymath} I_{\gamma(s,t)}^{\delta}(\Pi^{k}) \leq \delta^{-\kappa}, \qquad k \geq k_{0}, \end{displaymath} 
where $\gamma(s,t)$ is the constant defined in Proposition \ref{prop32}.
\end{proposition}

The notation $I_{\rho}^{\delta}(\mu)$ refers to the $\rho$-dimensional Riesz energy of a $\delta$-mollified measure $\mu \ast \psi_{\delta}$, see \cite[Section 4]{Orponen2024Jan}. However, the specifics of this notation are not important here.

\begin{remark}\label{rem4} The hypothesis $\spt \mu \subset B(R)$ is the same as in \cite[Proposition 4.7]{Orponen2024Jan}. It is crucial that $\epsilon > 0$ does not depend on $R$. In \cite{Orponen2024Jan}, one had to reduce the proof of \cite[Proposition 4.7]{Orponen2024Jan} to the case $\spt \mu \subset B(1)$, because  \cite[Proposition 4.3]{Orponen2024Jan} was only stated under the hypothesis $\spt \mu \subset B(1)$. This reduction was accomplished via parabolic rescaling. 

The parabolic rescaling step would fail under the generality of Proposition \ref{prop4}. More precisely: a na\"ive attempt to apply a rescaling argument would result in $\epsilon > 0$ depending on $R$. However, our Proposition \ref{prop32} has already been stated under the hypothesis $\spt \mu \subset B(R)$, so the rescaling step from \cite[Proposition 4.7]{Orponen2024Jan} is unnecessary. \end{remark}

With Proposition \ref{prop4} in hand, the remainder of the proof of Theorem \ref{t:smoothing} is exactly the same as the proof of  \cite[Theorem 1.1]{Orponen2024Jan}, and we omit repeating more details.


\appendix
\section{Proof of Lemma \ref{mainlem}}\label{appA}

The proof of Lemma \ref{mainlem} is an adaptation of \cite[Lemma 4.8]{2023arXiv230808819R} to the setting of transversal families, but we include all the details. Recall from Definition \ref{def:notation1} that 
\begin{equation}\label{form10} N_{\Delta,b}(p):=|\{\mathbf{F}\in \mathcal{D}_\Delta(\mathcal{F}): |3p\cap \mathbf{F}\cap \mathcal{F}|\geq b\}|, \qquad p \in \mathcal{D}_{\delta}, \, b \in \N, \end{equation}
where $\lambda p\cap \mathbf{F}\cap \mathcal{F}:=\{f\in \mathbf{F}\cap \mathcal{F}: z_p\in\Gamma_f(\lambda \delta)\}$ for $\lambda \in \N$. Clearly $N_{\Delta,b_{1}}(p)  \geq N_{\Delta,b_{2}}(p)$ for $b_{1} \leq b_{2}$. From Definition \ref{def:notation1} we also recall that, for $\rho \leq \Delta$, $p \in \mathcal{D}_{\rho}$, $\lambda \in \N$,
\begin{displaymath} \lambda p \cap \mathcal{D}_{\Delta}(\mathcal{F}) := \{\mathbf{F} \in \mathcal{D}_{\Delta}(\mathcal{F}) : \lambda p \cap \mathbf{F} \cap \mathcal{F} \neq \emptyset\}. \end{displaymath}
We record a simple lemma which connects the pieces of notation above:
\begin{lemma}\label{lemma5} Let $\delta \leq \Delta$ and $p \in \mathcal{D}_{\delta}$. Then, $N_{\Delta,1}(p) \leq |6p^\Delta \cap \mathcal{D}_{\Delta}(\mathcal{F})|$. \end{lemma} 

\begin{proof} Let $\mathbf{F} \in \mathcal{D}_{\Delta}(\mathcal{F})$ be such that $|3p \cap \mathbf{F} \cap \mathcal{F}| \geq 1$. It suffices to show that $5p^{\Delta} \cap \mathbf{F} \cap \mathcal{F} \neq \emptyset$. From the definition of $|3p \cap \mathbf{F} \cap \mathcal{F}| \geq 1$, there exists some $f \in \mathbf{F} \cap \mathcal{F}$ such that $(x,y) := z_{p} \in \Gamma_{f}(3\delta)$, i.e. $|f(x) - y| \leq 3\delta$. Then, writing $p^{\Delta} = (x^{\Delta},y^{\Delta})$, noting that $\max\{|x^{\Delta} - x|,|y^{\Delta} - y|\} \leq \Delta$, and using $\mathcal{F} \subset B_{C^{2}}(1)$ (in particular $f$ is $1$-Lipschitz):
	\begin{displaymath} |y^{\Delta} - f(x^{\Delta})| \leq |y^{\Delta} - y| + |f(x^{\Delta}) - f(x)| + |f(x) - y| \leq 2\Delta + 3\delta \leq 5\Delta. \end{displaymath} 
	This means that $z_{p^{\Delta}} \in \Gamma_{f}(5\Delta)$, therefore $f \in 5p^{\Delta} \cap \mathbf{F} \cap \mathcal{F} \subset 6p^{\Delta} \cap \mathbf{F} \cap \mathcal{F}$.   \end{proof}


We now restate Lemma \ref{mainlem}.

\begin{lemma}\label{lem-main1}
	Let $\mathcal{F} \subset B_{C^{2}}(1)$ be a finite transversal family (not necessarily $\delta$-separated) over $[-2,2]$ with constant $\mathfrak{T}\geq 1$. For every $\epsilon\in\frac{1}{\mathbb N}$, there exists $\Delta_0=\Delta_0(\epsilon,\mathfrak{T})>0$ such that the following holds for all $\delta,\Delta \in 2^{-\N} \cap (0,\Delta_{0}]$ with $\delta \leq \Delta$.
	
	Assume that $\mathcal{F}$ is $\{\Delta^{k\epsilon}\}_{j = 1}^{\epsilon^{-1}}$-uniform: for each $1 \leq k \leq \epsilon^{-1}$, there exists $N_{k} \in 2^{\N}$ such that 
	\begin{displaymath} |\mathbf{F} \cap \mathcal{F}| \in [N_{k},2N_{k}], \qquad \mathbf{F} \in \mathcal{D}_{\Delta^{k\epsilon}}(\mathcal{F}). \end{displaymath}
	Assume further that $\dist(\mathbf{F}_1, \mathbf{F}_2) \geq \Delta$ for distinct $\mathbf{F}_1,\mathbf{F}_2 \in \mathcal{D}_{\Delta}(\mathcal{F})$. For $a \geq 2$, $b \geq 1$, and $ab\geq \delta^{1-2\epsilon} |\mathcal{F}|$, assume that a family $\mathcal{P}_{a,b}\subset \mathcal{D}_\delta$ satisfies
	\begin{itemize}
		\item [\textup{(i)}]\phantomsection \label{i} $|6p^\rho \cap \mathcal{D}_\rho(\mathcal{F})| \leq \Delta^{-\epsilon} \cdot \rho |\mathcal{D}_\rho(\mathcal{F})|$ for all $p \in \mathcal{P}_{a,b}$ and $\rho \in \{ \Delta^\epsilon, \Delta^{2\epsilon}, \dots, \Delta\}$,
		\item [\textup{(ii)}]\phantomsection \label{ii} $N_{\Delta,b}(p) \geq a$ for all $p \in \mathcal{P}_{a,b}$.
	\end{itemize}
	Then
	\begin{equation}\label{form-221a}
		|\mathcal{P}_{a,b}| \leq \delta^{-10\epsilon} \frac{|\mathcal{F}|^2}{a^3 b^2}.
	\end{equation}
\end{lemma}

We first consider a special case when $\mathcal{F}$ satisfies a $1$-dimensional spacing condition.
\begin{lemma}\label{lem-auxi1}
	Let $\delta,\Delta \in 2^{-\mathbb N}$ with $\delta \leq \Delta$, $b\in \N$, and $\mathfrak{T} \geq 1$. Let $\mathcal{F} \subset B_{C^{2}}(\mathfrak{T})$ be a finite transversal family (not necessarily $\delta$-separated) with constant $\mathfrak{T}$. Assume moreover that $\mathcal{F}$ satisfies the spacing condition: $|\mathcal{F} \cap B(f,r)| \leq Kr|\mathcal{F}|$ for any $f\in\mathcal{F}$ and $r\in [\Delta,1]$.
	
	Let $\mathcal{F}_{\Delta}$ be a family of subsets of $\mathcal{F}$ such that every distinct pair $\mathbf{F}_{1},\mathbf{F}_{2} \in \mathcal{F}_{\Delta}$ is $\Delta$-separated. For $p \in \mathcal{D}_{\delta}$, write $\bar{N}_{\Delta,b}(p) := |\{\mathbf{F} \in \mathcal{F}_{\Delta} : |100p \cap \mathbf{F} \cap \mathcal{F}| \geq b\}|$. Then
	\begin{equation}\label{form-222}
		\sum_{p \in \mathcal{D}_\delta: N_{\Delta,b}(p) \geq 2} \bar{N}_{\Delta,b}(p)^2 \lesssim \mathfrak{T}^{2}K \log \tfrac{\mathfrak{T}}{\Delta} \cdot \frac{|\mathcal{F}|^2}{b^2},
	\end{equation}
	In particular, if $\mathcal{P}_{a,b} \subset \mathcal{D}_\delta$ is a family satisfying $\bar{N}_{\Delta,b}(p) \geq a\geq2$ for all $p \in \mathcal{P}_{a,b}$, then
	\begin{equation}\label{form-223}
		|\mathcal{P}_{a,b}| \lesssim \mathfrak{T}^{2}K \log \tfrac{\mathfrak{T}}{\Delta} \cdot \frac{|\mathcal{F}|^2}{(ab)^2}.
	\end{equation}
\end{lemma}
\begin{proof}
	Let $J$ be the number of triples $(f_1, f_2, p) \in \mathcal{F} \times \mathcal{F} \times\mathcal{D}_\delta$ such that $f_1, f_2$ belong to distinct elements $\mathbf{F}_{1},\mathbf{F}_{2} \in \mathcal{F}_\Delta$ and $z_p\in \Gamma_{f_1}(100\delta) \cap \Gamma_{f_{2}}(100\delta)$. We have
	\begin{align*} J = \mathop{\sum_{\mathbf{F}_{1},\mathbf{F}_{2} \in \mathcal{F}_{\Delta}}}_{\mathbf{F}_{1} \neq \mathbf{F}_{2}} & |\{(f_{1},f_{2},p) \in (\mathbf{F}_{1} \cap \mathcal{F}) \times (\mathbf{F}_{2} \cap \mathcal{F}) \times \mathcal{D}_{\delta} : z_{p} \in \Gamma_{f_{1}}(100\delta) \cap \Gamma_{f_{2}}(100\delta)\}|\\
		& = \sum_{p \in \mathcal{D}_{\delta}} \mathop{\sum_{\mathbf{F}_{1},\mathbf{F}_{2} \in \mathcal{F}_{\Delta}}}_{\mathbf{F}_{1} \neq \mathbf{F}_{2}} |100p \cap \mathbf{F}_{1} \cap \mathcal{F}| \cdot |100p \cap \mathbf{F}_{2} \cap \mathcal{F}|\\
		& \geq \mathop{\sum_{p \in \mathcal{D}_{\delta}}}_{N_{\Delta,b}(p) \geq 2} \mathop{\sum_{\mathbf{F}_{1} \neq \mathbf{F}_{2}}}_{|100p \cap \mathbf{F}_{i} \cap \mathcal{F}| \geq b} b^{2} = \mathop{\sum_{p \in \mathcal{D}_{\delta}}}_{N_{\Delta,b}(p) \geq 2} \bar{N}_{\Delta,b}(p)(\bar{N}_{\Delta,b} - 1) \cdot b^{2}\\
		& \geq \tfrac{1}{2} b^{2} \mathop{\sum_{p \in \mathcal{D}_{\delta}}}_{\bar{N}_{\Delta,b}(p) \geq 2} \bar{N}_{\Delta,b}(p)^{2}. \end{align*} 
	Next, we aim to prove an upper bound for $J$. To this end, fix $\mathbf{F}_{1}\neq \mathbf{F}_{2} \in \mathcal{F}_{\Delta}$, and $f_{j} \in \mathbf{F}_{j}$. Thus $d(f_{1},f_{2}) \geq \Delta$. Lemma \ref{lemma4} implies that $x$-projection of $\Gamma_{f_{1}}(100\delta) \cap \Gamma_{f_{2}}(100\delta)$ is contained in the union of $\lesssim \mathfrak{T}$ intervals of length $\lesssim \mathfrak{T}\delta/d(f,g)$. This yields
	\begin{displaymath} |\{p \in \mathcal{D}_{\delta} : z_{p} \in \Gamma_{f_{1}}(100\delta) \cap \Gamma_{f_{2}}(100\delta)\}| \lesssim \mathfrak{T}^{2}/d(f,g). \end{displaymath}
	Using also the spacing condition of $\mathcal{F}$, and recalling the hypothesis $\mathcal{F} \subset B_{C^{2}}(\mathfrak{T})$, we infer the following upper bound for $J$:
	\begin{align*} J & \leq \sum_{f_{1} \in \mathcal{F}} \sum_{\rho \in [\Delta,2\mathfrak{T}] \cap 2^{-\N}} \mathop{\sum_{f_{2} \in \mathcal{F}}}_{d(f_{1},f_{2}) \sim \rho} |\{p \in \mathcal{D}_{\delta} : z_{p} \in \Gamma_{f_{1}}(100\delta) \cap \Gamma_{f_{2}}(100\delta)\}|\\
		& \lesssim \sum_{f_1 \in \mathcal{F}} \sum_{\rho \in [\Delta, 2\mathfrak{T}]\cap 2^{-\N}} K \rho|\mathcal{F}| \cdot \frac{\mathfrak{T}^{2}}{\rho} \sim \mathfrak{T}^{2}K \log \tfrac{\mathfrak{T}}{\Delta} \cdot |\mathcal{F}|^2. \end{align*}
	Combining the lower and upper bounds for $J$ completes the proof. \end{proof}

We are then equipped to prove Lemma \ref{lem-main1}.

\begin{proof}[Proof of Lemma \ref{lem-main1}.] We write $\mathcal{P}:=\mathcal{P}_{a,b}$ and assume $\mathcal{P} \neq \emptyset$. Fix $\epsilon \in \tfrac{1}{\N}$, and let $\Delta_0 = \Delta_0(\epsilon,\mathfrak{T}) >0$ be a small constant to be determined. The proof is organised in five steps.
	
	\textbf{Step 1. A weak estimate.} We first verify that $\mathcal{F}$ satisfies the assumptions of Lemma \ref{lem-auxi1}. For any $p \in \mathcal{P}$ and $\rho = \Delta^{k\epsilon}$, $1 \leq k \leq \epsilon^{-1}$, we have $1 \leq N_{\rho,b}(p) \leq \Delta^{-\epsilon} \cdot \rho|\mathcal{D}_\rho(\mathcal{F})|$ by condition \nref{i} and Lemma \ref{lemma5}. Thus $|\mathcal{D}_\rho(\mathcal{F})| \geq \Delta^{\epsilon} \cdot \rho^{-1}$. Since each $\mathbf{F} \in \mathcal{D}_\rho(\mathcal{F})$ contains roughly the same number of functions in $\mathcal{F}$, we obtain  
	\[|\mathcal{F} \cap \mathbf{F}| \sim \frac{|\mathcal{F}|}{|\mathcal{D}_\rho(\mathcal{F})|} \leq \Delta^{-\epsilon} \cdot \rho |\mathcal{F}|.\]
	For general $r\in [\Delta,1]$ and $\mathbf{F} \in \mathcal{D}_{r}(\mathcal{F})$, the above implies $|\mathcal{F} \cap \mathbf{F}| \lesssim \Delta^{-2\epsilon} \cdot r|\mathcal{F}|$. Thus, the assumptions of Lemma \ref{lem-auxi1} are satisfied with $K \sim \Delta^{-2\epsilon}$, and the conclusion is  
	\begin{equation}\label{form-224}
		|\mathcal{P}| \lesssim_{\epsilon,\mathfrak{T}} \Delta^{-3\epsilon} \cdot \frac{|\mathcal{F}|^2}{(ab)^2}.
	\end{equation}
	This bound will mainly be used in \textbf{Step 4}, but here we also use it to handle a special case. If $\Delta \geq \Delta_0$, then $a \leq \Delta^{-1} \leq \Delta_0^{-1}$. Thus, \eqref{form-224} implies \eqref{form-221a} by absorbing $\Delta_0^{-1}$ into $\lesssim$.  
	
	\textbf{Step 2. Two base cases.} Fix $\Delta < \Delta_0$. The first base case is $\delta \geq \Delta^{1+\epsilon}$. In this case we claim that $\mathcal{P} = \emptyset$. Assume to the contrary that there exists $p \in \mathcal{P}$. Then there is some $\mathbf{F} \in \mathcal{D}_\Delta(\mathcal{F})$ such that $|\mathcal{F} \cap \mathbf{F}| \geq |3p \cap \mathcal{F}\cap \mathbf{F}|\geq b$, so by the uniformity of $\mathcal{F}$ at scale $\Delta$, we get $|\mathcal{F}| \gtrsim b |\mathcal{D}_\Delta(\mathcal{F})|$. Then, by condition \nref{ii}, and the hypothesis $ab \geq \delta^{1 - 2\epsilon}|\mathcal{F}|$,
	\[N_{\Delta,b}(p) \geq a \geq \delta^{1-2\epsilon} b^{-1} |\mathcal{F}| \gtrsim \delta^{1-2\epsilon} |\mathcal{D}_\Delta(\mathcal{F})|.\]
	On the other hand, by condition \nref{i} and Lemma \ref{lemma5},
	\[N_{\Delta,b}(p) \leq |6p^\Delta \cap \mathcal{D}_\Delta(\mathcal{F})| \leq \Delta^{1 - \epsilon}|\mathcal{D}_\Delta(\mathcal{F})|.\]
	From the two inequalities above, we get $\delta^{1-2\epsilon}\lesssim\Delta^{1-\epsilon}$. Chaining this with the case hypothesis $\delta \geq \Delta^{1+\epsilon}$ leads to $\delta^{1 - 2\epsilon} \lesssim \delta^{(1 - \epsilon)/(1 + \epsilon)}$, or equivalently $\delta^{1 - \epsilon - 2\epsilon^{2}} \lesssim \delta^{1 - \epsilon}$. This gives a contradiction for $\delta > 0$ small enough, depending only on $\epsilon$. Thus $\mathcal{P} = \emptyset$. 
	
	The second base case is $a \leq C_{\mathfrak{T}}\Delta^{-3\epsilon}$, where
	\begin{displaymath} C_{\mathfrak{T}} := 2\Delta^{2\epsilon}|\mathcal{D}_{\Delta^{\epsilon}}(\mathcal{F})| \lesssim_{\mathfrak{T}} 1, \end{displaymath}  
	by the upper $(2,20\mathfrak{T}^{2})$-regularity of $\mathcal{F}$. In this case \eqref{form-224} already implies \eqref{form-221a}, provided $\Delta_{0} > 0$ (hence $\delta$) is sufficiently small, depending on $\epsilon,\mathfrak{T}$. 
	
	For the remainder of the proof, we may and will assume that
	\begin{equation}\label{form-331}
		\delta < \Delta^{1+\epsilon}\quad \text{and} \quad a>C_{\mathfrak{T}}\Delta^{-3\epsilon}.
	\end{equation}
	We will prove \eqref{form-221a} by induction on $\delta$. The induction hypothesis is that Lemma \ref{lem-main1} holds for all $\delta\geq\Delta2^{-n+1}$ for some $n \geq 1$. We consider $\delta = \Delta 2^{-n}$. 
	
	\textbf{Step 3. Finding an intermediate scale.} Let $k_{0}$ be the largest integer $k$ such that
	\begin{equation}\label{form-225}
		a > 2\Delta^{-2\epsilon} \cdot \Delta^{k\epsilon} |\mathcal{D}_{\Delta^{k\epsilon}}(\mathcal{F})|, 
	\end{equation}
	and define
	\begin{equation}\label{form16} w := \Delta^{k_{0}\epsilon}. \end{equation}
	Note that by the definition of the constant $C_{\mathfrak{T}}$ in the previous section,
	\begin{displaymath} 2\Delta^{-2\epsilon} \cdot \Delta^\epsilon |\mathcal{D}_{\Delta^\epsilon}(\mathcal{F})| = C_{\mathfrak{T}}\Delta^{-3\epsilon} \stackrel{\eqref{form-331}}{<} a. \end{displaymath}
	This shows that the set of integers in question is non-empty, and $k_{0} \geq1$. Also $k_{0} <\epsilon^{-1}$, since otherwise we would have $a>2\Delta^{1-2\epsilon} |\mathcal{D}_\Delta(\mathcal{F})|$, but by Lemma \ref{lemma5} and assumption \nref{i} we also have $a\leq N_{\Delta,b}(p)\leq N_{\Delta,1}(p)\leq \Delta^{1-\epsilon}|\mathcal{D}_\Delta(\mathcal{F})|$, a contradiction. Thus, $k_{0} \in \N \cap [1,\epsilon^{-1})$. Now by the definition of $k_{0}$, for any $\rho = \Delta^{k\epsilon}$ with $k>k_{0}$, we have $a\leq 2\Delta^{-2\epsilon} \cdot \rho |\mathcal{D}_\rho(\mathcal{F})|$. From this, we deduce the following two conditions.
	\begin{itemize}
		\item [(C1)] \phantomsection \label{C1} By the upper $(2,20\mathfrak{T}^{2})$-regularity of $\mathcal{F}$, recall Lemma \ref{lem-Ahlforsregularity},
		\[
		a\leq 2\Delta^{-2\epsilon} \cdot (w \Delta^\epsilon) |\mathcal{D}_{w\Delta^\epsilon}(\mathcal{F})|\lesssim_{\mathfrak{T}} \Delta^{-3\epsilon} \cdot w |\mathcal{D}_w(\mathcal{F})|.
		\]
		In particular $a < \Delta^{-4\epsilon} \cdot w|\mathcal{D}_w(\mathcal{F})|$, provided that $\Delta > 0$ is sufficiently small in terms of $\epsilon,\mathfrak{T}$. This combined with the lower bound \eqref{form-225} yields
		\begin{equation}\label{form-226}
			2\Delta^{-2\epsilon} \cdot w |\mathcal{D}_w(\mathcal{F})| <a < \Delta^{-4\epsilon} \cdot w |\mathcal{D}_w(\mathcal{F})|. 
		\end{equation}	
		\item[(C2)] \phantomsection \label{C2} For any $\mathbf{F} \in \mathcal{D}_w(\mathcal{F})$, the rescaled version of $\mathcal{F} \cap \mathbf{F}$ is a $(\Delta/w,1)$-set:
		\[|T_\mathbf{F} (\mathcal{F}\cap \mathbf{F})\cap B(f,r)|\leq \Delta^{-2\epsilon}r|T_\mathbf{F} (\mathcal{F}\cap \mathbf{F})|,\quad f\in T_\mathbf{F} (\mathcal{F}), \quad r\in [\Delta/w,1].\]
		Here $T_{\mathbf{F}} \colon C^{2}([-2,2]) \to C^{2}([-2,2])$ is the $C^{2}$-homothetic map defined by $T_{\mathbf{F}}(g) := (g - f_{\mathbf{F}})/w$, where $f_{\mathbf{F}} \in \mathbf{F}$ is some "centre" of $\mathbf{F}$ satisfying $\mathbf{F} \subset B(f_{\mathbf{F}},2\mathfrak{T}w)$. 
	\end{itemize}
	To verify \nref{C2}, for $\rho = \Delta^{k\epsilon}$, $k>k_{0}$, by \eqref{form-226} we have $w |\mathcal{D}_w(\mathcal{F})| \leq \frac{a}{2\Delta^{-2\epsilon}} < \rho |\mathcal{D}_\rho(\mathcal{F})|$. Thus, for any $\bar{\mathbf{F}} \in \mathcal{D}_\rho(\mathcal{F})$ and $\mathbf{F} \in \mathcal{D}_w(\mathcal{F})$, we deduce by the above inequality
	\begin{equation}\label{form-339}
		|\mathcal{F}\cap\bar{\mathbf{F}}| \sim \frac{|\mathcal{F}|}{|\mathcal{D}_\rho(\mathcal{F})|} \leq \frac{|\mathcal{F}|}{|\mathcal{D}_w(\mathcal{F})|} \cdot \frac{\rho}{w} \sim |\mathcal{F}\cap \mathbf{F}| \cdot \frac{\rho}{w}.
	\end{equation}
	Let $r\in [\Delta/w,1]$ and $r\in [\Delta^{(k+1)\epsilon},\Delta^{k\epsilon})$ for some $0\leq k\leq \epsilon^{-1}-2$ (note $\tfrac{\Delta}{w}\in [\Delta^{1-\epsilon},1)$). Then, for $f \in T_{\mathbf{F}}(\mathcal{F})$, 
	\[|T_\mathbf{F} (\mathcal{F}\cap \mathbf{F})\cap B(f,r)| =|\mathcal{F}\cap \mathbf{F}\cap B(wf+f_\mathbf{F},wr)|\leq |\mathcal{F}\cap B(wf+f_\mathbf{F},w\Delta^{k\epsilon})|.\]
	Since $B(w f+f_\mathbf{F},w\Delta^{k\epsilon})$ can be covered by $\lesssim 1$ dyadic $w\Delta^{k\epsilon}$-cubes $\bar{\mathbf{F}}$, we infer from \eqref{form-339} with $\rho := w\Delta^{k\epsilon}$ that
	\[\begin{split}
		|T_\mathbf{F} (\mathcal{F}\cap \mathbf{F})\cap B(f,r)| &\lesssim \max_{\bar{\mathbf{F}} \in \mathcal{D}_{w\Delta^{k\epsilon}}(\mathcal{F})} |\mathcal{F}\cap \bar{\mathbf{F}}|\stackrel{\eqref{form-339}}{\lesssim} |\mathcal{F}\cap \mathbf{F}| \cdot \Delta^{k\epsilon}\\
		&\leq \Delta^{-\epsilon}r|T_\mathbf{F} (\mathcal{F}\cap \mathbf{F})|.
	\end{split}\]
	
	In the sequel, we will use properties (C1) and (C2) for the fixed scale $w=\Delta^{k_{0}\epsilon}$.
	
	\textbf{Step 4. High-low incidence estimates.} Abbreviate $\tau:=\delta/w$. For $Q\in \mathcal{D}_{\tau}(\mathcal{P})$, let $I_{Q} \subset \R$ be the $x$-projection of $Q$, and 
	\[\Gamma(Q):=\{\Gamma_{f|_{I_{Q}}}: Q\cap \Gamma_f\neq\emptyset\}\]
	be the \emph{curve segments intersecting $Q$}. Two curve segments $\Gamma_{f|_{I_{Q}}},\Gamma_{g|_{I_{Q}}}$ are \emph{comparable} if
	\begin{displaymath} |f(x) - g(x)| \leq 3\delta, \qquad x \in I_{Q}. \end{displaymath}
	Let $\mathbb{S}_Q$ be a maximal set of \emph{incomparable} curve segments in $\Gamma(Q)$ and write $\mathbb{S}:=\cup_Q \mathbb{S}_Q$. We have two remarks about incomparable segments. First, each $u\in \mathbb{S}$ belongs to at most two different $\mathbb{S}_Q$ since we assumed $\mathcal{F}\subset B_{C^2}(1)$, so the functions $f \in \mathcal{F}$ are $1$-Lipschitz, and therefore any graph segment $\Gamma_{f|_{I_{Q}}}$ can intersect at most two (vertically aligned) $\tau$-cubes.
	
	Second, for each $Q\in\mathcal{D}_\tau(\mathcal{P})$, we claim that each element of $\Gamma(Q)$ is contained in the vertical $3\delta$-neighborhood of one and at most $O(\mathfrak{T}^{4})$ segments in $\mathbb{S}_Q$. The containment in at least one vertical $3\delta$-neighbourhood follows from the maximality of $\mathbb{S}_{Q}$, so it remains to prove the $O(\mathfrak{T}^{4})$-part of the claim, that is,
	\begin{equation}\label{form2} |\{f \in \mathbb{S}_{Q} : \Gamma_{g|_{I_{Q}}} \subset \Gamma_{f|_{I_{Q}}}(3\delta)\}| \lesssim \mathfrak{T}^{4}, \qquad g \in \mathcal{F}. \end{equation}
	Here, and often below, we will abuse notation slightly by writing $f \in \mathbb{S}_{Q}$ when we actually mean that $f \in \mathcal{F}$ is such that 
	\begin{displaymath} u_{f,Q} := \Gamma_{f|_{I_{Q}}} \in \mathbb{S}_{Q}. \end{displaymath}
	If $Q$ is clear from context, we may also just write $u_{f} := u_{f,Q}$. We also write $u_{f}(r) :=  \Gamma_{f|_{I_{Q}}}(r)$ for $r > 0$. 
	We start with the following:
	\begin{claim}\label{claim1} For $\lambda \in \N$ and $g\in\mathcal{F}$, the set $\{f \in\mathcal{F}: u_g\subset u_f(\lambda \delta)\}$ is contained in a $C^{2}$-ball of radius $3\lambda \mathfrak{T} \delta/\tau=3\lambda \mathfrak{T}w$ around $g$.
	\end{claim}
	\begin{proof}
		Let $u_g\subset u_f(\lambda \delta)$. Assume to reach a contradiction that $d(f,g)> 3\lambda \mathfrak{T}w$. Let $x< y$ be the endpoints of $I_{Q}$; recall that $y - x = \tau$. By the triangle inequality and the mean value theorem, there exists $\xi \in I_{Q}$ such that 
		\begin{equation}\label{form1} \lambda \delta \geq |f(x)-g(x)|\geq |f'(\xi)-g'(\xi)|\cdot \tau -|f(y)-g(y)|. \end{equation}
		Since $u_g\subset u_f(\lambda \delta)$, and recalling \eqref{form-331}, we have 
		\begin{displaymath} |f(\xi) -g(\xi)|\leq \lambda \delta \leq \lambda \Delta^{1 + \epsilon} \leq \lambda \Delta \leq \lambda w < \frac{d(f,g)}{3\mathfrak{T}}. \end{displaymath}
		By the transversality of $\mathcal{F}$,
		\[|f'(\xi)-g'(\xi)|\geq \frac{2d(f,g)}{3\mathfrak{T}}> 2\lambda w.\]
		Plugging this into \eqref{form1} yields a contradiction: $\lambda\delta \geq |f(x)-g(x)| >  2\lambda w \tau -\lambda\delta =\lambda \delta$. \end{proof}
	
	We continue the proof of \eqref{form2}. By Claim \ref{claim1} with $\lambda := 3$, and the upper $(2,20\mathfrak{T}^{2})$-regularity of $\mathcal{F}$, 
	\[|\{f\in\mathcal{F}: u_g \subset u_f(3\delta)\}|_w\leq 20\mathfrak{T}^2 \Big(\frac{9\mathfrak{T}w}{w}\Big)^2 \sim \mathfrak{T}^4.\]
	To deduce \eqref{form2} from this, it suffices to show that if $B = B(f_{0},w)$ is a fixed $w$-ball in $C^{2}([-2,2])$, then $|\{f \in B \cap \mathbb{S}_{Q} : u_{g} \subset u_{f}(3\delta)\}| \leq 6$. To see this, we record a simple claim which will nonetheless be used multiple times:
	\begin{claim}\label{eq-verticalbound} If $B = B(f_{0},w) \subset C^{2}([-2,2])$ and $f,h \in B$ are such that $u_{f},u_{h} \in \mathbb{S}_{Q}$ are distinct, then
		\begin{displaymath} d_{V}(u_{f},u_{h}) := \min_{x\in I_{Q}}|f(x)-h(x)|>\delta. \end{displaymath}
		We will refer to the quantity $d_{V}(u_{f},u_{g})$ as the \emph{vertical distance} of the segments $u_{f},u_{h}$.
	\end{claim}
	
	\begin{proof} Note that $\|f' - h'\|_{L^{\infty}} \leq 2w$. If $|f(x_{0}) - h(x_{0})| \leq \delta$ for some $x_{0} \in I_{Q}$, then $|f(x) - h(x)| \leq 3\delta$ for $x \in I_{Q}$, contradicting the hypothesis that $f,h \in \mathbb{S}_{Q}$ are distinct. \end{proof} 
	
	The estimate $|\{f \in B \cap \mathbb{S}_{Q} : u_{g} \subset u_{f}(3\delta)\}| \leq 6$ is an immediate consequence of Lemma \ref{eq-verticalbound} and the pigeonhole principle. We have therefore proven \eqref{form2}.
	
	To proceed, for each $u\in\mathbb{S}_Q$ with $Q\in \mathcal{D}_{\tau}(\mathcal{P})$, let $N_{\Delta,b}(u)$ be the number of $\mathbf{F}\in\mathcal{D}_\Delta(\mathcal{F})$ such that $|u(10\delta)\cap \mathbf{F}\cap\mathcal{F}|\geq b$, where 
	\[u(10\delta)\cap \mathbf{F}\cap\mathcal{F}:=\{f\in \mathbf{F}\cap\mathcal{F}: u_{f} \subset u(10\delta)\}.\]
	The notation is similar to $N_{\Delta,b}(p)$ from the statement of Lemma \ref{lem-main1}, but now the input is a segment "$u$" rather than a square "$p$". We define
	\[\overline{\mathcal{I}}_{\omega}(\mathbb{S},\mathcal{P}):=\sum_{p\in\mathcal{P}}\sum_{u\in\mathbb S} N_{\Delta,b}(u) \mathbf{1}_{\{z_p\in u(6\delta)\}}\]
	This definition is a variant of the weighted incidences in Definition \ref{def:weightedIncidence}, the biggest difference being that $\mathbb{S}$ is now a family of curve fragments, instead of a transversal family of functions. We use the letter "$\omega$" in place of "$w$", since $w$ stands for the intermediate scale familiar from \eqref{form16}. We have the following lower bound for $\overline{\mathcal{I}}(\mathbb{S},\mathcal{P})$:
	
	\begin{claim}\label{claim2} $\overline{\mathcal{I}}_{\omega}(\mathbb{S},\mathcal{P})\geq a |\mathcal{P}|$.
	\end{claim}
	\begin{proof} 
		Fix $p \in \mathcal{P}$, and $Q\in\mathcal{D}_\tau(\mathcal{P})$ containing $p$. It suffices to show that 
		\begin{equation}\label{claim-lowerbound}
			\sum_{u\in\mathbb S} N_{\Delta,b}(u) \mathbf{1}_{\{z_p\in u(6\delta)\}} \geq a.
		\end{equation} 
		To prove \eqref{claim-lowerbound}, we recall that $N_{\Delta,b}(u) = |\{\mathbf{F} \in \mathcal{D}_{\Delta}(\mathcal{F}) : |u(10\delta) \cap \mathbf{F} \cap \mathcal{F}| \geq b\}|$, and write
		\begin{displaymath} \sum_{\substack{u\in\mathbb S\\z_p\in u(6\delta)}}N_{\Delta,b}(u) \geq \mathop{\sum_{\mathbf{F} \in \mathcal{D}_{\Delta}(\mathcal{F})}}_{|3p \cap \mathbf{F} \cap \mathcal{F}| \geq b} |\{u \in \mathbb{S} : z_{p} \in u(6\delta) \text{ and } |u(10\delta) \cap \mathbf{F} \cap \mathcal{F}| \geq b\}|. \end{displaymath}
		We then claim that for each $\mathbf{F}\in\mathcal{D}_\Delta(\mathcal{F})$ with $|3p\cap \mathbf{F}\cap \mathcal{F}|\geq b$, there exists at least one $u \in \mathbb{S}$ such that $z_p \in u(6\delta) $, and $|u(10\delta) \cap \mathbf{F} \cap \mathcal{F}| \geq b$. This will imply \eqref{claim-lowerbound}, since the number of $\mathbf{F}$, as above, equals $N_{\Delta,b}(p)\geq a$ by assumption \nref{ii}. 
		
		Fix $\mathbf{F}\in\mathcal{D}_\Delta(\mathcal{F})$ with $|3p\cap \mathbf{F}\cap \mathcal{F}|\geq b$. Let $g \in 3p \cap \mathbf{F} \cap \mathcal{F}$ be arbitrary. By the maximality of $\mathbb{S}_{Q}$, there exists $u \in \mathbb{S}_{Q}$ such that $u_{g,Q} \subset u(3\delta)$. In particular $z_{p} \in u_{g,Q}(3\delta) \subset u(6\delta)$. We now claim that $u_{h,Q} \subset u(10\delta)$ for every $h \in 3p \cap \mathbf{F} \cap \mathcal{F}$. Since $|3p \cap \mathbf{F} \cap \mathcal{F}| \geq b$, this will complete the proof.
		
		Let $h\in 3p\cap \mathbf{F}\cap\mathcal{F}$, and let $z_p=(x_p,y_p)$ be the centre of $p$. Then
		\[|h(x_p)-g(x_p)|\leq |y_p-h(x_p)|+|y_p-g(x_p)|\leq 3\delta+3\delta= 6\delta.\]
		Since $g,h\in \mathbf{F}\in\mathcal{D}_\Delta(\mathcal{F})$ and $\Delta^{1 - \epsilon} \leq w$, we have 
		\[|g(x)-h(x)|\leq \sqrt{2}\mathfrak{T}\Delta\cdot \tau+6\delta<7\delta, \quad x\in I_Q,\]
		provided that $\Delta > 0$ is sufficiently small in terms of $\epsilon,\mathfrak{T}$. This implies $u_{h,Q} \subset u_{g,Q}(7\delta) \subset u(10\delta)$, as claimed. \end{proof}
	
	We next claim that for any $p\in \mathcal{P}$ there holds
	\begin{equation}\label{eq-cl2}
		|\{u\in\mathbb S:z_p\in u(6\delta)\}| \leq 960\mathfrak{T}^4|6p^w \cap \mathcal{D}_w(\mathcal{F})|.
	\end{equation}
	It is clear that $|\mathcal{D}_w(\{f\in\mathcal{F}:z_p\in u_f(6\delta)\})|\leq |6p^w \cap \mathcal{D}_w(\mathcal{F})|$. Now \eqref{eq-cl2} will follow once we manage to show that for any $\mathbf{F}\in \mathcal{D}_w(\{f\in\mathcal{F}:z_p\in u_f(6\delta)\})$, 
	\begin{equation}\label{eqq-f4}
		|\{f\in\mathbf{F} \cap \mathcal{F} :z_p\in u_f(6\delta), \, u_{f} \in \mathbb{S}\}|\leq 960\mathfrak{T}^4.
	\end{equation}
	By Lemma \ref{lem-Ahlforsregularity}, $\mathbf{F}$ is contained in a ball of radius $\leq 2\mathfrak{T}w$. By the upper $(2,20\mathfrak{T}^{2})$-regularity of $\mathcal{F}$, $\mathbf{F}$ can be covered by $\leq 80\mathfrak{T}^4$ many $w$-balls. It follows from Claim \ref{eq-verticalbound} that in each $w$-ball there are at most $12$ functions $f \in \mathcal{F}$ such that $u_{f} \in \mathbb{S}$, and $z_p\in u_{f}(6\delta)$. Hence \eqref{eqq-f4} holds
	
	Now, using \eqref{eq-cl2}, assumption \nref{i} and \eqref{form-225}, and assuming $\Delta > 0$ to be so small that $960\mathfrak{T}^{4} \leq \Delta^{-\epsilon}$, we find
	\begin{equation}\label{eq-claim3}
		\begin{split}\sum_{p\in\mathcal{P}}|\{u\in\mathbb S:z_p\in u(6\delta)\}|&\leq 960\mathfrak{T}^4 \sum_{p\in\mathcal{P}}|6p^w \cap \mathcal{D}_w(\mathcal{F})| \\
			&\leq |\mathcal{P}|\Delta^{-2\epsilon} \cdot w |\mathcal{D}_w(\mathcal{F})|\stackrel{\eqref{form-225}}{\leq} \tfrac{a}{2}|\mathcal{P}|. \end{split}
	\end{equation}
	Combining Claim \ref{claim2} and inequality \eqref{eq-claim3}, we infer that the subset $\mathbb{S}':=\{u\in\mathbb{S}:N_{\Delta,b}(u)\geq 2\}$ satisfies 
	\[\overline{\mathcal{I}}_{\omega}(\mathbb{S}',\mathcal{P})=\sum_{p\in\mathcal{P}}\sum_{u\in\mathbb{S}'} N_{\Delta,b}(u)\mathbf{1}_{z_p\in u(6\delta)}\geq \tfrac{1}{2} \overline{\mathcal{I}}_{\omega}(\mathbb{S},\mathcal{P}).\]
	For simplicity, we still denote $\mathbb{S}'$ by $\mathbb S$ in the remaining of this proof and assume $N_{\Delta,b}(u)\geq 2$ for any $u\in\mathbb{S}$.

	Next, we decompose the incidences $\overline{\mathcal{I}}_{w}(\mathbb{S},\mathcal{P})$ as a sum over the squares $Q \in \mathcal{D}_{\tau}(\mathcal{P})$:
	\[\overline{\mathcal{I}}_{\omega}(\mathbb{S},\mathcal{P}) = \sum_{Q \in \mathcal{D}_{\tau}(\mathcal{P})} \sum_{p \in \mathcal{P} \cap Q} \sum_{u \in \mathbb{S}_{Q}} N_{\Delta,b}(u)\mathbf{1}_{z_{p} \in u(6\delta)} =: \sum_{Q \in \mathcal{D}_{\tau}(\mathcal{P})} \overline{\mathcal{I}}_{\omega}(\mathbb{S}_Q,\mathcal{P} \cap Q),\]
	We will next show that, for $Q \in \mathcal{D}_{\tau}(\mathcal{P})$ fixed, the quantity $\overline{\mathcal{I}}_{\omega}(\mathbb{S}_{Q},\mathcal{P} \cap Q)$ is bounded from above by a weighted incidence count $\mathcal{I}_{\omega}(\mathcal{F}_{Q},\mathcal{P}_{Q})$ in the sense of Definition \ref{def:weightedIncidence}. Then, we will be in a place to apply Proposition \ref{pro-highlow-curve} to $\mathcal{I}_{\omega}(\mathcal{F}_{Q},\mathcal{P}_{Q})$. 
	
	Fix $Q \in \mathcal{D}_{\tau}(\mathcal{P})$ with lower left corner $(x_{0},y_{0})$, and let $T_Q(x,y) := (x - x_{0},y - y_{0})/\tau$ be the rescaling map taking $Q$ to $[0,1)^2$. Then 
	\begin{displaymath} \mathcal{P}_{Q} := \{T_Q(p):p\in \mathcal{P} \cap Q\} \subset \mathcal{D}_{w}([0,1)^{2}). \end{displaymath}
	If $u\in \mathbb{S}_Q$ for some $Q\in \mathcal{D}_\tau$, we let $f_{u} \in \mathcal{F}$ be the function satisfying 
	\begin{displaymath} u=\Gamma_{f_{u}|_{I_Q}}. \end{displaymath}
	(The uniqueness of $f_{u}$ follows from the transversality hypothesis.) Recall from Lemma \ref{lem2} the notation  $f_{(x_{0},y_{0}),\tau}(x)=(f(\tau x+x_0)-y_0)/\tau$. Lemma \ref{lem2} stated that the family 
	\begin{displaymath} \mathcal{F}_{Q} := \mathcal{F}_{(x_0,y_0),\tau}:=\{(f_u)_{(x_0,y_0),\tau}: u\in \mathbb{S}_Q\} \end{displaymath}
	is transversal with constant $4\mathfrak{T} + 1$ on the interval $[-2,2]$ (note that $[-2,2]\subset ([-2,2]-x_0)/\tau$ since $x_0\in [0,1]$ and $\tau<1/2$). With this notation, 
	\begin{equation}\label{form5} \overline{\mathcal{I}}_{\omega}(\mathbb{S}_{Q},\mathcal{P} \cap Q) = \sum_{p \in \mathcal{P}_{Q}} \sum_{f \in \mathcal{F}_{Q}} N_{\Delta,b}(u_{f,Q})\mathbf{1}_{z_{p} \in \Gamma_{f}(6w)} = \mathcal{I}_{\omega}(\mathcal{F}_{Q},\mathcal{P}_{Q}), \end{equation}
	where the first equation follows from the observation that $z_{p} \in u_{f}(6\delta)$ if and only if $T_{Q}(z_{p}) \in \Gamma_{f_{(x_{0},y_{0}),\tau}}(6w)$. The right hand side in \eqref{form5} refers to the weighted incidences from Definition \ref{def:weightedIncidence} with the weight functions $w_{\mathcal{F}_Q}(f) := N_{\Delta,b}(u_{f,Q})$ and $w_{\mathcal{P}_{Q}} \equiv 1$.  To apply  Proposition \ref{pro-highlow-curve} to the right hand side, we need to check that $\mathcal{F}_{Q}$ is $\sim_{\mathfrak{T}} w$-separated:
	\begin{claim}\label{claim3}  $\mathcal{F}_{Q}$ is $(w/2\mathfrak{T})$-separated, provided that $\Delta > 0$ is sufficiently small in terms of $\epsilon,\mathfrak{T}$.
	\end{claim}
	\begin{proof} For $u, v\in\mathbb{S}_Q$, recall that
		\[d((f_u)_{(x_0,y_0),\tau},(f_v)_{(x_0,y_0),\tau}) \stackrel{\mathrm{def.}}{=}\|(f_u)_{(x_0,y_0),\tau}-(f_v)_{(x_0,y_0),\tau}\|_{C^2([-2,2])}.\] 
		Consider the first the case where $d(f_u,f_v)\leq w$. By Claim \ref{eq-verticalbound}, $|f_u(x_{0}) - f_v(x_{0})| >\delta$, so
		\[\begin{split}
			d((f_u)_{(x_0,y_0),\tau},(f_v)_{(x_0,y_0),\tau})&\geq \|(f_u)_{(x_0,y_0),\tau}-(f_v)_{(x_0,y_0),\tau}\|_{L^\infty([-2,2])}\\
			&\geq \frac{|f_u(x_0)-f_v(x_0)|}{\tau} >\frac{\delta}{\tau}=w.
		\end{split}\]
		Consider next the case $d(f_u,f_v)> w$. If $|f_u(x_0)-f_v(x_0)|>\delta$, we can argue as the previous case. Otherwise, note by \eqref{form-331} that $\delta < \Delta^{1 + \epsilon} < w/\mathfrak{T}$, provided that $\Delta > 0$ is sufficiently small in terms of $\epsilon,\mathfrak{T}$. Therefore, the transversality hypothesis implies 
		\begin{displaymath} d((f_u)_{(x_0,y_0),\tau},(f_v)_{(x_0,y_0),\tau}) \geq |f_u'(x_0)-f_v'(x_0)| \geq \frac{w}{2\mathfrak{T}}, \end{displaymath} 
		as desired. \end{proof}
	
	Now that Claim \ref{claim3} has been established, we may finally apply Proposition \ref{pro-highlow-curve} to the weighted incidences $\mathcal{I}_{\omega}(\mathcal{F}_{Q},\mathcal{P}_{Q})$, with the choice $S :=\Delta^{-\epsilon/100}$ (Proposition \ref{pro-highlow-curve} was stated under the assumption that $\mathcal{F}_{Q}$ is $w$-separated, but the same conclusion holds if $\mathcal{F}_{Q}$ is $(w/2\mathfrak{T})$-separated at the cost of increasing the multiplicative constant $\mathbf{C}$ below):
	\begin{align}\label{form-dichotomy} a|\mathcal{P}| \stackrel{\mathrm{C.\,} \ref{claim2}}{\leq} \overline{\mathcal{I}}_{\omega}(\mathbb{S},\mathcal{P}) & = \sum_{Q \in \mathcal{D}_{\tau}(\mathcal{P})} \overline{\mathcal{I}}_{\omega}(\mathbb{S}_Q,\mathcal{P} \cap Q) \notag\\
		&\lesssim \mathbf{C}(S^3w^{-1})^{1/2} \sum_{Q} \Big[ |\mathcal{P} \cap Q|^{1/2} \Big(\sum_{u\in\mathbb{S}_Q} N_{\Delta,b}(u)^2 \Big)^{1/2} \Big] \notag \\
		&\qquad + S^{-1} \sum_{Q} \mathcal{I}_\omega(\mathbb{S}_Q,(\mathcal{P} \cap Q)^{S\delta}):= \textup{(H)}+\textup{(L)}, \end{align}
	where $0 < \mathbf{C} \lesssim_{\mathfrak{T}} \log(1/w) \leq \log(1/\Delta)$; we abbreviate this in the sequel to $\mathbf{C} \lessapprox_{\Delta} 1$. Here we also note 
	\[\mathcal{I}_\omega(\mathbb{S}_Q,(\mathcal{P} \cap Q)^{S\delta}):=\sum_{p\in\mathcal{P}_Q}\sum_{u\in \mathbb{S}_Q} N_{\Delta,b}(u)\mathbf{1}_{\{z_{p^{S\delta}}\in u(6S\delta)\}}.\]
	For the remainder of the proof, we will estimate the high and low terms (H) and (L) separately.
	
	\textbf{High frequency case.} In this case, we assume that (H) $\geq$ (L). By the Cauchy-Schwarz inequality and noting that each $u$ belongs to at most two different $\mathbb{S}_Q$, we infer
	\begin{displaymath}
		\begin{split}
			a|\mathcal{P}|&\lessapprox_{\Delta} (S^3w^{-1})^{1/2}\Big(\sum_Q |\mathcal{P} \cap Q|\Big)^{1/2}\Big(\sum_Q \sum_{u\in \mathbb{S}_Q}N_{\Delta,b}(u)^2\Big)^{1/2}\\
			&\lesssim (S^3w^{-1})^{1/2}|\mathcal{P}|^{1/2}\Big(\sum_{u\in \mathbb{S}}N_{\Delta,b}(u)^2\Big)^{1/2}.
		\end{split}
	\end{displaymath}
	This implies
	\begin{equation}\label{form-332}
		|\mathcal{P}|\lessapprox_{\Delta} a^{-2}\cdot (S^3w^{-1})\cdot \sum_{u\in \mathbb{S}}N_{\Delta,b}(u)^2.
	\end{equation}
	For $\mathbf{F}\in\mathcal{D}_w(\mathcal{F})$, let $\mathbb{S}_\mathbf{F} := \{u = u_{g,Q} \in \mathbb{S} : g \in \mathbf{F}\}$. Then,
	\begin{equation}\label{form-333}
		\sum_{u\in \mathbb{S}}N_{\Delta,b}(u)^2\leq \sum_{\mathbf{F}\in\mathcal{D}_w(\mathcal{F})} \sum_{u\in\mathbb{S}_\mathbf{F}} N_{\Delta,b}(u)^2.
	\end{equation}
	
	\begin{remark}\label{rem2} Recall that $N_{\Delta,b}(u) = |\{\mathbb{F} \in \mathcal{D}_{\Delta}(\mathcal{F}) : |\{u(10\delta) \cap \mathbb{F} \cap \mathcal{F}\}| \geq b\}|$. (We use the notation $\mathbb{F}$ for $\Delta$-cubes and $\mathbf{F}$ for $w$-cubes in high frequency case.) Fix $\mathbf{F} \in \mathcal{D}_{w}(\mathcal{F})$, and let $f_{\mathbf{F}} \in \mathbf{F} \cap \mathcal{F}$ be an arbitrary "centre". Write $B(\mathbf{F}) := B(f_{\mathbf{F}},100\mathfrak{T}w)$. We claim that if $u \in \mathbb{S}_{\mathbf{F}}$, then in fact
		\begin{displaymath} N_{\Delta,b}(u) = |\{\mathbb{F} \in \mathcal{D}_{\Delta}(\mathcal{F}) \cap B(\mathbf{F}) : |u(10\delta) \cap \mathbb{F} \cap \mathcal{F}\}| \geq b\}|, \end{displaymath} 
		where $\mathcal{D}_{\Delta}(\mathcal{F}) \cap B(\mathbf{F}) := \{\mathbb{F} \in \mathcal{D}_{\Delta}(\mathcal{F}) : \mathbb{F} \subset B(\mathbf{F})\}$.
		
		Indeed, let $u = u_{g,Q} \in \mathbb{S}_{\mathbf{F}}$, and let $\mathbb{F} \cap \mathcal{D}_{\Delta}(\mathcal{F})$ satisfy $|\{u(10\delta) \cap \mathbb{F} \cap \mathcal{F}\}| \geq b$. In particular $\mathbb{F}$ contains at least one function $f \in u_{g,Q}(10\delta) \cap \mathcal{F}$, which by Claim \ref{claim1} is contained in $B(g,20\mathfrak{T}w)$. This implies that $\mathbb{F} \subset B(g,50\mathfrak{T}w) \subset B(f_{\mathbf{F}},100\mathfrak{T}w)$, using $g \in \mathbf{F}$. \end{remark}

	To estimate the sum \eqref{form-333}, the plan is to fix $\mathbf{F} \in \mathcal{D}_{w}(\mathcal{F})$, and apply Lemma \ref{lem-auxi1} to the inner sum after a rescaling. For $\mathbf{F} \in \mathcal{D}_{w}(\mathcal{F})$ with centre $f_{\mathbf{F}}$, let $T_{\mathbf{F}}$ be the map defined by 
	\begin{displaymath} T_{\mathbf{F}}(f)=(f-f_{\mathbf{F}})/w, \qquad f\in B(\mathbf{F})\cap\mathcal{F}. \end{displaymath}
	Then $T_{\mathbf{F}}(B(\mathbf{F})) = B_{C^{2}}(100\mathfrak{T})$, and Lemma \ref{lem1} tells us that $T_{\mathbf{F}}(B(\mathbf{F})\cap\mathcal{F})$ is a transversal family with constant $\mathfrak{T}$. For $u = u_{g,Q} \in \mathbb{S}_\mathbf{F}$, we define
	\[T_{\mathbf{F}}(u):=\Big\{(x,\tfrac{y-f_{\mathbf{F}}(x)}{w}):(x,y)\in u\Big\} = \Big\{(x,\tfrac{g(x) - f_{\mathbf{F}}(x)}{w}) : x \in I_{Q} \Big\}. \]
	\begin{remark}\label{rem1} To get the right intuition, note that $\diam(T_{\mathbf{F}}(u)) \lesssim \tau$. Indeed, each $u \in \mathbb{S}_{\mathbf{F}}$ has the form $u = u_{g,Q}$ for some $g \in \mathcal{F}$ with $\|g - f_{\mathbf{F}}\|_{C^{2}} \lesssim w$. In particular $G := [g - f_{\mathbf{F}}]/w$ is $O(1)$-Lipschitz, and $\diam(G(I_{Q})) \lesssim \ell(I_{Q}) = \tau$. \end{remark}
	
	Next we set up some notation required for the eventual application of Lemma \ref{lem-auxi1}. Fix $\mathbf{F} \in \mathcal{D}_{w}(\mathcal{F})$, and set $\mathcal{F}_{\mathbf{F}} := T_{\mathbf{F}}(\mathcal{F} \cap B(\mathbf{F}))$, and 
	\begin{equation}\label{form13} 	 \mathcal{F}_{\Delta/w} := \{T_{\mathbf{F}}(\mathbb{F} \cap \mathcal{F}) : \mathbb{F} \in \mathcal{D}_{\Delta}(\mathcal{F}) \cap B(\mathbf{F})\} \subset \mathcal{F}_{\mathbf{F}}, \end{equation}
	where (as in Remark \ref{rem2}) $\mathcal{D}_{\Delta}(\mathcal{F}) \cap B(\mathbf{F}) = \{\mathbb{F} \in \mathcal{D}_{\Delta}(\mathcal{F}) : \mathbb{F} \subset B(\mathbf{F})\}$. For $p \in \mathcal{D}_{\tau}$, write
	\begin{equation*}\label{aaa} \bar{N}_{\Delta/w,b}(p) := |\{\bar{\mathbb{F}} \in \mathcal{F}_{\Delta/w}  : |11p \cap \bar{\mathbb{F}} \cap \mathcal{F}_{\mathbf{F}}| \geq b\}|. \end{equation*} 
	\begin{claim}\label{claim4} For every $u = u_{g,Q} \in \mathbb{S}_{\mathbf{F}}$, there exists a square $p_{u} \in \mathcal{D}_{\tau}(T_{\mathbf{F}}(u))$ such that 
		\begin{equation}\label{form6} \bar{N}_{\Delta/w,b}(p_{u}) \geq N_{\Delta,b}(u). \end{equation} \end{claim} 
	\begin{remark}\label{rmk-unique} The square $p_{u}$ is essentially unique, since $|T_{\mathbf{F}}(u)|_{\tau} \lesssim 1$ by Remark \ref{rem1}. \end{remark} 
	
	\begin{proof}[Proof of Claim \ref{claim4}] Fix $p_{u} \in \mathcal{D}_{\tau}(T_{\mathbf{F}}(u))$ in such a way that the centre $z_{p_{u}}$ lies at vertical distance $\leq \tau$ from the segment $T_{\mathbf{F}}(u)$. To prove \eqref{form6}, recall from Remark \ref{rem2} that $N_{\Delta,b}(u) = |\{\mathbb{F} \in \mathcal{D}_{\Delta}(\mathcal{F}) \cap B(\mathbf{F}) : |u(10\delta) \cap \mathbb{F} \cap \mathcal{F}| \geq b\}|$. We need to show that if $\mathbb{F} \in \mathcal{D}_{\Delta}(\mathcal{F}) \cap B(\mathbf{F})$ is as in the definition of $N_{\Delta,b}(u)$, then $\bar{\mathbb{F}} = T_{\mathbf{F}}(\mathbb{F} \cap \mathcal{F}) \in \mathcal{F}_{w/\Delta}$ (defined in \eqref{form13}) satisfies
		\begin{displaymath} |11p_u \cap \bar{\mathbb{F}} \cap \mathcal{F}_{\mathbf{F}}| = |\{\bar{f} \in \bar{\mathbb{F}} : z_{p_{u}} \in \Gamma_{\bar{f}}(11\tau)\}| \geq b. \end{displaymath}
		To this end, fix $f \in u(10\delta) \cap \mathbb{F} \cap \mathcal{F}$, and recall $u = u_{g,Q}$. Since there are $\geq b$ such choices of $f$, it remains to show that $\bar{f} := T_{\mathbf{F}}(f)$ satisfies $z_{p_{u}} \in \Gamma_{\bar{f}}(11\tau)$. We first observe that
		\begin{displaymath} |\bar{f}(x) - T_{\mathbf{F}}(g)(x)| = \left|\frac{f(x)-f_{\mathbf{F}}(x)}{w}-\frac{g(x)-f_{\mathbf{F}}(x)}{w}\right|\leq \frac{10\delta}{w}=10\tau,\quad x \in I_{Q}. \end{displaymath}
		In other words $T_{\mathbf{F}}(u) \subset \Gamma_{\bar{f}}(10\tau)$. But since $z_{p_{u}}$ lies at vertical distance $\leq \tau$ from $T_{\mathbf{F}}(u)$, it follows that $z_{p_{u}} \in \Gamma_{\bar{f}}(11\tau)$. \end{proof}
	
	From \eqref{form-333} and Claim \ref{claim4}, we can obtain
	\begin{equation}\label{form21}
		\sum_{u\in \mathbb{S}}N_{\Delta,b}(u)^2\leq \sum_{\mathbf{F}\in\mathcal{D}_w(\mathcal{F})} \sum_{u\in\mathbb{S}_\mathbf{F}} \bar{N}_{\Delta/w,b}(p_u)^2.
	\end{equation}
	We next claim that
	\begin{equation}\label{form7} |\{u\in\mathbb{S}_\mathbf{F}: p_u=p\}| \lesssim_{\mathfrak{T}} 1, \qquad \mathbf{F} \in \mathcal{D}_{w}(\mathcal{F}), \, p \in \mathcal{D}_{\tau}. \end{equation}
	To see this, we count the number of $u\in\mathbb{S}_\mathbf{F}$ such that $z_p\in T_\mathbf{F}(u)(\tau)$. Take any two segments $u, v\in\mathbb{S}_\mathbf{F}$. If $d(f_u,f_v)\leq w$, by Claim \ref{eq-verticalbound} the vertical distance $d_V(u,v)>\delta$, thus $d_V(T_\mathbf{F}(u),T_\mathbf{F}(v))>\delta/w=\tau$.  This implies that in a fixed $w$-ball, there are $O(1)$ functions $h$ such that $z_p\in T_\mathbf{F}(u_{h,Q})(\tau)$ (for some $Q\in\mathcal{D}_\tau$). In other words,
	\begin{displaymath} |\{u \in \mathbb{S}_{\mathbf{F}} : p_{u} = p\}| \lesssim |\{f_{u} \in \mathcal{F} : u \in \mathbb{S}_{\mathbf{F}}\}|_{w}. \end{displaymath}
	But $u \in \mathbb{S}_{\mathbf{F}}$ implies $f_{u} \in \mathbf{F} \in \mathcal{D}_{w}(\mathcal{F})$, and evidently $|\mathbf{F}|_{w} \lesssim_{\mathfrak{T}} 1$. This proves \eqref{form7}.
	
	With \eqref{form7} in hand, \eqref{form21} can be further estimated by
	\begin{equation}\label{form22}
		\sum_{u\in \mathbb{S}}N_{\Delta,b}(u)^2\lesssim_\mathfrak{T} \sum_{\mathbf{F}\in\mathcal{D}_w(\mathcal{F})} \sum_{\substack{p\in \mathcal{D}_\tau\\ N_{\Delta/w,b}(p)\geq 2}} \bar{N}_{\Delta/w,b}(p)^2.
	\end{equation}
	
	We are now in a position to apply Lemma \ref{lem-auxi1} (at scales $\tau = \delta/w$ and $\Delta/w$ in place of $\delta$ and $\Delta$) to the inner sum in \eqref{form22}. Let us discuss the hypotheses of Lemma \ref{lem-auxi1}. 
	\begin{enumerate}
		\item The family $\mathcal{F}_{\mathbf{F}} = T_\mathbf{F}(\mathcal{F}\cap B(\mathbf{F})) \subset B_{C^{2}}(100\mathfrak{T})$ is transversal with constant $\mathfrak{T}$. 
		\item Distinct elements $\mathbb{F}_{1},\mathbb{F}_{2} \in \mathcal{D}_{\Delta}(\mathcal{F})$ are $\Delta$-separated by hypothesis. Therefore
		\begin{displaymath} \mathcal{F}_{\Delta/w} = \{T_{\mathbf{F}}(\mathbb{F} \cap \mathcal{F}) : \mathbb{F} \in \mathcal{D}_{\Delta}(\mathcal{F}) \cap B(\mathbf{F})\} \subset \mathcal{F}_{\mathbf{F}} \end{displaymath}
		consists of $\Delta/w$-separated elements.
		\item by \nref{C2}, $\mathcal{F}_{\mathbf{F}}$ satisfies the following spacing condition with $K := \Delta^{-3\epsilon}$:
		\[|\mathcal{F}_{\mathbf{F}}\cap B(f,r)|\leq Kr|\mathcal{F}_{\mathbf{F}}|,\quad f\in T_\mathbf{F} (\mathcal{F}), \quad r\in [\Delta/w,1].\]
	\end{enumerate}
	To be accurate, in \nref{C2} condition (3) was actually verified for $T_\mathbf{F}(\mathcal{F} \cap \mathbf{F})$ with constant $\Delta^{-2\epsilon}$. This gives (3) for $\Delta > 0$ small enough, since $|\mathcal{F} \cap B(\mathbf{F})|_{w} \lesssim \mathfrak{T}^{4}$.
	
	We now apply Lemma \ref{lem-auxi1} to the inner sum in \eqref{form22}:
	\begin{equation}\label{form-334}
		\begin{split}
			\sum_{u\in \mathbb{S}}N_{\Delta,b}(u)^2&\lesssim_\mathfrak{T} K\log \tfrac{\mathfrak{T}}{\Delta}\sum_{\mathbf{F}\in\mathcal{D}_w(\mathcal{F})} \frac{|\mathcal{F}_{\mathbf{F}}|^2}{b^2}\lesssim_\mathfrak{T} \sum_{\mathbf{F}\in\mathcal{D}_w(\mathcal{F})} \Delta^{-4\epsilon}\frac{|\mathcal{F}\cap \mathbf{F}|^2}{b^2}\\
			&\lesssim \Delta^{-4\epsilon}\frac{|\mathcal{F}|^{2}}{b^{2}|\mathcal{D}_w(\mathcal{F})|},
		\end{split}
	\end{equation}
	where we used the upper $(2,20\mathfrak{T}^{2})$-regularity of $\mathcal{F}$ in the second inequality, and the uniformity of $\mathcal{F}$ in the third inequality. Substituting \eqref{form-334} into \eqref{form-332}, we deduce (recall $S=\Delta^{-\epsilon/100}$)
	\begin{displaymath}|\mathcal{P}| \lessapprox_{\Delta,\mathfrak{T}} a^{-2}(S^3w^{-1})\Delta^{-4\epsilon}\frac{|\mathcal{F}|^{2}}{b^{2}|\mathcal{D}_w(\mathcal{F})|} = \Delta^{-3\epsilon/100-4\epsilon}\frac{a}{w|\mathcal{D}_w(\mathcal{F})|}\cdot \frac{|\mathcal{F}|^2}{a^3b^2}. \end{displaymath}
	Recalling from \eqref{form-226} that $a/(w|\mathcal{D}_{w}(\mathcal{F})|) \leq \Delta^{-4\epsilon}$, we infer that $|\mathcal{P}| \leq \Delta^{-10\epsilon}|\mathcal{F}|^{2}/(a^{3}b^{2})$ for $\Delta > 0$ sufficiently small, depending only on $\epsilon,\mathfrak{T}$.
	
	\textbf{Low frequency case.} In this case, (H) $<$ (L) in \eqref{form-dichotomy}, hence
	\[\begin{split}
		a|\mathcal{P}|&\lesssim S^{-1}\sum_{Q \in \mathcal{D}_{\tau}(\mathcal{P})} \mathcal{I}_{\omega}(\mathbb{S}_Q,(\mathcal{P} \cap Q)^{S\delta}) \stackrel{\mathrm{def.}}{=} S^{-1}\sum_Q \sum_{p\in\mathcal{P}\cap Q}\sum_{\substack{u\in \mathbb{S}_Q\\ z_{p^{S\delta}}\in u(6S\delta)}} N_{\Delta,b}(u)\\
		&\leq S^{-1}\sum_{p\in\mathcal{P}}\sum_{\substack{u\in \mathbb{S}\\ z_{p^{S\delta}}\in u(6S\delta)}} N_{\Delta,b}(u) =: S^{-1}\sum_{p\in\mathcal{P}}m(p).
	\end{split}\]
	Hence, there exists $\mathcal{P}' \subset \mathcal{P}$ with $|\mathcal{P}'|\geq \tfrac{1}{2}|\mathcal{P}|$ such that 
	\begin{equation}\label{form11}
		m(p)=\sum_{\substack{u\in \mathbb{S}\\ z_{p^{S\delta}}\in u(6S\delta)}} N_{\Delta,b}(u)\gtrsim Sa,\qquad p\in\mathcal{P}'.
	\end{equation}
	For each $\mathbf{F}\in\mathcal{D}_\Delta(\mathcal{F})$ and $p \in \mathcal{P}$, let 
	\begin{displaymath} \mathbb{S}_{\mathbf{F}}(p) := \{u \in \mathbb{S} : z_{p^{S\delta}}\in u(6S\delta) \text{ and } |u(10\delta)\cap \mathbf{F}\cap\mathcal{F}|\geq b\}. \end{displaymath}
	For fixed $f\in\mathbf{F}\cap\mathcal{F}$ and $p \in \mathcal{P}$, we claim the following bounded overlap property:
	\begin{equation}\label{equ-boverlaping}
		|\{u\in\mathbb S: z_{p^{S\delta}}\in u(6S\delta), f\in u(10\delta)\cap \mathbf{F}\cap\mathcal{F}\}| \lesssim \mathfrak{T}^4.
	\end{equation}
	To see this, let $u, v \in \mathbb{S}$ such that $z_{p^{S\delta}}\in u(6S\delta)\cap v(6S\delta)$ and $f\in u(10\delta)\cap v(10\delta)\cap\mathbf{F}\cap\mathcal{F}$. By Claim \ref{claim1}, we know $d(f,f_u), d(f,f_v)\leq 16\mathfrak{T}w$, hence $d(f_u,f_v)\leq 32\mathfrak{T}w$. Also, since $z_{p^{S\delta}}\in u(6S\delta)$ for any $u\in \{u\in\mathbb S: z_{p^{S\delta}}\in u(6S\delta), f\in u(8\delta)\cap \mathbf{F}\cap\mathcal{F}\}$, the $x$-projections of the segments in this set coincide with a common $\tau$-interval determined by $p$. Combining these facts with Claim \ref{eq-verticalbound} and using the upper $2$-regularity of $\mathcal{F}$, we get \eqref{equ-boverlaping}.
	
	We moreover note that if $u \in \mathbb{S}$ with $z_{p^{S\delta}} \in u(6S\delta)$ and $f \in u(10\delta) \cap \mathbf{F} \cap \mathcal{F}$, then $f \in 8p^{S\delta} \cap \mathbf{F} \cap \mathcal{F}$ by the triangle inequality (and since $10 \ll S$). Using \eqref{equ-boverlaping}, this implies
	\begin{equation}\label{form-338} |8p^{S\delta} \cap \mathbf{F} \cap \mathcal{F}| \gtrsim \mathfrak{T}^{-4} \sum_{u \in \mathbb{S}_{\mathbf{F}}(p)} |u(10\delta) \cap \mathbf{F} \cap \mathcal{F}| \geq \mathfrak{T}^{-4}|\mathbb{S}_{\mathbf{F}}(p)| \cdot b, \quad p \in \mathcal{P}'. \end{equation}
	\begin{remark} We also note that if $\mathbf{F} \in \mathcal{D}_{\Delta}(\mathcal{F})$ with $|\mathbb{S}_{\mathbf{F}}(p)| \geq 1$, then $|8p^{S\delta} \cap \mathbf{F} \cap \mathcal{F}| \geq b \geq 1$ for $p \in \mathcal{P}'$. Indeed one just picks any $u \in \mathbb{S}_{\mathbf{F}}(p)$ and uses $|8p^{S\delta} \cap \mathbf{F} \cap \mathcal{F}| \geq |u(10\delta) \cap \mathbf{F} \cap \mathcal{F}|$. \end{remark} 	
	We next claim that 
	\begin{equation}\label{form12} |\mathbb{S}_{\mathbf{F}}(p)|\lesssim_\mathfrak{T} S, \qquad \mathbf{F} \in \mathcal{D}_{\Delta}(\mathcal{F}). \end{equation}
	Indeed, let $u_1, u_2\in \mathbb{S}_{\mathbf{F}}(p)$, and pick $f_i\in u_i(10\delta)\cap\mathbf{F}\cap \mathcal{F}$ ($i=1,2$). Claim \ref{claim1} implies $d(f_{u_i},f_i)\lesssim \mathfrak{T}w$. Since $d(f_1,f_2)\leq \diam(\mathbf{F}) \lesssim \Delta \leq w$, we get $d(f_{u_1},f_{u_2})\lesssim\mathfrak{T}w$. This shows that $\diam \{f_u \in \mathcal{F} : u\in \mathbb{S}_{\mathbf{F}}(p)\} \lesssim \mathfrak{T}w$, and therefore $\{f_{u} \in \mathcal{F} : u \in \mathbb{S}_{\mathbf{F}}(p)\}$ can be covered by $n \lesssim \mathfrak{T}^{4}$ balls $B_{1},\ldots,B_{n}$ of radius $w$. Now it suffices to prove that $|\{f_{u} \in \mathcal{F} : u \in \mathbb{S}_{\mathbf{F}}(p)\} \cap B_{i}| \lesssim S$ for $1 \leq i \leq n$. This follows by recalling that $z_{p^{S\delta}} \in u(6S\delta)$ for all $u \in \mathbb{S}_{\mathbf{F}}(p)$, and using Claim \ref{eq-verticalbound}.
	
	Next, recall from \eqref{form11} that $m(p) \gtrsim Sa$ for all $p \in \mathcal{P}'$, and therefore
	\begin{displaymath} \sum_{\mathbf{F}\in\mathcal{D}_\Delta(\mathcal{F})} |\mathbb{S}_{\mathbf{F}}(p)| = \sum_{\substack{u\in \mathbb{S}\\z_{p^{S\delta}}\in u(6S\delta)}} |\{\mathbf{F}\in\mathcal{D}_\Delta(\mathcal{F}):|u(10\delta)\cap \mathbf{F}\cap\mathcal{F}\}|\geq b\}| = m(p) \gtrsim Sa. \end{displaymath}
	From this, and the upper bound $|\mathbb{S}_{\mathbf{F}}(p)| \lesssim_{\mathfrak{T}} S$ from \eqref{form12}, we can pigeonhole (for each $p \in \mathcal{P}'$) a dyadic number $1 \leq k_p\lesssim_\mathfrak{T} S$ and a subset $\mathcal{B}_p\subset\mathcal{D}_\Delta(\mathcal{F})$ such that
	\begin{equation}\label{form-337}
		|\mathcal{B}_p|\gtrsim_{\mathfrak{T}} \frac{Sa}{k_p \log S} \quad \text{and} \quad k_{p} \leq |\mathbb{S}_{\mathbf{F}}(p)| < 2k_p~\text{for any}~\mathbf{F}\in \mathcal{B}_p.
	\end{equation}
	We can then pigeonhole a further dyadic number $\mathfrak{K} \lesssim_\mathfrak{T} S$ and a subset $\mathcal{P}''\subset \mathcal{P}'$ with $|\mathcal{P}''|\gtrsim_{\mathfrak{T}} (\log S)^{-1}|\mathcal{P}'|$ such that $k_p = \mathfrak{K}$ for $p\in \mathcal{P}''$. Now we define $\mathcal{P}^\ast$ as the set of distinct elements in $\{p^{8S\delta}: p\in\mathcal{P}''\}$ (removing repeated cubes). Clearly 
	\begin{equation}\label{form15} |\mathcal{P}^\ast|\gtrsim S^{-2} |\mathcal{P}''| \gtrsim S^{-2}(\log S)^{-1}|\mathcal{P}|. \end{equation}
	\begin{claim}\label{claim5} There exists a constant $c = c_{\mathfrak{T}} > 0$ such that the following holds with constant $b^{\ast} := \max\{c\mathfrak{K}b,1\}$. If $p^{\ast} \in \mathcal{P}^{\ast}$, then 
		\begin{displaymath} N_{\Delta,b^{\ast}}(p^{\ast}) = |\{\mathbf{F} \in \mathcal{D}_{\Delta}(\mathcal{F}) : |3p^{\ast} \cap \mathbf{F} \cap \mathcal{F}| \geq b^{\ast}\}| \geq |\mathcal{B}_{p^{\ast}}| \gtrsim_{\mathfrak{T}} (\log S)^{-1}Sa/\mathfrak{K}. \end{displaymath} 
	\end{claim}
	
	\begin{proof} Fix $p^{\ast} = p^{8S\delta} \in \mathcal{P}^{\ast}$, where $p \in \mathcal{P}'' \subset \mathcal{P}$, and let $\mathbf{F} \in \mathcal{B}_{p^{\ast}}$. Then \eqref{form-338} (and the remark below it) implies $|8p^{S\delta} \cap \mathbf{F} \cap \mathcal{F}| \geq \max\{c|\mathbb{S}_{\mathbf{F}}(p)| b,1\} \geq b^{\ast}$. Spelling this out,
		\begin{displaymath} |\{f \in \mathbf{F} \cap \mathcal{F} : z_{p^{S\delta}} \in \Gamma_{f}(8S\delta)\}| \geq b^{\ast}. \end{displaymath}
		Writing $z_{p^\ast}=(x^{\ast},y^{\ast})$ and $z_{p^{S\delta}} = (x^{S\delta},y^{S\delta})$, note that $|x^{\ast} - x^{S\delta}| \leq 8S\delta$. Therefore, whenever $z_{p^{S\delta}} \in \Gamma_{f}(8S\delta)$, 
		\begin{align*} |y^{\ast} - f(x^{\ast})| & \leq |y^{\ast} - y^{S\delta}| + |f(x^{\ast}) - f(x^{S\delta})| + |y^{S\delta} - f(x^{S\delta})| \leq 3 \cdot 8S\delta, \end{align*}
		or in other words $z_{p^{\ast}} \in \Gamma_{f}(3 \cdot 8S\delta)$. This shows that
		\begin{displaymath} |3p^{\ast} \cap \mathbf{F} \cap \mathcal{F}| \geq |8p^{S\delta} \cap \mathbf{F} \cap \mathcal{F}| \geq b^{\ast}, \qquad \mathbf{F} \in \mathcal{B}_{p^{\ast}}, \end{displaymath}
		and so the claim follows. 	 \end{proof} 
	
	Write $\delta^{\ast} := 8S\delta$, so $\mathcal{P}^{\ast} \subset \mathcal{D}_{\delta^{\ast}}$. Our goal is now to apply the induction hypothesis (namely: Lemma \ref{lem-main1} at scale $\delta^{\ast}$) to $\mathcal{F},\mathcal{P}^\ast$, so we need to verify conditions \nref{i} and \nref{ii} of Lemma \ref{lem-main1}. Condition \nref{i} asks us to show that, for $p^{\ast} \in \mathcal{P}^{\ast}$,
	\begin{equation}\label{form14} |6(p^{\ast})^{\rho} \cap \mathcal{D}_{\rho}| \leq \Delta^{-\epsilon} \cdot \rho|\mathcal{D}_{\rho}(\mathcal{F})| \text{ for all } \rho \in \{\Delta^{\epsilon},\Delta^{2\epsilon},\ldots,\Delta\}. \end{equation} 
	Recall that $S=\Delta^{-\epsilon/100}$ and $\delta<\Delta^{1+\epsilon}$, thus $\delta^{\ast} = 8S\delta \leq \Delta$. Consequently, if $\rho \in \{\Delta^{\epsilon},\Delta^{2\epsilon},\ldots,\Delta\}$, the $\rho$-parent of each $p^{\ast} \in \mathcal{P}^{\ast}$ coincides with the $\rho$-parent of some $p \in \mathcal{P}'' \subset \mathcal{P}$. So, in fact \eqref{form14} follows immediately from hypothesis \nref{i} for $\mathcal{P}$.
	
	Regarding condition \nref{ii}, we have already shown in Claim \ref{claim5} that 
	\begin{displaymath} N_{\Delta,b^{\ast}}(p^{\ast}) \gtrsim_{\mathfrak{T}} a^{\ast} := \frac{Sa}{\mathfrak{K}\log S}, \qquad p^{\ast} \in \mathcal{P}^{\ast}. \end{displaymath}
	Recalling from \eqref{form-331} that $a > \Delta^{-3\epsilon}$, and that $\mathfrak{K} \lesssim_{\mathfrak{T}} S$, we have $a^{\ast} \geq 2$. Furthermore, by the initial hypothesis $ab\geq \delta^{1-2\epsilon} |\mathcal{F}|$, and since $S = \Delta^{-\epsilon/100}$, we also have
	\[a^\ast b^\ast=\frac{Sa}{\mathfrak{K}\log S}\cdot \max\{c\mathfrak{K}b,1\} \geq cS(\log S)^{-1}ab>(8S\delta)^{1-2\epsilon}|\mathcal{F}| = (\delta^{\ast})^{1 - 2\epsilon}|\mathcal{F}|,\]
	again provided that $\Delta > 0$ is sufficiently small in terms of $\epsilon,\mathfrak{T}$.
	
	We are ready to apply the inductive assumption to $\mathcal{F},\mathcal{P}^\ast$. This gives
	\[\begin{split}
		|\mathcal{P}|&\stackrel{\eqref{form15}}{\lesssim} S^2(\log S)|\mathcal{P}^\ast|\lesssim_{\mathfrak{T}} S^2(\log S)\cdot (8S\delta)^{-10\epsilon}\frac{|\mathcal{F}|^2}{{a^\ast}^3{b^\ast}^2}\\
		&\leq S^{2}(\log S)\frac{(\log S)^3\mathfrak{K}^{3}}{S^3(c\mathfrak{K})^{2}}\cdot (S\delta)^{-10\epsilon}\frac{|\mathcal{F}|^2}{a^3b^2}\lesssim_{\mathfrak{T}} S^{-10\epsilon}(\log S)^{4}\delta^{-10\epsilon}\frac{|\mathcal{F}|^2}{a^3b^2}.
	\end{split}\]
	Here $S^{-10\epsilon} = \Delta^{\epsilon/10}$, so we have shown that $|\mathcal{P}| \leq \delta^{-10\epsilon}|\mathcal{F}|^{2}/(a^{3}b^{2})$ if $\Delta > 0$ is sufficiently small in terms of $\epsilon,\mathfrak{T}$.	 This closes the induction and completes the proof. \end{proof}

\bibliographystyle{plain}
\bibliography{references}
		
\end{document}